\documentclass[english,11pt]{amsart}
\usepackage{amssymb,amsmath,amsthm,mathtools,amsfonts}

\usepackage{graphicx}
\usepackage[usenames,dvipsnames]{xcolor}
\usepackage{cite}
\usepackage{url}
\usepackage{hyperref}
\usepackage{xcolor}
\usepackage[all]{xy}
\usepackage[version=4]{mhchem}

\hypersetup{
    bookmarks=true,         	
    unicode=false,          		
    pdftoolbar=true,        		
    pdfmenubar=true,        	
    pdffitwindow=false,     		
    pdfstartview={FitH},    		
   pdftitle={long-time dynamics of the sine-Gordon equation },    					
   pdfauthor={Gong Chen, Jiaqi Liu, Bingying Lu},     	
    linktocpage=true,			
    pdfnewwindow=true,      	
    colorlinks=true,       			
    linkcolor=red,          		
    citecolor=PineGreen,    	
    filecolor=magenta,      		
    urlcolor=cyan           		
}

\usepackage{tikz}
\usetikzlibrary{decorations.pathmorphing}
\usetikzlibrary{arrows,decorations.pathmorphing}

\usepackage{float}
\usepackage[section]{placeins}		

\theoremstyle{plain}
\newtheorem{theorem}{Theorem}[section]
\newtheorem{corollary}[theorem]{Corollary}
\newtheorem{lemma}[theorem]{Lemma}
\newtheorem{proposition}[theorem]{Proposition}

\theoremstyle{definition}
\newtheorem{definition}[theorem]{Definition}
\newtheorem{problem}[theorem]{Problem}

\theoremstyle{remark}
\newtheorem{remark}[theorem]{Remark}

\numberwithin{figure}{section}
\numberwithin{equation}{section}

\DeclareMathOperator{\ad}{ad}

\DeclareMathOperator{\Real}{Re}
\DeclareMathOperator{\Imag}{Im}

\DeclareMathOperator{\Res}{Res}

\allowdisplaybreaks

\newenvironment{doublecases}
{
	\left\{ 
			\begin{array}{lllll}
}
{			
			\end{array} 
			\right.
}
			
\begin{document}

\title [sine-Gordon equation]{Long-time asymptotics and stability for the sine-Gordon equation }
\author{Gong Chen}
\author{Jiaqi Liu}
\author{Bingying Lu}
\address[Chen]{Fields Institute for Research in Mathematical Sciences, 222 College Street Toronto, Ontario M5S 2E4, Canada. Department of Mathematics, University of Toronto, Toronto, Ontario M5S 2E4, Canada.}
\email{gc@math.toronto.edu}
\address[Liu]{School of mathematics, University of Chinese Academy of Sciences. No.19 Yuquan Road, Beijing China }
\email{jqliu@ucas.ac.cn}
\address[Lu]{Institute of Mathematics, Academia Sinica and informatik und Mathematik Bibliothekstrae 5, 28359, Bremen, Germany }
\email{lubi@uni-bremen.de}

\date{\today}

\begin{abstract}

	
	In this paper, we study the long-time dynamics of the sine-Gordon equation
	\[
	\partial_{tt}f-\partial_{xx}f+\sin f=0,\ \left(x,t\right)\in\mathbb{R}\times\mathbb{R}^{+}.
	\]
	{Firstly, we use the nonlinear steepest descent for Riemann-Hilbert problems to compute the long-time asymptotics of the solutions to
	the sine-Gordon equation whose initial condition belongs to some \textit{weighted Sobolev spaces}.} Secondly, combining
	the long-time asymptotics with a refined approximation argument, we
	analyze the asymptotic stability of multi-soliton solutions to the sine-Gordon equation in weighted energy spaces.  It is known that the obstruction
	to the asymptotic stability of kink solutions to the sine-Gordon equation in the energy
	space is the existence of small breathers which is also closely related
	to the emergence of wobbling kinks. Our stability analysis gives
	a criterion for the weight which is sharp up to the endpoint so that the asymptotic stability holds.

\end{abstract}

\medskip

\maketitle
\tableofcontents

%
%

\newcommand{\eps}{\varepsilon}
\newcommand{\lam}{\lambda}

\newcommand{\bfN}{\mathbf{N}}
\newcommand{\calbR}{\mathcal{ \breve{R}}}
\newcommand{\rhobar}{\overline{\rho}}
\newcommand{\zetabar}{\overline{\zeta}}

\newcommand{\rarr}{\rightarrow}
\newcommand{\darr}{\downarrow}

\newcommand{\dee}{\partial}
\newcommand{\dbar}{\overline{\partial}}

\newcommand{\dint}{\displaystyle{\int}}

\newcommand{\dotarg}{\, \cdot \, }

%
%

\newcommand{\RHP}{\mathrm{LC}}			
\newcommand{\PC}{\mathrm{PC}}
\newcommand{\w}{w^{(2)}}
%
%

\newcommand{\zbar}{\overline{z}}

\newcommand{\bbC}{\mathbb{C}}
\newcommand{\bbR}{\mathbb{R}}

\newcommand{\calB}{\mathcal{B}}
\newcommand{\calC}{\mathcal{C}}
\newcommand{\calR}{\mathcal{R}}
\newcommand{\calS}{\mathcal{S}}
\newcommand{\calZ}{\mathcal{Z}}
\newcommand{\tgamma}{\tilde{\gamma}}

\newcommand{\ba}{\breve{a}}
\newcommand{\bb}{\breve{b}}
\newcommand{\bchi}{\breve{\chi}}

\newcommand{\balpha}{\breve{\alpha}}
\newcommand{\brho}{\breve{\rho}}

\newcommand{\tPhi}{{\widetilde{\Phi}}}

\newcommand{\tp}{\text{p}}
\newcommand{\tq}{\text{q}}
\newcommand{\tr}{\text{r}}

\newcommand{\bfe}{\mathbf{e}}
\newcommand{\bfn}{\mathbf{n}}

\newcommand{\tA}{\tilde{A}}
\newcommand{\tB}{\tilde{B}}
\newcommand{\tomega}{\tilde{\omega}}
\newcommand{\tc}{\tilde{c}}

\newcommand{\mhat}{\hat{m}}

\newcommand{\bphi}{\breve{\Phi}}
\newcommand{\bN}{\breve{N}}
\newcommand{\bV}{\breve{V}}
\newcommand{\bR}{\breve{R}}
\newcommand{\bdelta}{\breve{\delta}}
\newcommand{\bzeta}{\breve{\zeta}}
\newcommand{\bbeta}{\breve{\beta}}
\newcommand{\bm}{\breve{m}}
\newcommand{\br}{\breve{r}}
\newcommand{\bnu}{\breve{\nu}}
\newcommand{\bbfN}{\breve{\mathbf{N}}}
\newcommand{\rbar}{\overline{r}}

\newcommand{\One}{\mathbf{1}}

%
%

\newcommand{\bigO}[2][ ]
{
\mathcal{O}_{#1}
\left(
{#2}
\right)
}

\newcommand{\littleO}[1]{{o}\left( {#1} \right)}

\newcommand{\norm}[2]
{
\left\Vert		{#1}	\right\Vert_{#2}
}

%
%

\newcommand{\rowvec}[2]
{
\left(
	\begin{array}{cc}
		{#1}	&	{#2}	
	\end{array}
\right)
}

\newcommand{\uppermat}[1]
{
\left(
	\begin{array}{cc}
	0		&	{#1}	\\
	0		&	0
	\end{array}
\right)
}

\newcommand{\lowermat}[1]
{
\left(
	\begin{array}{cc}
	0		&	0	\\
	{#1}	&	0
	\end{array}
\right)
}

\newcommand{\offdiagmat}[2]
{
\left(
	\begin{array}{cc}
	0		&	{#1}	\\
	{#2}	&	0
	\end{array}
\right)
}

\newcommand{\diagmat}[2]
{
\left(
	\begin{array}{cc}
		{#1}	&	0	\\
		0		&	{#2}
		\end{array}
\right)
}

\newcommand{\Offdiagmat}[2]
{
\left(
	\begin{array}{cc}
		0			&		{#1} 	\\
		\\
		{#2}		&		0
		\end{array}
\right)
}

\newcommand{\twomat}[4]
{
\left(
	\begin{array}{cc}
		{#1}	&	{#2}	\\
		{#3}	&	{#4}
		\end{array}
\right)
}

\newcommand{\unitupper}[1]
{	
	\twomat{1}{#1}{0}{1}
}

\newcommand{\unitlower}[1]
{
	\twomat{1}{0}{#1}{1}
}

\newcommand{\Twomat}[4]
{
\left(
	\begin{array}{cc}
		{#1}	&	{#2}	\\[10pt]
		{#3}	&	{#4}
		\end{array}
\right)
}

%
%
%

\newcommand{\JumpMatrixFactors}[6]
{
	\begin{equation}
	\label{#2}
	{#1} =	\begin{cases}
					{#3} {#4}, 	&	z \in (-\infty,\xi) \\
					\\
					{#5}{#6},	&	z \in (\xi,\infty)
				\end{cases}
	\end{equation}
}


%
%
%

\newcommand{\RMatrix}[9]
{
\begin{equation}
\label{#1}
\begin{aligned}
\left. R_1 \right|_{(\xi,\infty)} 	&= {#2} &	\qquad\qquad		
\left. R_1 \right|_{\Sigma_1}		&= {#3} 
\\[5pt]
\left. R_3 \right|_{(-\infty,\xi)} 	&= {#4} 	&	
\left. R_3 \right|_{\Sigma_2} 	&= {#5}
\\[5pt]
\left. R_4 \right|_{(-\infty,\xi)} 	&= {#6} &	
\left. R_4 \right|_{\Sigma_3} 	&= {#7} 
\\[5pt]
\left. R_6 \right|_{(\xi,\infty)}  	&= {#8} &	
\left. R_6 \right|_{\Sigma_4} 	&= {#9}
\end{aligned}
\end{equation}
}

%
%

%
%
%
%
%
%

\newcommand{\SixMatrix}[6]
{
\begin{figure}
\centering
\caption{#1}
\vskip 15pt
\begin{tikzpicture}
[scale=0.75]
%
%
\draw[thick]	 (-4,0) -- (4,0);
\draw[thick] 	(-4,4) -- (4,-4);
\draw[thick] 	(-4,-4) -- (4,4);
%
%
\draw	[fill]		(0,0)						circle[radius=0.075];
\node[below] at (0,-0.1) 				{$z_0$};
%
%
\node[above] at (3.5,2.5)				{$\Omega_3$};
\node[below]  at (3.5,-2.5)			{$\Omega_7$};
\node[above] at (0,3.25)				{$\Omega_1$};
\node[below] at (0,-3.25)				{$\Omega_2$};
\node[above] at (-3.5,2.5)			{$\Omega_8$};
\node[below] at (-3.5,-2.5)			{$\Omega_4$};
%
%
\node[above] at (0,1.25)				{$\twomat{1}{0}{0}{1}$};
\node[below] at (0,-1.25)				{$\twomat{1}{0}{0}{1}$};
%
%
\node[right] at (1.20,0.70)			{$#3$};
\node[left]   at (-1.20,0.70)			{$#4$};
\node[left]   at (-1.20,-0.70)			{$#5$};
\node[right] at (1.20,-0.70)			{$#6$};
\end{tikzpicture}
\label{#2}
\end{figure}
}

\newcommand{\sixmatrix}[6]
{
\begin{figure}
\centering
\caption{#1}
\vskip 15pt
\begin{tikzpicture}
[scale=0.75]
%
%
\draw[thick]	 (-4,0) -- (4,0);
\draw[thick] 	(-4,4) -- (4,-4);
\draw[thick] 	(-4,-4) -- (4,4);
%
%
\draw	[fill]		(0,0)						circle[radius=0.075];
\node[below] at (0,-0.1) 				{$-z_0$};
%
%
\node[above] at (3.5,2.5)				{$\Omega_9$};
\node[below]  at (3.5,-2.5)			{$\Omega_5$};
\node[above] at (0,3.25)				{$\Omega_1$};
\node[below] at (0,-3.25)				{$\Omega_2$};
\node[above] at (-3.5,2.5)			{$\Omega_6$};
\node[below] at (-3.5,-2.5)			{$\Omega_{10}$};
%
%
\node[above] at (0,1.25)				{$\twomat{1}{0}{0}{1}$};
\node[below] at (0,-1.25)				{$\twomat{1}{0}{0}{1}$};
%
%
\node[right] at (1.20,0.70)			{$#3$};
\node[left]   at (-1.20,0.70)			{$#4$};
\node[left]   at (-1.20,-0.70)			{$#5$};
\node[right] at (1.20,-0.70)			{$#6$};
\end{tikzpicture}
\label{#2}
\end{figure}
}

%
%
%
%

%
%
%
%

\newcommand{\JumpMatrixRightCut}[6]
{
\begin{figure}
\centering
\caption{#1}
\vskip 15pt
\begin{tikzpicture}[scale=0.85]
%
%
\draw [fill] (4,4) circle [radius=0.075];						
\node at (4.0,3.65) {$\xi$};										
%
%
\draw 	[->, thick]  	(4,4) -- (5,5) ;								
\draw		[thick] 		(5,5) -- (6,6) ;
\draw		[->, thick] 	(2,6) -- (3,5) ;								
\draw		[thick]		(3,5) -- (4,4);	
\draw		[->, thick]	(2,2) -- (3,3);								
\draw		[thick]		(3,3) -- (4,4);
\draw		[->,thick]	(4,4) -- (5,3);								
\draw		[thick]  		(5,3) -- (6,2);
%
%
\draw [  thick, blue, decorate, decoration={snake,amplitude=0.5mm}] (4,4)  -- (8,4);				
\node at (1.5,4) {$0 < \arg (\zeta-\xi) < 2\pi$};
%
%
\node at (8.5,8.5)  	{$\Sigma_1$};
\node at (-0.5,8.5) 	{$\Sigma_2$};
\node at (-0.5,-0.5)	{$\Sigma_3$};
\node at (8.5,-0.5) 	{$\Sigma_4$};
%
%
\node at (7,7) {${#3}$};						
\node at (1,7) {${#4}$};						
\node at (1,1) {${#5}$};						
\node at (7,1) {${#6}$};						
\end{tikzpicture}
\label{#2}
\end{figure}
}

%
%

\section{Introduction}

%
%
\subsection{General introduction}
In this paper, we study the long-time asymptotics and  stability properties 
of soliton for the $1+1$ dimensional sine-Gordon (sG) equation,
\begin{align}
\label{eq: sG}
 \partial_{tt}f-\partial_{xx}f+\sin f &=0,\ \left(x,t\right)\in\mathbb{R}\times\mathbb{R}^{+}\\
\label{data}
 f(x,0) &=f_0(x), \\
\label{data-t}
f_t(x,0) &=f_1(x),
\end{align}
in laboratory coordinates $\left(x,t\right)\in\mathbb{R}^{2}$ where
$f$ is a real-valued function.
 
\smallskip

The sine-Gordon equation has been widely studied because of its ubiquitous
appearance in different fields. First proposed in differential geometry, it has been shown to have
close connection to relativistic field theory, and garnered further attention in 1970’s due to it
being discovered as an ``integrable system''. On the applications side, to mention but a few, it arises as the mechanical model for coupled pendula \cite{BEMS} and appears
in the theory of crystal dislocations \cite{FK} and 
superconducting Josephson junctions  \cite{SCR76}, 
vibrations of DNA molecules in \cite{Y04}, 
and quantum field theory in \cite{Coleman1975}. 
We refer the interested reader to monographs by Lamb \cite{Lamb},
Dauxois and Peyrard \cite{DP}, {Cuevas-Maraver}, Kevrekidis and Williams
\cite{CKW}. In particular, the sine-Gordon equation is a mechanical
model of a continuum of pendula parametrized by the $x$ variable that interact
elastically with each other. In this setting, $f\left(x,t\right)$
denotes the angle of the pendulum located at point $x$ at time $t$.
Thus, it is also natural to calculate $\sin(f)$ and $1-\cos(f)$ which
we will use later on in this paper.
\smallskip

As a wave-type equation, the sine-Gordon equation \eqref{eq: sG} enjoys
the conservation of energy
\begin{align}
E_{\text{sin}}\left(t\right) & :=\frac{1}{2}\int_{\mathbb{R}}\left|\partial_{t} f \right|^{2}+\left|\partial_{x} f \right|^{2}\,dx+\int_\bbR \left(1-\cos f \right)\,dx\label{eq:sGenergy}\\
& =\frac{1}{2}\int_{\mathbb{R}}\left|\partial_{t} f \right|^{2}+\left|\partial_{x} f \right|^{2}\,dx+2\int_\bbR\sin^{2}\left(\frac{f}{2}\right)\,dx=E_{\text{sin}}\left(0\right)\nonumber 
\end{align}
and the conservation of momentum
\begin{equation}
P\left(t\right):=\frac{1}{2}\int_{\mathbb{R}}f_{t}(x,t)f_{x}(x,t)\,dx=P(0).\label{eq:sGmomentum}
\end{equation}
Thus, using the general theory of semilinear wave equations, the natural Sobolev spaces to study the well-posedness
of the sine-Gordon equation are
\begin{equation}\label{eq:sobolevwave}
H_{\text{sin}}^{k}\left(\mathbb{R}\right)\times H^{k-1}\left(\mathbb{R}\right):=\left\{ \vec{ f }:=\left(f,f_{t}\right)\in\left(\dot{H}^{k}\left(\mathbb{R}\right)\times H^{k-1}\left(\mathbb{R}\right)\right):\ \sin\left(\frac{f}{2}\right)\in L^{2}\left(\mathbb{R}\right)\right\}
\end{equation}
where $\dot{H}^k$ is the homogeneous Sobolev space.
We refer the reader to de Laire-Gravejat \cite{deLG} and Shatah-Struwe \cite{SS} for full details. Solutions of \eqref{eq: sG} also satisfy several symmetries:
\begin{itemize}
\item[1.] Shifts in space and time: if $\vec{f}\left(x,t\right)$ is
a solution to \eqref{eq: sG} then $\vec{f}\left(x+x_{0},t+t_{0}\right)$ will also be
a solution for any $t_0, x_0\in\mathbb{R}$.
\item[2.] Lorentz boosts: for each $\beta\in\left(-1,1\right),$ if
$\vec{f}(x,t)=\left(f,f_{t}\right)(x,t)$ is a solution to \eqref{eq: sG} then
\begin{equation}
\left(f,f_{t}\right)_{\beta}:=\left(f,f_{t}\right)\left(\gamma\left(x-\beta t\right),\gamma\left(t-\beta x\right)\right)\qquad\gamma:=\frac{1}{\sqrt{1-\beta^{2}}}\label{eq:Lorentz}
\end{equation}
is also a solution.
\end{itemize}

\smallskip

Moreover,  the sine-Gordon equation is an
integrable system in the sense that it is the compatibility condition of
a Lax pair. Detailed descriptions will be given in  Section \ref{sec: d-i}. 
In the seminal work of Zakharov and Shabat \cite{ZS71}, the authors
developed powerful tools to study the initial value problem for a family of integrable equations. Under this framework, the initial value problem for the sine-Gordon equation in the characteristic coordinate was first done by Ablowitz, Kaup, Newell and Segur in \cite{AKNS73}
and independently by Takhatajan in \cite{Takhatajan1974}. 
In 1975, the initial value problem for the sine-Gordon equation in laboratory coordinates using the same methods was independently studied by Kaup in \cite{Kaup75},
and by Zakharov, Takhtajan and Fadeev in \cite{ZTF}.

\smallskip

One striking feature of the sine-Gordon equation is that it supports kinks and antikinks
	which are topological solitons and breathers, which can be regarded  as bound states of kinks and antikinks
(see \cite{Lamb}). The kink solution of the sine-Gordon equation
is known as
\begin{equation}
Q(x)=4\arctan e^{x}\label{eq:kink1}
\end{equation}
which is an exact solution connecting the final states $0$ and $2\pi$.
Using the Lorentz boost and translation invariance, given any
$\beta\in\left(-1,1\right)$ and any phase shift $x_{0}\in\mathbb{R},$ we can
find a family of kinks
\begin{equation}
Q\left(x,t;\beta,x_{0}\right)=4\arctan e^{\gamma(x-\beta t+x_{0})},\ \gamma=\frac{1}{\sqrt{1-\beta^{2}}}.\label{eq:kink2}
\end{equation}
The antikink of form $-Q\left(x,t;\beta,x_{0}\right)$ is another exact solution of the sine-Gordon equation. For $v\in\left(-1,1\right)$,
$\gamma=\frac{1}{\sqrt{1-\left|v\right|^{2}}}$, $\beta\in\left(0,\gamma\right)$
and $x_{1},x_{2}\in\mathbb{R}$, we can also construct a family of breather solutions
to the sine-Gordon equation:
\begin{equation}
B_{v}\left(x,t;\beta, x_{1},x_{2}\right):=4\arctan\left(\frac{\beta}{\alpha}\frac{\cos\left(\alpha (t-vx+x_{1})\right)}{\cosh\left(\beta (x-vt+x_{2})\right)}\right),\label{eq:breather}
\end{equation}
where $\alpha:=\sqrt{\gamma^{2}-\beta^{2}}$.

Kinks in one dimension \eqref{eq:kink2}
are the simplest example of \emph{topological solitons}, that is, non-spatially localized
special solutions, as opposed to the standard spatially localized solitons or breathers \eqref{eq:breather}.
These nonlinear structures do not disperse, a notable consequence of the interaction between the nonlinearity and the dispersion of the equation. The breathers and kink/antikinks can be combined into more complicated structures such as multi-kink, kink-antikink, wobbling kinks, etc.. These can be found in Alejo-Mu\~noz-Palacios \cite{AMP}, Cheng\cite{Cheng}, Mu\~noz-Palacios \cite{MP} and Tsuru-Wadati \cite{TW}.{\footnote{We will use slightly different notations for kinks, antikinks, breathers later on to be more consistent with the inverse scattering literature.} 
Given that topological solitons are not spatially localized,
it makes the linear theory and perturbation analysis around topological
solitons more difficult than
those needed for the study of localized solitons.


\subsection{Main results and related literature}

In our current work, we combine methods from the inverse
scattering transforms and classical PDEs to analyze the dynamics of
the sine-Gordon equation in certain Sobolev spaces. We first systematically
develop the mapping properties between the initial data and scattering
data. Then to obtain the long-time asymptotics,
we perform the nonlinear steepest descent method in the framework
of Deift and Zhou and its reformulation by Dieng and McLaughlin through
$\overline{\partial}$-derivatives. The second part of our work is
devoted to the study of the asymptotic stability. From the results of long-time asymptotics,
the full asymptotic stability\footnote{Here, by the full asymptotic stability, it means that detailed descriptions for the parameters of solitons and the modified scattering behavior of the radiation term are given after perturbations} is a byproduct. Then using a refined approximation
argument, we are able to study asymptotic stability in certain localized
energy space in the spirit of the work of Kowalczyk-Martel-Mu\~noz
\cite{KMM1,KMM2} with perturbations in weighted energy spaces.
We give a criterion for the weight which is sharp up to
the endpoint (not including) to determine the asymptotic stability/instability. In the following subsection, we introduce our main results in this paper and briefly survey the literature related to our topics.

\smallskip

In order to state our results, we first introduce the weighted version of the Sobolev space \eqref{eq:sobolevwave}. For $k\in \mathbb{N},\,s\in\mathbb{R}^+$, we define
\[
H_{\text{sin}}^{k,s}\left(\mathbb{R}\right):=\left\{ f\in L_{\text{loc}}^{1}\left(\mathbb{R}\right):\,\partial_{x}^{j}f\in L^{2,s}\left(\mathbb{R}\right),1\leq j\leq k,\,\text{and}\,\sin\left(f/2\right)\in L^{2,s}\left(\mathbb{R}\right)\right\} ,
\]
\begin{equation}
\label{H-weighted}
H^{k,s}\left(\mathbb{R}\right):=\left\{ f\in L_{\text{loc}}^{1}\left(\mathbb{R}\right):\,\partial_{x}^{j}f\in L^{2,s}\left(\mathbb{R}\right),0\leq j\leq k\right\} ,
\end{equation}
and
\[
L^{2,s}\left(\mathbb{R}\right):=\left\{ f\in L_{\text{loc}}^{1}\left(\mathbb{R}\right):\,\:\left\langle x\right\rangle^s f=\left(\sqrt{x^{2}+1}\right)^{s}f\in L^{2}\left(\mathbb{R}\right)\right\} .
\]
The pseudometric distance function associated to $H_{\text{sin}}^{k,s}$ is given by
\[
d^{k,s}_{\text{sin}}\left(f_{1},f_{2}\right):=\left(\left\Vert \sin\left(\frac{f_{1}-f_{2}}{2}\right)\right\Vert _{L^{2,s}}^{2}+\sum_{j=1}^k \left\Vert \partial_{x}^j\left(f_{1}-f_{2}\right)\right\Vert _{L^{2,s}}^{2}\right)^{\frac{1}{2}}.
\]
We study the sine-Gordon equation with initial data
\begin{equation}\label{eq:weightedspace}
\vec{f}(x,0)=\left(f_{0}(x),f_{1}(x)\right)\in H_{\sin}^{k,s}\left(\mathbb{R}\right)\times H^{k-1,s}\left(\mathbb{R}\right).
\end{equation}
First of all, observe that $f_{0}\in H_{\text{sin}}^{k,s}\left(\mathbb{R}\right)$ 
implies that there exist integers $\ell_{+}$ and $\ell_{-}$ such
that
\begin{equation}
\label{def:bdry}
f_{0}\rightarrow 2\pi \ell_{\pm},\ x\rightarrow\pm\infty.
\end{equation}

Without loss of generality, throughout this paper, we take $\ell_{-}=0$
and $\ell_{+}=n$ for a fixed integer $n$. Note that by adding a fixed multiple
of $2\pi$ to $f_{0}$, we can obtain the general cases.

\smallskip
Before introducing the main theorems, we mention that in order to avoid overly many {new notations and definitions in this introductory section, we shall only formally
state our main results. For the precise versions of these results,
see Section \ref{sec:summary} and Section \ref{sec:stability}.

\smallskip

\subsubsection{\bf{Long-time asymptotics}}

To study the long-time asymptotics for the sine-Gordon equation, we
first note that any solution to the sine-Gordon equation will
not scatter to the associated linear evolution, the free Klein-Gordon
equation. Due to the weakness of the nonlinearity, when no soliton appears, the solution will decay at the
same rate as linear solutions but its asymptotic behavior differs
from linear solutions by a logarithmic phase correction. To see this,
considering  small norm solutions to the sine-Gordon equation \eqref{eq: sG}, one can use the Taylor series to expand
$\sin(f)=f-\frac{f^{3}}{6}+\mathcal{O}\left(f^{5}\right).$ Then the
leading order behavior of the solution is given by
\[
f_{tt}-f_{xx}+f-\frac{f^{3}}{6}=0
\]
which is the cubic Klein-Gordon equation. The main feature of the
cubic nonlinearity is its criticality with respect to scattering:
linear solutions of the Klein-Gordon equation decay at best like $|t|^{-1/2}$
in $L_{x}^{\infty}$. Thus, when evaluating the nonlinearity on
linear solutions, one sees that $f^{3}\sim|t|^{-1}f$; the non-integrability
of $|t|^{-1}$ for $t\gg1$ results
in a ``Coulomb''-type contribution of the nonlinear terms which makes
the solution to exhibit modified scattering as time goes to infinity. In order to compute the precise behavior of this radiation, it is natural to work on weighted spaces. For works on the cubic Klein-Gordon equation in $1$d,
we refer to Delort \cite{JMD}, Lindblad-Soffer \cite{LS,LS2,LS3}, Lindblad-Luhrmann-Soffer \cite{LLS},
Hayashi-Naumkin \cite{HN1,HN2,HN3}, Sterbenz \cite{Sterb} and references therein. Overall,
although PDE techniques do not rely on integrability,
to our best knowledge, certain smallness assumptions on the data are
required in order to capture the modified scattering phenomenon.

\smallskip
In the large data regime, one of the most important problems in dispersive waves is the \emph{soliton resolution conjecture}. This conjecture asserts, roughly speaking, that any reasonable solution
to such equations eventually resolves into a superposition of a radiation
component and a finite number of \textquotedblleft nonlinear bound
states\textquotedblright{} or \textquotedblleft solitons\textquotedblright .
In the  works of Duyckaerts-Kenig-Merle \cite{DKM} and Duyckaerts-Jia-Kenig-Merle
\cite{DJKM} the authors establish this conjecture for the energy critical wave
equation in high dimensions (along a sequence of time for the nonradial
case) without using integrability. For integrable systems, this resolution phenomenon
is studied by Borghese-Jenkins-McLaughlin \cite{BJM16} for the cubic NLS
and by Jenkins-Liu-Perry-Sulem \cite{JLPS18} for the  derivative NLS and more recently by {Pelinovsky-Saalmann} \cite{PS19, SA18} for the massive Thirring model and Chen-Liu for \cite{CL2} the mKdV equation. 

\smallskip
In this paper, we make use of  the complete integrability
of the sine-Gordon equation.  Leading order asymptotic formulas for the long time behavior of those types of integrable equations have been formally studied in Zakharov and Manakov \cite{ZakharovM1976} 
in 1976 using the inverse scattering transform method, and then by Its \cite{Its1981} 
in 1981 using the isomonodromy method. Deift and Zhou in \cite{DZ93} developed the nonlinear  steepest descent method. Their work laid down the foundation of rigorously studying the asymptotics for integrable equations. In \cite{DZ93}, a key step in the nonlinear steepest descent method consists of  deforming the contour
associated to  the Riemann-Hilbert problem (RHP) in such a way that  the phase functions  with oscillatory 
dependence on parameters become exponential decay.
In Cheng \cite{Cheng} and Cheng-Venakedis-Zhou \cite{CVZ99}, the authors first used the nonlinear  steepest descent method to study the long time asymptotics of the sine-Gordon equation with soliton-free initial data. While \cite{Cheng} also considered  solitons, no detailed derivation of formulas is given. Moreover in both of these works, the authors assumed the initial condition to be in the Schwartz space. Later again assuming infinite order of smoothness and decay, Huang and Lenells in \cite{HL} considered the sine-Gordon equation in a quarter plane for $x\geq 0, t\geq 0$. 
Our goal in this paper is to study the long-time behavior of the sine-Gordon equation in weighted Sobolev spaces, in particular, to rigorously analyze the stability/instability of soliton solutions to the equation and prove the soliton resolution conjecture. We also want to point out the work of Deconinck-Trogdon-Yang \cite{DTY} in which the authors implemented numerical analysis of the inverse scattering transform for sine-Gordon equation. Their results give further insights into the decay rate of the error terms in the asymptotics formulas.

\smallskip

In \cite{Zhou98},  Zhou developed rigorous analysis of the direct and inverse scattering transform for the AKNS system with a class of initial conditions belonging to the space  $\mathcal{H}^{i,j}(\bbR)$.
Here,  $\mathcal{H}^{i,j}(\bbR)$,  different from $H^{k,s}$ in \eqref{H-weighted}, denotes  
the completion of $C_0^\infty(\bbR)$ in the norm
$$
\norm{u}{\mathcal{H}^{i,j}(\bbR)}
= \left( \norm{(1+|x|^j)u}{2}^2 + \norm{u^{(i)}}{2}^2 \right)^{1/2}. 
$$ This analysis later plays a fundamental role in  relaxing the regularity assumptions
of the initial data. In particular, among the most celebrated results
concerning nonlinear Schr\"odinger equations, we point out the work
of Deift-Zhou \cite{DZ03} where they derive the asymptotics for
the NLS in the weighted space $L^{2,1}$. Dieng and McLaughlin in \cite{DM08} (see also an extended version \cite{DMM18}) developed a variant of Deift-Zhou method. In their approach, the 
rational approximation of the reflection coefficient is replaced by some 
non-analytic extension of the jump matrices off the real axis, which  leads to a $\bar{\partial}$-problem to 
be solved in some regions of the complex plane. The
new  $\bar{\partial}$-problem can be reduced to an integral equation and is solvable through
a Neumann series.   
These ideas were originally implemented by Miller and McLaughlin \cite{MM08} to the 
study the  asymptotics of orthogonal polynomials. This method has shown its robustness in its application to other integrable models. Notably, it was successfully applied to address the soliton resolution in \cite{BJM16},  \cite{JLPS18} and \cite{CL2} for the focussing NLS, the derivative NLS and the mKdV equation respectively. We would like to point out that Kamvissis, McLaughlin and Miller generalized the Deift--Zhou approach to study the semi-classical limit for the focusing nonlinear Schr\"odinger equation in \cite{KamvissisMM2003}.
For semi-classical limit results concerning the sine-Gordon equation, we mention \cite{BuckinghamM2013} by Buckingham and Miller as an example for the reader's interest. 

\smallskip 

After systematically establishing the Sobolev mapping properties, we apply the nonlinear
steepest descent to study the long-time asymptotics of solutions to
the sine-Gordon equation without the smallness assumption on the initial
data. Our first result in this paper is that we give a full description of the long-time behavior of \textit{generic}  solutions  (open and dense subset)
in the weighted Sobolev space $H_{\text{sin}}^{2,s}\left(\mathbb{R}\right)\times H^{1,s}\left(\mathbb{R}\right)$
for $s>\frac{1}{2}$ which is sufficient to apply the nonlinear steepest descent method.
The asymptotic formulas obtained in Propositions \ref{lemma br.RHP.asy} and \ref{lemma kin.RHP.asy} immediately verifies of soliton resolution conjecture for the sine-Gordon  equation with generic data. 

\begin{theorem}[Soliton resolution]
	\label{thm:main1} Let the initial data $\vec{f}\left(0\right)=\left(f_{0},f_{1}\right)\in H_{\text{sin}}^{2,s}\left(\mathbb{R}\right)\times H^{1,s}\left(\mathbb{R}\right)$
	with $s>\frac{1}{2}$ be \emph{generic} in the sense of Definition \ref{genericity} and let $f$
	be  the unique \footnote{We also provide a  proof of the global well-posedness in  weighted Sobolev spaces in Theorem \ref{thm:gwpweighted}. }  solution to the sine-Gordon equation \eqref{eq: sG} obtained by solving the Riemann-Hilbert Problem \ref{RHP-1} with initial data $\vec{f}(0)$. Then the
	solution $f$ can be written as the superposition of breathers, kinks,
	anti-kinks and the radiation. More precisely, there exist two non negative
	integers $N_{1}$, $N_{2}$, two sets of velocities
	\[
	\left\{ v_{b,j}\right\} _{j=1,\ldots N_{1}},\ \left\{ v_{k,j}\right\} _{j=1,\ldots N_{2}}
	\]
	two sets of parameters
	\[
	\left\{ \epsilon_{j},1\leq j\leq N_{1}:\epsilon_{j}\in\left\{ \pm1\right\} \right\} ,\ \left\{ \beta_{j}\in\left(0,\gamma_{b,j}\right),1\leq j\leq N_{2}\right\} 
	\]
{and sets of trajectories}
	\[
	\left\{ x_{k,j}(t),1\leq j\leq N_{1}\right\} ,\ \left\{ x_{b,1,j}\left(t\right),x_{b,2,j}\left(t\right),1\leq j\leq N_{2}\right\} 
	\]
	(If $N_{\ell}$, $\ell=1,2$ is $0$ then the corresponding sets of
	velocities, parameters, will be empty sets) such that using the notations
	\eqref{eq:kink2} and \eqref{eq:breather}, one can write
	\begin{align*}
	f & =\sum_{j=1}^{N_{1}}\epsilon_{j}Q\left(x,t;v_{k,j},x_{k,j}(t)\right)+\sum_{j=1}^{N_{2}}B_{v_{b,j}}\left(x,t;\beta_{b,j},x_{b,1,j}(t),x_{b,2,j}(t)\right)+f_r(x,t)
	\end{align*}
{where the radiation term decays to $0$ and it exhibits the modified scattering inside the light-cone.}
	\end{theorem}
\begin{remark}
{We remark that in the leading order term of $f_r(x,t)$ is computed explicitly in terms of stationary points and the reflection coefficient. The error term from $f_r(x,t)$ depends on the weighted norm of the initial data and the stationary points. In particular,  strictly inside the light cone $|x/t|<1-\epsilon$ where $\epsilon>0$ is arbitrarily small,  the error term is uniform with respect to  $x/t$.  When $|x/t| \to 1$ as $t \to \infty$, the error term takes another form but the dependence on $x/t$ is still uniform. For the explicit expressions of the radiation terms, see Theorem \ref{thm:maindetail}.}
\end{remark}
\begin{remark}
	Note that in the asymptotic formula above, there are no wobbling kinks (see \eqref{eq:Wobbling-Kink}). The reason is that our Definition \ref{genericity} rules them out. The genericity condition forces breathers and kinks to have different speeds.  But this requirement is imposed just for the sake of simplicity. A wobbling kink formally is a superposition of a breather and kink having the same speed. From the inverse scattering point of view, the spectra associated to wobbling kinks are not stable under perturbation. It will become clear later on that our computations can also be applied to the setting with the appearance of wobbling kinks, since the spectral information of wobbling kinks is still nice, i.e., eigenvalues are simple. The reader can also see Remarks \ref{remark-generic} and \ref{rem:def1}.
\end{remark}
\begin{remark}
	We  point out  that by the continuous dependence of the scattering
	data on the initial data, if $\left\Vert \vec{f}(0)\right\Vert _{H_{\text{sin}}^{2,s}\times H^{1,s}}$
	is small enough, then $\vec{f}(0)$ is generic.
\end{remark}
\begin{remark}
	To illustrate why we need the weighted spaces,  consider the linear
	version of the direct and inverse scattering transform:
	\begin{itemize}
	\item For the direct scattering, we consider
	\begin{equation}
	r(z)=\int_\bbR e^{i(z-1/z)x}q(x)dx
	\end{equation}
	whose derivative is given by
	\begin{equation}
	\dfrac{d r(z)}{dz} =i\int_\bbR e^{i(z-1/z)x} \left( 1+\dfrac{1}{z^2}\right) x q(x)dx.
	\end{equation}
	To deal with the singularity at $0$, we introduce the change of variable $z\mapsto \gamma=1/z$. Then for $f(z)=g(\gamma)$, a simple change of variable gives
	\begin{align*}
\int_0^1|f'(z)|^2 dz= \int_1^\infty |g'(\gamma)|^2 \gamma^2 d\gamma.
\end{align*}
This leads to showing the following property
\begin{equation}
	\gamma\dfrac{d r(\gamma)}{d\gamma} =-i\int_\bbR e^{i(1/\gamma-\gamma)x} \left( \gamma+\dfrac{1}{\gamma}\right) x q(x)dx \in L^{2}_\gamma(\bbR)
	\end{equation}
	and the standard Fourier theory will require $xq(x)\in H^1(\bbR)$. We remark that this singularity at $z=0$ appears in other integrable PDEs including the defocussing NLS equation with non-vanishing boundary condition at $\infty$.
	\item Reading
	off from the evolution of the scattering data, the linear model is
	the following:
	\[
	\int e^{\left(k-\frac{1}{k}\right)ix+\left(k+\frac{1}{k}\right)it}r\left(k\right)\,dk.
	\]
	To obtain the asymptotics of the above integral, one can  apply the standard
	stationary phase. Taking $x=0$ as the example, we obtain that
	\[
	\left|\int e^{\left(k+\frac{1}{k}\right)it}r\left(k\right)\,dk\right|\lesssim\frac{1}{\sqrt{t}}|r\left(1\right)|+\frac{1}{t^{\frac{2s+1}{4}}}\left\Vert r\right\Vert _{H^{s,1}}.
	\]
	So we note that even in this simplified case we need $\left\Vert r\right\Vert _{H^{s,1}}$ to be finite in order to obtain
	the asymptotic. By the Sobolev mapping properties, to ensure that $\left\Vert r\right\Vert _{H^{s,1}}$ to be finite, it requires that $\vec{f}(0)\in H_{\text{sin}}^{2,s}\left(\mathbb{R}\right)\times H^{1,s}\left(\mathbb{R}\right)$, see Proposition \ref{prop:r}.
	\end{itemize}
	
\end{remark}

\smallskip

\subsubsection{\bf{Stability of solitons and topological solitons}}

The literature on the  stability of solitons  is extensive and without trying to be exhaustive, we refer
to the survey by Tao \cite{Tao}, the monograph Dauxois-Peyrard
\cite{DP} and references therein.  Regarding the study of the full asymptotic
stability of various kink models on the full real line, we refer the reader to recent
work of Germain-Pusateri \cite{GP}, Delort-Masmoudi \cite{DM}, Komech-Kopylova
\cite{KK1,KK2} and references therein. In particular, regarding kink
solutions, the classical orbital stability was proven in Henry-Perez-Wreszinski
\cite{HPW}. In Mu\~noz-Palacios \cite{MP}, the authors analyzed the orbital stability
of 2-soliton solutions for the sine-Gordon equation. Kowalczyk-Martel-Mu\~noz
\cite{KMM2} proved the asymptotic stability locally in the energy space
for odd perturbations of the kink of the $\phi^{4}$ equation. Regarding the interaction among kinks and antikinks, we refer to the work of Jendrej-Kowalczyk-Lawrie \cite{JKL}. In the very recent work of Kowalczyk-Martel-Mu\~noz-Van Den Bosch \cite{KMMV}, the authors gave a sufficient condition for asymptotic stability of kinks in general (1+1)-scalar field models.

\smallskip

While the mathematical theory on the stability (or instability) of
solitons is very well-developed in many models, this is not the case
for topological solitons. As mentioned above, the striking difference between solitons
and topological solitons is that the latter are not spatially localized.
This makes the linear theory and perturbation analysis around topological
solitons by PDE techniques difficult compared to
the analysis for the soliton. In the current sine-Gordon setting, our long-time asympotics in particular
imply the \emph{full asymptotic stability} of kinks, antikinks, breathers and
other nonlinear combination of them in the space $H_{\sin}^{2,s}\left(\mathbb{R}\right)\times H^{1,s}\left(\mathbb{R}\right)$
with $s>\frac{1}{2}$ in which we perform the nonlinear steepest descent.  
\smallskip


In order to avoid too much technicality and keep our introduction short, we illustrate the asymptotic stability results using kinks only (instead of the general $n$-soliton solutions). For full details and the asymptotic stability of multi-soliton solutions, see Section \ref{sec:summary} and Section \ref{sec:stability}.

\begin{corollary}[Full asymptotic stability of kinks]
\label{col-asy kink}
	Consider the kink solution to the \eqref{eq: sG}
	\[
	K(x,t)=Q\left(x,t;v_{k}^{0},x_{k}^{0}\right)
	\]
	with some velocity $v_{k}^{0}\in(-1,1)$ and shift $x_{k}^{0}$. Let $s>\frac{1}{2}$
	and suppose that
	\[
	\left\Vert \vec{f}\left(0\right)-\left(K(x,0),\partial_{t}K(x,0)\right)\right\Vert _{H_{\sin}^{2,s}\left(\mathbb{R}\right)\times H^{1,s}\left(\mathbb{R}\right)}\leq\eta
	\]
	for some $\eta$ small enough. Then there exist a new velocity $v_{k}\in(-1,1)$ and a trajectory
	$x_{k}(t)$ such that
	\[
	\left|v_{k}-v_{k}^{0}\right|+\left|x_{k}(t)-x_{k}^{0}\right|\lesssim\eta,
	\]
	 for the solution $f$ to the sine-Gordon equation \eqref{eq: sG} with initial data $\vec{f}(0)$,  one can write
	\[
	f=Q\left(x,t;v_{k},x_{k}(t)\right)+f_r
	\]
	{where the radiation term $f_r$ has the same behavior as Theorem \ref{thm:main1}.}
\end{corollary}
\begin{remark}
	For the full version with finitely many solitons, see Corollary \ref{cor:stabN} in Section \ref{subsec:fullasym}.
\end{remark}


On the other hand, from our discussion above, solutions of \eqref{eq: sG}
have the conservation of the energy and thus are globally well-defined
in the energy space $H_{\sin}^{1}\left(\mathbb{R}\right)\times L^{2}\left(\mathbb{R}\right)$.
Therefore, it is natural to consider the stability problem and dynamics
of the sine-Gordon equation in $H_{\text{sin}}^{1}\left(\mathbb{R}\right)\times L^{2}\left(\mathbb{R}\right)$.
Note that one can replace $\sin\left(v\right)$ by $v$ when $v$
is small. So additionally, small perturbations of a given solution
in $H_{\sin}^{1}\left(\mathbb{R}\right)\times L^{2}\left(\mathbb{R}\right)$
are essentially in $H^{1}\left(\mathbb{R}\right)\times L^{2}\left(\mathbb{R}\right)$
and vice-versa. Then by the explicit formula of breathers \eqref{eq:breather},
the energy norm of a breather can be arbitrarily
small as long as the parameter $\beta$ is small enough. Clearly, this is the obstruction to
the asymptotic stability of the sine-Gordon equation in the energy
space. Moreover, the existence of small breathers is also closely
related to the existence of the wobbling kink. Formally, wobbling
kink is the overlap of a kink with a breather sharing the same velocity.
In terms of the inverse scattering transform, the eigenvalues which
produce a wobbling kink consist of eigenvalue of a kink and eigenvalues
of a breather located on the same circle in the plane of the spectral parameter. By direct computations
in Kowalczyk-Martel-Mu\~noz \cite{KMM2} and our computations in Section
\ref{sec:stability}, the sine-Gordon kink is not asymptotically stable in the energy
space. Therefore, any result concerning the asymptotic stability of
solutions to the sine-Gordon equation will require to rule out this type of special solutions. In the work
of Kowalczyk-Martel-Mu\~noz \cite{KMM1}, the authors considered the
odd perturbation of the zero solution which rules out small breathers
in the energy space so that they can establish the asymptotic stability of
the zero solution in the localized energy norm. Regarding nontrivial
nonlinear structures for the sine-Gordon equation, we mention the recent
work of Alejo-Mu\~noz-Palacios \cite{AMP} where the authors identified
a smooth codimensional manifold in the energy space where the asymptotic
stability of the sine-Gordon kinks holds in the localized energy norm.

\smallskip

In the present work, instead of  appealing to parity properties of solutions, we
add weights to the energy space to quantify the small breathers and the difference between a wobbling kink and a kink. Then based on the long-time asymptotics, we use a refined
approximation argument to show the asymptotic
stability of the sine-Gordon solution under perturbations in 
\begin{equation}
\label{space: weight}
{\large H_{\text{sin}}^{1,s}\left(\mathbb{R}\right)\times L^{2, s}\left(\mathbb{R}\right):= \left\lbrace ( f_0, f_1):  \sin( f_0/2), \, f_{0, x},  \, f_1\in  L^{2,s}(\bbR) \right\rbrace },\quad s>\frac{1}{2}.
\end{equation}
 Here instead of giving the precise description of the radiation term as in the full asymptotic stability result above, the error terms are measured in the localized
energy space  in the spirit of the work of Kowalczyk-Martel-Mu\~noz \cite{KMM1,KMM2}.

\smallskip


For any given $\mathfrak{v}\in \mathbb{R}$ and $\mathfrak{L}>0$ fixed,
we define the localized energy norm for a vector-valued function $\vec{f}=\left(f_{1},f_{2}\right)$
as
\begin{align}
\left\Vert \vec{f}\right\Vert _{\mathcal{E},\mathfrak{v},\mathfrak{L}}^{2}:=\left\Vert f_{1}\right\Vert _{L^{2}\left(\left|x-\mathfrak{v} t\right|<\mathfrak{L}\right)}^{2}+\left\Vert \partial_{x}f_{1}\right\Vert _{L^{2}\left(\left|x-\mathfrak{v} t\right|<\mathfrak{L}\right)}^{2}+\left\Vert f_{2}\right\Vert _{L^{2}\left(\left|x-\mathfrak{v} t\right|<\mathfrak{L}\right)}^{2}.
\end{align}\label{eq:localenergy}

\begin{theorem}[Asymptotic stability of kinks]
	Consider the kink solution to the sine-Gordon equation \eqref{eq: sG}
	\[
	K(x,t)=Q\left(x,t;v_{k}^{0},x_{k}^{0}\right)
	\]
	with some velocity $v_{k}^{0}\in(-1,1)$ and shift $x_{k}^{0}$. Let $s>\frac{1}{2}$
	and suppose that
	\[
	\left\Vert \vec{f}\left(0\right)-\left(K(x,0),\partial_{t}K(x,0)\right)\right\Vert _{H_{\text{sin}}^{1,s}\left(\mathbb{R}\right)\times L^{2,s}\left(\mathbb{R}\right)}\leq\eta
	\]
	for some $\eta$ small enough. Then there exist a new velocity $v_{k}\in(-1,1)$ and a trajectory
	$x_{k}(t)$ such that
	\[
	\left|v_{k}-v_{k}^{0}\right|+\left|x_{k}(t)-x_{k}^{0}\right|\lesssim\eta,
	\]
	 for the solution $f$ to the sine-Gordon equation \eqref{eq: sG} with initial data $\vec{f}(0)$, one can write the solution to \eqref{eq: sG} as
	\[
	f=Q\left(x,t;v_{k},x_{k}(t)\right)+f_r
	\]
	where the radiation satisfies
	\begin{equation}\label{eq:localen}
		\lim_{t\rightarrow\infty}\left\Vert f_r(t)\right\Vert _{\mathcal{E},\mathfrak{v},\mathfrak{L}}=0
	\end{equation}
for any fixed $\mathfrak{v}\in\mathbb{R} $ and  $\mathfrak{L}>0$.
\end{theorem}
\begin{remark}
	The speed of the interval $\mathfrak{v}$ in \eqref{eq:localen} is only required to be a fixed constant in $\mathbb{R}$. In particular it can be the speed of the perturbed kink $v_k$. In this case, the localized energy norm measures the decay of the radiation around the kink.
\end{remark}
\smallskip
In the final part of our paper, we show that the weight required in the theorem above is almost optimal. By the explicit computations of the weighted energy norm of breathers in Section \ref{sec:stability}, we deduce the failure of the asymptotic stability for solutions to the sine-Gordon equation in $H^{1,s}_{\text{sin}}\times L^{2,s}$ for $0\leq s <1/2$. Here by the asymptotic instability we mean that after perturbations, there are additional/new type of non-decay components emerging in the asymptotics. The main obstruction is that there exist small breathers. As we show in Subsubsection  \ref{subsubsec:breather}, if the weight is not enough, one has the following estimate:
\begin{proposition}
	Consider $0\leq s <1/2$ and $0<\beta$ small. Using the notation for breathers \eqref{eq:breather}, one has that at $t=2\pi/\alpha$
	\[	
	\left\Vert B_{0}(\cdot,\frac{\pi}{2\alpha};\beta, 0,0)\right\Vert_{H_{\text{sin}}^{1,s}\times L^{2,s}} \sim \beta ^{\frac{1}{2}-s}
	\] where $\alpha=\sqrt{1-\beta^2}$. 
	In particular, this implies the failure of the asymptotic stability of the zero solution to the sine-Gordon equation for initial perturbations in $H_{\text{sin}}^{1,s}\times L^{2,s}$ for $0\leq s<1/2$.  
\end{proposition}
Similar computations can be carried out for the difference between a wobbling kink and a kink. Morally, a wobbling kink can be regarded as a nonlinear superposition of a breather and a kink. See \eqref{eq:Wobbling-Kink} for details.
\begin{proposition}
	The static kink solution $Q(x,t;0,0 )$ to the sine-Gordon equation \eqref{eq: sG} is not asymptotically stable for perturbations in  $H_{\text{sin}}^{1,s}\times L^{2,s}$ with $0\leq s<1/2$. 
\end{proposition}
One can always find wobbling kinks as close to the static kink as desired in the topology of $H_{\text{sin}}^{1,s}\times L^{2,s}$ for $0\leq s<1/2$. This leads to the failure of the asymptotic stability of the kink for $0\leq s<\frac{1}{2}$. 
\smallskip

 In general, for any given solution of the sine-Gordon equation, one can consider the perturbation to this solution from the inverse scattering point of view. Measuring this perturbation process using the norm $H_{\text{sin}}^{1,s}\times L^{2,s}$ for $0\leq s<1/2$, this perturbation can produce a breather as small as desired provided $\beta$ is small enough. Therefore, the solutions to the sine-Gordon equation is not asymptotically stable under perturbations in the weighted energy space if $0\leq s<\frac{1}{2}$.
 
 \begin{remark}
	By the smoothness of breathers, these computations also imply the failure of asymptotic stability of the zero solution and the kink solution to the sine-Gordon equation in Sobolev spaces with higher order regularity: $H^{k,s}_{\text{sin}}(\mathbb{R})\times H^{k-1,s}(\mathbb{R})$ for $s<1/2$ and $k\in \mathbb{N}$.
\end{remark}
}
\smallskip

\subsection{Some discussions}

 Next, we highlight some important features in this paper. 
 
 \begin{itemize}
 \item[1.]  The first preparatory step for the nonlinear steepest descent is to establish the Sobolev mapping properties from the initial data to the scattering data. In the current setting,  the $1/z$ factor in  the phase \eqref{theta} introduces extra difficulties in the analysis of the direct and inverse scattering transform. In the NLS and mKdV problems in Deift-Zhou \cite{DZ03} and Chen-Liu \cite{CL19,CL2} respectively, the regularities of the reflection coefficients deduced from the weights of the initial data resemble the standard Fourier duality between the physical and frequency spaces. Here to obtain the regularity of the reflection coefficients, we do need the additional smoothness of the initial data. Interestingly, the weighted estimate of the reflection coefficient is a byproduct of the $H^s$ estimate of it. Moreover the proof also gives the existence of the limits $\lim_{z\to 0} r(z)/z$.\footnote{It might
 	also be interesting to connect this fact to the analysis of the Klein-Gordon
 	equation using Fourier transforms. Generally speaking, the reflection coefficient
 	in the current setting is similar to the Fourier transform in the
 	PDEs analysis. Using the Fourier transform, if the solution vanishes
 	at zero frequency, typically one can get better dispersive estimates
 	than without vanishing. Regarding this, we refer to Germain-Pusateri \cite{GP}
 	for some discussions.} We need all of these to perform the $\dbar-$steepest descent method. As a direct consequence of estimates above, the error terms in the long-time asymptotics need the smoothness of the initial data.   We also point out that the $1/z$ factor in the Lax pair is particularly delicate in our analysis.  It not only introduces extra pieces of contours when we deform the real axis (see Fig \ref{fig:contour-def}) but also requires delicate handling when we take the limit $z\to 0$  
to compute the asymptotic formulas (see Lemma \ref{lemma:N.to.NRHP.asy}). This makes our analysis of the $\dbar$ problem more involved than in \cite{DM08}. If we simply use the fact that $r(0)=0$ and appply the standard estimate of the  modulus of $r(z)$ as in \cite{DM08} (also see \cite{CL19,CL2}), then the $\dbar$ problem \eqref{est-r-0} will not result in error terms. Instead, it will produce decay of the same rate as the leading order term $t^{-1/2}$. The same happens in the analysis of the asymptotic formula outside the light cone in \eqref{est-r-out}. To overcome this obstacle, we have to study the behavior of $r(z)$ near $z=0$ more carefully in the direct scattering analysis. Only after establishing
that $\lim_{z\to 0} r(z)/z$ is bounded with careful analysis of the $\dbar$ problem near $z=0$, can we conclude that the $\dbar$ problem will give us error terms with fast decay rates than the leading order terms,

 \smallskip
 
\item[2.] To investigate the stability of solutions in the weighed energy
 space $H_{\sin}^{1,s}\left(\mathbb{R}\right)\times L^{2,s}$, it is important to notice that
 we do not have sufficient regularity to make use  of the nonlinear steepest descent. One can use a sequence of data in $\vec{f_{n}}(0)\in H_{\text{sin}}^{2,s}\left(\mathbb{R}\right)\times H^{1,s}\left(\mathbb{R}\right)$
 to approximate the original data. But now the error terms are not
 uniform anymore since they depend on $\left\Vert \vec{f}_{n}\left(0\right)\right\Vert _{H_{\text{sin}}^{2,s}(\mathbb{R})\times H^{1,s}(\mathbb{R})}$
 which goes to $\infty$ as $n\rightarrow\infty$, whence the corresponding
 the radiation term in the asymptotics of solution $\vec{f}_{n}$ will
 decay to $0$ as $t\rightarrow\infty$ with different rates. Moreover,
 when we perform the standard PDEs approximation argument in Sobolev
 spaces, the convergence $\vec{f}_{n}\left(t\right)\rightarrow\vec{f}\left(t\right)$
 in $H_{\text{sin}}^{1}\left(\mathbb{R}\right)\times L^{2}\left(\mathbb{R}\right)$
 depends on time $t$ or $n$. Therefore, one could not simply interchange the
 order of limits to obtain the asymptotic behavior of $\vec{f}\left(t\right)$
 from that of $\vec{f}_{n}(t)$.
 
 \smallskip
 
 To resolve this, we observe that the reason that the convergence $\vec{f}_{n}\left(t\right)\rightarrow\vec{f}\left(t\right)$
 in $H_{\text{sin}}^{1}\left(\mathbb{R}\right)\times L^{2}\left(\mathbb{R}\right)$
 depends on time is the influence of solitons. Consider the Lipschitz
 estimate:
 \[
 \left\Vert \vec{f}_{n}(t)-\vec{f}(t)\right\Vert _{H_{\text{sin}}^{1}\left(\mathbb{R}\right)\times L^{2}\left(\mathbb{R}\right)}\leq A\left\Vert \vec{f}_{n}(0)-\vec{f}(0)\right\Vert _{H_{\text{sin}}^{1}\left(\mathbb{R}\right)\times L^{2}\left(\mathbb{R}\right)}.
 \]
The standard iteration using the Duhamel formula will give a constant
 $A(t)$ growing in $t$. The growth of this constant
 is due to solitons. If $\vec{f_{n}}$ and $\vec{f}$ support two solitons
 with different speeds and close initial positions, then clearly the
 Lipschitz constant will grow in $t$ or depend on $n$. To overcome this, we choose
 a well-designed sequence to approximate the rough solution. In this sequence, each soliton in $\vec{f}$
 will have a corresponding soliton in $\vec{f}_{n}$ with the same
 speed. By making use of the Beals-Coifman representation of solutions and appealing to the \emph{uniform resolvent estimates} developed in Section \ref{sec:stability}, we can get a global Lipschitz
 estimate with constants independent
 of $t$. This gives us a uniform estimate in the approximation which allows us to exchange the order of limits and obtain the asymptotic information
 of $\vec{f}(t)$ from $\vec{f}_{n}(t)$. 
 
 \smallskip
 \item[3]
 To give the intuition for the weights in the energy space, we first
 note that by a simple Cauchy-Schwarz argument, $L^{2,s}\subset L^{1}$
 for $s>\frac{1}{2}$. For the initial data in $H^{1,s}\left(\mathbb{R}\right)\times L^{2,s}\left(\mathbb{R}\right)$,
 the potentials in the AKNS system can produce nice Jost functions.
 Then by the continuity of the Jost functions, one can expect that
 small perturbations in this space will not produce small breathers. This can be found in Zhou \cite{Zhou95}.
 For $s<\frac{1}{2}$, by direct scaling arguments, there exist arbitrarily
 small breathers. Our analysis give a criterion for the weights to
 establish asymptotic stability which is sharp to the endpoint $s=\frac{1}{2}$. We mention that this borderline case $s=\frac{1}{2}$ remains as an interesting  open problem.
 
  \end{itemize}

\subsection{Notations} In this subsection, we fix some notations used this paper. 

\begin{itemize}
	\item[1.] Throughout this paper, we set
	$$\theta(z; x,t)=\dfrac{1}{4} \left(  \left(  z-\dfrac{1}{z}  \right )x+\left( z+\dfrac{1}{z}  \right)t  \right)$$
	$$\tilde{\theta}(z; x,t)=\dfrac{1}{4} \left(  \left(  z+\dfrac{1}{z}  \right )x+\left( z-\dfrac{1}{z}  \right)t  \right).$$
	
\item[2.] Let $\sigma_3$ be the third Pauli matrix:
\begin{equation}\label{eq:sigma3}
\sigma_3=\twomat{1}{0}{0}{-1}
\end{equation}
and define the matrix operation 
$$e^{\ad\sigma_3}\twomat{a}{ b}{c}{d} =\twomat{a}{e^{2} b}{e^{-2}c}{d}.$$
\item[3.] Let $A$ and $B$ be $2\times 2$ matrices, then define the commutator by
$$\left[A, B \right]=AB-BA.$$
\item[4.] $C^\pm$ is the Cauchy projection:
\begin{equation}
(C^\pm f)(z)= \lim_{z\to \Sigma_\pm}\dfrac{1}{2\pi i} \int_{\Sigma} \dfrac{f(s)}{s-z}ds.
\end{equation}
Here $+(-)$ denotes taking limit from the positive (negative) side of the oriented contour $\Sigma$.

Similarly, suppose $M(z)$ is a matrix-valued function in $\bbC$, then $M_\pm$ denotes its continuous boundary value from either side of the oriented contour.\\
\smallskip
\item[5.] We define the Fourier transform as 
	\begin{equation}
	\hat{h}\left(\xi\right)=\mathcal{F}\left[h\right]\left(\xi\right)=\frac{1}{2\pi}\int_\bbR e^{-ix\xi}h\left(x\right)\,dx.\label{eq:FT}
	\end{equation}
	Notice that by the duality between the physical space and the Fourier
	space, if $h\in L^{2,s}\left(\mathbb{R}\right)$ with $s>\frac{1}{2}$
	then
	\begin{equation}
	h\in L^{1}\left(\mathbb{R}\right),\,\hat{h}\in H^{s}\left(\mathbb{R}\right).\label{eq:weightL22}
	\end{equation}
	Then by the trivial Sobolev embedding, $\hat{h}\left(\xi\right)\in L^{\infty}.$

\item[6.] As usual, $``A:=B"$
or $``B=:A"$
is the definition of $A$ by means of the expression $B$. We use
the notation $\langle x\rangle=\left(1+|x|^{2}\right)^{\frac{1}{2}}$.
For positive quantities $a$ and $b$, we write
$a\lesssim b$ for $a\leq Cb$ where $C$ is some prescribed constant.
Also $a\simeq b$ for $a\lesssim b$ and $b\lesssim a$. 
 Throughout, we use $u_{t}:=\frac{\partial}{\partial_{t}}u$, for the derivative in the time variable and 
$u_{x}:=\frac{\partial}{\partial x}u$ for the derivative in the space variable. These two notations are used interchangebly.
\end{itemize}
\subsection{Organization of the paper}
From Section \ref{sec: d-i} to Section \ref{sec:summary}, we set up the inverse scattering formalism for the sine-Gordon equation and perform the nonlinear steepest descent to compute the long-time asymptotics. In Section \ref{sec:stability}, we analyze the asymptotic stability of multi-soliton solutions to the sine-Gordon equation in weighted energy spaces. When the weights are not enough, counterexamples to the asymptotic stability are also discussed.

\subsubsection{Outline of the nonlinear steepest descent}
We give an outline of the derivation of the long-time asymptotics (Sec \ref{sec: d-i} to Sec \ref{sec:summary}), in which the major part is devoted to study the space-time region $|x/t|<1$ where the presence of solitary waves can be observed. 

\smallskip
The first step (Section \ref{sec: d-i}) is to introduce the direct and inverse scattering transform for the sine-Gordon equation. The main purposes are to
\begin{itemize}
\item[1]  Characterize the scattering data corresponding to the initial data given by \eqref{space-1};
\item[2]  Derive the solution to equation \eqref{eq: sG} from solution of the associated Riemann-Hilbert problem \ref{RHP-1} and the reconstruction formula in Proposition \ref{prop:recon}.
\end{itemize}
This step will provide the building blocks for the application of the nonlinear steepest descent method.

\smallskip
The second step (Section \ref{sec:prep}),  is to conjugate 
the matrix $m$ with a scalar function $\delta(z)$  
which solves the  scalar model RHP Problem  \ref{prob:RH.delta}. This conjugation leads to a new RHP, Problem \ref{prob:mkdv.RHP1}. The purpose of this is to prepare for the lower/upper factorization of the jump matrix on the part of the real axis between two stationary points and also reverse the triangularity of certain residue conditions. This is needed in the contour deformation described in Section \ref{sec:mixed}.

\smallskip

The third  step
( Section \ref{sec:mixed}) is a  deformation of contour from $\bbR$ to a new
contour  $\Sigma^{(2)}$ (Figure \ref{fig:contour-def}). It is to guarantee that the phase factors in the jump matrix on $\bbR$
have the desired exponential decay in time along the deformed contours. Inevitably this transformation {will result }in certain non-analyticity in the sectors off the real axis which leads to
a mixed $\dbar$--RHP-problem,  Problem \ref{prob:DNLS.RHP.dbar}.

\smallskip

The fourth step is a factorization of  $m^{(2)}$ which is the solution to Problem \ref{prob:DNLS.RHP.dbar}
in the form $m^{(2)} =  m^{(3)} m^{\RHP}$ where $m^{\RHP}$ is solution of a 
localized RHP,
Problem \ref{MKDV.RHP.local}, and $m^{(3)}$ 
a solution of $\bar\partial$ problem, Problem \ref{prob:DNLS.dbar}. 
The term ``localized" refers to the fact that the reflection coefficient $r(z)$ is fixed at $\pm z_0$ and $0$ along the deformed contours. We then solve this localized RHP whose solution is given by parabolic cylinder functions and a system of linear equations whose solution is a soliton, and combine these solutions together to get the interaction between solitons and radiation.

\smallskip

The fifth step (Section \ref{sec:dbar}) is the solution of the $\dbar$-problem Problem \ref{prob:DNLS.dbar} through solving an integral equation. The integral operator has a small $L^\infty$-norm  at  large $t$, allowing solving the problem by a Neumann series. The contribution of this $\dbar$-problem is another higher order error term. In the context of sine-Gordon equation, there is a singularity in the kernel of the integral equation. To deal with this singularity, we
need to control the reflection coefficient $r$ at the origin, described in Proposition \ref{prop:r}.
\smallskip

The sixth step (Section \ref{sec:large-time}) groups together all the previous transformations from Section \ref{sec:prep} to Section \ref{sec:dbar} to  derive the long time asymptotics of the solution of the sine-Gordon equation in the region $|x/t|<1$,  using reconstruction formula (Proposition \ref{prop:recon}) and the expansion of the RHP solution near the origin.  

\smallskip

The seventh step is the study of region $|x/t|>1$ in which the asymptotic formulas decay rapidly in $t$.
\smallskip
\

\section{Direct and Inverse Scattering}
\label{sec: d-i}
We recall that \eqref{eq: sG} is the compatibility condition for the following Lax pair (cf. \cite[Appendix. A]{BM08}):
\begin{equation}
\label{L}
 \Psi_x =A \Psi
\end{equation}
\begin{equation}
\label{T}
 \Psi_t =B \Psi
\end{equation}
where 
$$ A= -\dfrac{i}{4}\diagmat{z-z^{-1}}{ -z+z^{-1} }+ \dfrac{i}{4 z}\twomat{\cos f-1}{\sin f}{\sin f}{-\cos f+1}+\dfrac{f_x+f_t}{4}\offdiagmat{-1}{1} ,$$
$$ B= -\dfrac{i}{4}\diagmat{z+z^{-1}}{ -z-z^{-1} }- \dfrac{i}{4 z}\twomat{\cos f-1}{\sin f}{\sin f}{-\cos f+1}+\dfrac{f_x+f_t}{4}\offdiagmat{-1}{1} .$$
{ Here $z\in \mathbb{R}\setminus \lbrace 0 \rbrace$ and $\Psi(x,t;z) \in SL(2, \mathbb{C})$.} 
More precisely, the compatibility condition here means that there exists\footnote{This basis is determined, say, by given two linearly independent vectors $\Psi$ at $x = t = 0$} a basis
of simultaneous solutions of \eqref{L} and \eqref{T} if and only if
$f=f(x,t)$ is a solution of the sine-Gordon equation \eqref{eq: sG}.
To study the direct scattering process, we have to analyze the spectral problem given by \eqref{L}. In the analysis $\Psi$ below in this section, the time variable $t$ is always fixed. To save the space, throughout this section,  we always drop the explicit time variable $t$ in the writing without any risk of confusion.

Now we rewrite \eqref{L} in the following form:
\begin{equation}
\label{Psi-x}
\Psi_x=(-J(z)\sigma_3+U(x, z))\Psi
\end{equation}
in which the matrix
$$J(z)\sigma_3:=\dfrac{i}{4}(z-z^{-1})\sigma_3$$
is perturbed by the matrix potential
$$U(x, z)= \twomat{ \dfrac{i}{4z} (\cos f-1)  }{\dfrac{i}{4z} \sin f-\dfrac{1}{4}(f_x+f_t)  }{ \dfrac{i}{4z} \sin f+\dfrac{1}{4}(f_x+f_t)  }{\dfrac{i}{4z} (1-\cos f) }. $$
To deal with $z^{-1}$ near the origin, we apply the gauge transform in \cite[Appendix A]{BM08}  to obtain a new spectral problem
\begin{equation}
\label{Phi-x}
\Phi_x=(-J(z)\sigma_3+V(x, z))\Phi
\end{equation}
where  $z\in \mathbb{R}\setminus \lbrace 0 \rbrace$ and $\Phi \in SL(2, \mathbb{C})$ and
$$V(x, z)= \twomat{\dfrac{iz}{4} (1-\cos f)  }{\dfrac{i z}{4} \sin f +\dfrac{1}{4}(f_x-f_t)  }{ \dfrac{iz}{4} \sin f-\dfrac{1}{4}(f_x-f_t)  }{-\dfrac{i z}{4} (1-\cos f) }. $$
Through the explicit forms of potentials $U$ and $V$, we will establish the long time asymptotics. For the major part of this paper,  we will  use focus on the computations for the initial data $ \vec{f}(0)\in H_{\sin}^{2,s}\left(\mathbb{R}\right)\times H^{1,s}\left(\mathbb{R}\right)$ with $1/2<s\leq 1$. For convenience,  we introduce the following function space
which is equivalent to { $ \vec{f}(0)\in H_{\sin}^{2,s}\left(\mathbb{R}\right)\times H^{1,s}\left(\mathbb{R}\right)$:
\begin{equation}
\label{space-1}
\mathcal{I}= \lbrace (f_0, f_1):  \sin f_0, \, 1-\cos f_0, \, f_{0, x}, \,  f_{0, xx}, \, f_1,\, f_{1, x} \in  L^{2,s}(\bbR) \rbrace.
\end{equation}
}
To see the equivalence, we claim that for any $s\geq 0$, we have $f\in H^{1,s}_\text{sin}$ if and only if $\sin{f},\,(1-\cos{f}),\,f_x\in L^{2,s}$.
\begin{proof}[Proof of the claim] \hspace*{\linewidth}
\begin{itemize}
\item \textit{Only if part}: suppose that $\int\left\langle x\right\rangle ^{2s}\sin^{2}\left(\frac{f}{2}\right)\,dx<\infty$. By double-angle formulas, one has
	\begin{align*}
	\int\left\langle x\right\rangle ^{2s}\sin^{2}\left(\frac{f}{2}\right)\,dx  =\frac{1}{2}\int\left\langle x\right\rangle ^{2s}\left(1-\cos f\right)\,dx \gtrsim\int\left\langle x\right\rangle ^{2s}\left(1-\cos f\right)^{2}\,dx
	\end{align*}
	and
	\begin{align*}
	\int\left\langle x\right\rangle ^{2s}\sin^{2}\left(\frac{f}{2}\right)\,dx & =\frac{1}{2}\int\left\langle x\right\rangle ^{2s}\left(1-\cos\left(f\right)\right)\,dx \gtrsim\int\left\langle x\right\rangle ^{2s}\left(1-\cos\left(f\right)\right)\left(1+\cos\left(f\right)\right)\,dx\\
	& \gtrsim\int\left\langle x\right\rangle ^{2s}\sin^{2}\left(f\right)\,dx.
	\end{align*}
	Therefore $
	\sin\left(f\right),\,1-\cos\left(f\right)\in L^{2,s}$. 
	\vspace{5mm}
\item \textit{If part:} For the other direction, we assume that $\sin{f},\,(1-\cos{f}),\,f_x\in L^{2,s}$. \begin{align*}
\int\left\langle x\right\rangle ^{2s}\sin^{2}\left(\frac{f}{2}\right)\,dx & \sim\int\left\langle x\right\rangle ^{2s}\left(1-\cos\left(f\right)\right)\,dx\sim\int\left\langle x\right\rangle ^{2s}\left(\frac{\sin^{2}(f)}{1+\cos\left(f\right)}\right)\,dx.
\end{align*}
Note that for any finite interval $[-M,M]$, the Cauchy-Schwarz inequality gives
\begin{equation*}
\int_{[-M,M]}\left\langle x\right\rangle ^{2s}\left(1-\cos\left(f\right)\right)\,dx\lesssim M^{\frac{2s+1}{2}}\left(\int_{[-M,M]}\left\langle x\right\rangle ^{2s}\left(1-\cos\left(f\right)\right)^{2}\,dx\right)^{\frac{1}{2}}.\label{eq:estiM1}
\end{equation*}
Therefore, it suffices to bound
\[
\int_{(-\infty,-M)\bigcup(M,+\infty)}\left\langle x\right\rangle ^{2s}\sin^{2}\left(\frac{f}{2}\right)\,dx
\]
for $M>0$ large enough. By symmetry, we only analyze the estimate
over $\left(M,+\infty\right)$. 

From the fundamental theorem of calculus, one has
\begin{align*}
\left|\cos\left(f(x)\right)-1\right| & \lesssim\int_{x}^{\infty}|\sin\left(f(y)\right)|\left|f'(y)\right|\,dy\\
 & \lesssim\left(\int_{x}^{\infty}\left\langle x\right\rangle ^{2s}\left|f'(y)\right|^{2}\,dy\right)^{\frac{1}{2}}\left(\int_{x}^{\infty}\left\langle x\right\rangle ^{2s}\left|\sin\left(f(y)\right)\right|^{2}\,dy\right)^{\frac{1}{2}}.
\end{align*}
Note that since the integrals $\left(\int_{0}^{\infty}\left\langle x\right\rangle ^{2s}\left|f'(y)\right|^{2}\,dy\right)^{\frac{1}{2}}$
and $\left(\int_{0}^{\infty}\left\langle x\right\rangle ^{2s}\left|\sin\left(f(y)\right)\right|^{2}\,dy\right)^{\frac{1}{2}}$
are both finite, one can pick $M$ large enough such that
\[
\left(\int_{x}^{\infty}\left\langle x\right\rangle ^{2s}\left|f'(y)\right|^{2}\,dy\right)^{\frac{1}{2}}\left(\int_{x}^{\infty}\left\langle x\right\rangle ^{2s}\left|\sin\left(f(y)\right)\right|^{2}\,dy\right)^{\frac{1}{2}}<1
\]
for all $x>M$. In this region, one has
\[
\left|1+\cos\left(f(x)\right)\right|\geq2-\left|\cos\left(f(x)\right)-1\right|>1,\,\,x>M.
\]
It follows that
\begin{align*}
\int_{(M,\infty)}\left\langle x\right\rangle ^{2s}\sin^{2}\left(\frac{f}{2}\right)\,dx&\lesssim\int_{(M,\infty)}\left\langle x\right\rangle ^{2s}\left(\frac{\sin^{2}(f)}{1+\cos\left(f\right)}\right)\,dx\\
&\lesssim \int_{(M,\infty)}\left\langle x\right\rangle ^{2s}\sin^{2}(f)\,dx<\infty.
\end{align*}
Putting the estimates together, we conclude
that
\[
\int\left\langle x\right\rangle ^{2s}\sin^{2}\left(\frac{f}{2}\right)\,dx<\infty
\]
which implies $f\in H^{1,s}_{\text{sin}}$ as desired.
\end{itemize}
\end{proof}

It has been shown in \cite{Kaup75} that if the initial condition belongs to the function space \eqref{space-1}, equation \eqref{Psi-x} admits bounded 
solutions for $z \in \mathbb{R}\setminus \lbrace 0 \rbrace$. There exist unique solutions $\Psi^\pm$ of \eqref{Psi-x} obeying the following space asymptotic conditions
$$\lim_{x \rarr \pm \infty} \Psi^\pm(x,z) e^{x J(z)\sigma_3} = \diagmat{1}{1},$$
and there is a matrix $S(z)$, the scattering matrix, with 
\begin{equation}
\label{Psi-T}
\Psi^+(x,z)=\Psi^-(x,z) S(z).
\end{equation}
The matrix $S(z)$ takes the form
\begin{equation} \label{matrixT}
 S(z) = \twomat{a(z)}{\bb(z)}{b(z)}{\ba(z)}
 \end{equation}
and  the determinant relation gives
$$ a(z)\ba(z) - b(z)\bb(z) = 1. $$
 By  uniqueness, the entries of $\Psi^\pm$ have the following symmetry relations: 
\begin{equation}
\label{symmetry-1}
\psi^\pm_{11}(z)=\overline{\psi^\pm_{22}(\overline{z})}, \quad \psi^\pm_{12}(z)=-\overline{\psi^\pm_{21}(\overline{z})},
\end{equation}
\begin{equation}
\label{symmetry-2}
\psi^\pm_{11}(z)={\psi^\pm_{22}(-{z})}, \quad \psi^\pm_{12}(z)=-\psi^\pm_{21}(-z).
\end{equation}
This leads to the symmetry relation of  the entries of $S$ (cf. \cite[Lemma A.8]{BM08}) :
\begin{align} \label{symmetry}
\ba(z)=\overline{a( \zbar )}, \quad \bb(z) = -\overline{ b(z)}, \quad \ba(z)=\overline{\ba(-z)}, \quad b(z)=\overline{b(-z)}
\end{align}
On $\mathbb{R}$, the determinant of $S(z)$ is given by
$$|a(z)|^2+ |b(z)|^2=1.$$
Making the change of variable 
$$\Psi^\pm=m^\pm e^{-xJ(z)\sigma_3 }$$
the system \eqref{L} then becomes 
\begin{equation}
\label{AKNS-m}
m_x^\pm= \left[ m^\pm, J(z)\sigma_3 \right]+Um^\pm.
\end{equation}
The standard AKNS method starts with the following two Volterra integral equations for the real $z$:
\begin{equation}
\label{IE-m-pm}
m^{\pm}(x, z)=I + \int_{\pm\infty}^x e^{(y-x) J(z) \ad \sigma_3  } \left[ U(y, z)m^{\pm}(y,z) \right] dy.
\end{equation}
and as a consequence (cf. \cite[(3.5c)]{DZ03})
\begin{equation}
\label{T-int}
S(z)=I - \int_\bbR e^{ yJ(z)\ad \sigma_3 } \left[ U(y, z)m^{+}(y,z) \right] dy.
\end{equation}
In \cite{Kaup75} it has been shown that  $m^{(+)}_1$ the first column of $m^+$ and $m^{(-)}_2$
the second column of $m^-$ have analytic extensions to $\bbC^- $. Similarly $m^{(-)}_1$ the first column of $m^-$ and $m^{(+)}_2$
the second column of $m^+$ have analytic extensions to $\bbC^+$.
It is also important to notice that the scattering matrix \eqref{matrixT} remains unchanged under the gauge transformation. So the scattering matrix \eqref{matrixT} can be extended to the entire $\bbR$.

By the standard inverse scattering theory, we formulate the reflection coefficient:
\begin{equation}
\label{reflection}
r(z)=-b(z)/\ba(z), \quad z\in\bbR.
\end{equation}
Also from the symmetry conditions \eqref{symmetry-1}-\eqref{symmetry-2} we deduce that
\begin{equation}
\label{minus}
r(-z)=\overline{r( z )}.
\end{equation}

Note that $\ba(z)$ and $a(z)$ has analytic continuation into the $\bbC^+$ and $\bbC^-$ half planes respectively.
From \eqref{Psi-T} we deduce that 
\begin{equation}
\label{ba-det}
\breve{a}(z)=\det\twomat{\psi^-_{11}(x, z )}{\psi^+_{12}(x, z)}{\psi^-_{21}(x, z )}{\psi^+_{22}(x,z)}.
\end{equation}
\begin{equation}
\label{a-det}
{a}(z)=\det\twomat{\psi^+_{11}(x,z)}{\psi^-_{12}(x,z)}{\psi^+_{21}(x,z)}{\psi^-_{22}(x,z)}.
\end{equation}
From \eqref{symmetry-2}-\eqref{symmetry} we read off directly that if $\ba(z_i)=0$ for some $z_i\in \bbC^+$, then $\overline{\ba(-  \overline{ z_i } )}=0$ by symmetry. Thus if $\ba(z_i)=0$, then either
\begin{itemize}
\item[(i)] $z_i$ is purely imaginary;\\
or
\item[(ii)] $-\overline{z_i}$ is also a zero $\ba$.
\end{itemize}
When $r\equiv 0$, Case (i) above corresponds to kinks/anti-kinks while Case (ii) introduces breathers.

\begin{remark}
\label{remark-generic}
It is proven in \cite{BC84} that for $n \times n$ AKNS spectral problems there is  an  open and dense subset $U_0 \subset L^1(\bbR)$ such that if the matrix potential belongs to $U_0$ , then the number of zeros of $\ba$ are finite,  and these zeros are simple and off the real axis. Later in Section \ref{sec: r} we are going to show that this \emph{generic} spectral property still holds for the system \eqref{L}. We restrict the initial data to such set in this paper.
\end{remark}
\begin{remark}
In \cite{Zhou95}, the author gives a functional analytic proof of the following statement:
if in equation \eqref{IE-m-pm} the kernel $U$ has
{$$\norm{U(x,z)}{L^\infty_z L^1_x(\mathbb{R} ) } <1$$}
then $a$ and $\ba$ have no zeros. Note that $L_x^{2,s}\subset L_x^1$ for $s>\frac{1}{2}$. In particular the small data condition  in the weighted energy space  $H^{1,s}_{\text{sin}}\times L^{2,s}$ with $s>1/2$ implies that the associated potential $U(x,z)$ satisfies the smallness condition in $L^1$ above. This rules out the possibility of breathers that have arbitrarily small norm in the weighted space.
\end{remark}

\begin{figure}[h!]
\caption{Zeros of $\ba$ and $a$ }
\vskip 0.4cm
\begin{center}
\begin{tabular}{ccc}


\setlength{\unitlength}{5.0cm}
\begin{picture}(1,1)

\put(0.5,0.5){\vector(1,0){0.25}}
\put(0.75,0.5){\line(1,0){0.25}}

\put(0.5,1){\line(0,-1){0.25}}
\put(0.5,0.5){\line(0,1){0.25}}

\put(0.5,0.5){\line(-1,0){0.25}}
\put(0.25,0.5){\line(-1,0){0.25}}

\put(0.5,0.25){\line(0,1){0.25}}
\put(0.5,0){\line(0,1){0.25}}

\put(0.5,0.5){\circle{0.025}}


\put(0.9,0.55){$\bbC^+$}
\put(0.9, 0.4){$\bbC^-$}
\put(0.55,0.9){$i\bbR$}




\put(0.7,0.8){{\color{red}\circle*{0.025}}}
\put(0.7,0.2){{\color{blue}\circle*{0.025}}}
\put(0.3,0.8){{\color{red}\circle*{0.025}}}
\put(0.3,0.2){{\color{blue}\circle*{0.025}}}

\put(0.6,0.7){{\color{red}\circle*{0.025}}}
\put(0.6,0.3){{\color{blue}\circle*{0.025}}}
\put(0.4,0.7){{\color{red}\circle*{0.025}}}
\put(0.4,0.3){{\color{blue}\circle*{0.025}}}

\end{picture}

 \\[0.2cm]
\end{tabular}

\vskip 0.2cm

\begin{tabular}{ccc}
Origin ({\color{black} $\circ$}) &
zeros of $\ba$ ({\color{red} $\bullet$})	&	
zeros of $a$  ( {\color{blue} $\bullet$}) 
\end{tabular}
\end{center}

\label{fig:spectra}
\end{figure}
Suppose that $\ba(z_i)=0$ for some $z_i\in\bbC^+$,  $i=1,2,..., N$, then we have the linear dependence of the columns :
\begin{align}
\label{b_i}
 \begin{bmatrix}
           \psi^-_{11}(x,z_i) \\
           \psi^-_{21}(x,z_i) \\
        \end{bmatrix}=b_i\begin{bmatrix}
           \psi^+_{12}(x, z_i) \\
           \psi^+_{22}(x,z_i) \\
        \end{bmatrix}
    \end{align}
  \begin{align} 
\label{b_0}
   \begin{bmatrix}
           m^-_{11}(x, z_i) \\
           m^-_{21}(x, z_i) \\
        \end{bmatrix}=b_i\begin{bmatrix}
           m^+_{12}(x, z_i) \\
           m^+_{22}(x, z_i) \\
        \end{bmatrix}e^{2x J(z_i)}.
  \end{align}
\begin{remark}
{
We assume that the zeros of $\ba$ are of order one and $\ba'(z_i)\neq 0$.
}
\end{remark}  

\subsubsection{Inverse Problem}
\label{subsec:inverse2}
In this subsection we construct the Beals-Coifman solutions needed for the RHP.  We need to find certain piece-wise analytic matrix functions. An obvious choice is 
\begin{equation}
 \begin{cases}
\left( m^{(-)}_1, m^{(+)}_2\right), \qquad \text{Im} z>0\\
\\
\left( m^{(+)}_1, m^{(-)}_2\right), \qquad \text{Im} z<0,
\end{cases}
\end{equation}
{ where $m^{(+)}_i$ is the $i$th column of $m^+$ and $m^{(-)}_i$ is the $i$th column of $m^-$, $i=1,2$.} Here we use the fact that 
$m^{(+)}_1$ 
and $m^{(-)}_2$
have analytic extension to $\bbC^- $. Similarly $m^{(-)}_1$ 
and $m^{(+)}_2$
have analytic extension to $\bbC^+$.
We want the solution to the RHP normalized as $x\rightarrow +\infty$, so we set
\begin{equation}
\label{BC}
M(z; x)= \begin{cases}
(m^{(-)}_1, m^{(+)}_2)\twomat{\ba^{-1}}{0}{0}{1}, \qquad \text{Im} z>0\\
(m^{(+)}_1, m^{(-)}_2)\twomat{1}{0}{0}{a^{-1}}, \qquad \text{Im} z<0.
\end{cases}
\end{equation}
We assume $a(z)\neq 0$ for all $z\in\mathbb{R}$ and define
\begin{equation}
\label{def:r}
{r}(z)=-\dfrac{b(z)}{\ba(z)}
\end{equation}
and by symmetry
$$\dfrac{\bb(z)}{a(z)}=\overline{r(z)}.$$
Following the same proof as \cite[Proposition A.14]{BM08}, we conclude that for $z\in \bbR$
\begin{equation}
\label{M+M-}
\left(  \dfrac{ m^{(-)}_1}{\ba}, m^{(+)}_2\right) \twomat{1}{0}{-e^{2xJ(z)}  \dfrac{b(z)}{\ba(z)} }{1}=\left( m^{(+)}_1,  \dfrac{ m^{(-)}_2}{a}\right) \twomat{1}{-e^{-2xJ(z)}\dfrac{\bb(z)}{a(z)} }{0}{1}
\end{equation}
Setting $M_{\pm}(z; x)=\lim_{\epsilon\to 0^+}M( z\pm i\epsilon, x)$, then $M_{\pm}$ satisfy the following jump condition on $\bbR$:
$$M_+(z; x)=M_-(z; x) \Twomat{1+|r(z)|^2}{e^{-2xJ(z)} \overline{r(z)} }{e^{2xJ(z)} r(z) }{1}.$$
Now we can calculate the residue at the pole $z_i$:
\begin{align}
\label{residue1}
\textrm{Res}_{z =z_i}M_{+}(z; x)&=\frac{1}{\breve{a}'(z_i)}\Twomat{m^-_{11}(x,z_i)}{0}{m^-_{21}(x,z_i)}{0}\\
\nonumber
                                    &=\frac{e^{2xJ{(z_i)}}b_i}{\breve{a}'(z_i)}\Twomat{m^+_{12}(x,z_i)}{0}{m^+_{22}(x,z_i)}{0}.
\end{align}
Similarly, at the pole $\overline{z_i}$:
\begin{align}
\label{residue2}
\textrm{Res}_{z =\overline{z_i}}\textrm{M}_{-}(z; x)
=-\frac{e^{-2xJ(\overline{z_i} ) }\overline{b_i}}{{a}'(\overline{z_i})}\Twomat{0}{m^+_{11}(x,\overline{z_i} ) }{0}{m^+_{21}(x,\overline{z_i})}.
\end{align}
If $z_i$ is not purely imaginary, $-\bar{z_i}$ and $-z_i$ are also poles, with 
\begin{align}
\label{residue3}
\textrm{Res}_{z =-\overline{z_i} }M_{+}(z; x)&=\frac{1}{\breve{a}'(-\overline{z_i} )}\Twomat{m^-_{11}(x,-\overline{z_i})}{0}{m^-_{21}(x,-\overline{z_i})}{0}\\
\nonumber
                                    &=\frac{e^{2xJ (-{\overline{z_i}} )}  \overline{b_i}   }{\breve{a}'(-\overline{z_i} )}\Twomat{m^+_{12}(x,-\overline{z_i})}{0}{m^+_{22}(x,-\overline{z_i})}{0}
\end{align}
and
\begin{align}
\label{residue4}
\textrm{Res}_{z =- {z_i}}\textrm{M}_{-}(z; x)
=- \frac{e^{-2x J(- {z_i} )} {b_i}}{{a}'(-{z_i})}\Twomat{0}{m^+_{11}(x, -{z_i} ) }{0}{m^+_{21}(x, -{z_i})}.
\end{align}
By symmetry $\breve{a}'(z_i)=\overline{a'(\overline{z_i})}$,
so we can define the norming constant
$$c_i=\frac{b_i}{\breve{a}'(z_i)}.$$
We will establish the following proposition which characterizes the map between the initial data and the scattering data in Subsection \ref{sec: r} :
{
\begin{proposition}
\label{prop:r}
If $\vec {f}(0)\in \mathcal{I}$ where $\mathcal{I}$ is given in \eqref{space-1}
then
\begin{itemize}
\item[1.] $r(z)\in H^{s}(\bbR)$;
\item[2.] $\lim_{z\to 0}{r(z)}/z=0$.
\end{itemize}
and we denote $r\in H^{s}_0(\bbR):= H^{s}(\bbR) \cap \lbrace r: \lim_{z\to 0}{r(z)}/z=0 \rbrace$ if $r(z)$ has the two properties above.
\end{proposition}
}
The importance of this proposition is two-fold. Not only it plays an important role in the reconstruction of the potential in Proposition \ref{prop:recon} but also it is the key ingredient in obtaining the error term in the long time asymptotics formulas. We define our set of  scattering data:
\begin{equation}
\label{scattering}
\mathcal{S}=\lbrace r(z),\lbrace i\zeta_k, \iota_k \rbrace_{k=1}^{N_1}, \lbrace z_j, c_j \rbrace_{j=1}^{N_2} \rbrace \in H_0^{s}(\bbR) \otimes \mathbb{C}^{2 N_1} \otimes\mathbb{C}^{ 2 N_2} .
\end{equation}
Here $\zeta_k>0$, $z_j=\rho_j e^{i \omega_j}$ with $ \rho_j>0$, $0<\omega_j<\frac{\pi}{2}$ and $\iota_k$ and $c_j$ are non-zero complex numbers playing the role of norming constants.

\subsection{Inverse Problem}
\label{subsec:inverse}

From the scattering data \eqref{scattering}, the standard direct scattering transform \cite[section 2.4]{Cheng} implies that $r(z)$, $c_j$ and $\iota_k$ have linear time evolution:
\begin{equation}
\label{time-evol}
r(z, t)=e^{1/2 (z+1/z)it}r(z), \quad c_j( t)=e^{1/2 (z_j+1/z_j)it}c_j, \quad \iota_k( t)=e^{1/2 (i\zeta_k+1/(i\zeta_k))it} \iota_k .
\end{equation}
 The long time asymptotics of sG equation will be obtained through a sequence of transformations of the following RHP:
\begin{problem}
\label{RHP-1}
For fixed $x\in\bbR$ and  $r(z)$ satisfying the two properties in Proposition \ref{prop:r}, find a meromorphic matrix $M(z, x, t)$ on $\bbC \setminus \bbR$ satisfying the following conditions:
\begin{enumerate}
\item[(i)] (Normalization) $M(z; x, t)\to I+\mathcal{O}(z^{-1})$ as $z\to \infty$.
\item[(ii)] (Jump relation) For each $z \in \bbR$, $M(z, x, t)$ has continuous non-tangential boundary value $M_\pm(z; x, t)$ as $z$ approaches $\bbR$ from $\bbC^\pm$ and the following jump relation holds
\begin{align}
\label{jump}
M_+(z; x, t) &=M_-(z; x, t)e^{-i\theta(z; x, t)\ad\sigma_3}v(z)\\
                   &=M_-(z; x, t)v_\theta(z)
\end{align}
where
$$v(z)=\Twomat{1+|r(z)|^2}{\overline{r(z)}} {r(z) }{1}$$
and 
\begin{equation}
\label{theta}
\theta(z; x, t)=\dfrac{1}{4}\left(  \left( z-\dfrac{1}{z}  \right)\dfrac{x}{t} +\left( z+\dfrac{1}{z}  \right) \right)t
\end{equation}
\item[(iii)] (Residue condition) For $k=1,2..., N_1$, $M(z; x, t)$ has simple poles at each $i\zeta_k, \overline{i\zeta_k}$ with
\begin{equation}
\label{res-1}
\Res_{i\zeta_k}M=\lim_{z\to i\zeta_k}M\twomat{0}{0}{e^{2i\theta(i\zeta_k)} \iota_k}{0}
\end{equation}
\begin{equation}
\label{res-2}
\Res_{\overline{i\zeta_k} }M=\lim_{z\to \overline{i\zeta_k} }M\twomat{0}{-e^{-2i\theta(\overline{i\zeta_k})} \overline{\iota_k}}{0}{0}.
\end{equation}
For $j=1,2,..., N_2$, $M(z; x, t)$ has simple poles at each $\pm z_j, \pm\overline{z_j}$ with
\begin{equation}
\label{res-3}
\Res_{z_j}M=\lim_{z\to z_j}M\twomat{0}{0}{e^{2i\theta(z_j)} c_j}{0},
\end{equation}
\begin{equation}
\label{res-4}
\Res_{\overline{z_j} }M=\lim_{z\to \overline{z_j} }M\twomat{0}{-e^{-2i\theta(\overline{z_j})} \overline{c_j}}{0}{0},
\end{equation}
\begin{equation}
\label{res-5}
\Res_{-z_j}M=\lim_{z\to -z_j}M\twomat{0}{e^{-2i\theta(-z_j)} c_j}{0} {0},
\end{equation}
\begin{equation}
\label{res-6}
\Res_{-\overline{z_j} }M=\lim_{z\to -\overline{z_j} }M\twomat{0}{0}{-e^{2i\theta(-\overline{z_j} )} \overline{c_j}}{0}.
\end{equation}
\end{enumerate}
\end{problem}
\begin{definition}
\label{genericity}
We say that the initial condition $\vec{f}_0$ is \emph{generic} if 
\begin{enumerate}
\item[1.] $\ba(z)$ and $a(z)$ associated to $\vec{f}_0$ only have finite and simple zeroes as stated in Remark \ref{remark-generic}.
\item[2.] For $\lbrace i\zeta_k \rbrace_{k=1}^{N_1}$ and  $\lbrace z_j \rbrace_{j=1}^{N_2}$ where  $z_j=\rho_j e^{i\omega_j}$, 
$$ \rho_j\neq \zeta_k, \, \rho_{j_1}\neq \rho_{j_2} $$
for all $j$, $k$. This will avoid the unstable structure in which kinks/breathers travel in the same velocity.
\end{enumerate}
\end{definition}
	\begin{remark}\label{rem:def1}
As it will become clear later on in our computations, the first requirement in the definition above is essential and the second one is just for the sake of simplicity. Our computations can also be applied without the second condition. For the NLS setting, see \cite{BJM16} \end{remark}
\begin{remark}
\label{re-arrange}
We arrange eigenvalues $\lbrace i\zeta_k \rbrace_{k=1}^{N_1}$ and  $\lbrace z_j \rbrace_{j=1}^{N_2}$ in the following way:
\begin{enumerate}
\item For $i\zeta_k$, $\zeta_k>0$, we have $\zeta_1<\zeta_2<...<\zeta_k<...<\zeta_{N_1}$.
\item For $\rho_j>0$, we have 
$$\rho_1<\rho_2<...<\rho_j<...<\rho_{\footnotesize{ N_2}}.$$
\end{enumerate}
\end{remark}

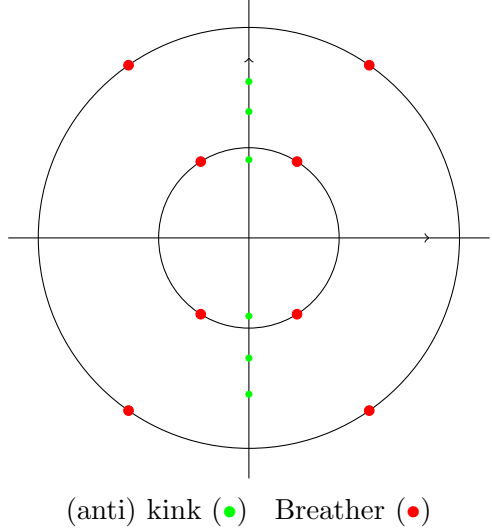
\begin{figure}[h]
\caption{kinks and breathers}
\begin{tikzpicture}[scale=0.8]
\draw [->] (-4,0)--(3,0);
\draw (4,0)--(3,0);
\draw [->] (0,-4)--(0,3);
\draw (0,3)--(0,4);
  
\draw	(0, 0)		circle[radius=1.5];	    
\draw	(0, 0)		circle[radius=3.5];	    

 \draw	[fill, green]  (0, 1.3)		circle[radius=0.05];	    
\draw	[fill, green]  (0, 2.1)		circle[radius=0.05];	    
\draw	[fill, green]  (0, 2.6)		circle[radius=0.05];	    

 \draw	[fill, green]  (0, -1.3)		circle[radius=0.05];	    
 \draw	[fill, green]  (0, -2)		circle[radius=0.05];	    

\draw	[fill, green]  (0, -2.6)		circle[radius=0.05];	   

\draw	[fill, red]  (2, 2.872 )		circle[radius=0.08];	 
\draw	[fill, red]  (-2, 2.872 )		circle[radius=0.08];	 
\draw	[fill, red]  (2, -2.872 )		circle[radius=0.08];	 
\draw	[fill, red]  (-2, -2.872 )		circle[radius=0.08];	 

\draw	[fill, red]  (0.8, 1.269 )		circle[radius=0.08];	 
\draw	[fill, red]  (0.8, -1.269 )		circle[radius=0.08];	 
\draw	[fill, red]  (-0.8, 1.269 )		circle[radius=0.08];	 
\draw	[fill, red]  (-0.8, -1.269 )		circle[radius=0.08];	 
    \end{tikzpicture}
 \begin{center}
  \begin{tabular} {ccc}
(anti) kink  ({\color{green} $\bullet$})	&	  
Breather ({\color{red} $\bullet$}) 
\end{tabular}
 \end{center}
\end{figure}

\begin{remark}
For each pole $i\zeta_k$ (or $z_j$) $ \in \mathbb{C}^+$, let $\tgamma_k$ (respectively $\gamma_j$) be a circle centered at $i\zeta_k$ ($z_j$) of sufficiently small radius to be in the open upper half-plane and to be disjoint from all other circles. By doing so we replace the residue conditions \eqref{res-1}-\eqref{res-6} of the Riemann-Hilbert problem with Schwarz invariant jump conditions across closed contours (see Figure \ref{figure-zeta}). The equivalence of this new RHP on augmented contours with the original one is a well-established result (see \cite{Zhou89} Sec 6). The purpose of this replacement is to
\begin{enumerate}
\item make use of the \textit{vanishing lemma} from \cite[Theorem 9.3]{Zhou89} ;
\item formulate the Beals-Coifman representation of the solution of \eqref{eq: sG}.
\end{enumerate}
We now rewrite the jump conditions of Problem \ref{RHP-1}:
$M(z; x, t)$ is analytic in $\mathbb{C}\setminus \Sigma$ where 
$$\Sigma= \bbR\cup  \left( \bigcup_{k=1}^{N_1} \tgamma_k  \right) \cup \left( \bigcup_{k=1}^{N_1} \tgamma_k^*  \right) \cup \left( \bigcup_{j=1}^{N_2} \pm \gamma_j  \right)\cup \left( \bigcup_{j=1}^{N_2} \pm \gamma_j^*  \right)$$
 is given in figure \ref{figure-zeta} below  and has continuous boundary values $M_\pm$ on $\Sigma$
and  $M_\pm$ satisfy
$$ M_+(z; x, t)=M_-(z; x, t)e^{-i\theta(z; x, t)\ad\sigma_3}v(z)$$
where 
\begin{align*}
v(z)=\Twomat{1+|r(z)|^2}{\overline{r(z)}} {r(z) }{1}, \quad z\in \bbR
\end{align*} 
and 
$$
v(z) = 	\begin{cases}
						\twomat{1}{0}{\dfrac{\iota_k }{z-i\zeta_k}}{1}	&	z\in \tgamma_k, \\
						\\
						\twomat{1}{\dfrac{\overline{\iota_k}}{z -\overline{i\zeta_k }}}{0}{1}
							&	z \in \tgamma_k^*
					\end{cases}
$$
and
$$
v(z) = 	\begin{cases}
						\twomat{1}{0}{\dfrac{c_j}{z-z_j}}{1}	&	z\in \gamma_j, \\
						\\
						\twomat{1}{\dfrac{\overline{c_j}}{z -\overline{z_j}}}{0}{1}
							&	z \in \gamma_j^*,\\
							\\
							\twomat{1}{\dfrac{-c_j}{z + z_j}}{0}{1}	&	z\in -\gamma_j, \\
						\\
						\twomat{1}{0}{\dfrac{-\overline{c_j}}{z +\overline{z_j }}}{1}
							&	z \in -\gamma_j^*.
					\end{cases}
$$					
\begin{figure}
\caption{The Augmented Contour $\Sigma$}
\vspace{.5cm}
\label{figure-zeta}
\begin{tikzpicture}[scale=0.75]
\draw[ thick] (0,0) -- (-3,0);
\draw[ thick] (-3,0) -- (-5,0);
\draw[thick,->,>=stealth] (0,0) -- (3,0);
\draw[ thick] (3,0) -- (5,0);
\node[above] at 		(2.5,0) {$+$};
\node[below] at 		(2.5,0) {$-$};
\node[right] at (3.5 , 2) {$\gamma_j$};
\node[right] at (3.5 , -2) {$\gamma_j^*$};
\node[left] at (-3.5 , 2) {$-\gamma_j^*$};
\node[left] at (-3.5 , -2) {$-\gamma_j$};
\draw[->,>=stealth] (-2.6,2) arc(360:0:0.4);
\draw[->,>=stealth] (3.4,2) arc(360:0:0.4);
\draw[->,>=stealth] (-2.6,-2) arc(0:360:0.4);
\draw[->,>=stealth] (3.4,-2) arc(0:360:0.4);
\draw [red, fill=red] (-3,2) circle [radius=0.05];
\draw [red, fill=red] (3,2) circle [radius=0.05];
\draw [red, fill=red] (-3,-2) circle [radius=0.05];
\draw [red, fill=red] (3,-2) circle [radius=0.05];
\node[right] at (5 , 0) {$\bbR$};
\draw [green, fill=green] (0, 1) circle [radius=0.05];
\draw[->,>=stealth] (0.3, 1) arc(360:0:0.3);
\draw [green, fill=green] (0, -1) circle [radius=0.05];
\draw[->,>=stealth] (0.3, -1) arc(0:360:0.3);
\draw [red, fill=red] (1,3) circle [radius=0.05];
\draw[->,>=stealth] (1.2, 3) arc(360:0:0.2);
\draw [red, fill=red] (-1,3) circle [radius=0.05];
\draw[->,>=stealth] (-0.8, 3) arc(360:0:0.2);
\draw [red, fill=red] (1,-3) circle [radius=0.05];
\draw[->,>=stealth] (1.2, -3) arc(0:360:0.2);
\draw [red, fill=red] (-1,-3) circle [radius=0.05];
\draw[->,>=stealth] (-0.8, -3) arc(0:360:0.2);

\node[above] at 		(0, 1.3) {$\tgamma_k$};
\node[below] at 		(0, -1.3) {$\tgamma_k^*$};
\end{tikzpicture}
\begin{center}
\begin{tabular}{ccc}
kink/antikink ({\color{green} $\bullet$})	&	
Breather ({\color{red} $\bullet$}) 
\end{tabular}
\end{center}
\end{figure}
It is well-known that $e^{-i\theta(z; x, t)\ad\sigma_3}v(z)$ admits the standard triangular factorization :
$$e^{-i\theta(z; x, t)\ad\sigma_3}v(z)=(I-w^-_{\theta})^{-1}(I+w^+_{\theta})$$ for triangular matrices $w_\theta^{\pm}$:

$$ w_\theta^+ = \lowermat{  { r(z) } e^{2i\theta(z)}}, \quad w_\theta^- = \uppermat{\overline{r(z)}e^{-2i\theta(z)}}\quad z\in\bbR $$
$$w_\theta^+ = \lowermat{\dfrac{\iota_k  e^{2i\theta(i\zeta_k) }}{z-i\zeta_k}}, \quad w_\theta^- = \lowermat{0}\quad z\in\tgamma_k $$
$$w_\theta^+ = \lowermat{0},\quad w_\theta^- =\uppermat{\dfrac{ \overline{\iota_i} \, e^{-2i \theta(\overline{i\zeta_k})} }{z-\overline{i\zeta_k}}}\quad z\in\tgamma^*_i $$

$$w_\theta^+ = \lowermat{\dfrac{c_j  e^{2i\theta(z_j) }}{z-z_j}}, \quad w_\theta^- = \lowermat{0}\quad z\in\gamma_j $$
$$w_\theta^+ = \lowermat{0},\quad w_\theta^- =\uppermat{\dfrac{ \overline{c_j} \, e^{-2i \theta(\overline{z_j})} }{z-\overline{z_j}}}\quad z\in\gamma^*_j$$
$$w_\theta^+ = \lowermat{0},\quad w_\theta^- =\uppermat{\dfrac{ -{c_j} \, e^{-2i \theta({-z_j})} }{z+ {z_j}}}\quad z\in-\gamma_j$$
$$\quad w_\theta^+ =\lowermat{\dfrac{ -\overline{c_j} \, e^{2i \theta(-\overline{z_j})} }{z+\overline{z_j}}}, \quad w_\theta^- = \uppermat{0},\quad z\in-\gamma^*_j.$$
We define {for $z\in \Sigma$}
\begin{align*}
\mu= M_-(I-w_\theta^- )^{-1}=M_+(I+w_\theta^+)^{-1}.
\end{align*}
Then the solvability of the RHP above is equivalent to the solvability of the following Beals-Coifman integral equation:
\begin{align}
\label{BC-int}
\mu(z; x,t) &= I +C_{w_\theta}\mu(z; x,t)\\
\nonumber
                &= I+C^+_\Sigma\mu w_\theta^- +C^-_\Sigma\mu w_\theta^+\\
                \nonumber
        &=I+C^+_\bbR\mu w_\theta^- +C^-_\bbR\mu w_\theta^+\\
        \nonumber
        &\qquad+ 
        \Twomat{\sum_{k=1}^{N_1} \dfrac{\mu_{12}(i\zeta_k )\iota_k e^{2i\theta(i\zeta_k)} }{z-i\zeta_k}  }
        { -\sum_{k=1}^{N_1} \dfrac{\mu_{11}( \overline{i\zeta_k}) {\overline{\iota_k}} e^{-2i\theta(   \overline{i\zeta_k }  )} }{z- \overline{i\zeta_k }}   }
        {\sum_{k=1}^{N_1} \dfrac{\mu_{22}(i\zeta_k )\iota_k e^{2i\theta(i\zeta_k )} }{z-i\zeta_k}  }
        {-\sum_{k=1}^{N_1} \dfrac{\mu_{21}( \overline{i\zeta_k }) {\overline{\iota_k}} e^{-2i\theta(   \overline{i\zeta_k}  )} }{z- \overline{i\zeta_k}} }\\
         \nonumber
        &\qquad + \Twomat{\sum_{j=1}^{N_2} \dfrac{\mu_{12}(z_j)c_j e^{2i\theta(z_j)} }{z-z_j}  }
        { -\sum_{j=1}^{N_2} \dfrac{\mu_{11}( \overline{z_j }) {\overline{c_j }} e^{-2i\theta(   \overline{z_j }  )} }{z- \overline{z_j }}   }
        {\sum_{j=1}^{N_2} \dfrac{\mu_{22}(z_j)c_j e^{2i\theta(z_j )} }{z-z_j }  }
        {-\sum_{j=1}^{N_2} \dfrac{\mu_{21}( \overline{z_j }) {\overline{c_j}} e^{-2i\theta(   \overline{z_j}  )} }{z- \overline{z_j}} }\\
         \nonumber
         &\qquad +\Twomat
        { -\sum_{j=1}^{N_2} \dfrac {\mu_{12}( -\overline{z_j }) {\overline{c_j }} e^{2i\theta(  - \overline{z_j }  )} }{z+ \overline{z_j }}   }
        {\sum_{j=1}^{N_2} \dfrac{\mu_{11}(-z_j)c_j e^{-2i\theta(-z_j)} }{z+z_j}  }
       {-\sum_{j=1}^{N_2} \dfrac{\mu_{22}( -\overline{z_j }) {\overline{c_j}} e^{2i\theta(  - \overline{z_j}  )} }{z + \overline{z_j}} }
        {\sum_{j=1}^{N_2} \dfrac{\mu_{21}(-z_j)c_j e^{-2i\theta(-z_j )} }{z+z_j }  }.
\end{align}
\begin{proposition}
\label{prop:recon}
Suppose that the reflection coefficient $r(z)\in H^{s}_0(\bbR)$ (cf. Proposition \ref{prop:r}), then Problem \ref{RHP-1} has a unique solution and the solution $M$ admits the following expansion in $z$:
\begin{equation}
\label{small z}
M(z; x,t)=M_0( x,t) + M_1 (z; x,t).
\end{equation}
 { Let $f$ be the unique solution to equation \eqref{eq: sG} (Theorem \ref{thm:gwpweighted}). Then from \eqref{small z} we obtain}
\begin{align}
\label{sin}
\sin f &=2m_{21}m_{22}\\
\label{cos}
\cos f&=1+2m_{12}m_{21}
\end{align}
where $m_{ij}$ are entries of
$$M_0 :=\twomat{m_{11}}{m_{12}}{m_{21}}{m_{22}}.$$
\end{proposition}
\begin{proof}
The existence and uniqueness of solution to Problem \ref{RHP-1} is standard \cite[P.895, problem 2.1]{BJM16}. By construction, we can see that $M(z; x, t) $ solves the equation \eqref{AKNS-m} for $z\in \bbC\setminus \lbrace 0 \rbrace$. To control the behavior of $M(z; x, t)$ near the origin, we invoke to the singular integral representation of $M(z; x, t)$:
\begin{align*}
M(z; x, t) &=I+ \dfrac{1}{2\pi i}\int_{\Sigma} \dfrac{\mu (w_\theta^+ + w_\theta^-)  }{s-z} ds\\
              &=\underbrace{\left( I+  \dfrac{1}{2\pi i}\int_{\Sigma} \dfrac{\mu (w_\theta^+ + w_\theta^-)  }{ s} ds \right) }_{M_0}+ \underbrace{ \left(   \dfrac{1}{2\pi i}\int_{\Sigma} \dfrac{\mu (w_\theta^+ + w_\theta^-)  }{s-z} ds -   \dfrac{1}{2\pi i}\int_{\Sigma} \dfrac{\mu (w_\theta^+ + w_\theta^-)  }{ s} ds  \right) }_{M_1} \\
              &=: M_0( x,t) + M_1 (z; x,t).
\end{align*}
The integrals make sense since $r(z)\in H^{s}_0(\bbR)$ and the fact that $\mu-I\in L^2_\Sigma$.  Following the same procedure in \cite[Theorem 4]{CVZ99},  we deduce that
\begin{align*}
M(z; x, t)_x &= z\left[J_1,   M_0( x,t) + M_1 (z; x,t)\right] + z^{-1} \left[J_2,   M_0( x,t) + M_1 (z; x,t)\right] \\
                 &= Q_1 \left(   M_0( x,t) + M_1 (z; x,t) \right) + z^{-1} Q_2  \left(   M_0( x,t) + M_1 (z; x,t) \right)
\end{align*}
where $J_1, J_2, Q_1, Q_2$ are the same as those in \cite[Theorem 4]{CVZ99}. If we multiply both sides by $z$ and assume
\begin{equation}
\label{lim-m1}
\lim_{z\to 0} M_1(z; x,t)=0
\end{equation}
and
\begin{equation}
\label{lim-m}
\lim_{z\to 0} z M(z; x,t)_x=0,
\end{equation}
we can establish
$$[J_2, M_0]+Q_2 M_0=0$$
from which \eqref{sin} and \eqref{cos} can be directly read off. The proof of \eqref{lim-m1} is a simple consequence of the dominated convergence theorem. To see this, we let
$z=i\sigma$ where $\sigma>0$, then
\begin{align*}
\left\vert \dfrac{r(s)}{s-z} \right\vert &= \left\vert \dfrac{(s + i\sigma)r(s)}{s^2+\sigma^2} \right\vert\\
                         &\leq \left\vert \dfrac{r(s)}{s} \right\vert + \left\vert \dfrac{\sigma r(s)}{s^2+\sigma^2} \right\vert\\
                          &\leq  \left\vert \dfrac{r(s)}{s} \right\vert + \left\vert \dfrac{\sigma s }{s^2+\sigma^2}\right\vert \left\vert \dfrac{r(s)}{s} \right\vert\\
                          & \leq \dfrac{3}{2}  \left\vert \dfrac{r(s)}{s} \right\vert.
\end{align*}
The last term is in $L^1$ by (1) and (2) in Proposition \ref{prop:r} applied to the regions with $|s|$ large and  $|s|$  small respectively.  For \eqref{lim-m}, using \eqref{BC-int}  a direct computation gives:
\begin{align*}
z M(z; x,t)_x &= \dfrac{1}{2\pi i}\int_{\bbR} \dfrac{z \mu_x (w_\theta^+ + w_\theta^-)  }{s-z} ds + \dfrac{1}{4\pi }\int_{\bbR} \dfrac{ zs \mu w_\theta^+   }{s-z} ds-\dfrac{1}{4\pi }\int_{\bbR} \dfrac{ zs \mu w_\theta^-   }{s-z}  \\
              & \quad  - \dfrac{1}{4\pi }\int_{\bbR} \dfrac{z\mu w_\theta^+   }{ s( s-z )} ds + \dfrac{1}{4\pi }\int_{\bbR} \dfrac{z\mu w_\theta^-  }{ s( s-z )} ds \\
              & \quad + z \times [\text{ $x$-derivative of the discrete part of \eqref{BC-int}}].
\end{align*}
We first mention that we ignore derivative of the discrete part since this part obviously vanishes as $z\to0$. The first two integrals hold and have zero limit since $\mu_x \in L^2_\bbR$ (see \cite[Lemma 6.2.2]{Liu}) and $r\in L^{2,1}(\bbR) $ by Proposition \ref{prop:r}. For the last integral, we notice that 
$$ \dfrac{1}{4\pi }\int_{\bbR} \dfrac{z\mu w_\theta^\pm  }{ s( s-z )} ds=  \dfrac{1}{4\pi } \left( \int_{\bbR} \dfrac{\mu w_\theta^\pm  }{s-z} ds -  \int_{\bbR} \dfrac{\mu w_\theta^\pm  }{ s} ds  \right)$$
which again goes to zero by the dominated convergence theorem similarly to the argument above with the explicit formulas of $w^{\pm}_{\theta}$ and the fact $(\mu-1)\in L^2$.
\end{proof}
\end{remark}

\subsection{Single kink/anti-kink and single breather solution}
If we assume $r\equiv 0$ and $\ba$ has exactly one simple zero at $z=i\zeta$, $\zeta>0$ and let $c=ib$ be the norming constant. Notice that $c$ is purely imaginary
and equation \eqref{eq: sG} admits the following single kink solution \cite{Cheng} :
\begin{subequations}
\begin{equation}
\label{kink-1}
\cos f=1+2m_{12}m_{21}=1-2\dfrac{b^2 e^{4\theta(i\zeta)}}{ \left( \zeta+ \dfrac{ b^2 e^{4\theta(i\zeta)} }{  4\zeta }   \right)^2 }
\end{equation}
\begin{equation}
\label{kink-2}
\sin f=2m_{21}m_{22}= \dfrac{2b e^{2\theta(i\zeta)}  \left( \zeta- \dfrac{ b^2 e^{4\theta(i\zeta)} }{  4\zeta }   \right)  }{ \left( \zeta+ \dfrac{ b^2 e^{4\theta(i\zeta)} }{  4\zeta }   \right)^2 }.
\end{equation}
\end{subequations}
If $b<0$ we have the kink solution and $b>0$ we have the anti-kink solution.

If we assume $r\equiv 0$ and $\ba$ has exactly two simple zeros at $z=\rho e^{i\omega}$ and $-\rho e^{-i\omega}$ and let $c=e^{p+iq}$ be the norming constant, then Equation \eqref{eq: sG} admits the following one-breather solution \cite[chapter 4]{Cheng} :
\begin{equation}
\label{1-breather}
f(x,t)=-4\arctan \left(  \dfrac{1}{2}e^{-K}\cos\left(  \dfrac{\cos\omega}{2} \theta(\rho)-q-\omega \right) \text{sech} \left( \dfrac{\sin \omega}{2} \tilde{\theta}(\rho)   -p-K \right)  \right) 
\end{equation}
with
\begin{align*}
e^{2K}=\dfrac{\cot^2 \omega}{4}.
\end{align*}
and
\begin{align*}
\tilde{\theta}(z; x,t)=\dfrac{1}{4} \left(  \left(  z+\dfrac{1}{z}  \right )x+\left( z-\dfrac{1}{z}  \right)t  \right)
\end{align*}
From above we observe that kink/anti-kink has velocity $\textrm{v}_k=(1-\zeta^2)/(1+\zeta^2)$,  breather has velocity $\textrm{v}_b=(1-\rho^2)/(1+\rho^2)$.

Now we make some observation on the stationary points given by the oscillatory phase and the velocities of kinks/anti-kinks and breathers.   We focus on velocities of breathers since the cases with kinks/anti-kinks are simpler. 

For a general complex number  $z=\xi+i\eta$, the real part of the phase is given by
\begin{equation}
\label{theta-re}
\text{Re}i\theta(z; x, t)=\dfrac{1}{4} \left[-\left(   1+\dfrac{x}{t} \right)\eta t +  \left( 1-\dfrac{x}{t} \right)\dfrac{\eta t}{\xi^2+\eta^2}   \right]
\end{equation}whence, for a given reference frame $\frac{x}{t}$
the stationary points  are 
\begin{equation}
\label{stationary}
\pm z_0=\pm \sqrt{\dfrac{t-x}{x+ t}  } .
\end{equation}
Let's fix the velocity $\textrm{v}_b=(1-\rho^2)/(1+\rho^2)$ of a breather and consider the reference frame with the same velocity $x/t=\textrm{v}_b$. Then the circle give by $\{z=\xi+i\eta: \xi^2+\eta^2=\rho^2\}$ passes through the stationary points \eqref{stationary}.

Conversely, if $a(z)$ has zeros on the circle 
$$ \xi^2+\eta^2=\left\vert \dfrac{t-x}{t+x} \right\vert$$ in the complex plane for some $x, t\in\mathbb{R}$. Then we expect the corresponding breather moves with the velocity $x/t$.

\begin{remark}
\label{painleve-soliton}
In \eqref{theta-re} if we choose $x>t$, then we have $z_0$ purely imaginary. And there is no kink/anti-kink and breathers in this region since $\textrm{v}_b,\textrm{v}_k<1$.
It is also clear that if we set $$x/t=\textrm{v}_{b_j}=(1-\rho^2_j)/(1+\rho^2_j),$$ then $ \text{Re} i\theta(z_j; x, t)=0$. This also give an alternative explanation that breathers and kinks/anti-kinks have no decay in time (up to translations).
\end{remark}

\section{The direct scattering map}
\label{sec: r}
In this section we prove Proposition \ref{prop:r}. As it will become clear later, the Sobolev bijectivity approach of \cite{Zhou98} will give a unified way of proof. 
Because of the factor $1/z$ in the spectral problem \eqref{Psi-x}, we divide our approach into two cases: $|z|<3/4$  and $|z|>1/2$. The purpose of the two constants are:
\begin{itemize}
    \item[1.] avoiding singularities during the analysis of the scattering map;
    \item[2.] guaranteeing the continuity of the scattering data and their derivatives.
\end{itemize}

\subsection{Away from the origin}For $z\in I_\infty\equiv \bbR\setminus [-1/2, 1/2]$,
 \eqref{IE-m-pm} is equivalent to the following integral equations:
\begin{subequations}
\begin{equation}
\label{m+}
m^+(x, z)=I+\int_{+\infty}^x e^{ (y-x)J(z)\ad \sigma_3 } \left[ U(y, z)m^+(y, z) \right] dy
\end{equation}
\begin{equation}
\label{m-}
m^-(x, z)=I+\int_{-\infty}^x e^{ (y-x)J(z)\ad \sigma_3} \left[  U(y, z)m^-(y, z) \right] dy
\end{equation}
\end{subequations}
and consequently the scattering matrix $S$ is given by:
\begin{equation}
\label{T-int1}
S(z)=I - \int_\bbR e^{ yJ(z)\ad \sigma_3}\left[ U(y, z)m^+(y, z) \right] dy.
\end{equation}
\begin{lemma}
\label{lemma:L2}
Suppose that $\vec {f}(0)\in \mathcal{I}$ where $\mathcal{I}$ is given in \eqref{space-1}
then $r(z)\in L^2(I_\infty) \cap L^\infty(I_\infty)$.
\end{lemma}
\begin{proof}
Clearly, we have the following operator bound:
\begin{equation}
    \left\|K^+_U \right\|_{L^{\infty}_x \rightarrow L^{\infty}_x} \leq\|U\|_{L^{1}_x(\bbR^+,   L^\infty_z(I_\infty))}
\end{equation}
where 
the explicit form of $K^+_U$ is given in \eqref{K_U-I}. And we can rewrite \eqref{m+} as
\begin{equation}
m^+(x, z)=\left(\mathbf{1}  -K^+_U\right)^{-1} \left( K^+_U I \right).
\end{equation}
And standard Volterra theory implies that 
\begin{equation}
\label{m-infty}
    \norm{m^+(x, z)}{L^\infty_x\left(\bbR^+; L^\infty_z(I_\infty\right) }< \infty.
\end{equation}
Recall from \eqref{def:r} the definition of the reflection coefficient $r$, to show that 
\begin{equation}
\label{est: r-1}
r(z)\in L^2(I_\infty) \cap L^\infty(I_\infty)
\end{equation}
we only need $S(z)-I\in L^2(I_\infty) \cap L^\infty(I_\infty)$.
We split $S(z)$ into three parts
\begin{align}
\label{S-K}
S(z) &=I-\int_\bbR e^{ yJ(z)\ad \sigma_3}\left[ U(y)\left(m^+(y, z) -I\right) \right] dy-\int_\bbR e^{ yJ(z)\ad \sigma_3}U(y)dy\\
\nonumber
                  &=I-\int_\bbR e^{ yJ(z)\ad \sigma_3 }\left[ U(y)\left(\mathbf{1}  -K^+_U\right)^{-1} \left( K^+_U I \right)\right] dy - \int_\bbR e^{ yJ(z)\ad \sigma_3}U(y)dy\\
                  &=:I- S_{1}(z) - S_{2}(z)
\end{align}
From \eqref{m-infty} it is easy to deduce that $S(z)-I\in L^\infty(I_\infty)$.  We then show that $S_2(z) \in L^2_z(I_\infty)$. Explicitly, $S_2(z)$ is given by the following formula:
\begin{equation}
S_2(z)=\twomat{\int_\bbR \dfrac{i}{4z} (\cos f-1) dy  }{ \int_\bbR e^{ iy (z/2-1/(2z)) } \left[\dfrac{i}{4z} \sin f-\dfrac{1}{4}(f_x+f_t) \right] dy }{ \int_\bbR e^{-iy (z/2-1/(2z)) } \left[\dfrac{i}{4z} \sin f +\dfrac{1}{4}(f_x+f_t) \right] dy}{\int_\bbR \dfrac{i}{4z} (-\cos f+1) dy}
\end{equation}
We only have to consider part of the off-diagonal entries without the $1/z$ factor:
\begin{equation}
\int_\bbR  e^{-iy (z/2-1/(2z)) }  (f_x+f_t) dy.
\end{equation}
Letting $\lambda=z-1/z$, then for $z\in (1/2, \infty)$, $\lambda\in (-3/2, \infty)$. Since we have the $L^\infty$ bound on $I_\infty \equiv \bbR\setminus [-1/2, 1/2] $, we only work on $z\in (1, \infty)$ which means $\lambda>1$. We first observe that
\begin{equation}
g(\lambda)=\int_{\bbR}  e^{-(1/2)iy \lambda }  (f_x(y)+f_t(y) ) dy
\end{equation}
belongs to $L^2_\lambda(\bbR)$ by Plancherel's theorem. Then let $\mathbf{f}(z)=g(z-1/z)=g(\lambda)$, we have 
\begin{equation}
\label{z-lambda}
\int_1^{\infty}|\mathbf{f}(z)|^2 dz\leq 
\int_{\bbR^+} |g(\lambda)|^2 \left( \dfrac{1}{2} +\dfrac{\lambda}{2\sqrt{\lambda^2+4}} \right) d\lambda <\infty.
\end{equation}
For  $S_1(z)$, we first study 
\begin{align}
\label{K_U-I}
 (K^+_U I)(y, z) & =\int_{+\infty}^y e^{(  w-y  )J(z)\ad \sigma_3} U(y) dw \\
 \nonumber
                       &=(K^{1,+}_{U_1} I)(y, z) +(K^{2,+}_{U_2} I)(y, z) \\
                       \nonumber
                       & :=\twomat{ \int_{+\infty}^y \dfrac{i}{4z} (\cos f-1) dw  }{ \int_{+\infty}^y e^{ i( w-y ) (z/2-1/(2z)) } \left[\dfrac{i}{4z} \sin f \right] dw }{ \int_{+\infty}^y e^{-i ( w-y ) (z/2-1/(2z)) } \left[\dfrac{i}{4z} \sin f \right] dw }{\int_{+\infty}^y \dfrac{i}{4z} (-\cos f+1) dw }\\
                       \nonumber
                       & + \twomat{0 }{ \int_{+\infty}^y  e^{ i ( w-y )  (z/2-1/(2z)) } \left[-\dfrac{1}{4}(f_x+f_t) \right] dw }{ \int_{+\infty}^y  e^{-i ( w-y )  (z/2-1/(2z)) } \left[\dfrac{1}{4}(f_x+f_t) \right] dw }{0}\
\end{align}
We only have to study the terms without the $1/z$ factor in the potential:
\begin{equation}
(K^{2,+}_{U_2} I)(y, z)_{21} :=\int_{+\infty}^y  e^{i( w-y )\ (z/2-1/(2z)) }  (f_x(w)+f_t(w)) dw.
\end{equation}
Again we set $\lambda=z-1/z $ and by the standard Fourier theory,
\begin{align}
 \label{test-phi}
 \left\Vert(K^{2,+}_{U_2} I)(y, \cdot )_{21} \right\Vert_{L^\infty_y L^2_\lambda} & =\sup_{\substack{\phi\in C_0^\infty\\ \left\Vert \phi \right\Vert_{L^2_\lambda}=1 }} \left\vert   \int_\bbR \phi(\lambda)  \left( \int_{+\infty}^y  e^{ (1/2) i( w-y ) \lambda }  (f_x+f_t) dw \right) d\lambda \right\vert \\
 \nonumber
                                                                                        &\lesssim \sup_{\substack{\phi\in C_0^\infty\\ \left\Vert \phi \right\Vert_{L^2_\lambda}=1 }}  \int_y^{+ \infty}  | \hat{\phi}(y-w )| \left\vert f_x(w)+f_t(w) \right\vert dw\\
  \nonumber                                                                                      
                                                                                        &\leq \left\Vert f_x+f_t \right\Vert_{L^2_y}
 \end{align}
 Also the  standard Volterra theory gives the following operator norm:
\begin{equation}
\label{resolvent-K}
\norm{\left(\mathbf{1} -K^+_U \right)^{-1}}{L^\infty_y \left( \bbR^+;L^2_\lambda  \right)\to L^\infty_y \left( \bbR^+;L^2_\lambda \right)  }\leq e^{\norm{U}{ L^1_y \left( \bbR^+;L^\infty_\lambda  \right)}}.
\end{equation}
Thus
\begin{align}
 \label{resolvent-K1}
\norm{m^+-I}{L^\infty_y \left( \bbR^+;L^2_\lambda \right)}&=\norm{\left( \mathbf{1}  -K^+_U \right)^{-1}K^+_U I}{L^\infty_y \left( \bbR^+;L^2_\lambda \right)}\\
\nonumber
&\lesssim e^{\norm{U}{L^1_y \left( \bbR^+;L^2_\lambda \right)}}\norm{U}{L^2_y \left( \bbR^+;L^2_\lambda \right)}. 
\end{align}
An application of \textit{Minkowski's inequality} implies $S_1\in L^2_\lambda(\bbR)$. This shows that 
\begin{equation}
\label{energy-1}
S(z)-I \in L^2(I_\infty).
\end{equation}
\end{proof}

Then we turn to the derivative of the reflection coefficient $r(z)$ for $z\in I_\infty $.
\begin{lemma}
\label{lemma:Hs-infty}
Suppose that $\vec {f}(0)\in \mathcal{I}$ where $\mathcal{I}$ is given in \eqref{space-1}
then $r(z)\in H^s (I_\infty)$.
\end{lemma}
\begin{proof}
We first notice that the scattering data are given by:
\begin{align}
    \label{b-det}
    b(z)&= m^-_{11}(0,z)m^+_{21}(0,z) - m^+_{11}(0,z) m^-_{21}(0,z) \\
\label{ba-det}
    \ba(z) &=m^-_{22}(0,z)m^+_{11}(0,z) - m^-_{12}(0,z) m^+_{21}(0,z)
\end{align}
So we only need to show $ m^\pm(0, z)$ in certain appropriate Sobolev spaces. By the standard Volterra theory, $\norm{m^\pm(x,\lambda)}{L^\infty_x L^\infty_z}<\infty$. From the difference quotient characterization of the Sobolev space, we will show that
\[
\left\Vert   m^\pm(x,z+h)-m^\pm(x,z)  \right\Vert _{L_{z}^{2}}\lesssim \left|h\right|^{s}\left\Vert \left\langle x\right\rangle ^{s}U\right\Vert _{H^{2}}.
\]
Before showing this, we will study the bound on the following derivative:
\begin{equation}
    \label{linear-1}
   \norm{ \dfrac{\partial}{\partial z}(K^\pm_U I)}{L^\infty\left(\bbR^\pm ; L^2_{z}\right)}<\infty.
\end{equation}
 We only deal with the $-$ sign. If we take the derivative of \eqref{K_U-I} with respect to $z$, we obtain the following term with an $1/z$ factor. Notice that in this case $(\cdot)\sin(\cdot)\notin L^1 $. Denote
\begin{align*}
h_1^- (x, z) &= \dfrac{\partial}{\partial z} \left( \int_{-\infty}^x e^{ i(x-y) (z/2-1/(2z)) } \left[\dfrac{i}{4z} \sin f \right] dy \right)\\
                      &= \int_{-\infty}^x  \dfrac{i}{2}(x-y) e^{ i(x-y) (z/2-1/(2z)) } \left[\dfrac{i}{4z}\sin f \right] dy\\
                      &\quad +\int_{-\infty}^x  \dfrac{i}{2 z^2}(x-y) e^{ i(x-y) (z/2-1/(2z)) } \left[\dfrac{i}{4z}\sin f \right] dy\\
                      &\quad - \int_{-\infty}^x e^{ i(x-y) (z/2-1/(2z)) } \left[\dfrac{i}{4z^2} \sin f \right] dy\\
                      &=:h_{1,1}^-(x, z) + h_{1,2}^-(x, z) +h_{1,3}^-(x, z). 
\end{align*}
We will only show that $h_{1,1}^- \in L^\infty_x \left( \bbR^-; L^2_z ((1, \infty) ) \right)$. The estimates for other terms in $H_1^-$ are similar. Indeed, setting $\lambda=z-1/z$
\begin{align*}
\norm{h_{1,1}^- (x,\lambda)}{L^2_\lambda} &=\sup_{\substack{\phi\in C_0^\infty\\ \left\Vert \phi \right\Vert_{L^2}=1 }} \left\vert   \int_\bbR \phi(\lambda)  \left( \int_{-\infty}^x \dfrac{i}{2}(x-y) e^{ i(x-y) \lambda/2 } \left[\dfrac{i}{4 (\lambda + \sqrt{\lambda^2 +4 })  }\sin f(y) \right] dy  \right) d\lambda \right\vert\\
                                                           &\lesssim \sup_{\substack{\phi\in C_0^\infty\\ \left\Vert \phi \right\Vert_{L^2}=1 }} \int^x_{ -\infty} \left \vert \mathcal{F} \left[  \dfrac{\phi (\cdot) }{ 4( (\cdot)+\sqrt{(\cdot)^2+4 }   )  } \right] (x-y)  \right\vert  \left\vert x-y \right\vert   \left\vert \sin f(y) \right\vert dy\\
                                                           &\lesssim \sup_{\substack{\phi\in C_0^\infty\\ \left\Vert \phi \right\Vert_{L^2}=1 }}\int^x_{ -\infty} \left \vert \mathcal{F} \left[  \dfrac{\phi (\cdot) }{ 4( (\cdot)+\sqrt{(\cdot)^2+4 }   )  } \right]  (x-y) \right\vert  \left\vert y \right\vert   \left\vert \sin f(y) \right\vert dy
\end{align*}
where we used the fact that $|x-y|<|y|$ for $y<x<0$.  Finally an application of the Schwarz inequality gives
\begin{equation}
\label{H_1}
\norm{h_{1,1}^-(x,\lambda)}{L_x^\infty\left( \bbR^+;L^2_\lambda \right)}\leq \norm{\sin f}{L^{2, 1} }.
\end{equation}
We then proceed to write $$n^\pm(x,z):=m^\pm(x,z+h)-m^\pm(x,z)$$
and define the operator 
\begin{equation}
    K^\pm_U[\mathbf{f}](x, z):=\int_{\pm\infty}^x e^{i\left(y-x\right)J(z)\ad\sigma_3}\left[ U\left(y\right) \mathbf{f}(y)\right]\,dy.
\end{equation}
We are only dealing with the "$-$" sign since the proof for the $+$ part follows similarly. With notations above, we deduce that for $x\leq 0$,
\begin{align}
\label{eq:gsplit}
    &(\mathbf{1}-K^-_U)n^-(x, z)\\
    \quad &= \int_{-\infty}^{x} \left(e^{\left(y-x\right)J(z+h)\ad\sigma_3}-e^{\left(y-x\right)J(z)\ad\sigma_3}\right)\left[U\left(y, z\right)\left(m^-\left(y,z+h\right)-I\right) \right]\,dy\\
    \nonumber
&\quad +\int_{-\infty}^{x}e^{\left(y-x\right)J(z+h)\ad\sigma_3} \left[\left(U(y, z+h)-U(y,z)\right)\left(m^-\left(y,z+h\right)-I\right) \right]\,dy\\
    \nonumber
      &\quad+\int_{-\infty}^{x} \left(e^{\left(y-x\right)J(z+h)\ad\sigma_3}-e^{\left(y-x\right)J(z)\ad\sigma_3}\right)U\left(y, z\right)\,dy\\
      \nonumber
      &\quad +\int_{-\infty}^{x} e^{\left(y-x\right)J(z+h)\ad\sigma_3}\left( U(y, z+h)-U(y, z)\right)\,dy\\
      & =:T_{1}\left(x,z; h\right)+T_{2}\left(x,z; h\right)+ T_{3}\left(x,z; h\right)+T_{4}\left(x,z; h\right).
\end{align}
It is obvious that 
\begin{align}
    \norm{T_4\left(x,z; h\right) }{L^2_z} &\lesssim |h|\norm{U(x, z)}{L^{2,s}_x}.
\end{align}
For $T_3\left(x,z; h\right)$, given estimates \eqref{test-phi} and \eqref{linear-1} above, applying the complex interpolation to the linear
operator $K^-_U I$, see Bergh- L\"ofstr\"om \cite{BL},  we obtain that
\[
\left\Vert K^-_U I(x,\cdot)\right\Vert _{L^\infty \left( \bbR^+; H^{s}_z \right) }\lesssim\left\Vert U \right\Vert _{L^{2,s}},
\]
whence, from the difference quotient characterization of Sobolev
spaces,
\begin{align}
\left\Vert T_{3}\left(x,z; h\right) \right\Vert _{L_{z}^{2}}
&\lesssim\left\Vert V\right\Vert _{H^{1,s}}\left|h\right|^{s}.\label{eq:M3S}
\end{align}
For the estimate on $T_2$ one needs $ \norm{m^--I}{ L^2_y L^2_z ( ( 1, \infty) ) } $ to be finite. To see this we recall that
\begin{equation}
\label{double norm}
m^--I = {\left( \mathbf{1}  -K^-_U \right)^{-1}K^-_U I}.
\end{equation}
Using \eqref{test-phi} and the decay of $U(y,z)$ we have that for $s>1/2$
\begin{align*}
    \norm{K^-_U I(y,\cdot)}{L^2_\lambda} &\leq \left(  \int_{-\infty}^y |w|^{2s}|U(w)|^2 dw \right)^{1/2}|y|^{-s}\\
                     & \leq \norm{U}{L^{2,s}_y(\bbR)}|y|^{-s}
\end{align*}
and we can conclude that 
\begin{align}
\label{K-double}
\norm{m^--I}{ L^2_y \left(\bbR^-; L^2_z ( ( 1, \infty) ) \right)} & \lesssim \norm{U}{L^{2,s}_x(\bbR)} \norm{\left(\mathbf{1} -K^-_U \right)^{-1}}{L^\infty_y \left( \bbR^-;L^2_\lambda  \right)\to L^\infty_y \left( \bbR^-;L^2_\lambda \right)   }\\
\nonumber
&\quad \times \left( \int \langle y\rangle^{-2s} dy\right)^{1/2}\\
\nonumber
& \lesssim \norm{U}{L_y^{2,s}\left(\bbR; L^\infty_\lambda \right) }e^{\norm{U}{L_y^{1}\left(\bbR; L^\infty_\lambda \right)}}.
\end{align}
An application of Cauchy-Schwarz inequality leads to 
\begin{align}
    \norm{T_2\left(x,z; h\right) }{L^\infty_x (\bbR^-; L^2_z)} &\lesssim |h|\norm{U(x, z)}{L^{2,s}_x}.
\end{align}
For $T_1$ we combine \eqref{K-double} with the inequality $|e^{i\theta}-1|<|\theta|^s$ to obtain for $z>1$.
\begin{align}
    \norm{T_1\left(x,z; h\right)}{L^\infty_x (\bbR^-; L^2_z)} &\lesssim \left\vert J(z+h)-J(z) \right\vert^s  \norm{U}{L_y^{2,s}\left(\bbR; L^\infty_\lambda \right) }e^{\norm{U}{L_y^{1}\left(\bbR; L^\infty_\lambda \right)}}\\
    \nonumber 
      &\leq |h|^s \norm{U}{L_y^{2,s}\left(\bbR; L^\infty_\lambda \right) }e^{\norm{U}{L_y^{1}\left(\bbR; L^\infty_\lambda \right)}}.
\end{align}
\end{proof}
\subsection{Near the origin}For $z\in I_0 \equiv (-3/4, 3/4)$, after performing the gauge transformation, we analyze  the following integral equation derived from the spectral problem \eqref{Phi-x}:
\begin{subequations}
\begin{equation}
\label{m+_}
\underline{m}^+(x, z)=I+\int_{+\infty}^x e^{ (y-x)J(z)\ad \sigma_3 } \left[ V(y, z)\underline{m}^+(y, z) \right] dy
\end{equation}
\begin{equation}
\label{m-_}
\underline{m}^-(x, z)=I+\int_{-\infty}^x e^{ (y-x)J(z)\ad \sigma_3} \left[  V(y, z)\underline{m}^-(y, z) \right] dy.
\end{equation}
\end{subequations}
We first have the following lemma:
\begin{lemma}
\label{lemma:L2-0}
Suppose that $\vec {f}(0)\in \mathcal{I}$ where $\mathcal{I}$ is given in \eqref{space-1}
then $r(z)\in L^2(I_0) \cap L^\infty(I_0)$.
\end{lemma}
Then we move on to prove the following lemma:
\begin{lemma}
\label{lemma:Hs-0}
Suppose that $\vec {f}(0)\in \mathcal{I}$ where $\mathcal{I}$ is given in \eqref{space-1}
then $r(z)\in H^s (I_0)$.
\end{lemma}

\begin{proof}
The proof of this proposition is similar to the proof of  Lemma \ref{lemma:Hs-infty}. The key difference is that the phase has certain nontrivial behavior near the origin. As in Section 2, we focus on $\underline{m}^-$ and write $$n^\pm(x,z; h):=m^\pm(x,z+h)-m^\pm(x,z)$$
and  
\begin{align}
\label{eq:gsplit}
  &  (\mathbf{1}-K^-_V)\underline{n}(x, z)\\
  \nonumber
   \quad &= \int_{-\infty}^{x} \left(e^{\left(y-x\right)J(z+h)\ad\sigma_3}-e^{\left(y-x\right)J(z)\ad\sigma_3}\right)\left[V\left(y, z\right)\left(\underline{m}^-\left(y,z+h\right)-I\right) \right]\,dy\\
    \nonumber
&\quad +\int_{-\infty}^{x}e^{\left(y-x\right)J(z+h)\ad\sigma_3} \left[\left(V(y, z+h)-V(y,z)\right)\left(\underline{m}^-\left(y,z+h\right)-I\right) \right]\,dy\\
    \nonumber
      &\quad+\int_{-\infty}^{x} \left(e^{\left(y-x\right)J(z+h)\ad\sigma_3}-e^{\left(y-x\right)J(z)\ad\sigma_3}\right)V\left(y, z\right)\,dy\\
      \nonumber
      &\quad +\int_{-\infty}^{x} e^{\left(y-x\right)J(z+h)\ad\sigma_3}\left( V(y, z+h)-V(y, z)\right)\,dy\\
      \nonumber
      & =:\tilde{T}_{1}\left(x,z; h\right)+\tilde{T}_{2}\left(x,z; h\right)+ \tilde{T}_{3}\left(x,z; h\right)+\tilde{T}_{4}\left(x,z; h\right).
\end{align}
It is obvious that 
\begin{align}
    \norm{\tilde{T}_4\left(x,z; h\right) }{L^2_z} &\lesssim |h|\norm{V(x, z)}{L^{2,s}_x}.
\end{align}
Similarly, using the same argument as in the estimate \eqref{K-double}, we obtain
\begin{align}
    \norm{\tilde{T}_2\left(x,z; h\right) }{L^\infty_x (\bbR^-; L^2_z)} &\lesssim |h|\norm{V(x, z)}{L^{2,s}_x}.
\end{align}
We then turn to
the linear part $\tilde{T}_3$. Consider the linear operator $L$ defined by the following expression:
\[
L[V] (x,z):=K_V[1](x, z)=\intop_{-\infty}^{x}e^{i\left(x-y\right)\left(z-\frac{1}{z}\right)\ad\sigma_3}V\left(y, z\right)\,dy.
\] 
Note that if we take derivative with in the $z$ variable, the $1/z$ term in the phase will bring down a $1/z^2$ factor. To begin with, we need the following change of variable:
$$z\mapsto \gamma = \dfrac{1}{z}.$$
First of all, we observe that
\begin{align*}
\int_{\left|z\right|<3/4}\left|\intop_{-\infty}^{x}e^{(y-x)J(z)\ad \sigma_3}V\left(y\right)\,dy\right|^{2}dz & \lesssim\int_{\left|\gamma\right|>4/3}\left|\intop_{-\infty}^{x}e^{i\left(x-y\right)\gamma \ad\sigma_3 }V\left(y\right)\,dy\right|^{2}\frac{1}{\gamma^{2}}d\gamma\\
 & \lesssim\left\Vert V\right\Vert _{L^{2}}^{2,s}
\end{align*}
and from integration by parts
\begin{align*}
\int_{\left|z\right|<3/4}\left|\frac{d}{dz}\intop_{-\infty}^{x}e^{(y-x)J(z)\ad \sigma_3}V\left(y\right)\,dy\right|^{2}dz & \lesssim\int_{\left|\gamma\right|>4/3}\gamma^{4}\left|\intop_{-\infty}^{x}e^{i\left(x-y\right)\gamma\ad\sigma_3}\left(x-y\right)V\left(y\right)\,dy\right|^{2}\frac{1}{\gamma^{2}}d\gamma\\
 & =\int_{\left|\gamma\right|>4/3}\gamma^{2}\left|\intop_{-\infty}^{x}e^{i\left(x-y\right)\gamma\ad\sigma_3}\left(x-y\right)V\left(y\right)\,dy\right|^{2}d\gamma\\
 & \sim\left\Vert \left\langle x\right\rangle V\right\Vert _{H^{1}}^{2}
    \end{align*}
where in the last line we applied the Plancherel's theorem.

Given estimates above, applying the complex interpolation to the linear
operator $L$, see Bergh- L\"ofstr\"om \cite{BL},  we obtain that
\[
\left\Vert L[V](x,\cdot)\right\Vert _{H^{s}}\lesssim\left\Vert V\right\Vert _{H^{s,s}}\lesssim\left\Vert V\right\Vert _{H^{1,s}},
\]
whence, from the difference quotient characterization of Sobolev
spaces,
\begin{align}
\left\Vert \tilde{T}_{2}\left(x,z; h\right) \right\Vert _{L_{z}^{2}}
&\lesssim\left\Vert V\right\Vert _{H^{1,s}}\left|h\right|^{s}.\label{eq:M3S}
\end{align}
For $\tilde{T}_1$, integration by parts  leads to
\begin{align}
&\intop_{-\infty}^{x}e^{(y-x)J(z)\ad \sigma_3}V\left(y\right)\left(\underline{m}^-\left(y,z\right)-I\right)\,dy\\
& \quad =\intop_{-\infty}^{x}-\left(\frac{z}{z^{2}-1}\right) \left(\frac{d}{dy}e^{(y-x)J(z)\ad \sigma_3} \right)V\left(y\right)\left(\underline{m}^-\left(y,z\right)-I\right)\,dy\nonumber \\
 & \quad=-\left(\frac{z}{z^{2}-1}\right)V\left(x\right)\left(\underline{m}^-\left(x,z\right)-I\right)\nonumber \\
 &\quad +\left(\frac{z}{z^{2}-1}\right)\intop_{-\infty}^{x}e^{(y-x)J(z)\ad \sigma_3}\frac{d}{dy}\left(V\left(y\right)\left(\underline{m}^-\left(y,z\right)-I\right)\right)\,dy\nonumber \\
 &\quad =:\tilde{T}_{11}\left(x,z\right)+\tilde{T}_{12}\left(x.z\right).\label{eq:diffs1}
\end{align}
For the  first term above $\tilde{T}_{11}\left(x,z; h\right)$, we bound:
\begin{align}
\left\Vert \tilde{T}_{11}\left(x,z+h\right)-\tilde{T}_{11}\left(x,z\right)\right\Vert _{L_{z}^{2}} & \lesssim\left\Vert \left(\frac{z+h}{\left(z+h\right)^{2}-1}\right)V\left(x\right)(\underline{m}^-\left(x,z\right)-I)\right. \\
&\quad \left. -\left(\frac{z}{z^{2}-1}\right)V\left(x\right)(\underline{m}^-\left(x,z\right)-I)\right\Vert _{L_{z}^{2}}\nonumber \\
 & \lesssim h\text{\ensuremath{\left|V(x)\right|}}\left\Vert V\left(x\right)(\underline{m}^-\left(x,z\right)-I)\right\Vert _{L_{z}^{2}}.\label{eq:M1S}
\end{align}
Then we analyze the second term $\tilde{T}_{12}\left(x,z\right)$. Taking the
difference, explicitly, we have
\begin{align*}
\tilde{T}_{12}\left(x,z+h\right)-\tilde{T}_{12}\left(x,z\right) & =\left(\frac{z+h}{\left(z+h\right)^{2}-1}-\frac{z}{z^{2}-1}\right)\intop_{-\infty}^{x}e^{(y-x)J(z+h)\ad \sigma_3}\frac{d}{dy}\left(V\left(y\right)(\underline{m}^-\left(x,z\right)-I)\right)\,dy\\
 & -\left(\frac{z}{z^{2}-1}\right)\intop_{-\infty}^{x}\left(e^{(y-x)J(z+h)\ad \sigma_3}-e^{(y-x)J(z)\ad \sigma_3}\right)\frac{d}{dy}\left(V\left(y\right)(\underline{m}^-\left(x,z\right)-I)\right)\,dy\\
 & =:\tilde{T}_{12,1}+\tilde{T}_{12,2}.
\end{align*}
For $\tilde{T}_{12,2}$, taking the absolute value, one has
\[
\left|\tilde{T}_{12,2}\right|\lesssim\left|\left(\frac{z}{z^{2}-1}\right)\right|\left|z\right|^{-2s}2^{1-s}\left|h\right|^{s}\intop_{-\infty}^{x}|x-y| ^{s}\left|\frac{d}{dy}\left(V\left(y\right)\left(\underline{m}^-\left(x,z\right)-I\right)\right)\right|\,dy.
\]
It suffices to bound the $L_{z}^{2}$ norm of the RHS above. Leaving the $|h|^s$ factor aside we make a change of variable $\gamma=1/z$. For brevity we are only dealing with the following inequalities explicitly
\begin{align}
\label{m-deriv-1}
   \left\Vert \intop_{-\infty}^{x}|x-y|^s
  |\left[ (f_x-f_t)(\underline{m}^-_{11}-1)  \right]_y|\,dy\right\Vert _{L_{\gamma}^{2}} &\leq \norm{\langle \cdot \rangle^s (f_x - f_t) }{H^{1}} \norm{\underline{m}^-_{11}-1}{L^2_y L^2_\gamma}\\
  \nonumber
  & \quad +  \norm{\langle \cdot \rangle^s (f_x-f_t) }{L^{2}} \norm{(\underline{m}^-_{11}-1)_{y}}{L^2_y L^2_\gamma}.
\end{align}
\begin{align}
\label{m-deriv-2}
   \left\Vert \intop_{-\infty}^{x}|x-y|^s
  |\left[ (f_x- f_t)(\underline{m}^-_{21})  \right]_y|\,dy\right\Vert _{L_{\gamma}^{2}} &\leq \norm{\langle \cdot \rangle^s (f_x - f_t) }{H^{1}} \norm{\underline{m}^-_{21}}{L^2_y L^2_\gamma}\\
  \nonumber
  & \quad +  \norm{\langle \cdot \rangle^s (f_x-f_t) }{L^{2}} \norm{(\underline{m}^-_{21})_{y}}{L^2_y L^2_\gamma}.
\end{align}
 Direct calculation gives:
\begin{align}
\label{double norm 11-x}
\underline{m}^-_{11 x}(x, \gamma )=\dfrac{1}{4i}\left[  V_{11}\underline{m}_{11}^- + V_{12}\underline{m}_{21}^- \right].
\end{align}
We deduce 
$$\underline{m}_{11x}^-\in L^2_x L^2_\gamma ((1,\infty) ) $$
from the fact that $V_{11}$ contains a $1/\gamma$ factor and inequality \eqref{K-double}. For the remaining term,
\begin{align}
\underline{m}^-_{21 x}(x, \gamma ) &=\dfrac{1}{4i} \int_{-\infty}^x e^{i (x-y)  (1/(2\gamma)-\gamma/2) } \left[   V_{21y}\underline{m}_{11}^- + V_{21}\underline{m}_{11y}^-  +V_{22y}\underline{m}_{21}^+  \right] dy\\
  &\quad +\dfrac{1}{4i} \int_{-\infty}^x e^{i (x-y)  (1/(2\gamma)-\gamma/2) }   V_{22}\underline{m}_{21y}^- dy\\
  &=P(x, \gamma)+\dfrac{1}{4i} \int_{-\infty}^x e^{i (x-y)  (1/(2\gamma)-\gamma/2) }   V_{22}\underline{m}_{21y}^- dy.
\end{align}
Similar to that of \eqref{K-double}, we have 
\begin{equation}
\label{double norm-12}
\norm{ \underline{m}^-_{21 x} }{ L^2_x L^2_\gamma ( ( 1, \infty) ) } \leq \norm{ P(x, \gamma) }{ L^2_x L^2_\gamma ( ( 1, \infty) ) }+\norm{K_{22} \underline{m}_{21y}^-  }{ L^2_x L^2_\gamma ( ( 1, \infty) ) }.
\end{equation}
where 
$$ \left[ K_{22} g \right](x, \lambda) =\dfrac{1}{4i} \int_{-\infty}^x  e^{i (x-y)  (1/(2\gamma)-\gamma/2) }   \dfrac{i }{4\gamma} (1-\cos f) g(y) dy.$$
Following the same procedure used to establish \eqref{K-double}, we conclude that
$$\underline{m}_{21x}^-\in L^\infty_x L^2_\gamma ((1,\infty) ) $$
and arrive at the bounds in \eqref{m-deriv-1}-\eqref{m-deriv-2}. From the analysis above, it follows that
\begin{equation}
\left\Vert \tilde{T}_{12,2}\right\Vert _{L_{z}^{2}}\lesssim2^{1-s}\left|h\right|^{s}\left\Vert V\right\Vert _{H^{1,s}}.\label{eq:M22S}
\end{equation}
Similarly, for the first term:
\begin{equation}
\left\Vert \tilde{T}_{12,1}\right\Vert _{L_{z}^{2}}\lesssim |h|^s \left\Vert \frac{d}{dy}\left(V\left(y\right)(m\left(x,z\right)-1)\right)\right\Vert _{L_{x}^{\infty}L_{z}^{2}} \lesssim\left|h\right|^{s}\left\Vert V\right\Vert _{H^{1,s}}.\label{eq:M21S}
\end{equation}

Combining \eqref{eq:M3S},\eqref{eq:M1S}  \eqref{eq:M22S} \eqref{eq:M21S}
 with the resolvent estimate \eqref{resolvent-K}, it follows that
\[
\left\Vert m(x,z+h)-m(x,z)\right\Vert _{L_{z}^{2}}\lesssim\left|h\right|^{s}\left\Vert V\right\Vert _{H^{1,s}}
\]
as desired.
\end{proof}
\subsection{Proof of Proposition \ref{prop:r}}
Lemmas \ref{lemma:L2}-\ref{lemma:Hs-0} lead to the first part of Proposition \ref{prop:r}. Now we turn to part (2) of Proposition \ref{prop:r}.
\begin{lemma}
\label{lemma:r2-0}
If $\vec{f}(0)\in \mathcal{I}$ where $\mathcal{I}$ is given in \eqref{space-1}
then $\lim_{z\to 0}{r(z)}/z=0$.
\end{lemma}
\begin{proof}
We have already proven that 
\begin{equation}
\label{partialS}
{ S(\gamma)}\in H^{s}(I_\infty)
\end{equation}
and by the Sobolev embedding, $H^s\subset C^{s-1/2}$ and  this leads to the following limit:
\begin{align*}
\lim_{z\to 0^+} S(z) -I &=\lim_{\gamma \to +\infty} S(\gamma)-I =0,\\
\lim_{z\to 0^-} S(z)-I &=\lim_{\gamma \to -\infty} S(\gamma)-I =0.
\end{align*}
To see
\begin{equation}
\label{r=0}
\lim_{z\to 0}\dfrac{r(z)}{z}=\lim_{\gamma\to \infty} \gamma r(\gamma)=0,
\end{equation}
we can just show that 
$ \gamma r(\gamma) \in L^2 $ and is uniformly continuous. We will work with
\begin{align}
S_{21}(\gamma) &= \left( \int_\bbR e^{ yJ(z)\ad \sigma_3} \left[ V(y)\left(\underline{m}^+(y, z) -I\right) \right] dy \right)_{21} + 
\left( \int_\bbR e^{- yJ(z)\ad \sigma_3}V(y)dy \right)_{21}\\
\nonumber
      &=: S^1_{21}(\gamma)  +S^2_{21}(\gamma)
\end{align}
and show that 
\begin{equation}
\label{S-small g}
S^1_{21}(\gamma) \in L^{2,1}_{\gamma }(( 1, \infty )).
\end{equation}
We write
\begin{align*}
\gamma S^1_{21}(\gamma) &= \int_\bbR \gamma e^{ iy (1/(2\gamma)-\gamma/2)   } \left[ V_{21}(y) (\underline{m}_{11}^+(y)-1 ) +V_{22}(y)\underline{m}_{21}^+(y)\right] dy.
\end{align*}
Again we only deal with the term without the $1/\gamma$ factor:
\begin{align}
\label{uni-cont}
  \int_\bbR  \gamma e^{ iy (1/(2\gamma)-\gamma/2)   } \left[ V_{21} (\underline{m}^+_{11}-1)  \right] dy &=- 2 \int_\bbR \dfrac{  d \left[  e^{iy(1/(2\gamma) -\gamma/2)   }  \right] } {dy}\left[ V_{21} (\underline{m}^+_{11}-1)  \right] dy\\
  \nonumber
       & \quad +\int_\bbR \dfrac{i}{ \gamma} e^{ iy(1/(2\gamma)-\gamma/2) } \left[ V_{21} (\underline{m}^+_{11}-1)  \right] dy\\
       \nonumber
       &= 2 \int_\bbR   e^{iy(1/(2\gamma) -\gamma/2)   }  \left[ V_{21 x} (\underline{m}^+_{11}-1)  \right] dy\\
       \nonumber
         & \quad +\int_\bbR e^{iy(1/(2\gamma) -\gamma/2)   }  V_{21} \underline{m}^+_{11y}  dy\\
         \nonumber
            & \quad +\int_\bbR \dfrac{i}{ \gamma} e^{ iy(1/(2\gamma)-\gamma/2) } \left[ V_{21} (\underline{m}_{11}-1)  \right] dy.
\end{align}
Then by Minkowski's inequality
$$\norm{\gamma S^1_{21}(\gamma) }{L^\infty_x L^2} \lesssim \norm{f_x - f_t}{L^{2,1}} \norm{\underline{m}^+_{11x} }{ L^\infty_x L^2_\gamma ((1,\infty) ) } + \norm{f_{xx} - f_{tx} }{L^{2,1}} \norm{\underline{m}^+_{11}-1 }{ L^\infty_x L^2_\gamma ((1,\infty) ) } $$
and the treatment of $S_{21}^2$  is the standard Fourier theory. The uniform continuity of follows directly from the integral representation of \eqref{uni-cont}.
\end{proof}

\begin{proof}[Proof of Proposition \ref{prop:r}]
Combining the four lemmas \ref{lemma:L2}-\ref{lemma:r2-0} completes the proof of Proposition \ref{prop:r}.
\end{proof} 
\begin{remark}
\label{rmk-generic}
We point out that a byproduct of Proposition \ref{prop:r} is the generic property mentioned in Remark \ref{remark-generic}. Indeed the proofs required here are similar to those of \cite[Section 4]{Liu} as long as we take care of the  behavior of $\ba$ and $a$ near the origin. We have proven in particular that
$$\ba-1, \, a-1 \in H^{s}(\bbR)$$
and
$$\lim_{z\to 0}a= \lim_{z\to 0}\ba=\pm 1 $$
where the $\pm$ depends on the boundary value of $f$ at $+\infty$, so we can infer that $\ba$ and $a$  are bounded away from $0$ near the origin. This will exclude the case that small perturbation of the matrix potential in \eqref{Psi-x} leading to new eigenvalues/spectral singularities. So we conclude that the space $\mathcal{I}$ in \eqref{space-1} has an open and dense subset $\mathcal{G}$ such that the generic properties listed in Remark \ref{remark-generic} are satisfied. This is another way to rule out the appearance of breathers with arbitrarily small norm in the sense of space \eqref{space-1}.
\end{remark}
\begin{remark}
From our proof of Proposition \ref{prop:r} that to make sense of the integral
$$\int_0^1\left \vert \dfrac{r(z)}{z} \right\vert^2 dz =\int_1^\infty \vert r(\gamma)\vert^2 d\gamma$$
we only need $\vec{f}(0)\in H_{\sin}^{1,s}\times L^{2,s}, s>1/2$. This fact will be useful in Section \ref{sec:stability}.
\end{remark}

\section{Conjugation}
\label{sec:prep}
From this section onwards we will assume that the initial condition $\vec{f}(0)$ of equation \eqref{eq: sG} belongs to $ \mathcal{G}$ described in remark \ref{rmk-generic}. Having characterized the properties of the scattering data in Proposition \ref{prop:r}, we begin the calculation of the long time asymptotics. Along a characteristic line $x=\text{v}t$ for $|\text{v}|<1$  we have the following signature table:
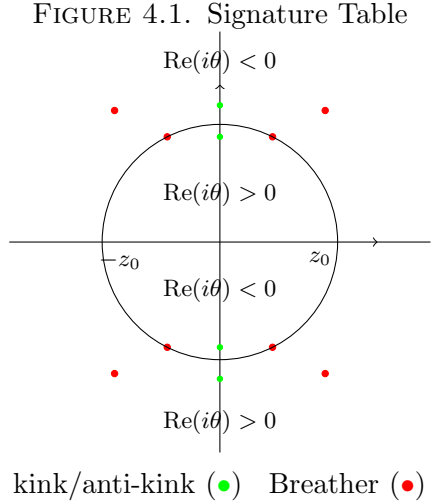
\begin{figure}[H]
\label{sig-table}
\caption{Signature Table}
\begin{tikzpicture}[scale=0.7]
\draw [->] (-4,0)--(3,0);
\draw (4,0)--(3,0);
\draw [->] (0,-4)--(0,3);
\draw (0,3)--(0,4);
 

 \draw	[fill, green]  (0, 2)		circle[radius=0.05];	    
 \draw	[fill, green]  (0, 2.6)		circle[radius=0.05];	    

 \draw	[fill, green]  (0, -2)		circle[radius=0.05];	    
 \draw	[fill, green]  (0, -2.6)		circle[radius=0.05];	    
 
 \draw	[fill, red]  (1, 2)		circle[radius=0.06];	    
 \draw	[fill, red]  (1, -2)		circle[radius=0.06];	    
 \draw	[fill, red]  (-1, 2)		circle[radius=0.06];	    
 \draw	[fill, red]  (-1, -2)		circle[radius=0.06];	    
  \draw	[fill, red]  (2, 2.5)		circle[radius=0.06];	    
 \draw	[fill, red]  (2, -2.5)		circle[radius=0.06];	    
 \draw	[fill, red]  (-2, 2.5)		circle[radius=0.06];	    
 \draw	[fill, red]  (-2, -2.5)		circle[radius=0.06];	    
 \draw (0, 0) circle[radius=2.236];

   \node [below] at (1.9,0) {\footnotesize $z_0$};
    \node [below] at (-1.9,0) {\footnotesize $-z_0$};
    \node[above]  at (0, 0.5) {\footnotesize $\text{Re}(i\theta)>0$};
    
      \node[below]  at (0, -0.5) {\footnotesize $\text{Re}(i\theta)<0$};
     
     \node[above]  at (0, 3) {\footnotesize $\text{Re}(i\theta)<0$};
     \node[below]  at (0, -3) {\footnotesize $\text{Re}(i\theta)>0$};
    \end{tikzpicture}
 \begin{center}
  \begin{tabular}{ccc}
kink/anti-kink ({\color{green} $\bullet$})	&	
Breather ({\color{red} $\bullet$}) 
\end{tabular}
 \end{center}
\end{figure}
In the figure above, we have chosen 
$$\mathrm{v}_\ell=\dfrac{x}{t}= \dfrac { 1-\rho^2_\ell}  {1+\rho^2_\ell } $$
where $  \lbrace  z_j \rbrace_{j=1}^{N_2} \ni z_\ell=\rho_\ell e^{i\omega_\ell}$ with $1\leq \ell\leq N_2$. Recall the following set:
\begin{equation}
\label{B-set}
\mathcal{B}_\ell=\lbrace z_j, i\zeta_k: ~  \zeta_k <\rho_\ell, \rho_j <\rho_\ell  \rbrace 
\end{equation}
and
\begin{equation}
\label{mathcal-z}
\tilde{\mathcal{Z}}=\lbrace i\zeta_k \rbrace_{k=1}^{N_1}, \quad \hat{\mathcal{Z}}=\lbrace z_j \rbrace_{k=1}^{N_2}, \quad  \mathcal{Z}=\tilde{\mathcal{Z}} \cup\hat{ \calZ}.
\end{equation}
Also define
\begin{equation}
\label{Up}
\Upsilon=\text{min}\lbrace \text{min}_{z,z'\in \calZ} |z-z'|, \quad \text{dist}(\calZ, \bbR)   \rbrace.
\end{equation}
We observe that for all $i\zeta_k, z_j\in \mathcal{B}_\ell $, 
$$\text{Re}(i\theta(i\zeta_k))>0,\, \text{Re}(i\theta(z_j))>0, \, \text{Re}(i\theta(- \overline{z_j } ))>0 .$$
Then we introduce a new matrix-valued function
\begin{equation}
\label{m1}
m^{(1)}(z;x,t) = M(z;x,t) [\delta(z)]^{-\sigma_3} 
\end{equation}
where $\delta(z)$  solves 
the scalar RHP 
Problem \ref{prob:RH.delta} below:

\begin{problem}
\label{prob:RH.delta}
Given $\pm z_0 \in \bbR$ and $r \in H^{s}_0(\bbR)$, find a scalar function 
$\delta(z) = \delta(z; z_0)$, meromorphic for
$z \in \bbC \setminus [-z_0, z_0]$ with the following properties:
\begin{enumerate}
\item		$\delta(z) \rarr 1$ as $z \rarr \infty$,
\item		$\delta(z)$ has continuous boundary values $\delta_\pm(z) =\lim_{\eps \to 0^+} \delta(z \pm i\eps)$ for $z \in \bbR$,
\item		$\delta_\pm$ obey the jump relation
			$$ \delta_+(z) = \begin{cases}
											\delta_-(z)  \left(1 + \left| r(z) \right|^2 \right),	&	 z\in (-z_0, z_0),\\
											\delta_-(z), &	z \in \bbR\setminus (-z_0, z_0),
										\end{cases}
			$$
\item		$\delta(z)$ has simple poles at $ i\zeta_k, z_j, -\overline{z_j} \in \mathcal{B}_\ell$.
\end{enumerate}
\end{problem}

\begin{lemma}
\label{lemma:delta}
Suppose $r \in H^{s}_0(\bbR)$. Then we have the following conclusions:
\begin{itemize}
\item[(i)]		Problem \ref{prob:RH.delta} has the unique solution
\begin{equation}
\label{RH.delta.sol}
\delta(z) =\left( \prod_{i\zeta_k\in B_\ell} \dfrac{z- \overline {i\zeta_k }}{ z-i\zeta_k}\right) \left(\prod_{z_j\in B_\ell} \dfrac{z-\overline{z_j}}{z-z_j} \dfrac{z+z_j}{z+\overline{z_j}}\right)  e^{\chi(z)}  
\end{equation}
where
\begin{equation}
\label{chi}
\chi(z)=\dfrac{1}{2\pi i}\int_{-z_0}^{z_0}\dfrac{\log( 1+|r(s)|^2)}{s-z} {d s}.
\end{equation}
Moreover, we can rewrite 
\begin{equation}
\label{exi}
e^{\chi(z)}=\left(  \dfrac{z-z_0}{z+z_0} \right)^{i\kappa} e^{\breve{\chi}(z)}
\end{equation}
where $\kappa $ is given by 
\begin{equation}
\kappa=-\frac{1}{2\pi}\log\left(1+\left|r\left( z_0 \right)\right|^{2}\right), \label{kappa}
\end{equation}
and
\begin{align*}
\breve{\chi}(z) &=\dfrac{1}{2\pi i}\int_{-z_0}^{z_0} \log \left( \dfrac{  1+|r(s )|^2 } { 1+|r(z_0)|^2   } \right) \dfrac{ ds} {s-z} 
\end{align*}
Here we have chosen the branch of the logarithm with $-\pi  < \arg(z) < \pi$. 
\bigskip
\item[(ii)] For $z\in \bbC\setminus [-z_0, z_0]$
\begin{equation*}
\delta(z) =\left(\overline{\delta(\zbar)}\right)^{-1}.
\end{equation*}
\bigskip

\item[(iii)] As $z\to 0$ nontangentially, 
$$\delta(0)= (-1)^\ell. $$

\item[(iv)]Along any ray of the form $ \lbrace z\in \mathbb{C}|\, z=\pm z_0+ e^{i\phi}\bbR^+ \rbrace$ with $0<\phi<\pi$ or $\pi < \phi < 2\pi$,  one has
				
				$$ 
						 \left| \delta(z) -\left(  \dfrac{z-z_0}{z+z_0} \right)^{i\kappa}  \delta_0(\pm z_0)\right| 
						 		\leq C_r
						 |z \mp z_0|^{1/2} $$
						 where
\begin{align*}
		\delta_0(\pm z_0) &=\left( \prod_{i\zeta_k \in B_\ell} \dfrac{\pm z_0- \overline {i\zeta_k }}{ \pm z_0-i\zeta_k}\right) \left(\prod_{z_j\in B_\ell} \dfrac{\pm z_0-\overline{z_j}}{\pm z_0-z_j} \dfrac{\pm z_0+z_j}{\pm z_0+\overline{z_j}}\right)  e^{  \breve{\chi}( \pm z_0) }
\end{align*}

and the implied constant depends on $r$ through its $H^{1}(\bbR)$-norm  
				and is independent of $\pm z_0\in \bbR$.
				\end{itemize}
\end{lemma}

\begin{proof}
The proofs of (i)-(ii) are formal computations. For (iii), we use the fact that as $z\to 0$
\begin{align*}
\dfrac{z- \overline{i\zeta_k}}{z-i\zeta_k} &=\dfrac{\overline{i\zeta_k} }{i\zeta_k}=-1\\
\dfrac{z-\overline{z_j}}{z-z_j} \dfrac{z+z_j}{z+\overline{z_j}} &=\dfrac{-\overline{z_j}}{-z_j} \dfrac{z_j}{\overline{z_j}}=1\
      \end{align*}
and $\lim_{s\to 0}r(s)/s=0$. We also need to invoke to the evenness of $|r(z)|^2$ which implies that 
\begin{equation}
\chi(0)=\dfrac{1}{2\pi i}\int_{-z_0}^{z_0}\dfrac{\log( 1+|r(s)|^2)}{s} {d\zeta} =0.
\end{equation}
 To establish (iv), we first note that 
$$ \left\vert \left( \dfrac{z-z_0}{z+z_0} \right)^{i\kappa}\right\vert \leq e^{\pi \kappa}.$$
To bound the difference $e^{\chi(z)}-e^{\chi(\pm z_0)}$, notice that
\begin{align*}
\left\vert e^{\chi(z)}-e^{\chi(\pm z_0)}\right\vert &\leq\left\vert e^{\chi(\pm z_0)}\right\vert 
\left\vert e^{\chi(z)-\chi(\pm z_0)}-1 \right\vert\\
     &\lesssim \left\vert \chi(z)-\chi(\pm z_0) \right\vert\\
       &\lesssim \norm{r}{H^s(\bbR)} |z\pm z_0|^{s-1/2}
    \end{align*}
    where the last inequality follows from \cite[Lemma 3.4]{CuPe}.
\end{proof}

It is straightforward to check that if $M(z;x,t)$ solves Problem \ref{RHP-1}, then the new matrix-valued function $m^{(1)}(z;x,t)=M(z;x,t)[\delta(z)]^{-\sigma_3}$ is the solution to the  following RHP.  

\begin{problem}
\label{prob:mkdv.RHP1}
Given 
$$\mathcal{S}=\lbrace r(z),\lbrace i\zeta_k, \iota_k \rbrace_{k=1}^{N_1}, \lbrace z_j, c_j \rbrace_{j=1}^{N_2} \rbrace \in H^{s}_0(\bbR)\otimes \mathbb{C}^{2 N_1} \otimes \mathbb{C}^{ 2 N_2}$$ 
and the augmented contour $\Sigma$ in Figure \ref{figure-Sigma}
and set
$$\dfrac{x}{t}=\dfrac { 1-\rho^2_\ell}  {1+\rho^2_\ell }$$
where $ \lbrace z_j\rbrace_{j=1}^{N_2} \ni z_\ell=\rho_\ell e^{\omega_\ell}$, 
find a matrix-valued function $m^{(1)}(z;x,t)$ on $\bbC \setminus\Sigma$ with the following properties:
\begin{enumerate}
\item		$m^{(1)}(z;x,t) \rarr I$ as $|z| \rarr \infty$,
\item		$m^{(1)}(z;x,t)$ is analytic for $z \in  \bbC \setminus\Sigma$
			with continuous boundary values
			$m^{(1)}_\pm(z;x,t)$.
\item		On $\bbR$, the jump relation $$m^{(1)}_+(z;x,t)=m^{(1)}_-(z;x,t)	
			e^{-i\theta\ad\sigma_3}v^{(1)}(z)$$
			holds,
			 where $$v^{(1)}(z) = [\delta_-(z)]^{\sigma_3} v(z) [\delta_+(z)]^{-\sigma_3}.$$
			 \noindent
			The jump matrix $e^{-i\theta\ad\sigma_3} v^{(1)} $ is factorized as 
			\begin{align}
			\label{mkdv.V1}
			e^{-i\theta\ad\sigma_3}v^{(1)}(z)	=
			\begin{cases}
					\Twomat{1}{0}{\dfrac{\delta_-^{-2}  r}{1+|r|^2}  e^{2i\theta}}{1}
					\Twomat{1}{\dfrac{\delta_+^2 \rbar}{1+|r|^2} e^{-2i\theta}}{0}{1},
						& z \in (-z_0, z_0),\\
						\\
						\Twomat{1}{\rbar \delta^2 e^{-2i\theta}}{0}{1}
					\Twomat{1}{0}{r \delta^{-2} e^{2i\theta}}{1},
						& z \in(-\infty, -z_0)\cup (z_0,\infty) .
			\end{cases}
			\end{align}
\item Let $\delta(z)$ be the solution to Problem \ref{prob:RH.delta} and we have the following jump conditions
$m^{(1)}_+(z;x,t)=m^{(1)}_-(z;x,t)	
			e^{-i\theta\ad\sigma_3}v^{(1)}(z)$
			where for $i\zeta_k \in \mathcal{B}_\ell $
$$
e^{-i\theta\ad\sigma_3}v^{(1)}(z) = 	\begin{cases}
						\twomat{1}{\dfrac{ \left[ (1/\delta)'(i\zeta_k) \right]^{-2} }{\iota_k  (z-i\zeta_k)}e^{-2i\theta} }{0}{1}	&	z\in \tgamma_k, \\
						\\
						\twomat{1}{0}{\dfrac{ \left[ \delta'\left( \overline{i\zeta_k } \right) \right]^{-2}}{\overline{\iota_k}  (z -\overline{i\zeta_k })} e^{2i\theta}}{1}
							&	z \in \tgamma_k^*
					\end{cases}
$$
and for $i\zeta_k \in \lbrace i\zeta_k \rbrace_{k=1}^{N_1} \setminus \mathcal{B}_\ell $	
$$
e^{-i\theta\ad\sigma_3}v^{(1)}(z) = 	\begin{cases}
						\twomat{1}{\dfrac{\iota_k \left[ \delta(i\zeta_k) \right]^{-2} }{z-i\zeta_k} e^{2i\theta(i\zeta_k)}}{0}{1}	&	z\in \tgamma_k, \\
						\\
						\twomat{1}{0} { \dfrac{\overline{\iota_k} \left[ \delta(\overline{i\zeta_k })\right]^2 }{z -\overline{i\zeta_k }} e^{-2i\theta (\overline{i\zeta_k })} } {1}
							&	z \in \tgamma_k^*
					\end{cases}
$$

and for $z_j\in \mathcal{B}_\ell $
$$
e^{-i\theta\ad\sigma_3}v^{(1)}(z)= 	\begin{cases}
						\twomat{1}{\dfrac{ \left[ (1/\delta)'(z_j)\right]^{-2} }{c_j  (z-z_j)}e^{-2i\theta(z_j)} }{0}{1}	&	z\in \gamma_j, \\
						\\
						\twomat{1}{0}{\dfrac{ \left[\delta'( \overline{z_j } )\right]^{-2}}{\overline{c_j}  (z -\overline{z_j })} e^{2i\theta( \overline{z_j } )}}{1}
							&	z \in \gamma_j^*,\\
							\\
							\twomat{1}{0}{-\dfrac{ \left[\delta'( -{z_j } )\right]^{-2}}{ c_j  (z + z_j ) } e^{2i\theta( -{z_j } )}}{1}	&	z\in -\gamma_j, \\
						\\
						\twomat{1} {-\dfrac{ \left[ (1/\delta)'(  -\overline{z_j } )\right]^{-2} }{ \overline{c_j}  (z +\overline{z_j } ) }  e^{-2i\theta(  -\overline{z_j } )} }{0}{1}
							&	z \in -\gamma_j^*
					\end{cases}
$$								
and for $z_j \in \lbrace z_j \rbrace_{j=1}^{N_2} \setminus \mathcal{B}_\ell $	
$$
e^{-i\theta\ad\sigma_3}v^{(1)}(z)= 	\begin{cases}
						\twomat{1}{0}{\dfrac{c_j \left[ \delta(z_j)\right]^{-2} }{z-z_j} e^{2i\theta(z_j)}}{1}	&	z\in \gamma_j, \\
						\\
						\twomat{1}{\dfrac{\overline{c_j}\left[ \delta(\overline{z_j})\right]^2 }{z -\overline{z_j}} e^{-2i\theta(\overline{z_j})} }{0}{1}
							&	z \in \gamma_j^*,\\
							\\
							\twomat{1}{\dfrac{-c_j \, \left[ \delta(-{z_j})\right]^{2} }{z + z_j} e^{-2i\theta(-{z_j})}  }{0}{1}	&	z\in -\gamma_j, \\
						\\
						\twomat{1}{0}{\dfrac{-\overline{c_j} \, \left[\delta(-\overline{z_j})\right]^{-2}e^{2i\theta(-\overline{z_j})} }{z +\overline{z_j }}}{1}
							&	z \in -\gamma_j^*.
					\end{cases}
$$					
\end{enumerate}
\end{problem}
\begin{remark}
\label{circles}
We set
\begin{equation}
\label{Gamma}
\Gamma=  \left( \bigcup_{k=1}^{N_1} \tgamma_k  \right) \cup \left( \bigcup_{k=1}^{N_1} \tgamma_k^*  \right) \cup \left( \bigcup_{j=1}^{N_2} \pm \gamma_j  \right)\cup \left( \bigcup_{j=1}^{N_2} \pm \gamma_j^*  \right).
\end{equation}
From the signature table Figure  \ref{sig-table} 
and the triangularities of the jump matrices, we observe that along the characteristic line $x=\mathrm{v}t$ where $\mathrm{v}={ (1-\rho^2_\ell )} / { (1+\rho^2_\ell) }$,  by choosing the radius of each element of $\Gamma$ small enough, we have for $z\in \Gamma \setminus \left( \pm \gamma_\ell \cup \pm \gamma^*_\ell   \right)   $ 
$$ |e^{-i\theta\ad\sigma_3}v^{(1)}(z) -I|\lesssim e^{-ct},  \quad t\to \infty.$$
For technical purpose which will become clear later, we want that the radius of each element of $\Gamma$ less than $\Upsilon/3$ where $\Upsilon$ is given by \eqref{Up}. Also we make each element of $\Gamma$ is invariant under the Schwarz reflection.
\end{remark}

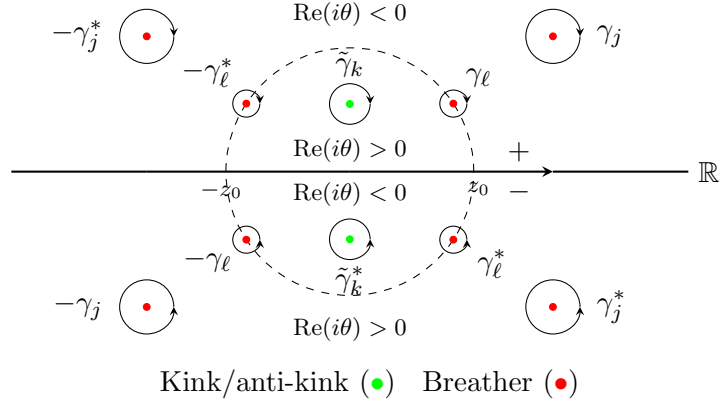
\begin{figure}
\caption{The Augmented Contour $\Sigma$}
\vspace{.5cm}
\label{figure-Sigma}
\begin{tikzpicture}[scale=0.9]
 \draw[ thick] (0,0) -- (-3,0);
\draw[ thick] (-3,0) -- (-5,0);
\draw[thick,->,>=stealth] (0,0) -- (3,0);
\draw[ thick] (3,0) -- (5,0);
\node[above] at 		(2.5,0) {$+$};
\node[below] at 		(2.5,0) {$-$};
\node[right] at (3.5 , 2) {$\gamma_j$};
\node[right] at (3.5 , -2) {$\gamma_j^*$};
\node[left] at (-3.5 , 2) {$-\gamma_j^*$};
\node[left] at (-3.5 , -2) {$-\gamma_j$};
\draw[->,>=stealth] (-2.6,2) arc(360:0:0.4);
\draw[->,>=stealth] (3.4,2) arc(360:0:0.4);
\draw[->,>=stealth] (-2.6,-2) arc(0:360:0.4);
\draw[->,>=stealth] (3.4,-2) arc(0:360:0.4);
\draw [red, fill=red] (-3,2) circle [radius=0.05];
\draw [red, fill=red] (3,2) circle [radius=0.05];
\draw [red, fill=red] (-3,-2) circle [radius=0.05];
\draw [red, fill=red] (3,-2) circle [radius=0.05];
\node[right] at (5 , 0) {$\bbR$};

\draw (0,0) [dashed] circle [ radius=1.83];

\draw [green, fill=green] (0, 1) circle [radius=0.05];
\draw[->,>=stealth] (0.3, 1) arc(360:0:0.3);
\draw [green, fill=green] (0, -1) circle [radius=0.05];
\draw[->,>=stealth] (0.3, -1) arc(0:360:0.3);
\node[above] at 		(0, 1.25) {$\tgamma_k$};
\node[below] at 		(0, -1.25) {$\tgamma_k^*$};
 \node [below] at (1.9,0) {\footnotesize $z_0$};
    \node [below] at (-1.9,0) {\footnotesize $-z_0$};
     \draw	[fill, red]  (1.529, 1.006)		circle[radius=0.05];	  
     \draw[->,>=stealth] (1.729, 1.006) arc(360:0:0.2);
   \draw	[fill, red]  (-1.529, 1.006)	circle[radius=0.05];	  
    \draw[->,>=stealth] (-1.329, 1.006) arc(360:0:0.2);
     \draw	[fill, red]  (1.529, -1.006)		circle[radius=0.05];	  
     \draw[->,>=stealth] (1.729, -1.006) arc(0:360:0.2);
   \draw	[fill, red]  (-1.529, -1.006)	circle[radius=0.05];	  
   \draw[->,>=stealth] (-1.329, -1.006) arc(0:360:0.2); 
    \node[above]  at (0, 0) {\footnotesize $\text{Re}(i\theta)>0$};
    
      \node[below]  at (0, 0) {\footnotesize $\text{Re}(i\theta)<0$};
    
     \node[above]  at (0, 2) {\footnotesize $\text{Re}(i\theta)<0$};
     \node[below]  at (0, -2) {\footnotesize $\text{Re}(i\theta)>0$};
     \node[above] at 		(1.9, 1.1) {$\gamma_\ell$};
\node[above] at 		(-2.1, 1.1) {$-\gamma_\ell^*$};
\node[below] at 		(-2.1, -1.0) {$-\gamma_\ell$};
\node[below] at 		(2.1, -1.0) {$\gamma_\ell^*$};
\end{tikzpicture}
\begin{center}
\begin{tabular}{ccc}
Kink/anti-kink ({\color{green} $\bullet$})	&	
Breather ({\color{red} $\bullet$}) 
\end{tabular}
\end{center}
\end{figure}

\section{Contour Deformation}
\label{sec:mixed}

We now perform contour deformation on Problem \ref{prob:mkdv.RHP1}, following the standard procedure outlined in \cite{DZ93} and also \cite{BJM16} in the presence of  discrete spectrum.
Since the phase function \eqref{theta} has two critical points
at $\pm z_0$, our new contour is chosen to be
\begin{equation}
\label{new-contour}
\Sigma^{(2)} = \Sigma_1 \cup \Sigma_2 \cup \Sigma_3 \cup \Sigma_4\cup  \Sigma_5 \cup \Sigma_6 \cup \Sigma_7 \cup \Sigma_8
\end{equation}
shown in Figure \ref{new-contour} (Also see \cite[Figure 8]{CVZ99}). 

\begin{figure}[H]
\label{fig:contour-def}
\caption{Deformation from $\mathbb{R}$ to $\Sigma^{(2)}$}
\vskip 15pt
\begin{tikzpicture}[scale=0.9]

\draw 	[->,thick,>=stealth] 	(-3.564, 0)-- (-5.296, 0);	
\draw 	[thick] 	 (-5.296, 0)--(-7.028,0 );

\draw 	[thick] 	(3.564, 0)-- (5.296, 0);	
\draw 	[->,thick,>=stealth] 	(7.028,0 )-- (5.296, 0);

\draw 	[->,thick,>=stealth] 	(-3.564, 0)-- (-1.732, 0);	
\draw 	[thick] 	 (-1.732, 0)--(0,0 );

\draw 	[->,thick,>=stealth] 	(0, 0)-- (1.732, 0);	
\draw 	[thick] 	 (1.732, 0)--(3.564,0 );

\draw[thick] 			(3.564 ,0) -- (5.296, 1);								
\draw[->,thick,>=stealth] 	(   7.028 ,2  )--	(5.296, 1);

\draw[->,thick,>=stealth] 			(1.732 ,1)--(2.598 ,0.5);
\draw[thick]                                 (2.598, 0.5)--(3.564 ,0); 

\draw[->,thick,>=stealth] 			(1.732 ,-1)--(2.598 ,-0.5);
\draw[thick]                                 (2.598, -0.5)--(3.564 ,0); 

\draw[->,thick,>=stealth] 			(0,0)--(1.732 ,1);

\draw[thick] 			(3.564 ,0) -- (5.296, -1);								
\draw[->,thick,>=stealth] 	(   7.028 , -2  )--	(5.296, -1);

\draw[->,thick,>=stealth] 			(0,0)--(1.732 ,-1);

\draw[->,thick,>=stealth] 			 (-3.564 ,0)--(-1.732 ,-1) ;
\draw[->,thick,>=stealth] 			(-1.732 ,-1)--(-0.866 ,-0.5);
\draw[thick]                                 (0,0)--(-0.866 ,-0.5);
\draw[->,thick,>=stealth] 			 (-3.564 ,0)--(-1.732 ,1) ;
\draw[->,thick,>=stealth] 			(-1.732 ,1)--(-0.866 ,0.5);
\draw[thick]                                 (0,0)--(-0.866 ,0.5);

\draw[->,thick,>=stealth] 			(-3.564 ,0) -- (-5.296, 1);							
\draw[thick] 	(   -7.028 ,2  )--	(-5.296, 1);
\draw[->,thick,>=stealth] 			(-3.564 ,0) -- (-5.296, -1);							
\draw[thick] 	(   -7.028 , -2  )--	(-5.296, -1);
\draw[thick] (4,0) arc(0:30:0.6);
\node[right] at (4.2,0.15){$\pi/6$};

\draw	[fill]							(-3.564 ,0)		circle[radius=0.1];	
\draw	[fill]							(3.564, 0)		circle[radius=0.1];
\draw							(0,0)		circle[radius=0.1];
\node[below] at (-3.564,-0.1)			{$-z_0$};
\node[below] at (3.564,-0.1)			{$z_0$};
\node[right] at (5, 1.6)					{$\Sigma^{(2)}_1$};
\node[left] at (-5, 1.6)					{$\Sigma^{(2)}_2$};
\node[left] at (-5,-1.6)					{$\Sigma^{(2)}_3$};
\node[right] at (5,-1.6)				{$\Sigma^{(2)}_4$};
\node[left] at (-2.4, 1)					{$\Sigma^{(2)+}_5$};
\node[left] at (-2.4,-1)					{$\Sigma^{(2)+}_7$};
\node[left] at (-0.5, 1.2)					{$\Sigma^{(2)-}_5$};
\node[left] at (-0.5,-1.2)					{$\Sigma^{(2)-}_7$};
\node[right] at (0.5,1.2)					{$\Sigma^{(2)-}_6$};
\node[right] at (0.5,-1.2)					{$\Sigma^{(2)-}_8$};
\node[right] at (2.4, 1)					{$\Sigma^{(2)+}_6$};
\node[right] at (2.4,-1)					{$\Sigma^{(2)+}_8$};

\node[above] at (0, 1.5)				{$\Omega_1$};
\node[below] at (0, -1.5)				{$\Omega_2$};

\node[right] at (5, 0.5)				{$\Omega_3$};
\node[right] at (5, -0.5)				{$\Omega_7$};
\node[right] at (1.73, 0.5)				{$\Omega_8$};
\node[right] at (1.73, -0.5)				{$\Omega_4$};
\node[left] at (-1.73, 0.5)				{$\Omega_9$};
\node[left] at (-1.73, -0.5)				{$\Omega_5$};
\node[left] at (-5, 0.5)				{$\Omega_6$};
\node[left] at (-5, -0.5)				{$\Omega_{10}$};
\node[below] at (0, -0.1)				{0};

\end{tikzpicture}

\end{figure}
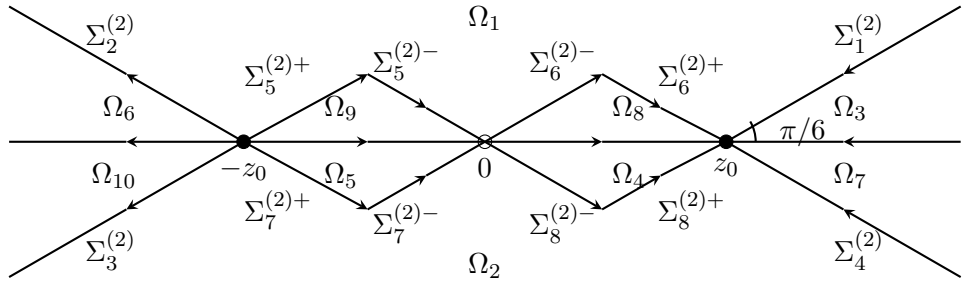

For some technical reasons (see Remark \ref{radius}), we define the following smooth cutoff function:
\begin{equation}
\label{chi-z}
\Xi_{\calZ}(z)=\begin{cases}
1 &\quad \text{dist}(z, \calZ\cup \calZ^*)\leq\Upsilon/3\\
0 &\quad \text{dist}(z, \calZ\cup \calZ^*)>2\Upsilon/3 .\\

\end{cases}
\end{equation}
Here recall that $\calZ$ is given by \eqref{mathcal-z} and $\Upsilon$ is defined in \eqref{Up}.
We now introduce another matrix-valued function $m^{(2)}$:
$$ m^{(2)}(z) = m^{(1)}(z)  \calR^{(2)}(z). $$
Here $\calR^{(2)}$ will be chosen to remove the jump on the real axis and bring about new analytic jump matrices with the desired exponential decay 
along the contour $\Sigma^{(2)}$. Straight forward computations give
\begin{align*}
m^{(2)}_+	&=m^{(1)}_+ \calR^{(2)}_+ \\
				&= m^{(1)}_- \left( e^{-i\theta\ad\sigma_3} v^{(1)} \right) \calR^{(2)}_+ \\
				&= m^{(2)}_- \left(\calR^{(2)}_-\right)^{-1}
						\left( e^{-i\theta\ad\sigma_3} v^{(1)} \right) \calR^{(2)}_+.
\end{align*}
We want to make sure that the following condition is satisfied
$$ 
\left(\calR^{(2)}_-\right)^{-1} \left( e^{-i\theta\ad\sigma_3} v^{(1)} \right) \calR^{(2)}_+ = I
$$
where $\calR_\pm^{(2)}$ are the boundary values of $\calR^{(2)}(z)$ as $z$ approaches $\sigma^{(2)}$ from the $\pm$ side of the contour. In this case the jump matrix associated to $m^{(2)}_\pm$ will be the identity matrix on $\bbR$ .

We can easily check that the function $e^{2i\theta}$ is exponentially decreasing on $\Sigma_3$  $\Sigma_4$, $\Sigma_5$, $\Sigma_6$ and increasing on $\Sigma_1$, $\Sigma_2$, $\Sigma_7$, $\Sigma_8$ while the reverse is true for $e^{-2i\theta}$. 
Letting
\begin{equation} \label{eta}
\eta(z; \pm z_0) = \left( \prod_{i\zeta_k\in B_\ell } \dfrac{\pm z_0- \overline {i\zeta_k }}{ \pm z_0-i\zeta_k}\right) \left(\prod_{z_j\in B_\ell} \dfrac{\pm z_0-\overline{z_j}}{ \pm z_0-z_j} \dfrac{\pm z_0+z_j}{\pm z_0+\overline{z_j}}\right) \left( \dfrac{ z-z_0}{z+z_0} \right)^{i\kappa}
\end{equation}
\begin{equation} \label{eta-0}
\eta_0(\pm z_0) = \left( \prod_{ i\zeta_k\in B_\ell    } \dfrac{\pm z_0- \overline {i\zeta_k }}{ \pm z_0-i\zeta_k}\right) \left(\prod_{z_j\in B_\ell} \dfrac{\pm z_0-\overline{z_j}}{ \pm z_0-z_j}\dfrac{\pm z_0+z_j}{\pm z_0+\overline{z_j}}\right) 
\end{equation}
and we define $\calR^{(2)}$ as follows (Figure \ref{fig R-2+}-\ref{fig R-2-}): 
 the functions $R_1$, $R_2$, $R_3$, $R_4$, $R_5$, $R_6$, $R_7$, $R_{8}$ satisfy 
\begin{align}
\label{R1}
R_1(z)	&=	\begin{cases}
							-{r(z)} [\delta(z)]^{-2}			
								&	z \in (z_0,\infty)\\[10pt]
						-{r(z_0 )} e^{-2\bchi(z_0)} [\eta(z; z_0)]^{-2}(1-\Xi_\calZ)
								&	z	\in \Sigma_1,
					\end{cases}\\[10pt]
\label{R2}
R_2(z)	&=	\begin{cases}
						-{r(z)} [\delta(z)]^{-2}		
								&	z \in (-\infty, -z_0)\\[10pt]
						-{r(-z_0 )} e^{-2\bchi(-z_0)} [\eta(z; -z_0)]^{-2} (1-\Xi_\calZ)
								&	z	\in \Sigma_2,
					\end{cases}\\[10pt]
\label{R3}
R_3(z)	&=	\begin{cases}
						-\overline{r(z)} [\delta(z)]^{2}			
								&	z \in (-\infty, -z_0)\\[10pt]
						-\overline{r(-z_0)} e^{2\bchi(-z_0)} [\eta(z; z_0)]^{2} (1-\Xi_\calZ)
								&	z	\in \Sigma_3,
					\end{cases}\\[10pt]
\label{R4}
R_4(z)	&=	\begin{cases}
						-\overline{r(z)}[\delta(z)]^{2}			
								&	z \in (z_0, \infty)\\[10pt]
						-\overline{r(z_0)} e^{2\bchi(z_0)} [\eta(z; z_0)]^{2} (1-\Xi_\calZ)
								&	z	\in \Sigma_4,
					\end{cases}
\end{align}

\begin{align}
\label{R5}
R_5(z)	&=	\begin{cases}
						\dfrac{\delta_+^{2}(z)  \overline{r(z)} }{1+ |r(z)|^2}		
								& z \in (-z_0, z_0)\\[10pt]
					{\dfrac{e^{2\bchi(-z_0)} [\eta(z; -z_0)]^{2} \overline{r(-z_0)} }{1+ |r(-z_0)|^2}} (1-\Xi_\calZ)
						& z \in \Sigma_5^+\\
						0 & z \in \Sigma_5^-
					\end{cases}
					\\[10pt]
\label{R8+}
R_6(z)	&=	\begin{cases}
						\dfrac{\delta_+^{2}(z)  \overline{r(z)} }{1+ |r(z)|^2}		
								& z \in (-z_0, z_0)\\[10pt]
					{\dfrac{e^{2\bchi(z_0)} [\eta(z; z_0)]^{2} \overline{r(z_0)} }{1+ |r(z_0)|^2}} (1-\Xi_\calZ)
						& z \in \Sigma_6^-\\
						0 & z \in \Sigma_6^+
					\end{cases}
					\\[10pt]
\label{R7-}
R_{7}(z)	&=	\begin{cases}
						\dfrac{\delta_-^{-2}(z)r(z)}{1+|r(z)|^2}		
								& z \in (-z_0, z_0)\\[10pt]
					\dfrac{e^{-2\bchi(-z_0)} [\eta(z; -z_0)]^{-2} r(-z_0)}{1+|r(-z_0)|^2}  (1-\Xi_\calZ)\quad
                                  & z \in \Sigma_7^+\\
                                  0  &z \in \Sigma_7^-
						\end{cases}
						\\[10pt]
\label{R8-}
R_8(z)	&=	\begin{cases}
						\dfrac{\delta_-^{-2}(z)r(z)}{1+|r(z)|^2}		
								& z \in (-z_0, z_0)\\[10pt]
					\dfrac{e^{-2\bchi(z_0)} [\eta(z; z_0)]^{-2} r(z_0)}{1+|r(z_0)|^2}  (1-\Xi_\calZ)\quad
                                  & z \in \Sigma_8^+\\
                                  0 &  z \in \Sigma_8^-.
                                  \end{cases}
\end{align}

{
\SixMatrix{The Matrix  $\calR^{(2)}$ for Region I, near $z_0$}{fig R-2+}
	{\twomat{1}{0}{R_1 e^{2i\theta}}{1}}
	{\twomat{1}{R_6 e^{-2i\theta}}{0}{1}}
	{\twomat{1}{0}{R_8 e^{2i\theta}}{1}}
	{\twomat{1}{R_4 e^{-2i\theta}}{0}{1}}
}

{
\sixmatrix{The Matrix  $\calR^{(2)}$ for Region I, near $-z_0$}{fig R-2-}
	{\twomat{1}{R_{5} e^{-2i\theta}}{0}{1}}
	{\twomat{1}{0}{R_2 e^{2i\theta}}{1}}
	{\twomat{1}{R_3 e^{-2i\theta}}{0}{1}}
	{\twomat{1}{0} {R_{7} e^{2i\theta}}{1}}
}

Each $R_i(z)$ in $\Omega_i$ is constructed in such a way that the jump matrices on the contour and $\dbar R_i(z)$ enjoys the property of exponential decay as $t\to \infty$.
We formulate Problem \ref{prob:mkdv.RHP1} into a mixed RHP-$\dbar$ problem. In the following sections we will separate this mixed problem into a localized RHP and a pure $\dbar$ problem whose long-time contribution to the asymptotics of $f(x,t)$ is  smaller than the leading term.

The following lemma (\cite[Proposition 2.1]{DM08}) will be used in the error estimates of  the
$\bar \partial$-problem in Section \ref{sec:dbar}.

We first denote the entries that appear in \eqref{R1}--\eqref{R8-} by
\begin{align}
\label{func;r-p-1}
p_1(z)=p_2(z)	&=	-r(z).	&
p_3(z)=p_4(z)	&=	- \overline{r(z)},&\\
\label{func;r-p-2}
p_{5}(z)=p_{6}(z)	&= \dfrac{ \overline{r(z)}}{1 + |r(z)|^2},& p_{7}(z)=p_{8}(z)	&= \dfrac{r(z)}{1+|r(z)|^2}.
\end{align}

For technical purpose, we further split the regions $\Omega_4$, $\Omega_5$, $\Omega_8$ and $\Omega_9$ in Figure \ref{fig:contour-2'}:
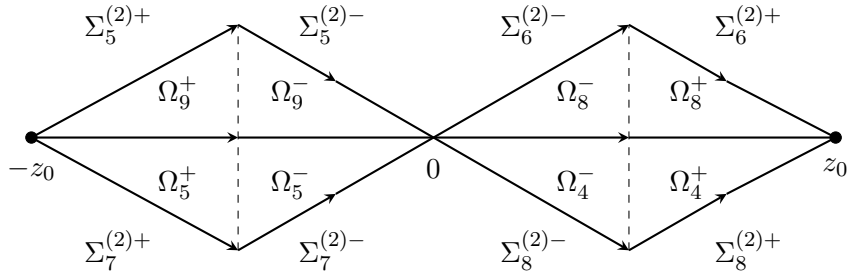
\begin{figure}[H]
\caption{Splitting the regions }
\vskip 15pt
\begin{tikzpicture}[scale=1.5]

\draw 	[->,thick,>=stealth] 	(-3.564, 0)-- (-1.732, 0);	
\draw 	[thick] 	 (-1.732, 0)--(0,0 );

\draw 	[->,thick,>=stealth] 	(0, 0)-- (1.732, 0);	
\draw 	[thick] 	 (1.732, 0)--(3.564,0 );

\draw[->,thick,>=stealth] 			(0,0)--(1.732 ,1);

\draw[->,thick,>=stealth] 			(0,0)--(1.732 ,-1);

\draw[->,thick,>=stealth] 			 (-3.564 ,0)--(-1.732 ,-1) ;
\draw[->,thick,>=stealth] 			(-1.732 ,-1)--(-0.866 ,-0.5);
\draw[thick]                                 (0,0)--(-0.866 ,-0.5);
\draw[->,thick,>=stealth] 			 (-3.564 ,0)--(-1.732 ,1) ;
\draw[->,thick,>=stealth] 			(-1.732 ,1)--(-0.866 ,0.5);
\draw[thick]                                 (0,0)--(-0.866 ,0.5);

\draw[->,thick,>=stealth] 			(1.732 ,1)--(2.598 ,0.5);
\draw[thick]                                 (2.598, 0.5)--(3.564 ,0); 

\draw[->,thick,>=stealth] 			(1.732 ,-1)--(2.598 ,-0.5);
\draw[thick]                                 (2.598, -0.5)--(3.564 ,0);

\draw	[fill]							(-3.564 ,0)		circle[radius=0.05];	
\draw	[fill]							(3.564, 0)		circle[radius=0.05];

\node[below] at (-3.564,-0.1)			{$-z_0$};
\node[below] at (3.564,-0.1)			{$z_0$};

\node[left] at (-2.4, 1)					{$\Sigma_5^{(2)+}$};
\node[left] at (-2.4,-1)					{$\Sigma_7^{(2)+}$};
\node[left] at (-0.5, 1)					{$\Sigma_5^{(2)-}$};
\node[left] at (-0.5,-1)					{$\Sigma_7^{(2)-}$};
\node[right] at (0.5,1)					{$\Sigma_6^{(2)-}$};
\node[right] at (0.5,-1)					{$\Sigma_8^{(2)-}$};
\node[right] at (2.4, 1)					{$\Sigma_6^{(2)+}$};
\node[right] at (2.4,-1)					{$\Sigma_8^{(2)+}$};

\node[right] at (2, 0.4)				{$\Omega_8^{+}$};
\node[right] at (2, -0.4)				{$\Omega_4^{+}$};
\node[left] at (-2, 0.4)				{$\Omega_9^{+}$};
\node[left] at (-2, -0.4)				{$\Omega_5^{+}$};

\node[right] at (1, 0.4)				{$\Omega_8^-$};
\node[right] at (1, -0.4)				{$\Omega_4^-$};
\node[left] at (-1, 0.4)				{$\Omega_9^-$};
\node[left] at (-1, -0.4)				{$\Omega_5^-$};
\node[below] at (0, -0.1)				{$0$};

\draw [dashed]  (-1.732, -1)--(-1.732, 1);
\draw [dashed]  (1.732, -1)--(1.732, 1);

\end{tikzpicture}
\label{fig:contour-2'}
\end{figure}
\begin{lemma}
\label{lemma:dbar.Ri}
Suppose $r \in H^{s}_0(\bbR)$ as stated in proposition \ref{prop:r}. There exist functions $R_i$ on $\Omega_3, \Omega_6, \Omega_7, \Omega_{10}$ and $\Omega_4^+, \Omega_5^+, \Omega_8^+, \Omega_9^+$ satisfying \eqref{R1}--\eqref{R8-}, so that
$$ 
|\dbar R_i(z)| \lesssim 
	 |\overline{\partial}[\mathbf{p}_i(\Real z)]| + |z-\breve{z}|^{-1/2}  +\dbar \left( \Xi_\calZ(z) \right) , 	
				$$ 
where $\breve{z}=\pm z_0$. And for functions $R_i$ on $ \Omega_4^-, \Omega_5^-, \Omega_8^-, \Omega_9^-$  we have
$$ 
|\dbar R_i(z)| \lesssim 
	 |\overline{\partial}[\mathbf{p}_1(\Real z)]| +\dfrac{\mathbf{p}_1(\Real z)}{ |z|}  +\dbar \left( \Xi_\calZ(z) \right) .
			$$ 
All the implied constants are uniform for $r $ in a bounded subset of $H^{s}(\bbR)$.
\end{lemma}

\begin{proof}
We define 
\begin{align}
\label{smooth-r}
   \mathbf{r}(z)= \begin{cases}
    r(\Real z); \quad \Imag z=0\\
   \left[ \mathcal{P}_{\Imag z}* r\right](\Real z); \quad \Imag z\neq 0
    \end{cases}
\end{align}
where we introduce an approximation to identity function ${P}(x)$ 
with compact support and
\begin{align}
   \int_\bbR\mathcal{P}(x) dx=1
\end{align}
and for $z=u+iv$
\begin{equation}
    \left[ \mathcal{P}_{\Imag z}* r\right](\Real z)=\dfrac{1}{v}\int_\bbR  \mathcal{P}\left(\dfrac{u-t}{v}  \right)r(t) dt.
\end{equation}
We use $\mathbf{p}_i$, $i=1,2...8$ to denote functions of the form as in \eqref{func;r-p-1}-\eqref{func;r-p-2} where $r$ is replaced by $\mathbf{r}$. We first prove the lemma for $R_1$. Define $f_1(z)$ on $\Omega_3$ by
$$ f_1(z) = p_1(z_0) e^{-2\chi(z_0)} [\eta(z; z_0)]^{-2} [\delta(z)]^{2} $$
and let
\begin{equation}
\label{interpol}
\ R_1(z) = \left( f_1(z) + \left[\mathbf{p}_1(z) - f_1(z) \right] \mathcal{K}(\phi) \right) [\delta(z)]^{-2} 
(1-\Xi_\calZ)
\end{equation}
where $\phi = \arg (z-z_0)$ and $\mathcal{K}$ is a smooth function on $(0, \pi/6)$ with
$$ 
\mathcal{K}(\phi)=
	\begin{cases}
			1
			&	z\in [0, \pi/24], \\
			0
			&	z \in[\pi/12, \pi/6].
	\end{cases}
$$ 
 It is easy to see that $R_1$ as constructed has the boundary values \eqref{R1}.
Writing $z-z_0= \rho e^{i\phi}$, we have
$$ \dbar = \frac{1}{2}\left( \frac{\dee}{\dee u} + i \frac{\dee}{\dee v} \right)
			=	\frac{1}{2} e^{i\phi} \left( \frac{\dee}{\dee \rho} + \frac{i}{\rho} \frac{\dee}{\dee \phi} \right).
$$
We calculate
\begin{align*}
\dbar R_1 (z) & = \dfrac{1}{2}\left(  \overline{\partial}[\mathbf{p}_1(z)] \mathcal{K}(\phi)  ~ [\delta(z)]^{-2} -
		\left[ \mathbf{p}_1(z) - f_1(z) \right][\delta(z)]^{-2}  \frac{ie^{i\phi}}{|z-z_0|}  \mathcal{K}'(\phi) \right)\\
		 & \quad \times \left( 1-\Xi_\calZ  \right) e^{2i\theta}-\left( f_1(z) + \left[ \mathbf{p}_1(z) - f_1(z) \right] \mathcal{K}(\phi) \right)[ \delta(z)]^{-2} \dbar \left( \Xi_\calZ(z) \right) e^{2i\theta}.
\end{align*}
We also have the following estimate:
\begin{align*}
  \left[ \mathcal{P}_{\Imag z}* r\right](\Real z)-r(z_0) &=  \dfrac{1}{v}\int_\bbR  \mathcal{P}\left(\dfrac{u-t}{v}  \right)\left(r(t)-r(z_0)\right) dt\\
  &=\dfrac{1}{v}\int_{B(u, v)}\mathcal{P}\left(\dfrac{u-t}{v}  \right)\left(r(t)-r(z_0)\right) dt\\
  &\lesssim \dfrac{1}{v}\int_{B(u, v)}\mathcal{P}\left(\dfrac{u-t}{v} \right)\norm{r}{H^{s}(\bbR)}|t-z_0|^{s-1/2} dt\\
   &\lesssim |z-z_0|^{s-1/2} 
  \end{align*}

Given that $\Xi(z)$ is infinitely smooth and compactly supported, it follows from Lemma \ref{lemma:delta} (iv) that
$$ 
 \left|\left( \dbar R_1 \right)(z)  \right| \lesssim
\left[  \left\vert \overline{\partial}[\mathbf{p}_1( z)]\right\vert + |z-z_0|^{s-3/2} +\dbar \left( \Xi_\calZ(z) \right) \right] |e^{2i\theta}|.
$$
 We then move to region $\Omega_8^-$.
We set
$$ f_8(z) = p_8(0) e^{-2\chi(0)} [\eta(z; 0)]^{-2} [\delta(z)]^{2} =0$$
and calculate
\begin{align*}
\dbar R_8 (z) &=\dbar  \left(\mathbf{p}_8(\Real z) \mathcal{K}(\phi) [ \delta(z)]^{-2} (1-\Xi_\calZ)  \right) \\
                     & = \left( \frac{1}{2} \overline{\partial}[\mathbf{p}_8(z)] \mathcal{K}(\phi)  ~ [\delta(z)]^{-2} -
	\mathbf{p}_8( z)[\delta(z)]^{-2}  \frac{ie^{i\phi}}{| z |}  \mathcal{K}'(\phi) \right) \left( 1-\Xi_\calZ  \right)\\
		 &\quad - \mathbf{p}_8( z)  \mathcal{K}(\phi)  [\delta(z)]^{-2} \dbar \left( \Xi_\calZ(z) \right) .
\end{align*}
The estimates in the remaining regions are analogous.
\end{proof}
\begin{remark}
We note that the interpolation defined through \eqref{interpol} introduces  new jumps on $\Sigma'^{(2)}_9$ and $\Sigma'^{(2)}_{10}$  with  the jump matrix given by 
\begin{equation}
\label{jump v9}
v_9=\begin{cases} I, & z\in \left(-i ( z_0/2) \tan(\pi/24) , i ( z_0/2) \tan(\pi/24) \right)\\
\\
  \unitupper{ (R_5^+-R_5^-)e^{-2i\theta} },  & z\in \left(i( z_0/2) \tan(\pi/24)-z_0/2, iz_0/(2\sqrt{3} ) -z_0/2\right) \\
        \\
          \unitupper{ (R_6^--R_6^+)e^{-2i\theta} },  & z\in \left(i ( z_0/2)  \tan(\pi/24)+z_0/2,  iz_0/(2\sqrt{3} )+z_0/2 \right) \\
        \\
        \unitlower{ (R_7^+ -R_7^-)e^{2i\theta} } , & z\in  \left( -i( z_0/2)  \tan(\pi/24)-z_0/2, - iz_0/(2\sqrt{3} )-z_0/2 \right) \\
        \\
         \unitlower{ (R_8^--R_8^+)e^{2i\theta} },  & z\in  \left ( -i( z_0/2) \tan(\pi/24)+z_0/2, -iz_0/(2\sqrt{3} )+z_0/2 \right)  .
\end{cases}
\end{equation}
\end{remark}

\begin{figure}[H]
\caption{$\Sigma'^{(2)}$}
\vskip 15pt
\begin{tikzpicture}[scale=0.8]

\draw[thick] 			(3.564 ,0) -- (5.296, 1);								
\draw[->,thick,>=stealth] 	(   7.028 ,2  )--	(5.296, 1);

\draw[->,thick,>=stealth] 			(0,0)--(1.732 ,1);

\draw[thick] 			(3.564 ,0) -- (5.296, -1);								
\draw[->,thick,>=stealth] 	(   7.028 , -2  )--	(5.296, -1);

\draw[->,thick,>=stealth] 			(0,0)--(1.732 ,-1);

\draw[->,thick,>=stealth] 			 (-3.564 ,0)--(-1.732 ,-1) ;

\draw[->,thick,>=stealth] 			 (-3.564 ,0)--(-1.732 ,1) ;

\draw[->,thick,>=stealth] 			(-3.564 ,0) -- (-5.296, 1);							
\draw[thick] 	(   -7.028 ,2  )--	(-5.296, 1);
\draw[->,thick,>=stealth] 			(-3.564 ,0) -- (-5.296, -1);							
\draw[thick] 	(   -7.028 , -2  )--	(-5.296, -1);

\draw[->,thick,>=stealth] 			(-1.732 , -1) -- (-1.732, 0);	
\draw	[thick]	(-1.732 , 0) -- (-1.732, 1);	
\draw[->,thick,>=stealth] 			(1.732 , -1) -- (1.732, 0);	
\draw	[thick]	(1.732 , 0) -- (1.732, 1);

\draw[->,thick,>=stealth] 			(1.732 ,1)--(2.598 ,0.5);
\draw[thick]                                 (2.598, 0.5)--(3.564 ,0); 

\draw[->,thick,>=stealth] 			(1.732 ,-1)--(2.598 ,-0.5);
\draw[thick]                                 (2.598, -0.5)--(3.564 ,0);

\draw	[fill]							(-3.564 ,0)		circle[radius=0.1];	
\draw	[fill]							(3.564, 0)		circle[radius=0.1];

\draw[->,thick,>=stealth] 			(-1.732 ,1)--(-0.866 ,0.5);
\draw[thick]                                 (0,0)--(-0.866 ,0.5);

\draw[->,thick,>=stealth] 			(-1.732 ,-1)--(-0.866 ,-0.5);
\draw[thick]                                 (0,0)--(-0.866 ,-0.5);

\node[below] at (-3.564,-0.1)			{$-z_0$};
\node[below] at (3.564,-0.1)			{$z_0$};

\node[right] at (5, 1.6)					{$\Sigma_1'^{(2)} $};
\node[left] at (-5, 1.6)					{$\Sigma_2'^{(2)}$};
\node[left] at (-5,-1.6)					{$\Sigma_3'^{(2)}$};
\node[right] at (5,-1.6)				{$\Sigma_4'^{(2)}$};
\node[left] at (-1,1.4)					{$\Sigma_5'^{(2)}$};
\node[left] at (-1,-1.4)					{$\Sigma_7'^{(2)}$};
\node[right] at (1,1.4)					{$\Sigma_6'^{(2)}$};
\node[right] at (1,-1.4)					{$\Sigma_8'^{(2)}$};
\node[right] at (-1.732, 0) {$\Sigma_9'^{(2)} $};
\node[right] at (1.732, 0) {$\Sigma_{10}'^{(2)} $};
\node[below] at (0, -0.2) {$0 $};

\end{tikzpicture}
\label{fig:contour-2}
\begin{center}
  \begin{tabular}{ccc}
adding $\Sigma'^{(2)}_9$ and $\Sigma'^{(2)}_{10}$ to $\Sigma^{(2)}$
\end{tabular}
 \end{center}
\end{figure}

The unknown $m^{(2)}$ satisfies a mixed $\dbar$-RHP. We first identify the jumps of $m^{(2)}$ along the contour $\Sigma'^{(2)}$. Recall that $m^{(1)}$ is analytic along the contour,  the jumps are determined entirely by 
$\mathcal{R}^{(2)}$, see \eqref{R1}--\eqref{R8-}. Away from $\Sigma'^{(2)}$, using the triangularity of $\mathcal{R}^{(2)}$, we   have that 
\begin{equation}
\label{N2.dbar}
 \dbar m^{(2)} = m^{(2)} \left( \calR^{(2)} \right)^{-1} \dbar \calR^{(2)} = m^{(2)} \dbar \calR^{(2)}. 
 \end{equation}
 
 \begin{remark}
 \label{radius}
By construction of $\mathcal{R}^{(2)}$ (see \eqref{R1}-\eqref{R8-} and \eqref{interpol}) and the choice of the radius of the circles in the set $ \Gamma$ (see Remark \ref{circles}), the right multiplication of $\mathcal{R}^{(2)}$ to $m^{(1)}$ will not change the jump conditions on circles in the set $ \Gamma$ . Thus over circles in the set $ \Gamma$, $m^{(2)}$ has the same jump matrices as given by (4) of Problem \ref{prob:mkdv.RHP1}.
\end{remark}

\begin{problem}
\label{prob:DNLS.RHP.dbar}
Given $r \in H_0^{s}(\bbR)$, find a matrix-valued function $m^{(2)}(z;x,t)$ on $\bbC \setminus  \left( \Sigma'^{(2)} \cup \Gamma \right)$ where the contour $\Gamma$ is given by \eqref{Gamma} with the following properties:
\begin{enumerate}
\item		$m^{(2)}(z;x,t) \rarr I$ as $ z \rarr \infty$ in $ \bbC \setminus \left( \Sigma'^{(2)} \cup \Gamma \right)$.
\item		$m^{(2)}(z;x,t)$ is continuous for $z \in  \bbC \setminus\left( \Sigma'^{(2)} \cup \Gamma \right)$ 
			with continuous boundary values 
			$m^{(2)}_\pm(z;x,t) $
			(where $\pm$ is defined by the orientation in Figure \ref{fig:contour-def}).
\item		The jump relation $m^{(2)}_+(z;x,t)=m^{(2)}_-(z;x,t)	
			e^{-i\theta\ad\sigma_3}v^{(2)}(z)$ holds, where
			$e^{-i\theta\ad\sigma_3}v^{(2)}(z)	$ is given in and part (4) of Problem \ref{prob:mkdv.RHP1}  and \eqref{jump v9} (see also Figure \ref{fig:jumps-1}-\ref{fig:jumps-2}).
\item		The equation 
			$$
			\dbar m^{(2)} = m^{(2)} \, \dbar \calR^{(2)}
			$$ 
			holds in $\bbC \setminus  \left( \Sigma'^{(2)} \cup \Gamma \right)$, where
			$$
			\dbar \calR^{(2)}=
				\begin{doublecases}
					\Twomat{0}{0}{(\dbar R_1) e^{2i\theta}}{0}, 	& z \in \Omega_3	&&
					\Twomat{0}{(\dbar R_5)e^{-2i\theta}}{0}{0}	,	& z \in \Omega_9	\\
					\\
					\Twomat{0}{0}{(\dbar R_2)e^{2i\theta}}{0},	&	z \in \Omega_6	&&
					\Twomat{0}{(\dbar R_6)e^{-2i\theta}}{0}{0}	,	&	z	\in \Omega_8	 \\
					\\
					\Twomat{0}{(\dbar R_3) e^{-2i\theta}}{0}{0}, 	& z \in \Omega_{10}	&&
					\Twomat{0}{0} {(\dbar R_7)e^{2i\theta}}  {0}	,	& z \in \Omega_5	\\
					\\
					\Twomat{0} {(\dbar R_4)e^{-2i\theta}} {0}{0},	&	z \in \Omega_7	&&
					\Twomat{0}{0}{(\dbar R_8)e^{2i\theta}}{0}	,	&	z	\in \Omega_4 \\
					\\
					0	&\hspace{-5pt} z\in \Omega_1\cup\Omega_2	.
				\end{doublecases}
			$$
\end{enumerate}
\end{problem}

The following picture is an illustration of the jump matrices of RHP Problem \ref{prob:DNLS.RHP.dbar}. For brevity we ignore the discrete scattering data.

\begin{figure}[H]
\caption{Jump Matrices  $v^{(2)}$  for $m^{(2)}$ near $z_0$ }
\vskip 15pt
\begin{tikzpicture}[scale=0.7]
\draw[dashed] 				(-6,0) -- (6,0);							
\draw [thick] 	(0,0 )-- (1.732, 1);						
\draw [->,thick,>=stealth]   (3.464,2 ) -- (1.732, 1);
\draw [thick] 	(0,0 )-- (1.732, -1);						
\draw  [->,thick,>=stealth]  (3.464,-2 ) -- (1.732, -1);
\draw [thick] 	(0,0 )-- (-1.732, 1);						
\draw [->,thick,>=stealth] (-3.464,2 ) -- (-1.732, 1);
\draw [thick] 	(0,0 )-- (-1.732, -1);						
\draw [->,thick,>=stealth]  (-3.464, -2 ) -- (-1.732, -1);
\draw[fill]						(0,0)	circle[radius=0.075];		
\node [below] at  			(0,-0.15)		{$z_0$};
\node[right] at					(3.2,3)		{$\unitlower{R_1 e^{2i\theta}}$};
\node[left] at					(-3.2,3)		{$\unitupper{R_6 e^{-2i\theta}}$};
\node[left] at					(-3.2,-3)		{$\unitlower{R_8 e^{2i\theta}}$};
\node[right] at					(3.2,-3)		{$\unitupper{R_4 e^{-2i\theta}}$};
\node[left] at					(2.5, 1.8)		{$\Sigma_1$};
\node[right] at					(-2.5,1.8)		{$\Sigma_6$};
\node[right] at					(-2.5,-1.8)		{$\Sigma_8$};
\node[left] at					(2.5,-1.8)		{$\Sigma_4$};
\end{tikzpicture}
\label{fig:jumps-1}
\end{figure}

\begin{figure}[H]
\caption{Jump Matrices  $v^{(2)}$  for $m^{(2)}$ near $-z_0$}
\vskip 15pt
\begin{tikzpicture}[scale=0.7]
\draw[dashed] 				(-6,0) -- (6,0);							
\draw [->,thick,>=stealth] 	(0,0 )-- (1.732, 1);						
\draw  [thick]  (3.464,2 ) -- (1.732, 1);
\draw [->,thick,>=stealth] 	(0,0 )-- (1.732, -1);						
\draw  [thick]  (3.464,-2 ) -- (1.732, -1);
\draw [->,thick,>=stealth] 	(0,0 )-- (-1.732, 1);						
\draw  [thick]  (-3.464,2 ) -- (-1.732, 1);
\draw [->,thick,>=stealth] 	(0,0 )-- (-1.732, -1);						
\draw  [thick]  (-3.464, -2 ) -- (-1.732, -1);
\draw[fill]						(0,0)	circle[radius=0.075];		
\node [below] at  			(0,-0.15)		{$-z_0$};
\node[right] at					(3.2,3)		{$\unitupper{R_5 e^{-2i\theta}}$};
\node[left] at					(-3.2,3)		{$\unitlower{R_2 e^{2i\theta}}$};
\node[left] at					(-3.2,-3)		{$\unitupper{R_3 e^{-2i\theta}}$};
\node[right] at					(3.2,-3)		{$\unitlower{R_7 e^{2i\theta}}$};
\node[left] at					(2.5, 1.8)		{$\Sigma_5$};
\node[right] at					(-2.5, 1.8)		{$\Sigma_2$};
\node[right] at					(-2.5,-1.8)		{$\Sigma_3$};
\node[left] at					(2.5,-1.8)		{$\Sigma_7$};
\end{tikzpicture}
\label{fig:jumps-2}
\end{figure}
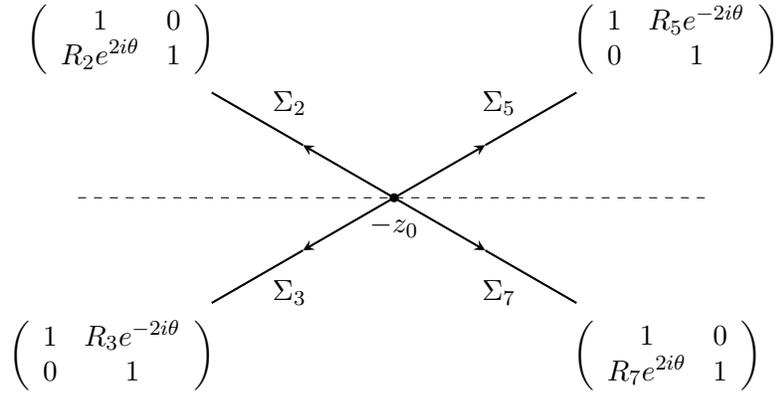

%
%

\section{The Localized Riemann-Hilbert Problem}
\label{sec:local}
We perform the following factorization of $m^{(2)}$:
\begin{equation}
\label{factor-LC}
m^{(2)} = m^{(3)} m^{\RHP}.
\end{equation}
Here we require that $m^{(3)} $ to be the solution of the pure $\dbar$-problem, hence no jump, and $ m^{\RHP}$ solution of the localized  RHP Problem \ref{MKDV.RHP.local} below
with the jump matrix $v^\RHP=v^{(2)}$. The current section focuses on finding $ m^{\RHP}$.
\begin{problem}
\label{MKDV.RHP.local}
Find a $2\times 2$ matrix-valued function $m^\RHP(z; x,t)$, analytic on $\bbC \setminus   \left( \Sigma'^{(2)} \cup \Gamma \right)$,
with the following properties:
\begin{enumerate}
\item	$m^\RHP(z;x,t) \rarr I$ as $|z| \rarr \infty$ in $\bbC \setminus (\Sigma'^{(2)}\cup\Gamma)$, where $I$ is the $2\times 2$ identity matrix,
\item	$m^\RHP(z; x,t)$ is analytic for $z \in \bbC \setminus  (\Sigma'^{(2)}\cup\Gamma)$ with continuous boundary values $m^\RHP_\pm$
		on $\Sigma'^{(2)}\cup\Gamma $,
\item	The jump relation $m^\RHP_+(z;x,t) = m^\RHP_-(z; x,t ) v^\RHP(z)$ holds on $\Sigma'^{(2)}\cup\Gamma $, where
		\begin{equation*}	
		v^\RHP(z) =	v^{(2)}(z).
		\end{equation*}
\end{enumerate}
\end{problem}

\begin{figure}[H]
\caption{ $\Sigma'^{(2)}\cup\Gamma$ }

\vskip 15pt
\begin{tikzpicture}[scale=0.9]

\draw[thick] 		(3.564 ,0) -- (5.296, 1);								
\draw[->,thick,>=stealth] [red]	(   7.028 ,2  )--	(5.296, 1);

\draw[->,thick,>=stealth] 	 [red]			(0,0)--(1.732 ,1);

\draw[thick] 			(3.564 ,0) -- (5.296, -1);								
\draw[->,thick,>=stealth]  [red]		(   7.028 , -2  )--	(5.296, -1);

\draw[->,thick,>=stealth] 		 [red]		(0,0)--(1.732 ,-1);

\draw[->,thick,>=stealth] 			 (-3.564 ,0)--(-1.732 ,-1) ;

\draw[->,thick,>=stealth] 			 (-3.564 ,0)--(-1.732 ,1) ;

\draw[->,thick,>=stealth] 			(-3.564 ,0) -- (-5.296, 1);							
\draw[thick]  [red]		(   -7.028 ,2  )--	(-5.296, 1);
\draw[->,thick,>=stealth] 			(-3.564 ,0) -- (-5.296, -1);							
\draw[thick]  [red]		(   -7.028 , -2  )--	(-5.296, -1);

\draw[->,thick,>=stealth]  [red]				(-1.732 , -1) -- (-1.732, 0);	
\draw	[thick]  [red]		(-1.732 , 0) -- (-1.732, 1);	
\draw[->,thick,>=stealth] 	 [red]			(1.732 , -1) -- (1.732, 0);	
\draw	[thick]	 [red]	 (1.732 , 0) -- (1.732, 1);

\draw [blue][dashed] (0,0) circle[radius=3.564];

\draw	[fill]							(-3.564 ,0)		circle[radius=0.06];	
\draw	[fill]							(3.564, 0)		circle[radius=0.06];

\node[below] at (-3.564,-0.1)			{$-z_0$};
\node[below] at (3.564,-0.1)			{$z_0$};

\node[right] at (5, 1.6)					{$\Sigma_1'^{(2)} $};
\node[left] at (-5, 1.6)					{$\Sigma_2'^{(2)}$};
\node[left] at (-5,-1.6)					{$\Sigma_3'^{(2)}$};
\node[right] at (5,-1.6)				{$\Sigma_4'^{(2)}$};
\node[left] at (-2.1, 1)					{$\Sigma_5^{(2)+}$};
\node[left] at (-2.4,-1)					{$\Sigma_7^{(2)+}$};
\node[left] at (-0.2, 1)					{$\Sigma_5^{(2)-}$};
\node[left] at (-0.2,-1)					{$\Sigma_7^{(2)-}$};
\node[right] at (0.2,1)					{$\Sigma_6^{(2)-}$};
\node[right] at (0.2,-1)					{$\Sigma_8^{(2)-}$};
\node[right] at (2.1, 1)					{$\Sigma_6^{(2)+}$};
\node[right] at (2.1,-1)					{$\Sigma_8^{(2)+}$};
\node[right] at (-1.732, 0) {$\Sigma_9'^{(2)} $};
\node[right] at (1.732, 0) {$\Sigma_{10}'^{(2 )} $};
\node[below] at (0, -0.1) {0};

 \draw [red, fill=red] (1, 1.732) circle [radius=0.05];
\draw[->,>=stealth] [red] (1.2, 1.732) arc(360:0:0.2);
\draw [red, fill=red] (-1,1.732) circle [radius=0.05];
\draw[->,>=stealth] [red] (-0.8, 1.732) arc(360:0:0.2);
\draw [red, fill=red] (1,-1.732) circle [radius=0.05];
\draw[->,>=stealth] [red] (1.2, -1.732) arc(0:360:0.2);
\draw [red, fill=red] (-1,-1.732) circle [radius=0.05];
\draw[->,>=stealth] [red] (-0.8, -1.732) arc(0:360:0.2);
\draw [green, fill=green]  (0, 2.7) circle [radius=0.05];
\draw[->,>=stealth] [red] (0.3, 2.7) arc(360:0:0.3);
\draw [green, fill=green] (0, -2.7) circle [radius=0.05];
\draw[->,>=stealth] [red]  (0.3, -2.7) arc(0:360:0.3);
\draw[->,>=stealth] (-2.6,2) arc(360:0:0.4);
\draw[->,>=stealth] (3.4,2) arc(360:0:0.4);
\draw[->,>=stealth] (-2.6,-2) arc(0:360:0.4);
\draw[->,>=stealth] (3.4,-2) arc(0:360:0.4);
\draw [red, fill=red] (-3,2) circle [radius=0.1];
\draw [red, fill=red] (3,2) circle [radius=0.1];
\draw [red, fill=red] (-3,-2) circle [radius=0.1];
\draw [red, fill=red] (3,-2) circle [radius=0.1];

\draw[->,thick,>=stealth] [red]			(-1.732 ,1)--(-0.866 ,0.5);
\draw[thick]           [red]                      (0,0)--(-0.866 ,0.5);

\draw[->,thick,>=stealth] [red]			(-1.732 ,-1)--(-0.866 ,-0.5);
\draw[thick]           [red]                      (0,0)--(-0.866 ,-0.5);

\draw[->,thick,>=stealth] 			(1.732 ,1)--(2.598 ,0.5);
\draw[thick]                                 (2.598, 0.5)--(3.564 ,0); 

\draw[->,thick,>=stealth] 			(1.732 ,-1)--(2.598 , -0.5);
\draw[thick]                                 (2.598, -0.5)--(3.564 ,0); 

\end{tikzpicture}
\label{fig:contour-d}
\begin{center}
  \begin{tabular}{ccc}
$v^{(2)}$ decays exponentially on red contours
\end{tabular}
 \end{center}
\end{figure}
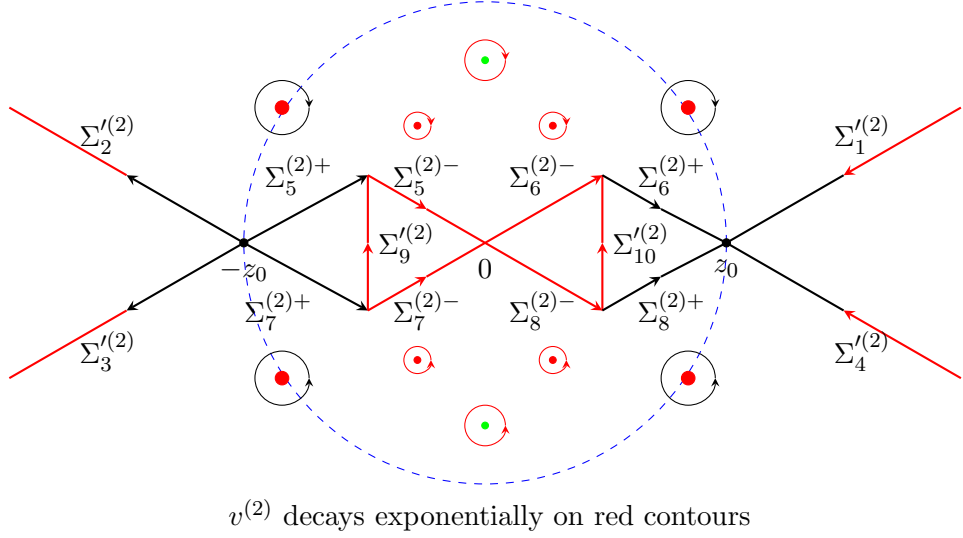
We will eventually construct a classical solution to problem \ref{MKDV.RHP.local}. We make the following observations.  On figure \ref{fig:contour-d}, for some fixed $\eps>0$, we define 
\begin{align*}
L_\eps &=\lbrace z: z=u z_0 e^{i\pi/6},  0 \leq u\leq 1/ \sqrt{3} \rbrace\\
           &\cup   \lbrace z: z=u z_0 e^{i 5\pi/6},  0\leq u\leq  1/ \sqrt{3} \rbrace \\
           & \cup \lbrace z: z= z_0+ u z_0 e^{i\pi/6}, \eps\leq u <+\infty\rbrace\\
             &\cup   \lbrace z: z=-z_0+u z_0 e^{5i\pi/6},  \eps\leq u < +\infty \rbrace \\
           \Sigma'&= \left(\Sigma'^{(2)}\setminus (L_\eps\cup {L_\eps^*}\cup \Sigma'^{(2)}_9 \cup\Sigma'^{(2)}_{10} )\right) \cup \left( \pm \gamma_\ell  \right)  \cup \left( \pm \gamma^*_\ell\right) . 
\end{align*}

\begin{figure}[H]
\caption{ $\Sigma'$}
\vskip 15pt
\begin{tikzpicture}[scale=1.2]

\draw[thick]		(2,0) -- (2.866, 0.5);								

\draw[thick] 		(-2,0) -- (-2.866, 0.5);					

\draw[thick] 		(-2,0) -- (-2.866, -0.5);					

\draw[thick]		(2,0) -- (2.866,-0.5);								

\draw[thick] 	(-1.134, 0.5)--(-2, 0) ; 
\draw[thick] 	(-1.134, -0.5)--(-2, 0) ;

\draw[thick]	(2,0) -- (1.134, 0.5);
\draw[thick]	(2,0) -- (1.134, -0.5);

\draw	[fill]							(-2,0)		circle[radius=0.1];	
\draw	[fill]							(2,0)		circle[radius=0.1];
\node[below] at (-2,-0.1)			{$-z_0$};
\node[below] at (2,-0.1)			{$z_0$};

\draw [blue] [dashed] (0,0) circle[radius=2];	

\draw [red, fill=red] (1, 1.732) circle [radius=0.07];
\draw[->,>=stealth]  (1.2, 1.732) arc(360:0:0.2);
\draw [red, fill=red] (-1,1.732) circle [radius=0.07];
\draw[->,>=stealth]  (-0.8, 1.732) arc(360:0:0.2);
\draw [red, fill=red] (1,-1.732) circle [radius=0.07];
\draw[->,>=stealth] (1.2, -1.732) arc(0:360:0.2);
\draw [red, fill=red] (-1,-1.732) circle [radius=0.07];
\draw[->,>=stealth] (-0.8, -1.732) arc(0:360:0.2);

  \node[above]  at (0, 0) {\footnotesize $\text{Re}(i\theta)>0$};
    
      \node[below]  at (0, 0) {\footnotesize $\text{Re}(i\theta)<0$};
    
     \node[above]  at (0, 2.5) {\footnotesize $\text{Re}(i\theta)<0$};
     \node[below]  at (0, -2.5) {\footnotesize $\text{Re}(i\theta)>0$};
     \draw[dashed ] (0,0) -- (-3,0);
\draw[ dashed] (-3,0) -- (-5,0);
\draw[->,>=stealth] [dashed](0,0) -- (3,0);
 \draw[ dashed] (3,0) -- (5,0);
 \node[right] at (5,0) {$\bbR$};
\end{tikzpicture}
\label{fig:sigma'}
\end{figure}

Here $\Sigma'$ is the black portion of the contour $\Sigma'^{(2)}\cup\Gamma $ given in Figure \ref{fig:contour-d}. Our goal is to build explicitly solvable models out of this contour. Now we decompose $w^{ ( 2 )}_\theta = v^{ ( 2 )}_\theta -I $ into two parts:
\begin{equation}
w^{ ( 2 )}_\theta=w^e+w'
\end{equation}
where $w'=w^{ ( 2 )}_\theta\restriction_{  \Sigma' }$ and $w^e=w^{ ( 2 )}_\theta\restriction_{   \left( \Sigma^{(2)}\cup\Gamma \right) \setminus\Sigma'   }$. 

Near $\pm z_0$, we write
$$\text{Re} i\theta(z; x, t)=\dfrac{1}{2}( \text{Im} z ) \left( \dfrac{1}{1+z_0^2}  \right) \left(   \dfrac{z_0^2}{(\text{Re}z )^2 + (\text{Im}z )^2 } -1\right)$$
and set 
\begin{equation}
\label{tau}
\tau=\dfrac{t z_0}{1+z_0^2}.
\end{equation}
On $L_\eps$, away from $\pm z_0$, for $i=1,  2, 7, 8$ we estimate:
\begin{equation}
\label{Ri-decay}
\left\vert R_i e^{2i\theta} \right\vert \leq C_r e^{-k \eps \tau}.
\end{equation}
Similarly, on ${L_\eps^*}$  for $j=3,  4, 5, 6$
\begin{equation}
\label{Rj-decay}
\left\vert R_j e^{-2i\theta} \right\vert  \leq C_r e^{-k \eps \tau}.
\end{equation}
Also notice that
on $\Sigma'^{(2)}_9$ and $\Sigma^{(2)}_{10}$  by the construction of  $\mathcal{K}(\phi)$ and $v_9$, one obtains
\begin{equation}
\label{R9-decay}
\left\vert v_9 -I \right\vert \lesssim e^{-c\tau}.
\end{equation}
Combining Remark \ref{circles} with the discussion above we conclude that 
\begin{equation}
\label{expo}
|w^e|\lesssim e^{-c\tau}.
\end{equation}

\subsection{Construction of local parametrices}
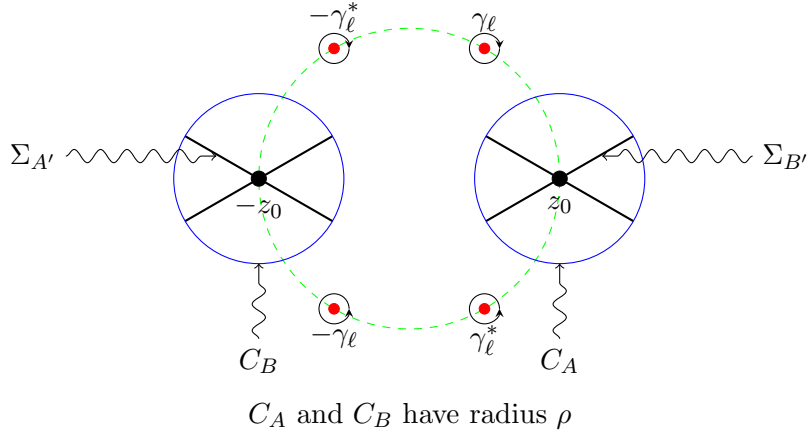
\begin{figure}[H]
\caption{ $\Sigma'=\Sigma_{A'}\cup \Sigma_{B'}\cup \pm \gamma_\ell \cup \gamma^*_\ell$}
\vskip 15pt
\begin{tikzpicture}[ photon/.style={decorate,decoration={snake,post length=0.8mm}} ]

\draw [green] [dashed] (0,0) circle [radius=2];

\draw  [thick]   (2.9794, 0.5655)-- (1.02, -0.5655);
\draw  [thick]   (2.9794, -0.5655)-- (1.02, 0.5655);
\draw  [thick]   (-2.9794, 0.5655)-- (-1.02, -0.5655);
\draw  [thick]   (-2.9794, -0.5655)-- (-1.02, 0.5655);
\draw	[fill]						(-2,0)		circle[radius=0.1];	
\draw	[fill]					(2,0)		circle[radius=0.1];
\draw		[blue]				(-2,0)		circle[radius=1.131];	
\draw		[blue]					(2,0)		circle[radius=1.131];
\node[below] at (-2,-0.1)			{$-z_0$};
\node[below] at (2,-0.1)			{$z_0$};
\node[left] at (-4.56, 0.3)					{$\Sigma_{A'}$};
\node[right] at (4.56, 0.3)				{$\Sigma_{B'}$};

\draw [red, fill=red] (1, 1.732) circle [radius=0.07];
\draw[->,>=stealth]  (1.2, 1.732) arc(360:0:0.2);
\draw [red, fill=red] (-1,1.732) circle [radius=0.07];
\draw[->,>=stealth] (-0.8, 1.732) arc(360:0:0.2);
\draw [red, fill=red] (1,-1.732) circle [radius=0.07];
\draw[->,>=stealth] (1.2, -1.732) arc(0:360:0.2);
\draw [red, fill=red] (-1,-1.732) circle [radius=0.07];
\draw[->,>=stealth]  (-0.8, -1.732) arc(0:360:0.2);

\draw[->,photon] ( 4.565,0.3) --  (2.565, 0.3); 
\draw[->,photon] ( -4.565,0.3) --  (-2.565, 0.3); 
\draw[->,photon] (-2, -2.131) --  (-2, -1.131); 
\draw[->,photon] (2, -2.131) --  (2, -1.131); 
\node [below] at  (-2, -2.131)  {$C_B$};
\node [below] at  (2, -2.131) {$C_A$};
\node [above] at (1, 1.82) {$\gamma_\ell$};
\node [above] at (-1, 1.82) {$-\gamma_\ell^*$};
\node [below] at (-1, -1.82) {$-\gamma_\ell$};
\node [below] at (1, -1.82)  {$\gamma_\ell^*$};
\end{tikzpicture}
\begin{center}
  \begin{tabular}{ccc}
$C_A$ and $C_B$ have radius $\rho$
\end{tabular}
\end{center}
\label{fig:contour-2}
\end{figure}
In this subsection we construct some local parametrices which are needed to obtain $m^\RHP$. To achieve this, we build the solutions of the following three exactly solvable RHPs:
\begin{problem}
\label{prob:mkdv.br}
Find a matrix-valued function $m^{(br)}(z;x,t)$ on $\bbC \setminus\Sigma$ with the following properties:
\begin{enumerate}
\item		$m^{(br)}(z;x,t) \rarr I$ as $|z| \rarr \infty$,
\item		$m^{(br)}(z;x,t)$ is analytic for $z \in  \bbC \setminus ( \pm \gamma_\ell \cup \pm\gamma^*_\ell ) $
			with continuous boundary values
			$m^{(br)}_\pm(z;x,t)$.
\item On $ \pm \gamma_\ell \cup \pm\gamma^*_\ell$, let $\delta(z)$ be the solution to Problem \ref{prob:RH.delta} and we have the following jump conditions
$m^{(br)}_+(z;x,t)=m^{(br)}_-(z;x,t)	
			e^{-i\theta\ad\sigma_3}v^{(br)}(z)$
			where
$$
e^{-i\theta\ad\sigma_3}v^{(br)}(z)= 	\begin{cases}
						\twomat{1}{0}{\dfrac{c_\ell \left[ \delta(z_\ell ) \right]^{-2} }{z-z_\ell} e^{2i\theta(z_\ell )}}{1}	&	z\in \gamma_\ell, \\
						\\
						\twomat{1}{\dfrac{\overline{c_\ell} \left[ \delta(\overline{z_\ell })\right]^2 }{z -\overline{z_\ell }} e^{-2i\theta (\overline{z_\ell })} }{0}{1}
							&	z \in \gamma_\ell^*\\
							\\
							\twomat{1}{\dfrac{-c_\ell \,\left[ \delta(-{z_\ell })\right]^{2} }{z + z_\ell } e^{-2i\theta (-{z_\ell })}  }{0}{1}	&	z\in -\gamma_\ell, \\
						\\
						\twomat{1}{0}{\dfrac{-\overline{c_\ell} \, \left[ \delta(-\overline{z_\ell })\right]^{-2}e^{2i\theta(-\overline{z_\ell })} }{z +\overline{z_\ell }}}{1}
							&	z \in -\gamma_\ell^* .
					\end{cases}
$$					
\end{enumerate}
\end{problem}

\begin{problem}
\label{prob:mkdv.A}
Find a matrix-valued function $m^{A'}(z;x,t)$ on $\bbC \setminus\Sigma_A'$ with the following properties:
\begin{enumerate}
\item		$m^{A'}(z;x,t) \rarr I$ as $ z \rarr \infty$.
\item		$m^{A' }(z;x,t)$ is analytic for $z \in  \bbC \setminus \Sigma_A' $
			with continuous boundary values
			$m^{A'}_\pm(z;x,t)$.
\item On $ \Sigma_A'$ we have the following jump conditions
$$m^{A'}_+(z;x,t)=m^{A'}_-(z;x,t)	
			e^{-i\theta\ad\sigma_3}v^{A'}(z)$$
			where $v^{A'}=v^{(2)}\restriction _{\Sigma_A' }$.
\end{enumerate}
\end{problem}			

\begin{problem}
\label{prob:mkdv.B}
Find a matrix-valued function $m^{B'}(z;x,t)$ on $\bbC \setminus\Sigma_B'$ with the following properties:
\begin{enumerate}
\item		$m^{B'}(z;x,t) \rarr I$ as $ z \rarr \infty$.
\item		$m^{B'}(z;x,t)$ is analytic for $z \in  \bbC \setminus \Sigma_B' $
			with continuous boundary values
			$m^{B’}_\pm(z;x,t)$.
\item On $ \Sigma_B'$ we have the following jump conditions
$$m^{B'}_+(z;x,t)=m^{B’}_-(z;x,t)	
			e^{-i\theta\ad\sigma_3}v^{B'}(z)$$
			where $v^{B’}=v^{(2)}\restriction_{\Sigma_B' }$.
\end{enumerate}
\end{problem}		
We first study the solution to Problem \ref{prob:mkdv.br}. Since this problem consists of only discrete data, \eqref{BC-int} reduces to a linear system. More explicitly, we have a closed system:
\begin{align}
\label{BC-int-br}
\twomat{ \mu_{11}(\overline{z_l} ) }{\mu_{12}( {z_l} )}{\mu_{21}(\overline{z_l} )}{\mu_{22}( {z_l} )} &= I +\Twomat { \dfrac{\mu_{12}(z_l )c_l  \left[ \delta(z_\ell )\right]^{-2} e^{2i\theta(z_l )} }{\overline{z_l} -z_l }  }
        { -\dfrac{\mu_{11}( \overline{z_l }) {\overline{c_l }} \left[ \delta(\overline{z_\ell })\right]^2 e^{-2i\theta(   \overline{z_l }  )} }{z_l- \overline{z_l }}   }
        {\dfrac{\mu_{22}(z_l ) c_l\left[ \delta(z_\ell )\right]^{-2} e^{2i\theta(z_l )} }{ \overline{z_l} -z_l }  }
        {- \dfrac{\mu_{21}( \overline{z_l }) {\overline{c_l }}\left[  \delta(\overline{z_\ell })\right]^2 e^{-2i\theta(   \overline{z_l }  )} }{z_l- \overline{z_l }} }\\
         \nonumber
         &\qquad +\Twomat
        { -\dfrac {\mu_{12}( -\overline{z_l }) {\overline{c_l }} \left[ \delta(-\overline{z_\ell })\right]^{-2} e^{2i\theta(  - \overline{z_l }  )} }{\overline{z_l} + \overline{z_l }}   }
        {\dfrac{\mu_{11}(-z_l )c_l \left[ \delta(-{z_\ell })\right]^{2}  e^{-2i\theta(-z_j)} }{z_l+z_l }  }
       {- \dfrac{\mu_{22}( -\overline{z_l }) {\overline{c_l }} \left[ \delta(-\overline{z_\ell })\right]^{-2} e^{2i\theta(  - \overline{z_l }  )} }{\overline{z_l}  + \overline{z_l }} }
        {\dfrac{\mu_{21}(-z_l )c_l \left[ \delta(-{z_\ell })^{2}\right]  e^{-2i\theta(-z_l )} }{z_l+z_l }  },
\end{align}
\begin{align}
\label{BC-int-br-}
\twomat{ \mu_{11}( -{z_l} ) }{\mu_{12}( -\overline{z_l} )}{\mu_{21}(-{z_l} )}{\mu_{22}(-\overline {z_l} )} &= I +\Twomat { \dfrac{\mu_{12}(z_l )c_l \left[ \delta(z_\ell )\right]^{-2} e^{2i\theta(z_l )} }{-{z_l} -z_l }  }
        { -\dfrac{\mu_{11}( \overline{z_l }) {\overline{c_l }} \left[ \delta(\overline{z_\ell })\right]^2 e^{-2i\theta(   \overline{z_l }  )} }{- \overline{z_l } - \overline{z_l } }   }
        {\dfrac{\mu_{22}(z_l )c_l \left[ \delta(z_\ell )\right]^{-2} e^{2i\theta(z_l )} }{ -{z_l} -z_l }  }
        {- \dfrac{\mu_{21}( \overline{z_l }) {\overline{c_l }} \left[ \delta(\overline{z_\ell })\right]^2 e^{-2i\theta(   \overline{z_l }  )} }{- \overline{z_l } - \overline{z_l }} }\\
         \nonumber
         &\qquad +\Twomat
        { -\dfrac {\mu_{12}( -\overline{z_l }) {\overline{c_l }} \left[ \delta(-\overline{z_\ell })\right]^{-2} e^{2i\theta(  - \overline{z_l }  )} }{-{z_l} + \overline{z_l }}   }
        {\dfrac{\mu_{11}(-z_l )c_l\left[ \delta(-{z_\ell })\right]^{2} e^{-2i\theta(-z_j)} }{-\overline{z_l} +z_l }  }
       {- \dfrac{\mu_{22}( -\overline{z_l }) {\overline{c_l }} \left[ \delta(-\overline{z_\ell })\right]^{-2} e^{2i\theta(  - \overline{z_l }  )} }{-{z_l}  + \overline{z_l }} }
        {\dfrac{\mu_{21}(-z_l )c_l \left[\delta(-{z_\ell })\right]^{2} e^{-2i\theta(-z_l )} }{-\overline{z_l} +z_l }  }.
\end{align}
Given that 
$$\delta(z)=\left(  \overline{ \delta (\overline{z}) } \right)^{-1},  $$
the Schwarz invariant condition of the jump matrices $ e^{-i\theta\ad\sigma_3}v^{(br)}(z) $ is satisfied and the solvability of this linear system \eqref{BC-int-br}-\eqref{BC-int-br-} follows. And we arrive at the following expressions:
\begin{equation}
\label{br-0}
m^{br}(0)=\twomat {m_{11}^{br}(0)   }{m_{12}^{br}(0)  }{m_{21}^{br}(0)  }{ m_{22}^{br}(0)   }
\end{equation}
where
\begin{align*}
m_{11}^{br}(0) &=m_{22}^{br}(0) \\
                       &= \dfrac{\left(1-\dfrac{\left\vert c_l \left[  \delta(z_\ell )\right]^{-2} e^{2i\theta(z_l )}\right\vert^2}{4z_\ell \zbar_\ell }-\dfrac{\left\vert c_l  \left[\delta(z_\ell )\right]^{-2}  e^{2i\theta(z_l )}\right\vert^2}{(z_\ell-\zbar_\ell   )^2}\right)^2
				-\left(\dfrac{c_l \left[ \delta(z_\ell )\right]^{-2} e^{2i\theta(z_l )}}{2z_\ell}+\dfrac{ {\overline{c_l } } \left[ \delta(\overline{z_\ell })\right]^2  e^{-2i\theta(   \overline{z_l }  )}   }{2\zbar_\ell }\right)^2} {\left(1-\dfrac{\left\vert c_l \left[ \delta(z_\ell )\right]^{-2}  e^{2i\theta(z_l )} \right\vert^2}{4z_\ell \zbar_\ell }-\dfrac{\left\vert c_l \left[ \delta(z_\ell )\right]^{-2}  e^{2i\theta(z_l )} \right\vert^2}{(z_\ell -\zbar_\ell)^2}\right)^2
				+\left(\dfrac{c_l \left[ \delta(z_\ell )\right]^{-2} e^{2i\theta(z_l )}}{2z_\ell}+\dfrac{ {\overline{c_l } } \left[ \delta(\overline{z_\ell })\right]^2  e^{-2i\theta(   \overline{z_l }  )}   }{2\zbar_\ell }\right)^2 }
\end{align*}

\begin{align*}
m_{21}^{br}(0) &=-m_{12}^{br}(0)\\
    &=-\dfrac{2\left(\dfrac{ c_l \left[ \delta(z_\ell )\right]^{-2}  e^{2i\theta(z_l )}   }{2z_\ell}+\dfrac{ \overline{c_l } \left[ \delta(\overline{z_\ell })\right]^2   e^{-2i\theta(   \overline{z_l }  )}    }{2 \zbar_\ell}\right)\left(1-\dfrac{\left\vert c_l \left[ \delta(z_\ell )\right]^{-2}  e^{2i\theta(z_l )} \right\vert^2   }{4z_\ell \zbar_\ell}-\dfrac{\left\vert c_l \left[ \delta(z_\ell )\right]^{-2}  e^{2i\theta(z_l )} \right\vert^2}{(z_\ell-\zbar_\ell)^2}\right)} {\left(1-\dfrac{\left\vert c_l \left[ \delta(z_\ell )\right]^{-2}  e^{2i\theta(z_l )} \right\vert^2}{4z_\ell \zbar_\ell}-\dfrac{\left\vert c_l \left[ \delta(z_\ell )\right]^{-2}  e^{2i\theta(z_l )} \right\vert^2}{(z_\ell-\zbar_\ell)^2}\right)^2
				+\left(\dfrac{ c_l \left[  \delta(z_\ell )\right]^{-2}  e^{2i\theta(z_l ) }   }{2z_\ell}+\dfrac{ \overline{c_l }  \left[\delta(\overline{z_\ell })\right]^2  e^{-2i\theta(   \overline{z_l }  )}  }{2 \zbar_\ell}\right)^2}.
\end{align*}
Recall that 
$$z_\ell=\rho_\ell e^{i\omega_\ell}=\xi_\ell+ i\eta_\ell.$$
We split $\theta(z_n; x, t)$ into real and imaginary parts,
\begin{equation}
		\begin{split}
				\theta(z_\ell; x, t)&=\theta_\mathbb{R}(z_\ell; x, t)+i\theta_\mathbb{I}(z_\ell; x, t),\\
				\text{with}\quad&\theta_\mathbb{R}(\xi_\ell,\eta_\ell; x, t)=
				-\frac{1}{4}\left[\left(\eta_\ell+\frac{\eta_\ell}{\xi^2_\ell+\eta^2_\ell}\right)x
				+\left(\eta_\ell-\frac{\eta_\ell}{\xi^2_\ell+\eta^2_\ell}\right)t\right],\\
				&\theta_\mathbb{I}(\xi_\ell,\eta_\ell; x, t)=
				\frac{1}{4}\left[\left(\xi_\ell+\frac{\xi_\ell}{\xi^2_\ell+\eta^2_\ell}\right)x
				+\left(\xi_\ell-\frac{\xi_\ell}{\xi^2_\ell+\eta^2_\ell}\right)t\right].\\
		\end{split}
\end{equation}
Therefore,
\begin{equation}
\begin{split}
				&1-\dfrac{\left\vert c_l \left[ \delta(z_\ell )\right]^{-2}  e^{2i\theta(z_l )} \right\vert^2}{4z_\ell \zbar_\ell}-\dfrac{\left\vert c_l \left[ \delta(z_\ell )\right]^{-2}  e^{2i\theta(z_l )} \right\vert^2}{(z_\ell-\zbar_\ell)^2}
				=e^{2\theta_\mathbb{R}}\left(e^{-2\theta_\mathbb{R}}
				+\dfrac{\xi_\ell^2\left\vert c_l  \left[\delta(z_\ell )\right]^{-2}  \right\vert^2e^{2\theta_\mathbb{R}}}{4\eta_\ell^2(\xi_\ell^2+\eta_\ell^2)}\right)\\
				&\qquad=\dfrac{\xi_\ell \left\vert c_l \left[ \delta(z_\ell )\right]^{-2} \right\vert e^{2\theta_\mathbb{R}}}{2\eta_\ell \sqrt{\xi_\ell^2+\eta_\ell^2}}
				2\cosh\left(2\theta_\mathbb{R}(\xi_\ell,\eta_\ell; x, t)+\mathrm{log}\left[\frac{\xi_\ell \left\vert c_l \left[ \delta(z_\ell )\right]^{-2}  \right\vert}{2\eta_\ell \sqrt{\xi_\ell^2+\eta_\ell^2}}\right]\right)\\
				&\dfrac{c_l \left[ \delta(z_\ell )\right]^{-2}  e^{2i\theta(z_l )} }{2z_\ell}+\dfrac{ \overline{c_l } \left[ \delta(\overline{z_\ell })\right]^2   e^{-2i\theta(   \overline{z_l }  )}    }{2 \zbar_\ell } \\
				&=\dfrac{\left\vert c_l  \left[\delta(z_\ell )\right]^{-2}   \right\vert e^{2\theta_\mathbb{R}}}{2\sqrt{\xi_\ell^2+\eta_\ell^2}}
				2\cos\left(2\theta_\mathbb{I}(\xi_\ell,\eta_\ell; x, t)+\arg\left(  c_l  \left[\delta(z_\ell )\right]^{-2}     \right)-\mathrm{arctan}\left(\frac{\eta_\ell }{\xi_\ell}\right)\right)
		\end{split}
\end{equation}
Finally,
\begin{align}
\label{br-cos}
\cos(u)=&1-\dfrac{8\,\dfrac{\eta_\ell^2}{\xi_\ell^2}\,\mathrm{sech}\left[2\theta_\mathbb{R}(z_\ell; x, t)+h_\ell\right]^2\cos\left[2\theta_\mathbb{I}(z_\ell; x, t)+\alpha_\ell\right]^2}
				{\left(1+\frac{\eta_\ell^2}{\xi_\ell^2}\mathrm{sech}\left[2\theta_\mathbb{R}(z_n; x, t)+ h_\ell\right]^2\cos\left[2\theta_\mathbb{I}(z_\ell; x, t )+\alpha_\ell\right]^2\right)^2},\\
				\label{br-sin}
				\sin(u)=&-4\frac{\eta_\ell}{\xi_\ell }\mathrm{sech}\left[2\theta_\mathbb{R}(z_\ell; x, t)+ h_\ell \right]\cos\left[2\theta_\mathbb{I}(z_\ell; x, t)+\alpha_\ell\right]\\
				\nonumber
				&\times\dfrac{
				\left(1-\dfrac{\eta_\ell^2}{\xi_\ell^2}\mathrm{sech}\left[2\theta_\mathbb{R}(z_\ell; x, t)+ h_\ell \right]^2\cos\left[2\theta_\mathbb{I}(z_\ell; x, t )+\alpha_\ell\right]^2\right)}
				{\left(1+\dfrac{\eta_\ell^2}{z_\ell; x, t}\mathrm{sech}\left[2\theta_\mathbb{R}(z_\ell; x, t)+ h_\ell \right]^2\cos\left[2\theta_\mathbb{I}(z_\ell; x, t)+\alpha_\ell\right]^2\right)^2},
\end{align}
with
\begin{equation}
\begin{split}
				2\theta_\mathbb{R}(z_\ell; x, t)=&-\frac{\eta_\ell}{2}\left[\left(1+\frac{1}{\rho_\ell^2}\right)x+\left(1-\frac{1}{  \rho_\ell^2}\right)t\right],\\
				2\theta_\mathbb{I}(z_\ell; x, t)=&\frac{\xi_\ell }{2}\left[\left(1+\frac{1}{\rho_\ell^2}\right)x+\left(1-\frac{1}{ \rho_\ell^2}\right)t\right],\\
				h_\ell=&\mathrm{log}\left(\frac{\xi_\ell \left\vert c_l \left[ \delta(z_\ell )\right]^{-2} \right\vert  }{2\eta_\ell \sqrt{\xi_\ell^2+\eta_\ell^2}}\right),\\
				\alpha_\ell=&\arg\left(  c_l  \left[\delta(z_\ell )\right]^{-2} \right)-\mathrm{arctan}\left(\frac{\eta_\ell}{\xi_\ell}\right).
		\end{split}
\end{equation}
We then find the solution to Problem \ref{prob:mkdv.A} and Problem \ref{prob:mkdv.B}.
We mention that unlike solving the model problem related to the NLS RHP, in this place, a direct scaling does not lead to the exactly solvable models. So we give the complete picture. Extend the contours $\Sigma_{A'}$ and $\Sigma_{B'}$ to
\begin{subequations}
\begin{equation}
\widehat{\Sigma }_{A'}=\left\lbrace\lbrace z=-z_0+z_0 u e^{\pm i\pi /6}: -\infty<u<\infty \right\rbrace,
\end{equation}
\begin{equation}
\widehat{\Sigma }_{B'}=\left\lbrace z=z_0+z_0 u e^{\pm i 5\pi /6}: -\infty<u<\infty \right\rbrace
\end{equation}
\end{subequations}
respectively and define $\hat{v}^{A'}$, $\hat{v}^{B'}$ on $\widehat{\Sigma }_{A'}$, $\widehat{\Sigma }_{B'}$ through
\begin{subequations}
\begin{equation}
\hat{v}^{A‘} =\begin{cases}
v^{A'}(z), &\quad z\in \Sigma_{A'}\subset \widehat{\Sigma}_{A'}, \\
       0,   &\quad z\in \widehat{\Sigma}_{A'}\setminus {\Sigma}_{A'},
\end{cases}
\end{equation}
\begin{equation}
\hat{v}^{B'} =\begin{cases}
v^{B'}(z), &\quad z\in \Sigma_{B'}\subset \widehat{\Sigma}_{B'} \\
       0,   &\quad z\in \widehat{\Sigma}_{B'}\setminus {\Sigma}_{B'}.
\end{cases}
\end{equation}
\end{subequations}
\begin{figure}[H]
\caption{$\Sigma_A,\Sigma_B$}
\vskip 15pt
\begin{tikzpicture}[scale=0.6]
\draw[dashed] 				(-4,0) -- (4,0);							
\draw [->,thick,>=stealth] 	(0,0 )-- (1.732, 1);						
\draw  [thick]  (3.464,2 ) -- (1.732, 1);
\draw [->,thick,>=stealth] 	(0,0 )-- (1.732, -1);						
\draw  [thick]  (3.464,-2 ) -- (1.732, -1);
\draw [->,thick,>=stealth] 	(0,0 )-- (-1.732, 1);						
\draw  [thick]  (-3.464,2 ) -- (-1.732, 1);
\draw [->,thick,>=stealth] 	(0,0 )-- (-1.732, -1);						
\draw  [thick]  (-3.464, -2 ) -- (-1.732, -1);

\draw[fill]						(0,0)	circle[radius=0.075];		
\node [below] at  			(0,-0.15)		{$0$};
\node[left] at					(2.5, 2.3)		{$\Sigma_A^1$};
\node[right] at					(-2.5,2.3)		{$\Sigma_A^2$};
\node[right] at					(-2.5,-2.3)		{$\Sigma_A^3$};
\node[left] at					(2.5,-2.3)		{$\Sigma_A^4$};
\end{tikzpicture}
\qquad
\begin{tikzpicture}[scale=0.6]
\draw[dashed] 				(-4,0) -- (4,0);							
\draw[dashed] 				(-4,0) -- (4,0);							
\draw [thick] 	(0,0 )-- (1.732, 1);						
\draw [->,thick,>=stealth]   (3.464,2 ) -- (1.732, 1);
\draw [thick] 	(0,0 )-- (1.732, -1);						
\draw  [->,thick,>=stealth]  (3.464,-2 ) -- (1.732, -1);
\draw [thick] 	(0,0 )-- (-1.732, 1);						
\draw [->,thick,>=stealth] (-3.464,2 ) -- (-1.732, 1);
\draw [thick] 	(0,0 )-- (-1.732, -1);						
\draw [->,thick,>=stealth]  (-3.464, -2 ) -- (-1.732, -1);

\draw[fill]						(0,0)	circle[radius=0.075];		
\node [below] at  			(0,-0.15)		{$0$};
\node[left] at					(2.5, 2.3)		{$\Sigma_B^1$};
\node[right] at					(-2.5,2.3)		{$\Sigma_B^2$};
\node[right] at					(-2.5,-2.3)		{$\Sigma_B^3$};
\node[left] at					(2.5,-2.3)		{$\Sigma_B^4$};
\end{tikzpicture}
\label{fig:jumps-A-B}
\end{figure}
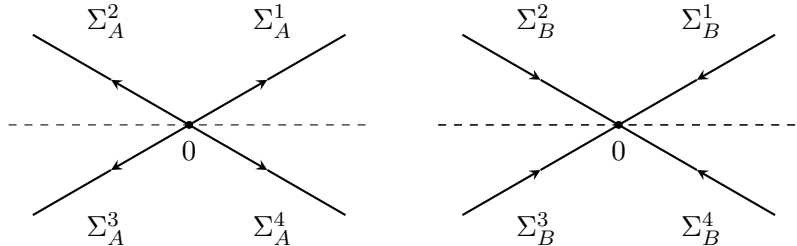
 Let $\Sigma_A$ and $\Sigma_B$ denote the contours
\begin{equation*}
\left\lbrace z=z_0 u e^{\pm i\pi/6} : -\infty<u<\infty \right\rbrace
\end{equation*}
with the same orientation as those of $\Sigma_{A'}$ and $\Sigma_{B'}$ respectively.  On 
$\widehat{\Sigma}_{A'}$ ($\widehat{\Sigma}_{B'}$) we carry out the following change of variable 
$$z\mapsto \zeta=    \sqrt{\dfrac{2t   }{ (1+z_0^2)z_0  } }  (z \pm z_0) $$
and introduce the scaling operators 
\begin{subequations}
\begin{equation}
\begin{cases}
N_A:  L^{2} ( \widehat{\Sigma}_{A'} ) \rightarrow L^2(\Sigma_A)\\
f(z)\mapsto (N_A f)(z)=f\left(  \sqrt{\dfrac{ (1+z_0^2)z_0  }{2t   }}  \zeta -z_0  \right),
\end{cases}
\end{equation}
\begin{equation}
\begin{cases}
N_B:  L^{2} ( \widehat{\Sigma}_{B'} ) \rightarrow L^2(\Sigma_B)\\
f(z)\mapsto (N_B f)(z)=f\left(\sqrt{\dfrac{ (1+z_0^2)z_0  }{2t   }}\zeta + z_0  \right).
\end{cases}
\end{equation}
\end{subequations}
We also define
\begin{equation}
\label{operator a-b}
1_A=1\restriction_{\Sigma_A}, \quad 1_B=1\restriction_{\Sigma_B}.
\end{equation}
We first consider the case $\Sigma_B$. The rescaling gives 
\begin{equation*}
N_B \left( e^{\chi(z_0)} \eta(z; z_0)  e^{-it\theta}\right)=\delta^0_B\delta^1_B(\zeta)
\end{equation*}
with 
\begin{align*}
\delta^0_B &= (8\tau)^{-i\kappa/2} e^{-i\tau} e^{\chi(z_0)} \eta_0(z_0) \\
\delta^1_B(\zeta) &= \zeta^{i\kappa} \left( \dfrac{2 z_0}{\sqrt{\dfrac{ (1+z_0^2)z_0  }{2t   }}\zeta+2z_0}\right)^{i\kappa} \exp\left[ {(-i\zeta^2/4) \left(1 -\dfrac{ z_0^4 }{\sqrt{2} \xi^4 \sqrt{\tau} }\zeta\right) } \right] 
\end{align*}
where 
$$\xi=z_0+k \sqrt{\dfrac{ (1+z_0^2)z_0  }{2t   }}\zeta, \quad 0<k<1. $$
Note that $\delta^0_B(z)$ is independent of $z$ and that $|\delta^0_B(z)|=1$. 
Set 
\begin{align*}
\Delta^0_B &=(\delta_B^0)^{\sigma_3}\\
w^B(\zeta)     &=(\Delta^0_B)^{-1}  (N_B \hat{w}^{B'})\Delta^0_B
\end{align*}
and define the operator $B: L^2(\Sigma_B) \to  L^2(\Sigma_B)$
\begin{align*}
B &=C_{ w^B}\\
  &= C^+ \left( \cdot (\Delta ^0_B)^{-1} (N_B \hat{w}^{B'}_- )  \Delta ^0_B \right)+ C^- \left( \cdot (\Delta ^0_B)^{-1} (N_B \hat{w}^{B'}_+)  \Delta ^0_B \right).
\end{align*}
On 
\begin{align*}
L_B\cup \overline{L}_B & =\left \lbrace z=u z_0 \sqrt{\dfrac{2t   }{ (1+z_0^2)z_0  } }  e^{-i\pi/6}: -\eps<u< \eps \right\rbrace\\
       &\quad \cup\left \lbrace z=u z_0  \sqrt{\dfrac{2t   }{ (1+z_0^2)z_0  } }  e^{i\pi/6}: -\eps<u<\eps \right \rbrace
\end{align*}
From the list of entries stated in \eqref{R1}, \eqref{R4}, \eqref{R8+} and \eqref{R8-}, we have
\begin{align}
\label{wB-1}
\left(  (\Delta_B^0)^{-1} \left(N_B \hat{w}^{B'}_- \right)  \Delta_B^0 \right)(\zeta) &= \twomat{0}{0}{ -r(z_0) \left[ \delta^1_B( \zeta ) \right]^{-2} }{0},\\
\label{wB-2}
\left(  (\Delta_B^0)^{-1} \left(N_B \hat{w}^{B'}_- \right)  \Delta_B^0 \right)(\zeta ) &= \twomat{0}{0}{
\dfrac{r(z_0)}{1 +  |r(z_0)|^2} \left[  \delta^1_B( \zeta )\right]^{-2} }{0},\\
\label{wB-3}
\left(  (\Delta_B^0)^{-1} \left(N_B \hat{w}^{B'}_+ \right)  \Delta_B^0 \right)(\zeta) &= \twomat{0}{-\overline{ r(z_0)} \left[  \delta^1_B(\zeta  )\right]^{2} } {0}{0},\\
\label{wB-4}
\left(  (\Delta_B^0)^{-1} \left(N_B \hat{w}^{B'}_+ \right)  \Delta_B^0 \right)(\zeta) &= \twomat{0}
{\dfrac{ \overline{r(z_0)}}{1 +  |r(z_0)|^2}\left[ \delta^1_B(\zeta)\right]^{2} }{0}{0}.
\end{align}
\begin{lemma}
\label{delta-B }
Let $\nu$ be a small but fixed positive number with $0<2\nu<1$. Then 
$$\left\vert  \left[ \delta^1_B(\zeta) \right]^{\pm 2}-\zeta^{\pm 2i\kappa} e^{\mp i \zeta^2/2 } \right\vert \leq c |e^{\mp i\nu \zeta^2/2} \zeta^3 |\tau^{-1/2} $$
and as a consequence
\begin{equation}
\norm{ \left[\delta^1_B(\zeta)\right]^{\pm 2}-\zeta^{\pm 2i\kappa} e^{\mp i \zeta^2/2} }{L^1\cap L^2 \cap L^\infty} \leq c \tau^{-1/2} 
\end{equation}
where the $\pm$ sign corresponds to $\zeta \in L_B$ and $\zeta\in \overline{L}_B$ respectively. Moreover, 
\begin{equation}
\label{zeta-tau}
\left\vert \zeta^{\pm 2i\kappa} e^{\mp i \zeta^2/2}\right\vert \lesssim \left\vert e^{\mp i\nu \zeta^2/2} \right\vert e^{-\eps^2(1-\nu)\tau}\lesssim  \left\vert e^{\mp i\nu z^2/2} \right\vert  \tau^{-1/2}
\end{equation}
where the $\pm$ sign corresponds to $\zeta \in  (\Sigma_B^1\cup \Sigma_B^3)\setminus L_B$ and $ \zeta\in (\Sigma_B^2\cup \Sigma_B^4)\setminus \overline{L}_B$ respectively. 
\end{lemma}
\begin{proof}
We only deal with the $-$ sign. One can write
\begin{align*}
&\left[\delta^1_B(\zeta)\right]^{- 2}-\zeta^{- 2i\kappa} e^{ i \zeta^2/2}\\
 & \quad = e^{i\nu \zeta^2/2}\left( e^{i\nu \zeta^2/2} \left[  \left( \dfrac{2 z_0}{\sqrt{\dfrac{ (1+z_0^2)z_0  }{2t   }}\zeta+2z_0}\right)^{-2i\kappa} \zeta^{-2i\kappa} \exp \left[ (i (1-2\nu  )\zeta^2/2) \left(1 -\dfrac{ z_0^4 }{\sqrt{2}  (1-2\nu )  \xi^4 \sqrt{\tau} }\zeta\right)\right]  \right. \right. \\
&\qquad \qquad \qquad \qquad \qquad \left. \left.  -\zeta^{-2i\kappa }  e^{i(1-2\nu)\zeta^2/2}\right]   \right).
 \end{align*}
Each of the terms in the expression above is uniformly bounded for fixed $z_0$ (cf. \cite[Appendix ]{CVZ99}). 
Following the proof of \cite[Lemma 3.35]{DZ93}, we estimate 
$$\left\vert e^{i\nu \zeta^2/2} \left(  \left( \dfrac{2 z_0}{\sqrt{\dfrac{ (1+z_0^2)z_0  }{2t   }}\zeta+2z_0} \right)^{-2i\kappa} -1  \right) \right\vert \leq c |e^{i\nu \zeta^2/2} |\tau^{-1/2}$$
and 
\begin{align*}
& \left\vert e^{i\nu \zeta^2/2} \zeta^{-2i\kappa} \left(  e^{(i (1-2\nu  )\zeta^2/2) \left(1 -\dfrac{ z_0^4 }{\sqrt{2}  (1-2\nu )  \xi^4 \sqrt{\tau} }\zeta\right) }-  e^{i(1-2\nu)\zeta^2/2} \right) \right\vert \\
 &\quad \leq c |e^{i\nu \zeta^2/2} |\tau^{-1/2} 
\end{align*} as desired. And the inequality in \eqref{zeta-tau} is an easy consequence of \eqref{Ri-decay}-\eqref{Rj-decay}.
\end{proof}
We then consider the case $\Sigma_A$. Again the rescaling gives 
\begin{equation*}
N_A \left( e^{\chi(z_0)} \eta(z; z_0)  e^{-it\theta}\right)=\delta^0_A \delta^1_A(\zeta)
\end{equation*}
with 
\begin{align*}
\delta^0_A &= (8\tau)^{i\kappa/2} e^{i\tau} e^{\chi(-z_0)} \eta_0(-z_0)\\
\delta^1_A (\zeta) &= \zeta^{-i\kappa} \left( \dfrac{2 z_0}{\sqrt{\dfrac{ (1+z_0^2)z_0  }{2t   }}\zeta+2z_0}\right)^{-i\kappa} e^{( i\zeta^2/4) \left(1 +\dfrac{ z_0^4 }{\sqrt{2} \xi^4 \sqrt{\tau} }\zeta\right) }.
\end{align*}
Note that $\delta^0_A$ is independent of $\zeta$ and that $|\delta^0_A|=1$.
Set
\begin{align*}
\Delta^0_A &=(\delta_A^0)^{\sigma_3}\\
w^A(\zeta)     &=(\Delta^0_A)^{-1}  (N_A \hat{w}^{A'})\Delta^0_A
\end{align*}
and define the operator $A: L^2(\Sigma_A) \to  L^2(\Sigma_A)$
\begin{align*}
A &=C_{ (\Delta^0_A)^{-1}  (N_A \hat{w}^{A'})\Delta^0_A }\\
  &= C^+ \left( \cdot (\Delta ^0_A)^{-1} (N_A \hat{w}^{A'}_- )  \Delta ^0_A \right)+ C^- \left( \cdot (\Delta ^0_A)^{-1} (N_A \hat{w}^{A'}_+)  \Delta ^0_A \right).
\end{align*}
On 
\begin{align*}
L_A\cup \overline{L}_A & =\left\lbrace z=u z_0 \sqrt{\dfrac{2t   }{ (1+z_0^2)z_0  } }  e^{-i5\pi/6}: -\eps<u<\eps \right\rbrace\\
       &\quad \cup \left\lbrace z=u z_0\sqrt{\dfrac{2t   }{ (1+z_0^2)z_0  } }  e^{i 5\pi/6}: -\eps<u<\eps \right\rbrace
\end{align*}
we have from the list of entries stated in \eqref{R2}, \eqref{R3}, \eqref{R5} and \eqref{R7-}
\begin{align}
\left(  (\Delta_A^0)^{-1} \left(N_A \hat{w}^{A'}_- \right)  \Delta_A^0 \right)(z) &= \twomat{0}{0}{-r(-z_0) \left[  \delta^1_A(z)\right]^{-2} }{0},\\
\left(  (\Delta_A^0)^{-1} \left(N_A \hat{w}^{A'}_- \right)  \Delta_A^0 \right)(z ) &= \twomat{0}{0}{
\dfrac{r(-z_0)}{1 + |r(z_0)|^2} \left[ \delta^1_A(z)\right]^{-2} }{0},\\
\left(  (\Delta_A^0)^{-1} \left(N_A \hat{w}^{A'}_+ \right)  \Delta_A^0 \right)(z) &= \twomat{0}{-\overline{ r(-z_0)}\left[ \delta^1_A(z)\right]^{2} } {0}{0},\\
\left(  (\Delta_A^0)^{-1} \left(N_A \hat{w}^{A'}_+ \right)  \Delta_A^0 \right)(z) &= \twomat{0}
{\dfrac{ \overline{r(-z_0)}}{1+ |r(z_0)|^2}\left[ \delta^1_A(z)\right]^{2} }{0}{0}.
\end{align}
\begin{lemma}
\label{delta-A}
Let $\nu$ be a small but fixed positive number with $0<2\nu<1$. Then 
$$ \left\vert  \left[  \delta^1_A(\zeta)\right]^{\pm 2}-(-\zeta)^{\mp 2i\kappa} e^{\pm i \zeta^2/2} \right\vert\leq  c |e^{\pm i\nu \zeta^2/2} |\tau^{-1/2}  $$
and as a consequence,
\begin{equation}
\norm{\left[ \delta^1_A(\zeta)\right]^{\pm 2}-(-\zeta)^{\mp 2i\kappa} e^{\pm i\zeta^2/2}}{L^1\cap L^2 \cap L^\infty}  \leq c \tau^{-1/2} 
\end{equation}
where the $\pm$ sign corresponds to $\zeta\in L_A$ and $\zeta\in \overline{L}_A$ respectively. Moreover,
\begin{equation}
\label{zeta-tau'}
\left\vert (-\zeta)^{\pm 2i\kappa} e^{\mp i \zeta^2/2}\right\vert \lesssim \left\vert e^{\mp i\nu z^2/2} \right\vert e^{-\eps^2(1-\nu)\tau} \lesssim \left\vert e^{\mp i\nu z^2/2} \right\vert  \tau^{-1/2}
\end{equation}
where the $\pm$ sign corresponds to $\zeta \in  (\Sigma_A^2\cup \Sigma_A^3)\setminus L_A$ and $ \zeta\in (\Sigma_A^1\cup \Sigma_A^4)\setminus \overline{L}_A$ respectively. 
\end{lemma}
We now define
\begin{align*}
w^{A^0}(\zeta) &=\lim_{\tau\to\infty} (\Delta^0_A)^{-1}  (N_A \hat{w}^{A'})\Delta^0_A(\zeta),\\
w^{B^0}(\zeta) &=\lim_{\tau\to\infty} (\Delta^0_B)^{-1}  (N_B \hat{w}^{B'})\Delta^0_B(\zeta),\\
A^0 &=C^+(\cdot w^{A^0}_-)+C^-(\cdot w^{A^0}_+),\\
B^0 &=C^+(\cdot w^{B^0}_-)+C^-(\cdot w^{B^0}_+).
\end{align*}
\begin{proposition}
\label{prop:resolvent}
\begin{equation}
\norm{({1}_A-A)^{-1} }{L^2(\Sigma_A)} , \, \norm{(1_B-B)^{-1}}{L^2(\Sigma_B)} \leq c 
\end{equation}
as $\tau\to\infty$.
\end{proposition}
\begin{proof}
From Lemma \ref{delta-B } and Lemma \ref{delta-A},  it is easily seen that 
\begin{equation}
\label{fix}
\norm{A-A^0}{L^2(\Sigma_A)}, \, \norm{B-B^0}{L^2(\Sigma_B)} \leq c \tau^{-1/2}.
\end{equation}
We will only establish the boundedness of  $(1_B-B)^{-1}$ since the case for  $(1_A-A)^{-1}$ is similar. From Lemma \ref{delta-B } we deduce that  on $\Sigma_B$
\begin{align}
\label{w-B-0}
w^{B^0}(\zeta) =\begin{cases}
\twomat{0}{0}{-r(z_0) \zeta^{- 2i\kappa} e^{ i \zeta^2/2}}{0}, \quad  \zeta\in \Sigma_B^1,\\
\\
\twomat{0}{\dfrac{ \overline{r(z_0)}}{1 +  |r(z_0)|^2} \zeta^{ 2i\kappa} e^{ -i \zeta^2/2} }{0}{0}, \quad \zeta\in \Sigma_B^2, \\
\\
\twomat{0}{0}{
\dfrac{r(z_0)}{1 +  |r(z_0)|^2} \zeta^{- 2i\kappa} e^{ i \zeta^2/2} }{0}, \quad  \zeta\in \Sigma_B^3,\\
\\
\twomat{0}{-\overline{ r(z_0)} \zeta^{ 2i\kappa} e^{ -i \zeta^2/2}  } {0}{0}, \quad  \zeta\in \Sigma_B^4.
\end{cases}
\end{align}
Setting 
$$v^{B^0}(\zeta) =I+w^{B^0}(\zeta)  $$
we first notice that after contour reorientation,
$v^{B^0}(\zeta) $ is precisely the jumps of the exactly solvable parabolic cylinder problem.  The solution of this problem is standard and can be found in \cite[Appendix A]{BJM16}. More importantly,  $v^{B^0}(\zeta) $ satisfies the 
Schwarz invariant condition:
$$v^{B^0}(\zeta)= v^{B^0}(\overline{\zeta})^\dagger$$
which will guarantee the uniqueness of the solution. By standard arguments in \cite{Zhou89} and \cite[Sec 7.5]{Deift99}, this implies the existence and boundedness of the resolvent operator $(1_B-B^0)^{-1}$. And the boundedness of $(1_B-B)^{-1}$ is a consequence of \eqref{fix} and the second resolvent identity.
\end{proof}
Indeed, for $\zeta\in\Sigma_B$ we let
\begin{equation}
m^{B^0}(\zeta)=I+\dfrac{1}{2\pi i}\int_{\Sigma_B}\dfrac{\left( (1_B-B^0)^{-1} I\right)(s) w^{B^0}(s)}{s-\zeta} ds
\end{equation}
then $m^{B^0}(\zeta)$ solves the following Riemann-Hilbert problem
\begin{align}
\label{RHP-B0}
\begin{cases}
m^{B^0}_+(\zeta) &=m^{B^0}_-(\zeta)v^{B^0}(\zeta), \quad \zeta\in\Sigma_B\\
m^{B^0}(\zeta) &\to I , \quad \zeta\to \infty.
\end{cases}
\end{align}
In the large $\zeta$ expansion, 
$$m^{B^0}(\zeta) =I-\dfrac{m^{B^0}_1}{\zeta}+O(\zeta^{-2}), \quad \zeta\to\infty$$
thus
$$m^{B^0}_1=\dfrac{1}{2\pi i}\int_{\Sigma_B}{\left( (1_B-B)^{-1} I\right)(s) w^{B^0}(s)} ds.$$
Similarly, setting 
\begin{equation}
m^{B}(\zeta)=I+\dfrac{1}{2\pi i}\int_{\Sigma_B}\dfrac{\left( (1_B-B)^{-1} I\right)(s) w^{B}(s)}{s-\zeta} ds
\end{equation}
then $m^{B}(\zeta)$ solves the following Riemann-Hilbert problem
\begin{align}
\label{RHP-B}
\begin{cases}
m^{B}_+(\zeta) &=m^{B}_-(\zeta)v^{B}(\zeta), \quad \zeta\in\Sigma_B\\
m^{B}(\zeta) &\to I , \quad \zeta\to \infty.
\end{cases}
\end{align}
Here $v^{B}(\zeta)=I+w^{B}(\zeta)$ where $w^B(\zeta)$ is given by \eqref{wB-1}-\eqref{wB-4}.
In the large $\zeta$ expansion, 
$$m^{B}(\zeta) =I-\dfrac{m^{B}_1}{\zeta}+O(\zeta^{-2}), \quad \zeta\to\infty$$
thus
$$m^{B}_1=\dfrac{1}{2\pi i}\int_{\Sigma_B}{\left( (1_B-B)^{-1} I\right)(s) w^{B}(s)} ds.$$
Setting $w^d=w^{B}-w^{B^0}$, a simple computation shows that
\begin{align*}
\int_{\Sigma_B }  \left( ( 1_B-B)^{-1}I \right)w^{B}  - \int_{\Sigma_B   } { \left( (1_B-B^0 )^{-1}I \right)w^{B^0} } &= \int_{\Sigma^{B} }w^d+  \int_{\Sigma_B  } \left( ( 1_B-B^0   )^{-1} (C_{w^d} I)  \right) w^{B}\\
&\quad+ \int_{\Sigma_B} \left( ( 1_B-B^0 )^{-1} (B^0 I)  \right)w^d \\
&\quad+\int_{\Sigma_B  } \left( (1_B-B^0 )^{-1} C_{w^d}  (1_B-B )^{-1}   \right) \left( B ( I) \right) { w^{B} } \\
&=:\text{I} + \text{II}  + \text{III} + \text{IV}. 
\end{align*}
From Lemma \ref{delta-B } and Proposition \ref{prop:resolvent}, it is clear that
\begin{align*}
\left\vert \text{I}\right\vert  &\lesssim \tau^{-1/2},\\
\left\vert \text{II}  \right\vert   & \leq   \norm{( 1_B-B^0 )^{-1} }{L^2(\Sigma_B )}\norm{C_{w^d }I }{L^2(\Sigma_B) } \norm{ w^{B} }{L^2(\Sigma_B ) } \\
                                    & \lesssim \tau^{-1/2},\\                    
\left\vert \text{III}  \right\vert &\leq \norm{( 1_B-B^0 )^{-1} }{L^2(\Sigma_B )}\norm{B^0 I }{L^2(\Sigma_B) } \norm{ w^{d} }{L^2(\Sigma_B ) }\\
                                    & \lesssim \tau^{-1/2}.
\end{align*}
For the last term
\begin{align*}
\left\vert \text{IV} \right\vert & \leq  \norm{ (1_B-B^0 )^{-1} }{L^2(\Sigma_B )} \norm{( 1_B-B )^{-1} }{L^2(\Sigma_B )} \norm{C_{w^d}}{L^2(\Sigma_B )} \\
     &\quad \times  \norm{B( I ) }{L^2(\Sigma_B )} \norm{ w^B  }{L^2(\Sigma_B )} \\
     &\leq c \norm{ w^d}{L^\infty (\Sigma_B ) }  \norm{w^B }{L^2(\Sigma^{(3 )} )}^2\\
     & \lesssim \tau^{-1/2}.
\end{align*}
So we conclude that
\begin{equation}
\label{differ-B}
\left\vert \int_{\Sigma_B }  \left( ( 1_B-B)^{-1}I \right)w^{B}  - \int_{\Sigma_B   } { \left( (1_B-B^0 )^{-1}I \right)w^{B^0} }   \right\vert \lesssim \tau^{-1/2}.
\end{equation}
Clearly there is a parallel case for $\Sigma_A$:
\begin{equation}
\label{differ-A}
\left\vert \int_{\Sigma_A }  \left( ( 1_A-A)^{-1}I \right)w^{A}  - \int_{\Sigma_A  } { \left( (1_A-A^0 )^{-1}I \right)w^{A^0} }   \right\vert \lesssim \tau^{-1/2}.
\end{equation}
The explicit form of $m^{B^0}_1$ is given as follows (see \cite[Appendix A]{BJM16}) :
\begin{equation}
\label{explicit-B0}
m^{B^0}_1=\twomat{0}{-i\beta_{12}}{i\beta_{21}}{0}
\end{equation}
where
$$\beta_{12}=\dfrac{\sqrt{2\pi } e^{i\pi/4} e^{-\pi \kappa/2 } }{r(z_0) \Gamma(-i\kappa)}, \qquad \beta_{21}=\dfrac{-\sqrt{2\pi } e^{-i\pi/4} e^{-\pi \kappa/2 } }{ \overline{ r(z_0) } \Gamma(i\kappa)}$$
and $\Gamma(z)$ is the \textit{Gamma} function.
Recall that on $\Sigma_B$, $\zeta=  \sqrt{\dfrac{2t   }{ (1+z_0^2)z_0  } }  (z - z_0) $, thus by \eqref{differ-B}, we have
\begin{equation}
\label{differ-m-B}
\left\vert   \dfrac{m_1^B}{\zeta}-\dfrac{m_1^{B^0}}{\zeta}   \right\vert \lesssim \dfrac{1}{t(z-z_0)}.
\end{equation}
Using the explicit form of $w^{B^0}$ given by \eqref{w-B-0}, symmetry reduction given by \eqref{minus} and their analogue for $w^{A^0}$, we verify that
\begin{equation}
v^{A^0}(z)=  \sigma_3 \overline{ v^{B^0}( - \overline{z} )  } \sigma_3  
\end{equation}
which  in turn implies by uniqueness that
\begin{equation}
m^{A^0}(z)= \sigma_3 \overline{ m^{B^0}( - \overline{z} )  } \sigma_3  
\end{equation}
and from this we deduce that
\begin{align}
\label{mA-mB}
m^{A^0}_1 &=-\sigma_3 \overline{ m^{B^0}_1  } \sigma_3  \\
\nonumber
&=\twomat{0}{i \overline{\beta}_{12}}{-i \overline{\beta}_{21} }{0}.
\end{align}
We also have an analogue of \eqref{differ-m-B} for $m_1^{A^0}$: 
\begin{equation}
\label{differ-m-A}
\left\vert   \dfrac{m_1^A}{\zeta}-\dfrac{m_1^{A^0}}{\zeta}   \right\vert \lesssim \dfrac{1}{\tau(z+z_0)}.
\end{equation}
Collecting all the computations above, we write down the asymptotic expansions of solutions to Problem \ref{prob:mkdv.A} and Problem \ref{prob:mkdv.B} respectively.
\begin{proposition}
\label{solution-A-B}
Setting $\zeta=  \sqrt{\dfrac{2t   }{ (1+z_0^2)z_0  } }  (z + z_0)  $, the solution to RHP Problem  \ref{prob:mkdv.A}  $m^{A'}$ admits the following expansion:
\begin{equation}
\label{expansion-A}
m^{A'}(z(\zeta) ;x,t)=I +\dfrac{1}{\zeta}\twomat{0}{i ( \delta^0_A)^2 \overline{\beta}_{12}}{-i ( \delta^0_A)^{-2}\overline{\beta}_{21} }{0} +\mathcal{O}(\tau^{-1}).
\end{equation}
Similarly, setting $  \zeta=\sqrt{\dfrac{2t   }{ (1+z_0^2)z_0  } }  (z - z_0)  $, the solution to RHP Problem  \ref{prob:mkdv.B}  $m^{B'}$ admits the following expansion:
\begin{equation}
\label{expansion-B}
m^{B'}(z(\zeta) ;x,t)=I +\dfrac{1}{\zeta}\twomat{0}{-i ( \delta^0_B)^2 {\beta}_{12}}{i ( \delta^0_B)^{-2}{\beta}_{21} }{0} +\mathcal{O}(\tau^{-1}).
\end{equation}
\end{proposition}
\subsection{Finding $m^\RHP$}
Now we construct $m^\RHP$ defined by problem \ref{MKDV.RHP.local}. 
In Figure \ref{fig:contour-2}, we let $\rho$ be the radius of the circle $C_A$ ($C_B$) centered at $z_0$ ($-z_0$). We seek a solution of the form
\begin{equation}
\label{parametrix}
m^\RHP(z)=\begin{cases}
E(z)m^{(out)}(z) \quad  &\left\vert z\pm z_0 \right\vert>\rho \\
E(z)m^{(out)}(z)  m^{A'}(z) \quad &\left\vert z + z_0 \right\vert\leq\rho \\
E(z)m^{(out)}(z)  m^{B'}(z) \quad & \left\vert z - z_0 \right\vert\leq \rho 
\end{cases}
\end{equation}
where $m^{(out)}$ solves the discrete Riemann-Hilbert problem whose jump condition are given by problem \ref{prob:mkdv.RHP1} with $r\equiv 0$.
Since $ m^{A'}$ and $ m^{B'}$ solve  Problem \ref{prob:mkdv.A} and Problem \ref{prob:mkdv.B} respectively and discrete RHPs have explicit solutions, we can construct the solution $m^\RHP(z)$ if we find $E(z)$. Indeed,  $E$ solves the following Riemann-Hilbert problem on $\Sigma_E$ given in figure \ref{fig:contour-E} :
\begin{problem}
\label{prob: E-2}
Find a matrix-valued function $E (z)$ on $\bbC \setminus \Sigma_E $ with the following properties:
\begin{enumerate}
\item		$E(z) \rarr I$ as $ z \rarr \infty$,
\item		$E (z) $ is analytic for $z \in  \bbC \setminus \Sigma_E$
			with continuous boundary values
			$E_{\pm}(z)$.
\item On $ C_A\cup C_B $, we have the following jump conditions
\begin{equation}
\label{jump:E}
E_{+}(z)=E_{-}(z) v^{ (E) }(z)
\end{equation}
			where 
			\begin{equation}
			\label{v-E}
			 v^{ (E) }(z)=\begin{cases}
			m^{(out)} (z)m^{A'}(z(\zeta) ) \left[m^{(out )} (z)\right]^{-1}, \quad z\in C_A\\
			m^{(out )} (z)m^{B'}(z(\zeta) )\left[m^{(out )} (z)\right]^{-1}, \quad z\in C_B\\
			m^{(out )} (z)v^{(2)} \left[m^{(out )} (z)\right]^{-1}, \quad z\in \Sigma_E\setminus\left(  C_A\cup C_B\right)
			\end{cases}
			\end{equation}
\end{enumerate}
\end{problem}
\begin{figure}[H]
\caption{$\Sigma_E$}
\vskip 15pt

\begin{tikzpicture}[ photon/.style={decorate,decoration={snake,post length=0.8mm}} ][scale=0.9]

\draw[thick]   [red]			(4.472,0.467) -- (5.296, 1);								
\draw[->,thick,>=stealth] [red]	(   7.028 ,2  )--	(5.296, 1);

\draw[->,thick,>=stealth] 	  [red]			(1.732 ,1) -- (2.656 ,0.467) ;

\draw[thick] 		  [red]		(4.472,-0.467) -- (5.296, -1);								
\draw[->,thick,>=stealth]  [red]		(   7.028 , -2  )--	(5.296, -1);

\draw[->,thick,>=stealth] 	  [red]			(1.732 ,-1) -- (2.656 ,-0.467) ;

\draw[->,thick,>=stealth] 	  [red]			 (-2.656 ,-0.467)--(-1.732 ,-1) ;
\draw[->,thick,>=stealth] 	  [red]			 (-2.656 ,0.467)--(-1.732 ,1) ;

\draw[->,thick,>=stealth]   [red]				(-4.472 ,0.467) -- (-5.296, 1);						
\draw[thick]  [red]		(   -7.028 ,2  )--	(-5.296, 1);
\draw[->,thick,>=stealth] 	  [red]			(-4.472 ,-0.467) -- (-5.296, -1);							
\draw[thick]  [red]		(   -7.028 , -2  )--	(-5.296, -1);

\draw[->,thick,>=stealth]  [red]				(-1.732 , -1) -- (-1.732, 0);	
\draw	[thick]  [red]		(-1.732 , 0) -- (-1.732, 1);	
\draw[->,thick,>=stealth] 	 [red]			(1.732 , -1) -- (1.732, 0);	
\draw	[thick]	 [red]	 (1.732 , 0) -- (1.732, 1);

\draw [blue][dashed] (0,0) circle[radius=3.564];

\draw	[fill]							(-3.564 ,0)		circle[radius=0.06];	
\draw	[fill]							(3.564, 0)		circle[radius=0.06];
\draw		[fill]					(0,0)		circle[radius=0.1];

\node[below] at (-3.564,-0.1)			{$-z_0$};
\node[below] at (3.564,-0.1)			{$z_0$};

\node[right] at (5, 1.6)					{$\Sigma_1'^{(2)} $};
\node[left] at (-5, 1.6)					{$\Sigma_2'^{(2)}$};
\node[left] at (-5,-1.6)					{$\Sigma_3'^{(2)}$};
\node[right] at (5,-1.6)				{$\Sigma_4'^{(2 )}$};
\node[left] at (-1,1.2)					{$\Sigma_5'^{(2)}$};
\node[left] at (-1,-1.2)					{$\Sigma_7'^{( 2 )}$};
\node[right] at (1,1.2)					{$\Sigma_6'^{(2 )}$};
\node[right] at (1,-1.2)					{$\Sigma_8'^{(2 )}$};
\node[right] at (-1.732, 0) {$\Sigma_9'^{(2)} $};
\node[left] at (1.732, 0) {$\Sigma_{10}'^{(2 )} $};

 \draw [red, fill=red] (1, 1.732) circle [radius=0.05];
\draw[->,>=stealth] [red] (1.2, 1.732) arc(360:0:0.2);
\draw [red, fill=red] (-1,1.732) circle [radius=0.05];
\draw[->,>=stealth] [red] (-0.8, 1.732) arc(360:0:0.2);
\draw [red, fill=red] (1,-1.732) circle [radius=0.05];
\draw[->,>=stealth] [red] (1.2, -1.732) arc(0:360:0.2);
\draw [red, fill=red] (-1,-1.732) circle [radius=0.05];
\draw[->,>=stealth] [red] (-0.8, -1.732) arc(0:360:0.2);
\draw [green, fill=green]  (0, 2.7) circle [radius=0.05];
\draw[->,>=stealth] [red] (0.3, 2.7) arc(360:0:0.3);
\draw [green, fill=green] (0, -2.7) circle [radius=0.05];
\draw[->,>=stealth] [red]  (0.3, -2.7) arc(0:360:0.3);
\draw[->,>=stealth] (-2.6,2) arc(360:0:0.4);
\draw[->,>=stealth] (3.4,2) arc(360:0:0.4);
\draw[->,>=stealth] (-2.6,-2) arc(0:360:0.4);
\draw[->,>=stealth] (3.4,-2) arc(0:360:0.4);
\draw [red, fill=red] (-3,2) circle [radius=0.1];
\draw [red, fill=red] (3,2) circle [radius=0.1];
\draw [red, fill=red] (-3,-2) circle [radius=0.1];
\draw [red, fill=red] (3,-2) circle [radius=0.1];

\draw[->,thick,>=stealth] (4.564,0) arc(360:0:1);
\draw[->,thick,>=stealth] (-2.564,0) arc(360:0:1);

\draw[->,photon] (-4, -2.) --  (-4, -1.); 
\draw[->,photon] (4, -2) --  (4, -1); 
\node [below] at  (-4, -2.131)  {$C_B$};
\node [below] at  (4, -2.131) {$C_A$};

\end{tikzpicture}
\label{fig:contour-E}
\begin{center}
  \begin{tabular}{ccc}
exponential decay on red portion of the contour
\end{tabular}
 \end{center}
\end{figure}
Note that we omitted the contours passing through the origin since the jumps are identities by definition (see \eqref{R5}-\eqref{R8-}) .
\begin{proposition}
 $E(z)$ admits a classical solution, i.e jump condition \eqref{jump:E} holds pointwise on the contour $\Sigma_E$. 
\begin{proof}
Here we invoke to the well-established existence and uniqueness theory from \cite{Zhou89} (see also chapter 2 \cite{TO16}). First it is easy to check that 
$$v^{ (E) }(z)= v^{ (E) \dagger}(z)$$
where the $\dagger$ denotes the Hermitian conjugate of the given matrix. We then take care of the zero sum condition at the self-intersecting points of $\Sigma_E$. Since the remaining cases follows from symmetry, we will only look $\Sigma'^{(2)}_6 \cap \Sigma'^{(2)}_{10} $ and $\Sigma'^{(2)}_6\cap C_A $. The zero sum condition holds at the first point by comparing \eqref{R8+} and the third line of \eqref{jump v9}. For $\Sigma'^{(2)}_6\cap C_A $, (after adding contour with identity jumps and reorientation see P. 1058 \cite{DZ03} ) we explicitly compute
\begin{align*}
I & = \left[ m^{(out)} (z)m^{A'}(z(\zeta) )\left[m^{( out )} (z)\right]^{-1}  \right] \left[  m^{(out )} (z)v^{(2)} \left[m^{( out )} (z)\right]^{-1}  \right]^{-1} \\
  &\quad \times \left[ m^{(out)} (z)m^{A'}(z(\zeta) )\left[m^{( out )} (z)\right]^{-1}  \right] \\
   &= m^{(out)} (z) m^{A'}_+(z(\zeta) )\left( v^{(2)} \right)^{-1} \left( m^{A'}_-(z(\zeta))\right)^{-1}\left[m^{( out )} (z)\right]^{-1} .
\end{align*}
Since $v^{(2)}$ is smooth away from the intersections and zero sum conditions have been verified, this completes the proof.
\end{proof}
\end{proposition}
Furthermore, for $t$ large enough, we can solve a small norm RHP to obtain some desired estimates. Setting 
$$\eta(z)=E_{-}(z)-I$$
then by the standard theory, we have the following singular integral equation
$$ \left( \mathbf{1}-C_{v^{(E)}}  \right) \eta =C_{v^{(E)}} I$$
where the singular integral operator is defined by:
$$  C_{v^{(E)}}f =C^-\left( f \left( v^{(E)}-I  \right)  \right). $$
We first deduce from \eqref{expansion-A}-\eqref{expansion-B} that 
\begin{equation}
\norm{v^{(E)}-I}{L^\infty} \lesssim t^{-1/2}
\end{equation}
hence the operator norm of $C_{v^{(E)}} $ 
\begin{equation}
\label{norm-vE}
\norm{C_{v^{(E)}}f }{ L^2} \leq \norm{f}{L^2} \norm{v^{(E)}-I}{L^\infty} \lesssim t^{-1/2}.
\end{equation}
Then the resolvent operator $(1-C_{v^{(E)}})^{-1}$ can be obtained through \textit{Neumann} series and we obtain the unique solution to Problem \ref{prob: E-2}:
\begin{align*}
E(z) &=I+ \dfrac{1}{2\pi i}\int_{C_A\cup C_B} \dfrac{ (1+\eta(s))(v^{(E)}(s)-I )   }{s-z}ds\\
    &\quad +\dfrac{1}{2\pi i}\int_{  \Sigma_E \setminus ( C_A\cup C_B) } \dfrac{ (1+\eta(s))(v^{(E)}(s)-I )   }{s-z}ds\\
    &=I+ \dfrac{1}{2\pi i}\int_{C_A\cup C_B} \dfrac{ (1+\eta(s))(v^{(E)}(s)-I )   }{s-z}ds +\mathcal{O}(e^{-ct})
\end{align*}
where we make use of the fact that $v^{(E)}(s)-I $ decays exponentially on $ \Sigma_E \setminus ( C_A\cup C_B) $.

We  also have the bound on $\eta$ on $C_A\cup C_B$:
\begin{equation}
\label{norm-eta}
\norm{\eta}{L^2}\lesssim t^{-1/2}.
\end{equation}
Letting $z\to 0$ and using the bounds given by \eqref{norm-vE} and \eqref{norm-eta} and an application of Cauchy-Schwarz inequality, we obtain
\begin{align}
\label{bound-E-1}
E (0) - \mathcal{O}(e^{-ct})&=I+ \dfrac{1}{2\pi i}\int_{C_A\cup C_B}   \dfrac{ { (1+\eta(s))(v^{(E)}(s)-I )   }  } {s}ds  \\
\nonumber
              &=I+\dfrac{1}{2\pi i}\int_{C_A\cup C_B} \dfrac{ v^{(E)}(s)-I    }{  s }ds +\dfrac{1}{2\pi i}\int_{C_A\cup C_B} \dfrac{\eta(s) (v^{(E)}(s)-I)    }{  s }ds \\
              \nonumber
              &\lesssim I+\dfrac{1}{2\pi i}\int_{C_A\cup C_B} \dfrac{ v^{(E)}(s)-I    }{  s }ds +
              \norm{\eta}{L^2} \norm{v^{(E)}-I}{L^\infty\cap L^2 } \\
              \nonumber
             & \lesssim  I+\dfrac{1}{2\pi i}\int_{C_A\cup C_B} \dfrac{ v^{(E)}(s)-I    }{  s }ds +\mathcal{O}(t^{-1}).
\end{align}
Given the form of $v^{(E)}$ in \eqref{v-E} and the asymptotic expansions \eqref{expansion-A}-\eqref{expansion-B}, an application of Cauchy's integral formula leads to
\begin{align}
\label{E-1-2 cauchy}
E(0) &=I + \dfrac{1}{z_0   \sqrt{\dfrac{2t   }{ (1+z_0^2)z_0  } } } m^{(br)}(z_0) \twomat{0}{i ( \delta^0_B)^2 {\beta}_{12}}{ -i ( \delta^0_B)^{-2}{\beta}_{21} }{0}  \left[m^{(br)}(z_0)\right]^{-1}\\
\nonumber
    & \quad + \dfrac{1}{z_0   \sqrt{\dfrac{2t   }{ (1+z_0^2)z_0  } } } m^{(br)}(-z_0) \twomat{0}{- i ( \delta^0_A)^2 \overline{\beta}_{12} }{i ( \delta^0_A)^{-2}\overline{\beta}_{21} }{0} \left[m^{(br)}(-z_0)\right]^{-1}\\
    \nonumber
     & \quad + \mathcal{O}(t^{-1}).
\end{align}
Note that in the formula above we replaced $m^{(out)}$ with $m^{(br)}$ which is the solution to problem \ref{prob:mkdv.br} since the contributions from all data other than $\pm \gamma_l \cup \pm \gamma^*_l $ are exponentially small (see figure \ref{fig:contour-d}).
We now completed the construction of the matrix-valued function $E (z)$ hence $m^\RHP(x,t; z)$. Combining this with Proposition \ref{expo}, we obtain $m^\RHP(z)$ in \eqref{factor-LC}.

\section{The $\dbar$-Problem}
\label{sec:dbar}
 
From \eqref{factor-LC} we have the matrix-valued function
\begin{equation}
\label{N3}
m^{(3)}(z;x,t) = m^{(2)}(z;x,t) \left[m^\RHP(z; x,t)\right]^{-1}. 
\end{equation}
The goal of this section is to show that $m^{(3)}$ only results in an error term $E(x, t)$ with a higher order decay rate than the leading order term of the asymptotic formula. The computations and proofs are standard. We follow \cite[ Section 5 ]{ CL19 } with slight modifications. 

Since $m^\RHP(z; x, t)$ is analytic in $\bbC \setminus \left( \Sigma'^{(2)} \cup \Gamma  \right)$, we may compute
\begin{align*}
\dbar m^{(3)}(z;x,t) 	&=	\dbar m^{(2)}(z;x,t) \left[m^\RHP(z; x,t)\right]^{-1}\\	
								&=	m^{(2)}(z;x,t) \, \dbar \calR^{(2)}(z) \left[m^\RHP(z; x,t)\right]^{-1}	&\text{(by \eqref{N2.dbar})}\\
								&=	m^{(3)}(z;x,t) m^{\RHP}(z; x, t) \, \dbar \calR^{(2)}(z) \left[m^\RHP(z; x,t)\right]^{-1}	& \text{(by \eqref{N3})}\\
								&=	m^{(3)}(z;x,t)  W(z;x,t)
\end{align*}
where
 \begin{equation}
 \label{W-bound}
 W(z;x,t) = m^{\RHP}(z;x,t ) \, \dbar \calR^{(2)}(z) \left[ m^\RHP(z;x,t )\right]^{-1}. 
 \end{equation}
We thus arrive at the following pure $\dbar$-problem:

\begin{problem}
\label{prob:DNLS.dbar}
Give $r(z)\in H^{s}_0(\bbR)$ as given by proposition \ref{prop:r} , find a continuous matrix-valued function
$m^{(3)}(z;x,t)$ on $\bbC$ with the following properties:
\begin{enumerate}
\item		$m^{(3)}(z;x,t) \rarr I$ as $|z| \rarr \infty$.
\smallskip
\item		$\dbar m^{(3)}(z;x,t) = m^{(3)}(z;x,t) W(z;x,t)$.
\end{enumerate}
\end{problem}

It is well understood (cf. \cite[Chapter 7]{AF}) that the solution to this $\dbar$ problem is equivalent to the solution of a Fredholm-type integral equation involving the solid Cauchy transform
$$ (Pf)(z) = \frac{1}{\pi} \int_\bbC \frac{1}{\zeta-z} f(\zeta) \, d\zeta $$
where $d$ denotes Lebesgue measure on $\bbC$.  Also throughout this section, $\zeta$ refers to complex numbers, not to be confused with $\zeta$ in the previous section.
\begin{lemma}
A bounded and continuous matrix-valued function $m^{(3)}(z;x,t)$ solves Problem \ref{prob:DNLS.dbar} if and only if
\begin{equation}
\label{DNLS.dbar.int}
m^{(3)}(z;x,t) =I+ \frac{1}{\pi} \int_\bbC \frac{1}{\zeta-z} m^{(3)}(\zeta;x,t) W(\zeta;x,t) \, d\zeta.
\end{equation}
\end{lemma}

 Using the integral equation formulation \eqref{DNLS.dbar.int}, we will prove:

\begin{proposition}
\label{prop:N3.est}
Suppose that $r \in H^{s}_0(\bbR)$.
Then, for $t\gg 1$, there exists a unique solution $m^{(3)}(z;x,t)$ for Problem \ref{prob:DNLS.dbar} with the property that 
\begin{equation}
\label{N3.exp}
m^{(3)}(z;x,t) =   m^{(3)}_1(x,t) +m^{(3)}_2(z; x,t)
\end{equation}
with 
\begin{equation}
\label{0-dbar}
\lim_{z\to 0} m^{(3)}_2(z; x,t)=0.
\end{equation}
Here
\begin{equation}
\label{N31.est}
\left| m^{(3)}_1(x,t) \right| \lesssim \tau^{-s}
\end{equation}
 where the implicit constant in \eqref{N31.est} is uniform for 
$r$ in a bounded subset of $H^{1,1}_0(\bbR)$.
\end{proposition}

\begin{proof} Given Lemmas \ref{lemma:dbar.R.bd}--\ref{lemma:N31.est},
 as in \cite{LPS}, we first show that, for large $t$, the integral operator $K_W$ 
defined by
\begin{equation*}
\left( K_W f \right)(z) = \frac{1}{\pi} \int_\bbC \frac{1}{\zeta-z} f(\zeta) W(\zeta) \, d \zeta
\end{equation*}
is bounded by
\begin{equation}
\label{dbar.int.est1}
\norm{K_W}{L^\infty \rarr L^\infty} \lesssim \tau^{(1-2s)/4}
\end{equation}
where the implied constants depend only on $\norm{r}{H_0^{s}}$.  This is the goal of Lemma \ref{lemma:KW}.
It implies that 
\begin{equation}
\label{N3.sol}
m^{(3)} = ( \mathbf{1} -K_W)^{-1}I
\end{equation}
exists as an $L^\infty$ solution of \eqref{DNLS.dbar.int}.
We  finally in Lemma \ref{lemma:N31.est} estimate \eqref{N31.est}
where the constants are uniform in $r$ belonging to a bounded subset of $H_0^{s}(\bbR)$.
Estimates \eqref{N3.exp}, \eqref{N31.est}, and \eqref{dbar.int.est1} result from the  bounds obtained in the next four lemmas. For \eqref{N3.exp}, we can decompose $m^{(3)}(z;x,t) $:
\begin{align*}
m^{(3)}(z;x,t) &=I+ \frac{1}{\pi} \int_\bbC \frac{1}{\zeta-z} m^{(3)}(\zeta;x,t) W(\zeta;x,t) \, d\zeta\\
                   &= \underbrace{ \left( I+  \frac{1}{\pi} \int_\bbC \frac{1}{\zeta} m^{(3)}(\zeta;x,t) W(\zeta;x,t) \, d\zeta \right)}_{m_1^{(3)}} \\
                   &+\underbrace{ \left( \frac{1}{\pi} \int_\bbC \frac{1}{\zeta-z} m^{(3)}(\zeta;x,t) W(\zeta;x,t) \, d\zeta - \frac{1}{\pi} \int_\bbC \frac{1}{\zeta} m^{(3)}(\zeta;x,t) W(\zeta;x,t) \, d\zeta\right) }_{m_2^{(3)}}\\
                   &= : m^{(3)}_1(x,t) +m^{(3)}_2(z; x,t).
\end{align*}
We set $z=i\sigma$ and let $\sigma\to 0$. Then near the origin, simple geometrical observation gives
\begin{align*}
\left\vert \dfrac{1}{s-z} \right\vert &= \left\vert \dfrac{s}{s-z} \right\vert \left\vert \dfrac{1}{s} \right\vert \\
                                                    &= \left\vert \dfrac{u^2+v^2}{u^2+(v-\sigma)^2 } \right\vert^{1/2} \left\vert \dfrac{1}{u^2+v^2} \right\vert^{1/2}\\
                                                    &\leq  \dfrac{2}{\sqrt{3}}\left\vert \dfrac{1}{s} \right\vert.
\end{align*}
 So the dominated convergence theorem will lead to \eqref{0-dbar} if we can show that the integral representation of $m_1^{(3)}$ make sense which will be proven in Lemma \ref{lemma:N31.est}. 
\end{proof}
For simplicity we only work with regions $\Omega_8$ and $\Omega_3$. For technical reasons we further divide $\Omega_3$ into two parts. See Figure \ref{fig:three reigons} below.

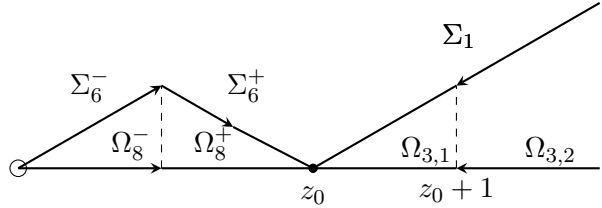
\begin{figure}[H]
\caption{Four regions }
\vskip 15pt
\begin{tikzpicture}[scale=1.1]

\draw 	[thick] 	(3.564, 0)-- (5.296, 0);	
\draw 	[->,thick,>=stealth] 	(7.028,0 )-- (5.296, 0);
\draw [dashed]  (5.296, 0)--(5.296, 1);
\node[below] at   (5.296, 0) {$z_0+1$};


\draw 	[->,thick,>=stealth] 	(0, 0)-- (1.732, 0);	
\draw 	[thick] 	 (1.732, 0)--(3.564,0 );

\draw[thick] 			(3.564 ,0) -- (5.296, 1);								
\draw[->,thick,>=stealth] 	(   7.028 ,2  )--	(5.296, 1);

\draw[->,thick,>=stealth] 			(1.732 ,1)--(2.598 ,0.5);
\draw[thick]                    (2.598, 0.5)--(3.564 ,0); 
\draw[->,thick,>=stealth] 		(0,0)--(1.732 ,1);

%

\draw	[fill]							(3.564, 0)		circle[radius=0.05];
\draw							(0,0)		circle[radius=0.1];

\node[below] at (3.564,-0.1)			{$z_0$};

\node[right] at (0.5,1)					{$\Sigma_6^-$};
\node[right] at (2.4, 1)					{$\Sigma_6^+$};
\node[right] at (5, 1.6)					{$\Sigma_1$};

\node[right] at (2, 0.3)				{$\Omega_8^+$};

\node[right] at (4.5, 0.2)				{$\Omega_{3, 1}$};
\node[right] at (6, 0.2)				{$\Omega_{3,2}$};

\node[right] at (1, 0.3)				{$\Omega_8^-$};
\node[right] at (5, 1.6)					{$\Sigma_1$};

\draw [dashed]  (1.732, 0)--(1.732, 1);

\end{tikzpicture}
\label{fig:three reigons}
\end{figure}

\begin{lemma}
\label{lemma:dbar.R.bd}
\begin{equation}
\label{dbar.R_1.bd}
\left| \dbar R_1 e^{ 2i\theta}  \right|
\lesssim		
\begin{cases}
 	\left( |\dbar [\mathbf{p}_1(z)]| + |z - z_0 |^{s-3/2} + \left\vert \dbar \left( \Xi_\calZ(z) \right)  \right\vert \right) e^{-uv\tau };  z\in \Omega_{3,1} \\
 	\\
 	\left( |\dbar [\mathbf{p}_1(z)]| + |z - z_0 |^{s-3/2} + \left\vert \dbar \left( \Xi_\calZ(z) \right)  \right\vert \right) e^{-v\tau };  z\in \Omega_{3,2} 
		\end{cases}
\end{equation}
\begin{equation}
\label{dbar.R_6.bd}
\left| \dbar R_6 e^{- 2i\theta}  \right|
\lesssim		
\begin{cases}
 	\left( |\dbar [\mathbf{p}_6( z)]| +\dfrac{\mathbf{p}_6(z)}{ |z  | }+ \left\vert \dbar \left( \Xi_\calZ(z) \right)  \right\vert \right) e^{-v \tau / u^2 }; \, z\in \Omega_8^-\\
 	\\
 	\left( |\dbar [\mathbf{p}_6(z)]| + |z - z_0 |^{s-3/2} + \left\vert \dbar \left( \Xi_\calZ(z) \right)  \right\vert \right) e^{-|u v| \tau }; \, z\in \Omega_8^+
		\end{cases}
\end{equation}
\end{lemma}

\begin{proof}
 In $\Omega_3$, set $z=(u + z_0)+iv$ and $0\leq v<u$.
\begin{align*}
\text{Re}(2i\theta)  &= \dfrac{1}{2}( \text{Im} z ) \left( \dfrac{1}{1+z_0^2}  \right) \left(   \dfrac{z_0^2}{(\text{Re}z )^2 + (\text{Im}z )^2 } -1\right)t \\ 
                            & =\dfrac{1}{2} v \left( \dfrac{1}{1+z_0^2}  \right) \left(   \dfrac{ -u^2-2z_0 u -v^2}{ ( u+z_0)^2+ v^2 } \right)t\\
                            &\leq -\dfrac{1}{2}v \dfrac{1}{1+z_0^2}\dfrac{u^2+2z_0 u}{( u+z_0)^2+v^2}t
                           \end{align*}
 Near $z_0$ we take
 \begin{equation}
 \label{omega8}
 \text{Re}(2i\theta) \leq - \dfrac{1}{1+z_0^2}\dfrac{z_0 uv}{( u+z_0)^2+v^2}t \lesssim -|u||v|\tau
 \end{equation}
and away from $z_0$, we take
\begin{equation}
 \label{omega8'}
 \text{Re}(2i\theta) \leq -\dfrac{1}{2}v \dfrac{1}{1+z_0^2}\dfrac{u^2 }{( u+z_0)^2+v^2}t \lesssim -|v|\tau.
 \end{equation}
In $\Omega_8^-$, we set $z=u+iv$ and use the facts that $u\geq 0$, $v\geq 0$ and $u^2+v^2\leq z_0^2/3$ to deduce
\begin{align*}
-\text{Re}(2i\theta)  &=- \dfrac{1}{2}( \text{Im} z ) \left( \dfrac{1}{1+z_0^2}  \right) \left(   \dfrac{z_0^2}{(\text{Re}z )^2 + (\text{Im}z )^2 } -1\right)t \\ 
                            & =-\dfrac{1}{2} v \left( \dfrac{1}{1+z_0^2}  \right) \left(   \dfrac{z_0^2 }{ u^2+ v^2 } -1\right)t\\
                            &\leq -\dfrac{z_0^2vt}{6(1+z_0^2)u^2} .\\
                            \end{align*}
Finally, in $\Omega_8^+$, we set $z=(u + z_0)+iv$ and notice that $- z_0/2 <u<0$ and $|u|>|v|$
\begin{align*}
-\text{Re}(2i\theta)  &=- \dfrac{1}{2}( \text{Im} z ) \left( \dfrac{1}{1+z_0^2}  \right) \left(   \dfrac{z_0^2}{(\text{Re}z )^2 + (\text{Im}z )^2 } -1\right)t \\ 
                            & =\dfrac{1}{2} v \left( \dfrac{1}{1+z_0^2}  \right) \left(   \dfrac{ u^2 + 2z_0 u + v^2}{ ( u+z_0)^2+ v^2 } \right)t\\
                            &\leq \dfrac{3\sqrt{3}-1}{4\sqrt{3}} \dfrac{z_0}{1+z_0^2}\dfrac{ uv }{( u+z_0)^2+v^2}t\\
                            & \lesssim - |u||v|\tau.
                           \end{align*}
Estimates \eqref{dbar.R_1.bd} and  \eqref{dbar.R_6.bd}  then follow from Lemma \ref{lemma:dbar.Ri}. 

\end{proof}

\begin{lemma}
\label{lemma:RHP.bd}For the localized Riemann-Hilbert problem from Problem \ref{MKDV.RHP.local}, we have
\begin{align}
\label{RHP.bd1}
\norm{m^\RHP(\dotarg; x, t)}{\infty}	&	\lesssim		1,\\[5pt]
\label{RHP.bd2}
\norm{m^\RHP(\dotarg; x,t )^{-1}}{\infty}	&	\lesssim	1.
\end{align}
All implied constants are uniform  for  $r$ in a bounded subset of $H^{s}_0(\bbR)$.
\end{lemma}

The proof of this lemma is a consequence of the previous section.

\begin{lemma}
\label{lemma:KW}
Suppose that $r\in H_0^{s}(\bbR)$.
Then, the estimate \eqref{dbar.int.est1}
holds, where the implied constant is uniform  for  $r$ in a bounded subset of $H_0^{s}(\bbR)$.
\end{lemma}

\begin{proof}
To prove \eqref{dbar.int.est1}, first note that
\begin{align}
 \norm{K_W f}{\infty} &\leq \norm{f}{\infty} \int_\bbC \frac{1}{|z-\zeta|}|W(\zeta)| \, dm(\zeta) 
                                \end{align}
so that we need only estimate the right-hand integral. We will prove the estimate in the region $ z\in\Omega_3$ first. From \eqref{W-bound}, it follows
$$ |W(\zeta)| \leq \norm{m^{\RHP}}{\infty} \norm{(m^{\RHP})^{-1}}{\infty} \left| \dbar R_1\right| |e^{2i\theta}|.$$
Setting $z=\alpha+i\beta$ and $\zeta=(u+z_0)+iv$, the region $\Omega_{3,1}$ corresponds to $u\geq \sqrt{3}v \geq 0 $. We then have from \eqref{dbar.R_1.bd} \eqref{RHP.bd1}, and \eqref{RHP.bd2} that
$$
 \int_{\Omega_{3,1}}  \frac{1}{|z-\zeta|} |W(\zeta)| \, d\zeta  \lesssim  I_1 + I_2 +I_3
$$
where
\begin{align*}
I_1 	&=	\int_0^{1/\sqrt{3}} \int_{\sqrt{3} v}^1 \frac{1}{|z-\zeta|} |\overline{\partial}[\mathbf{p}_1(\zeta)]| e^{-uv\tau} \, du \, dv, \\[5pt]
I_2	&=	\int_0^{1/\sqrt{3}} \int_{\sqrt{3} v}^1 \frac{1}{|z-\zeta|} \left| u+iv \right|^{s-3/2} e^{- uv\tau } \, du \, dv,\\
I_3	&=	\int_0^{1/\sqrt{3}} \int_{\sqrt{3}v}^1 \frac{1}{|z-\zeta|} \left|  \dbar (\Xi_\calZ(\zeta  ) ) \right| e^{- uv\tau} \, du \, dv.
\end{align*}
It now follows from \cite[proof of Lemma 3.7]{CuPe} that
$$
|I_1|, \, |I_2|, \, |I_3| \lesssim \tau^{(1-2s)/4}.
$$
It then follows that
\begin{equation}
\label{omega31}
\int_{\Omega_{3,1}} \frac{1}{|z-\zeta|} |W(\zeta)| \, d\zeta \lesssim \tau^{(1-2s)/4}
\end{equation}
Similar estimates hold for the integrations over the remaining $\Omega_8^+$:
\begin{equation}
\label{omega8+}
 \int_{\Omega_{8}^+} \frac{1}{|z-\zeta|} |W(\zeta)| \, d\zeta \lesssim \tau^{(1-2s)/4}. 
\end{equation}
Now we turn to region $\Omega_8^-$ and write
$$
\int_{\Omega_{8}^-}  \frac{1}{|z-\zeta|} |W(\zeta)| \, d\zeta  \lesssim  I_1 + I_2 +I_3
$$
where
\begin{align*}
I_1 	&=	\int_0^{z_0/(2\sqrt{3} )} \int_{\sqrt{3} v}^{z_0/2} \frac{1}{|z-\zeta|} |\overline{\partial}[\mathbf{p}_6(\zeta)]| e^{-v \tau / u^2} \, du \, dv, \\[5pt]
I_2	&=	\int_0^{z_0/(2\sqrt{3})} \int_{\sqrt{3} v}^{z_0/2}  \frac{1}{|z-\zeta|} \left| u+iv \right|^{s-3/2} e^{-v \tau / u^2 } \, du \, dv,\\
I_3	&=	\int_0^{z_0 / (2\sqrt{3}) } \int_{\sqrt{3}v}^{z_0/2}  \frac{1}{|z-\zeta|} \left|  \dbar (\Xi_\calZ(\zeta  ) ) \right| e^{-v \tau / u^2} \, du \, dv.
\end{align*}
We can rewrite
\begin{equation}
  \left[ \mathcal{P}_{\Imag \zeta}* r\right](\Real \zeta)=\dfrac{1}{2\pi}\int_\bbR e^{i\xi u}\hat{r}(\xi)\hat{\mathcal{P}}(v\xi) d\xi
  \end{equation}
  and calculate that 
  \begin{align}
      \dbar[ \left[ \mathcal{P}_{v}* r\right](u) ]=\dfrac{1}{2}\int_\bbR i\xi\left(  e^{i\xi u}\hat{r}(\xi)\hat{\mathcal{P}}(v\xi) + e^{i\xi u}\hat{r}(\xi)\hat{\mathcal{P}}'(v\xi) \right) d\xi.
      \end{align}
An application of the \textit{Plancheral}'s theorem leads to
      \begin{equation}
        \norm{ \dbar[ \left[ \mathcal{P}_{v}* r\right](u) }{L^2_u}\lesssim \norm{r}{H^s(\bbR)} v^{s-1}.
      \end{equation}
Setting $z=\alpha+i\beta$, then using the fact that
$$\norm{\dfrac{1}{z-\zeta}}{L^2_u(v, \infty)}^2\leq \dfrac{\pi}{|v-\beta|}$$
we estimate
\begin{align*}
|I_1| &\lesssim \int_0^\infty  \int_{\sqrt{3} v}^\infty \dfrac{\overline{\partial}[\mathbf{p}_6(\zeta)]}{|\zeta-z|}e^{-v \tau / u^2}dudv\\
&\lesssim \norm{r}{H^s}\int_0^\infty \dfrac{e^{-v\tau}}{v^{1-s}|v-\beta|^{1/2}} dv\\
&\lesssim \tau^{1/2-s}.
\end{align*}
For $I_2$, we choose some appropriate $p$ and use \textit{H\"older}'s inequality to bound
\begin{align*}
|I_2|  & \lesssim \int_0^\infty e^{-v\tau} v^{1/p+(s-3/2)}|v-\beta|^{1/q-1}dv\\
        &=  \int_0^\beta e^{-v\tau} v^{1/p+(s-3/2)}|v-\beta|^{1/q-1}dv +  \int_\beta^\infty e^{-v\tau} v^{1/p+(s-3/2)}|v-\beta|^{1/q-1}dv\\
        &\leq \int_0^1\beta^{1/2}e^{-\beta w\tau}w^{1/p+(s-3/2)}(1-w)^{1/q-1}dw +\int_\beta^\infty e^{-v\tau} v^{1/p+(s-3/2)}|v-\beta|^{1/q-1}dv\\
        &\lesssim  \tau^{1/2-s}.
\end{align*}
The estimate on $I_3$ is similar to that of $I_1$.
So we arrive at
\begin{equation}
\label{omega8-}
\int_{\Omega_{8}^-}  \frac{1}{|z-\zeta|} |W(\zeta)| \, d\zeta  \lesssim \tau^{1/2-s}.
\end{equation}

Similar procedure gives the following estimate on $\Omega_{3,2}$
\begin{equation}
\label{omega32}
\int_{\Omega_{3,2}}  \frac{1}{|z-\zeta|} |W(\zeta)| \, d\zeta  \lesssim \tau^{1/2-s}.
\end{equation}
Combining \eqref{omega31}-\eqref{omega32} leads to \eqref{dbar.int.est1}.
\end{proof}

\begin{lemma}
\label{lemma:N31.est}
The estimate   \eqref{N31.est} 
 holds with constants uniform in $r$ in a bounded subset of $H_0^{s}(\bbR)$.
\end{lemma}

\begin{proof}
We have defined in \eqref{N3.1}
\begin{equation}
\label{N3.1}
m^{(3)}_1(x,t) = I+\dfrac{1}{\pi} \int_{\bbC}  \dfrac{  m^{(3)}(\zeta;x,t) W(\zeta;x,t)}{\zeta} \, d\zeta . 
\end{equation}
From the representation formula \eqref{N3.1} and Lemma \ref{lemma:KW} we have
$$ \left|m^{(3)}_1(x,t) \right| \lesssim I+ \int_\bbC \dfrac{ |W(\zeta;x,t)|}{|\zeta|} \, d\zeta. $$
We will bound this integral by $\tau^{-3/4}$ modulo constants with the required uniformities. Again we only work with $\Omega_8$ and $\Omega_3$ as shown in Figure
\ref{fig:three reigons}. In $\Omega_8^-$
$$
\int_{\Omega_8^-} \dfrac{ |W(\zeta;x,t)| }{|\zeta|}\, d \zeta	\lesssim  I_1+I_2 +I_3
$$
where
\begin{align*}
I_1 	&=	\int_0^{z_0/(2\sqrt{3} )} \int_{\sqrt{3} v}^{z_0/2} \dfrac{1}{|\zeta|} |\overline{\partial}[\mathbf{p}_6(\zeta)]| e^{-v \tau / u^2} \, du \, dv, \\[5pt]
I_2	&=	\int_0^{z_0/(2\sqrt{3})} \int_{\sqrt{3} v}^{z_0/2}  \dfrac{1}{|\zeta|^2 }|\mathbf{p}_6(\zeta)|  e^{-v \tau / u^2 } \, du \, dv,\\
I_3	&=	\int_0^{z_0 / (2\sqrt{3}) } \int_{\sqrt{3}v}^{z_0/2}  \dfrac{1}{|\zeta|} \left|  \dbar (\Xi_\calZ(\zeta  ) ) \right| e^{-v \tau / u^2} \, du \, dv.
\end{align*}
For $I_1$ we estimate that 
\begin{align*}
I_1 	&=	\int_0^{z_0/(2\sqrt{3} )} \int_{\sqrt{3} v}^{z_0/2} \dfrac{1}{|\zeta|} |\overline{\partial}[\mathbf{p}_6(\zeta)]| e^{-v \tau / u^2} \, du \, dv\\
        &\leq \int_0^{\infty}  \left( \int_0^{z_0/2} {|\overline{\partial}[\mathbf{p}_6(\zeta)] |^2}du  \right)^{1/2} \left( \int_{\sqrt{3} v}^{z_0/2} \dfrac{ e^{-2v \tau / u^2} }{ u^2  } du  \right)^{1/2} dv\\
        &\lesssim  \int_0^{\infty}  \norm{r}{H^s(\bbR)} v^{s-1}  \left( \int_{{2/z_0} }^{\infty} {e^{-4vw^2\tau}} dw  \right)^{1/2} dv\\
        &  \lesssim  \int_0^{\infty}  v^{s-1}  \left( \int_{{2/z_0} }^{\infty} {e^{-4vw\tau}} dw  \right)^{1/2} dv\\
        & \lesssim \tau^{-1/2}\int_0^{\infty} \dfrac{e^{-(8/z_0)v\tau}}{v^{3/2-s}} dv\\
         & \lesssim \tau^{-1/2-1+(3/2-s)}\int_0^{\infty} \dfrac{e^{-(8/z_0)m}}{m^{3/2-s}} dm\\
         &\lesssim \tau^{-s}.
        \end{align*}
For $I_2$ we rewrite and apply \textit{H\"older}'s inequality with $p=4/3$:
\begin{align}
\label{est-r-0}
I_2	 &= \int_0^{z_0/(2\sqrt{3} )}  \int_{\sqrt{3} v}^{z_0/2} \dfrac{ |\mathbf{P}_6(\zeta)|}{ |u^2+v^2|^{1/2}  }\dfrac{e^{-v\tau/u^2}}{| u^2+v^2  |^{1/2}  } \, du \, dv\\
\nonumber
   &\lesssim \norm{r}{H^s(\bbR)} \int_0^{z_0/2} \dfrac{1}{u^2} \int_{ 0}^{u/\sqrt{3}} e^{-v\tau/u^2} dv   du\\
   \nonumber
    & \lesssim \tau^{-1}.
 \end{align}
 Notice that we have shown that $r(z)/z$ is bounded near the origin in \eqref{r=0}. Thus $\mathbf{P}_6(u)/u$ is bounded on the interval $(0, z_0/2)$.
 And 
 $$I_3\lesssim \tau^{-s}$$
 follows from the same argument as that of $I_1$. So we establish
 \begin{equation}
 \label{omega8-'}
 \int_{\Omega_8^-} \dfrac{ |W(\zeta;x,t)| }{|\zeta|}\, d \zeta	 \lesssim \tau^{-s}.
 \end{equation}
 Given the fact that $|\zeta|\geq z_0/2 $, it then follows from previous calculations that
\begin{align}
 \int_{\Omega_8^+} \dfrac{ |W(\zeta;x,t)| }{|\zeta|}\, d \zeta	  &\lesssim \tau^{-s}, \\
  \int_{\Omega_{3,1}} \dfrac{ |W(\zeta;x,t)| }{|\zeta|}\, d \zeta	 &\lesssim \tau^{-(2s+1)/4}.
\end{align}
 We finally turn to region $\Omega_{3,2}$. In this region,
$$
\int_{\Omega_{3,2}} \dfrac{ |W(\zeta;x,t)| }{|\zeta|}\, d \zeta	\lesssim  I_1+I_2 +I_3
$$
where
\begin{align*}
I_1 	&=	\int_0^{\infty} \int_{z_0+1}^{\infty} \dfrac{1}{|\zeta|} |\overline{\partial}[\mathbf{p}_1(\zeta)]| e^{-v\tau} \, du \, dv, \\[5pt]
I_2	&=	\int_0^{\infty} \int_{z_0+1}^{\infty}\dfrac{1}{|\zeta-z_0|^{3/2-s}} \dfrac{1}{|\zeta|} e^{- v\tau } \, du \, dv,\\
I_3	&=	\int_0^{\infty} \int_{z_0+1}^{\infty} \dfrac{1}{|\zeta|} \left|  \dbar (\Xi_\calZ(\zeta  ) ) \right| e^{- v\tau} \, du \, dv.
\end{align*}
For $I_1$ we rewrite and apply the Cauchy-Schwarz inequality:
\begin{align*}
I_1 &=\int_0^{\infty} \int_{z_0+1}^{\infty} \dfrac{1}{\sqrt{u^2+v^2}} |\overline{\partial}[\mathbf{p}_1 (\zeta)]| e^{-v\tau} \, du \, dv,\\
     &\lesssim  \int_0^{\infty} v^{s-1}e^{-v\tau}  dv\lesssim \tau^{-s}.
 \end{align*}
 Similarly,
 \begin{align*}
I_2 &=\int_0^{\infty} \int_{z_0+1}^{\infty} \dfrac{1}{\left[{(u-z_0)^2+v^2} \right]^{(3-2s)/4}}\dfrac{1}{\sqrt{u^2+v^2}} e^{-v\tau} \, du \, dv,\\
     &\lesssim  \int_0^{\infty} e^{-v\tau} \left( \int_1^\infty \dfrac{1}{u^{(5-2s)/2}} du \right) dv \lesssim \tau^{-1}.
 \end{align*}
 And 
 $$I_3\lesssim \tau^{-1}$$
 follows from the same argument as that of $I_1$. So we establish
 \begin{equation}
 \label{omega32'}
 \int_{\Omega_{3,2}} \dfrac{ |W(\zeta;x,t)| }{|\zeta|}\, d \zeta	 \lesssim \tau^{-(2s+1)/4}.
 \end{equation}
 Combining \eqref{omega8-'}-\eqref{omega32'} we arrive at the inequality in \eqref{N31.est}.
\end{proof}

\section{Long-Time Asymptotics inside the Light Cone}
\label{sec:large-time}

We now put together our previous results and formulate the long-time asymptotics of $f(x,t)$  inside the light cone.  Undoing all transformations we carried out previously,  we get back $M$:
\begin{equation}
\label{N3.to.N}
M(z;x,t) = m^{(3)}(z;x,t) m^\RHP(z; z_0) \left( \calR^{(2)}(z)\right)^{-1} \delta(z)^{\sigma_3}.
\end{equation}
By Proposition \ref{prop:recon}, the non-zero term in the small-$z$ expansion for $M(z;x,t)$ will formulate the solution to the sine-Gordon equation.  In the previous sections, we have proven:
\begin{lemma}
\label{lemma:N.to.NRHP.asy}
For $z=i\sigma$ and $\sigma \rarr 0^+$, the following asymptotic relations hold
\begin{align}
\label{N.asy}
m^{(3)}(0; x,t) 				&=	I +\mathcal{O}\left( \tau^{-(2s+1)/4}\right) \\
\label{expan-mLC}
m^\RHP (0; x, t)  &= E (0; x, t)m^{(br)}(0; x, t)\\
 \nonumber
             &= \left( I+\mathcal{O} \left( \tau^{-1/2}\right) \right) m^{(br)}(0; x, t)\\
             \label{delta.sigma.asy} 
\delta(0)^{\sigma_3}  &= \twomat{(-1)^l}{0}{0}{ (-1)^{-l} }     \\
\nonumber
          \calR^{(2)}(z) &=I.
\end{align}

\end{lemma}

With Lemma \ref{lemma:N.to.NRHP.asy},  we arrive at the asymptotic formula under the reference frame of a given breather:
$$\mathrm{v}_\ell=\dfrac{x}{t}=\dfrac{1-\rho_\ell^2}{1+\rho_\ell^2}$$
where $\rho_\ell$ is indicated in Problem \ref{prob:mkdv.RHP1}.

\begin{proposition}
\label{lemma br.RHP.asy}
The function 
\begin{equation}
\label{cos.recon.bis}
\cos f(x,t) = 1+2 m_{12}(0 ;x,t) m_{21} (0 ;x,t)
\end{equation}
takes the form 
$$ \cos f(x,t)= 1+ 2m_{12}^{(br)} (0; x,t)m_{21}^{(br)} (0; x,t)+  \dfrac{\mathrm{R_{cos} }}{\tau^{1/2}}+\mathcal{O}\left( \tau^{-(2s+1)/4}\right) $$
where  $ m^{(br)} (0; x,t)$ is given by \eqref{br-0} and

\begin{align}
\label{cos-asym}
\dfrac{ \mathrm{R_{cos} }(x, t) }{\tau^{1/2}}&=    \dfrac{2}{   \sqrt{\dfrac{t  z_0 }{ 1+z_0^2 } }  }\left\lbrace -\left[m_{11}^{(br)}( -z_0)\right]^2 \left( i ( \delta^0_A)^2 \overline{\beta}_{12} \right) - \left[m_{12}^{(br)} (-z_0)\right]^2 \left(  i \left( \delta^0_A\right)^{-2}\overline{\beta}_{21} \right)  \right. \\
\nonumber
&\quad   \left.+ \left[ m_{11}^{(br)} (z_0)\right]^2 \left( i ( \delta^0_B)^2 {\beta}_{12} \right)  + \left[ m_{12}^{(br)}(z_0)\right]^2 \left( i ( \delta^0_B)^{-2}{\beta}_{21} \right) \right\rbrace m_{21}^{(br)}(0)m_{22}^{(br)}(0) \\
\nonumber
& \quad +  \dfrac{2}{   \sqrt{\dfrac{t  z_0 }{ 1+z_0^2 } }  }\left\lbrace \left[m_{21}^{( br )}( -z_0)\right]^2 \left( i ( \delta^0_A)^2 \overline{\beta}_{12} \right) + \left[m_{12}^{(br)} (-z_0)\right]^2 \left(  i ( \delta^0_A)^{-2}\overline{\beta}_{21} \right)  \right. \\
\nonumber
&\quad   \left.- \left[ m_{21}^{(br)} (z_0)\right]^2 \left( i ( \delta^0_B)^2 {\beta}_{12} \right)    - \left[m_{22}^{(br)}(z_0)\right]^2 \left( i ( \delta^0_B)^{-2}{\beta}_{21} \right) \right\rbrace m_{11}^{( br )}(0)m_{12}^{(br)}(0) 
\end{align}
and the function
\begin{equation}
\label{sin.recon.bis}
\sin f(x,t) =2 m_{21}(0 ;x,t) m_{22} (0 ;x,t)
\end{equation}
takes the form 
$$ \sin f(x,t)= 2m_{21}^{(br)} (0; x,t)m_{22}^{(br)} (0; x,t)+  \dfrac{\mathrm{R_{sin} }}{\tau^{1/2}}+\mathcal{O}\left( \tau^{-(2s+1)/4}\right) $$
with
\begin{align}
\label{sin-asym}
\mathrm{R_{sin} }(x, t)&= \dfrac{2}{   \sqrt{\dfrac{t  z_0 }{ 1+z_0^2 } }  }\left\lbrace \left[ m_{21}^{(br)}( -z_0)\right]^2 \left( i ( \delta^0_A)^2 \overline{\beta}_{12} \right) + \left[m_{12}^{( br )} (-z_0)\right]^2 \left(  i ( \delta^0_A)^{-2}\overline{\beta}_{21} \right)  \right. \\
\nonumber
&\quad   \left.- \left[ m_{21}^{(br)} (z_0)\right]^2 \left( i ( \delta^0_B)^2 {\beta}_{12} \right)    -\left[ m_{22}^{( br )}(z_0)\right]^2 \left( i ( \delta^0_B)^{-2}{\beta}_{21} \right) \right\rbrace \\
\nonumber
  &\quad \times \left\lbrace m_{11}^{( br )}(0)m_{22}^{(br)}(0) + m_{12}^{(br)}(0)m_{21}^{(br)}(0) \right\rbrace.
\end{align}
All the terms above have been obtained in Section \ref{sec:local}.
\end{proposition}
Next we choose the reference frame of a kink:
$$\mathrm{v}_\ell=\dfrac{x}{t}=\dfrac{1-\zeta_\ell^2}{1+\zeta_\ell^2}$$
where $z_\ell=i\zeta_\ell \in \lbrace i\zeta_k   \rbrace_{k=1}^{N_1}$ is indicated in Definition \ref{genericity} and Remark \ref{re-arrange}. 
\begin{proposition}
\label{lemma kin.RHP.asy}
The function 
\begin{equation}
\label{cos.recon.bis}
\cos f(x,t) = 1+2 m_{12} (0 ;x,t)m_{21} (0 ;x,t)
\end{equation}
takes the form 
$$ \cos f(x,t)= 1+ 2m_{12}^{(kin)} (0; x,t)m_{21}^{(kin)} (0; x,t)+  \dfrac{\mathrm{R_{cos} }}{\tau^{1/2}}+\mathcal{O}\left( \tau^{-(2s+1)/4}\right) $$
where  
\begin{align}
\label{cos-asym k}
\dfrac{ \mathrm{R_{cos} }(x, t) }{\tau^{1/2}}&=    \dfrac{2}{   \sqrt{\dfrac{t  z_0 }{ 1+z_0^2 } }  }\left\lbrace -\left[m_{11}^{(kin)}( -z_0)\right]^2 \left( i ( \delta^0_A)^2 \overline{\beta}_{12} \right) - \left[m_{12}^{(kin)} (-z_0)\right]^2 \left(  i ( \delta^0_A)^{-2}\overline{\beta}_{21} \right)  \right. \\
\nonumber
&\quad   \left.+ \left[ m_{11}^{(kin)} (z_0)\right]^2 \left( i ( \delta^0_B)^2 {\beta}_{12} \right)  + \left[ m_{12}^{(kin)}(z_0)\right]^2 \left( i ( \delta^0_B)^{-2}{\beta}_{21} \right) \right\rbrace m_{21}^{(kin)}(0)m_{22}^{(kin)}(0) \\
\nonumber
& \quad +  \dfrac{2}{   \sqrt{\dfrac{t  z_0 }{ 1+z_0^2 } }  }\left\lbrace \left[m_{21}^{( kin )}( -z_0)\right]^2 \left( i ( \delta^0_A)^2 \overline{\beta}_{12} \right) + \left[m_{12}^{(kin)} (-z_0)\right]^2 \left(  i ( \delta^0_A)^{-2}\overline{\beta}_{21} \right)  \right. \\
\nonumber
&\quad   \left.- \left[ m_{21}^{(kin)} (z_0)\right]^2 \left( i ( \delta^0_B)^2 {\beta}_{12} \right)    - \left[m_{22}^{(kin)}(z_0)\right]^2 \left( i ( \delta^0_B)^{-2}{\beta}_{21} \right) \right\rbrace m_{11}^{( kin )}(0)m_{12}^{(kin)}(0) 
\end{align}
and the function
\begin{equation}
\label{sin.recon.bis}
\sin f(x,t) =2 m_{21}(0 ;x,t) m_{22} (0 ;x,t)
\end{equation}
takes the form 
$$ \sin f(x,t)= 2m_{21}^{(kin)} (0; x,t)m_{22}^{(kin)} (0; x,t)+  \dfrac{\mathrm{R_{sin} }}{\tau^{1/2}}+\mathcal{O}\left( \tau^{-(2s+1)/4}\right) $$
with
\begin{align}
\label{sin-asym k}
\mathrm{R_{sin} }(x, t)&= \dfrac{2}{   \sqrt{\dfrac{t  z_0 }{ 1+z_0^2 } }  }\left\lbrace \left[ m_{21}^{(kin)}( -z_0)\right]^2 \left( i ( \delta^0_A)^2 \overline{\beta}_{12} \right) + \left[m_{12}^{( kin )} (-z_0)\right]^2 \left(  i ( \delta^0_A)^{-2}\overline{\beta}_{21} \right)  \right. \\
\nonumber
&\quad   \left.- \left[ m_{21}^{(kin)} (z_0)\right]^2 \left( i ( \delta^0_B)^2 {\beta}_{12} \right)    -\left[ m_{22}^{( kin )}(z_0)\right]^2 \left( i ( \delta^0_B)^{-2}{\beta}_{21} \right) \right\rbrace \\
\nonumber
  &\quad \times \left\lbrace m_{11}^{( kin )}(0)m_{22}^{(kin)}(0) + m_{12}^{(kin)}(0)m_{21}^{(kin)}(0) \right\rbrace.
\end{align}
\end{proposition}
\begin{proof}
In this case we only have to replace Problem \ref{prob:mkdv.br} with the following exactly solvable problem:
\begin{problem}
\label{prob:mkdv.kin}
Find a matrix-valued function $m^{(kin)}(z;x,t)$ on $\bbC \setminus\Sigma$ with the following properties:
\begin{enumerate}
\item		$m^{(kin)}(z;x,t) \rarr I$ as $|z| \rarr \infty$,
\item		$m^{(kin)}(z;x,t)$ is analytic for $z \in  \bbC \setminus (  \tgamma_\ell \cup \tgamma^*_\ell ) $
			with continuous boundary values
			$m^{(kin)}_\pm(z;x,t)$.
\item On $  \tgamma_\ell \cup \tgamma^*_\ell$, let $\delta(z)$ be the solution to Problem \ref{prob:RH.delta} and we have the following jump conditions
$m^{(kin)}_+(z;x,t)=m^{(kin)}_-(z;x,t)	
			e^{-i\theta\ad\sigma_3}v^{(kin)}(z)$
			where
$$
e^{-i\theta\ad\sigma_3}v^{(kin)}(z)= 	\begin{cases}
						\twomat{1}{0}{\dfrac{\iota_\ell \left[ \delta(i\zeta_\ell )\right]^{-2} }{z-i\zeta_\ell} e^{2i\theta(i\zeta_\ell )}}{1}	&	z\in \tgamma_\ell, \\
						\\
						\twomat{1}{\dfrac{\overline{\iota_\ell} \left[ \delta(\overline{i\zeta_\ell })\right]^2 }{z -\overline{i\zeta_\ell }} e^{-2i\theta (\overline{i\zeta_\ell })} }{0}{1}
							&	z \in \tgamma_\ell^* .\\
							
\end{cases}
$$					
\end{enumerate}
\end{problem}
Recall that $\iota_\ell =ib_\ell$ and $i\zeta_\ell$ are both purely imaginary and notice that 
\begin{align*}
{\chi(i\zeta_\ell)} &=\dfrac{1}{2\pi i}\int_{-z_0}^{z_0}\dfrac{\log( 1+|r(s)|^2)}{s-i\zeta_\ell} {d\zeta} \\
  &=\dfrac{1}{2\pi i}\int_{-z_0}^{z_0}\dfrac{\log( 1+|r(s)|^2) (s+i\zeta_\ell   )   }{s^2 +\zeta^2_\ell} {d\zeta} \\
   &=\dfrac{1}{2\pi }\int_{-z_0}^{z_0}\dfrac{\log( 1+|r(s)|^2) \zeta_\ell  }{s^2 +\zeta^2_\ell} {d\zeta}
\end{align*}
is real by the evenness of $|r(z)|^2$.
And we can explicitly calculate that 
\begin{equation}
		m^{(kin)}(z)=\twomat{1+\dfrac{ia}{z-i\zeta_\ell}}{\dfrac{ib}{z+i\zeta_\ell}}
				{\dfrac{ib}{z-i\zeta_\ell}}{1-\dfrac{ia}{z+i\zeta_\ell}}.
\end{equation}
where
\begin{align}
		a=&\dfrac{2 \left( b_\ell  \left[  \delta(i\zeta_\ell )\right]^{-2}  e^{2\theta(i\zeta_\ell )}  \right)^2 \zeta_\ell }{ \left( b_\ell \left[  \delta(i\zeta_\ell )\right]^{-2}   e^{2\theta(i\zeta_\ell )}  \right )^2+4\zeta_\ell^2},\\
	   b=&\dfrac{4 \left( b_\ell \left[  \delta(i\zeta_\ell )\right]^{-2} e^{2\theta(i\zeta_\ell )} \right) \zeta_\ell^2}{\left( b_\ell  \left[ \delta(i\zeta_\ell )\right]^{-2} e^{2\theta(i\zeta_\ell )}   \right)^2+4\zeta_\ell^2}.
	\label{eq:Kink-1/2u}
\end{align}
And consequently
\begin{equation}
		\begin{split}
				m^{(kin)}_{11}(0) &=m^{(kin)}_{22}(0)=\dfrac{4 \zeta_\ell^2-\left( b_\ell   \left[\delta(z_\ell )\right]^{-2} e^{2\theta(i\zeta_\ell )}   \right) ^2}{ \left(b_\ell  \left[ \delta(i\zeta_\ell )\right]^{-2} e^{2\theta(i\zeta_\ell )}  \right) ^2+4\zeta_\ell^2},\\
				m^{(kin)}_{21}(0) &=-m^{(kin)}_{12}(0)=-\dfrac{4 \left( b_\ell  \left[ \delta(i\zeta_\ell )\right]^{-2} e^{2\theta(i\zeta_\ell )} \right) \zeta_\ell}{  \left(b_\ell \left[  \delta(i\zeta_\ell )\right]^{-2} e^{2\theta(i\zeta_\ell )}  \right) ^2+4\zeta_\ell^2},
		\end{split}
\end{equation}

\begin{align}
\label{kink-cos}
 2m_{12}^{(kin)} (0; x,t)m_{21}^{(kin)} (0; x,t) &=-2\dfrac{ \left(b_\ell  \left[ \delta(i\zeta_\ell )\right]^{-2} \right)^2 e^{4\theta(i\zeta_\ell )}}{ \left( \zeta_\ell+ \dfrac{  \left(b_\ell \left[  \delta(i\zeta_\ell )\right]^{-2} \right) ^2 e^{4\theta(i\zeta_\ell)} }{  4\zeta_\ell }   \right)^2 }\\
 \nonumber
        &=-2\,\mathrm{sech}\left(\dfrac{1}{2}\left[\left(\zeta_\ell+\dfrac{1}{\zeta_\ell}\right)x+\left(\zeta_\ell-\dfrac{1}{\zeta_\ell}\right)t\right]-\log\dfrac{b_\ell  \left[ \delta(i\zeta_\ell )\right]^{-2} }{2\zeta_\ell}\right)^2,
\end{align}

\begin{align}
\label{kink-sin}
 2m_{21}^{(kin)} (0; x,t)m_{22}^{(kin)} (0; x,t) &=2\dfrac{ b_\ell \left[  \delta(i\zeta_\ell )\right]^{-2}  e^{2\theta(i\zeta_\ell )}  \left( \zeta_\ell- \dfrac{  \left(b_\ell \left[  \delta(i\zeta_\ell ) \right]^{-2} \right)^2 e^{4\theta(i\zeta_\ell)} }{  4\zeta_\ell }   \right)   }   { \left( \zeta_\ell+ \dfrac{  \left( b_\ell  \left[ \delta(i\zeta_\ell )\right]^{-2} \right) ^2 e^{4\theta(i\zeta_\ell)} }{  4\zeta_\ell }   \right)^2 }\\
 \nonumber
 &=-\frac{2\sinh\left(\dfrac{1}{2}\left[\left(\zeta_\ell+\dfrac{1}{\zeta_\ell}\right)x+\left(\zeta_\ell-\dfrac{1}{\zeta_\ell}\right)t\right]-\log\dfrac{b_\ell  \left[ \delta(i\zeta_\ell )\right]^{-2} }{2\zeta_\ell}\right)}{\cosh\left(\dfrac{1}{2}\left[\left(\zeta_\ell+\dfrac{1}{\zeta_\ell}\right)x+\left(\zeta_\ell-\dfrac{1}{\zeta_\ell}\right)t\right]-\log\dfrac{b_\ell \left[  \delta(i\zeta_\ell )\right]^{-2} }{2\zeta_\ell}\right)^2}.
 \end{align}

\end{proof}

Finally in the solitonless region we have:
\begin{proposition} If we choose the frame  $x=\mathrm{v} t$ with $| \mathrm{v} |< 1$ and $\mathrm{v} \neq ( 1-\zeta_k^2)/(1+\zeta_k^2) $ for all $1\leq k \leq N_1$ and $\mathrm{v} \neq ( 1-\rho_j^2)/(1+\rho_j^2) $ for all $1\leq j \leq N_2$, then
\begin{align*}
\cos f(x,t)  &= \cos f_{as}(x,t) +\mathcal{O}\left( \tau^{-(2s+3)/4}\right)\\
\sin f(x,t)    &=\sin f_{as}(x,t) +\mathcal{O}\left( \tau^{-(2s+1)/4}\right)
\end{align*}
where
\begin{align}
\label{as-solitonless}
\cos f_{as}(x,t) &=1-\dfrac{4|\kappa |}{\tau} \cos^2 \left(  2\tau  +\kappa \log(8\tau) +\phi(z_0)\right)\\
\label{as-solitonless'}
\sin f_{as}(x, t) &= \sqrt{ \dfrac{8|\kappa|}{\tau}  } \cos \left(  2\tau  +\kappa \log(8\tau) +\phi(z_0)\right)\
\end{align}
with
\begin{align*}
\phi(z_0)&=-\arg \Gamma(i\kappa) + \dfrac{\pi}{4}-\arg  \overline{r(z_0)}+\dfrac{1}{\pi}\int_{-z_0}^{z_0}\log\left( \dfrac{1+|r(\zeta)|^2}{1+|r(z_0)|^2} \right)\dfrac{d\zeta}{\zeta-z_0}\\
             & \quad -4 \left( \sum_{i\zeta_k\in \mathcal{B}_\ell }  \arg(z_0-i\zeta_k)  +  \sum_{z_j \in \mathcal{B}_\ell } \arg(z_0- {z_j}) +  \sum_{z_j\in \mathcal{B}_\ell } \arg(z_0 + \overline{z_j})   \right)
\end{align*}
\end{proposition}
\begin{proof}
We use the signature table \ref{figure-Sigma1}
and define the same  $\delta$ function as \eqref{RH.delta.sol}. We follow the same procedure as in Section 3 and deform contours. It is easy to see that the localized RHP $m^{\text{LC}}$ in this case reduces to Problem \ref{prob:mkdv.A} and Problem \ref{prob:mkdv.B}. We then follow \cite{CL19} and  \cite[Section 4 ]{DZ93} to derive the explicit formula of $u_{as}$ in \eqref{as-solitonless}-\eqref{as-solitonless'}.
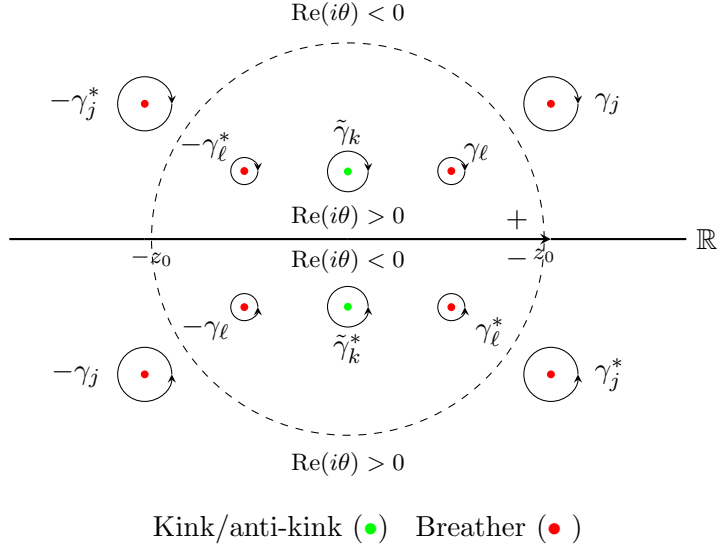
\begin{figure}
\caption{The Augmented Contour $\Sigma$}
\vspace{.5cm}
\label{figure-Sigma1}
\begin{tikzpicture}[scale=0.9]
 \draw[ thick] (0,0) -- (-3,0);
\draw[ thick] (-3,0) -- (-5,0);
\draw[thick,->,>=stealth] (0,0) -- (3,0);
\draw[ thick] (3,0) -- (5,0);
\node[above] at 		(2.5,0) {$+$};
\node[below] at 		(2.5,0) {$-$};
\node[right] at (3.5 , 2) {$\gamma_j$};
\node[right] at (3.5 , -2) {$\gamma_j^*$};
\node[left] at (-3.5 , 2) {$-\gamma_j^*$};
\node[left] at (-3.5 , -2) {$-\gamma_j$};
\draw[->,>=stealth] (-2.6,2) arc(360:0:0.4);
\draw[->,>=stealth] (3.4,2) arc(360:0:0.4);
\draw[->,>=stealth] (-2.6,-2) arc(0:360:0.4);
\draw[->,>=stealth] (3.4,-2) arc(0:360:0.4);
\draw [red, fill=red] (-3,2) circle [radius=0.05];
\draw [red, fill=red] (3,2) circle [radius=0.05];
\draw [red, fill=red] (-3,-2) circle [radius=0.05];
\draw [red, fill=red] (3,-2) circle [radius=0.05];
\node[right] at (5 , 0) {$\bbR$};

\draw (0,0) [dashed] circle [ radius=2.9 ];

\draw [green, fill=green] (0, 1) circle [radius=0.05];
\draw[->,>=stealth] (0.3, 1) arc(360:0:0.3);
\draw [green, fill=green] (0, -1) circle [radius=0.05];
\draw[->,>=stealth] (0.3, -1) arc(0:360:0.3);
\node[above] at 		(0, 1.25) {$\tgamma_k$};
\node[below] at 		(0, -1.25) {$\tgamma_k^*$};
 \node [below] at (2.9,0) {\footnotesize $z_0$};
    \node [below] at (-2.9,0) {\footnotesize $-z_0$};
     \draw	[fill, red]  (1.529, 1.006)		circle[radius=0.05];	  
     \draw[->,>=stealth] (1.729, 1.006) arc(360:0:0.2);
   \draw	[fill, red]  (-1.529, 1.006)	circle[radius=0.05];	  
    \draw[->,>=stealth] (-1.329, 1.006) arc(360:0:0.2);
     \draw	[fill, red]  (1.529, -1.006)		circle[radius=0.05];	  
     \draw[->,>=stealth] (1.729, -1.006) arc(0:360:0.2);
   \draw	[fill, red]  (-1.529, -1.006)	circle[radius=0.05];	  
   \draw[->,>=stealth] (-1.329, -1.006) arc(0:360:0.2); 
    \node[above]  at (0, 0) {\footnotesize $\text{Re}(i\theta)>0$};
    
      \node[below]  at (0, 0) {\footnotesize $\text{Re}(i\theta)<0$};
    
     \node[above]  at (0, 3) {\footnotesize $\text{Re}(i\theta)<0$};
     \node[below]  at (0, -3) {\footnotesize $\text{Re}(i\theta)>0$};
     \node[above] at 		(1.9, 1.0) {$\gamma_\ell$};
\node[above] at 		(-2.1, 1.0) {$-\gamma_\ell^*$};
\node[below] at 		(-2.1, -1.0) {$-\gamma_\ell$};
\node[below] at 		(2.1, -1.0) {$\gamma_\ell^*$};
\end{tikzpicture}
\begin{center}
\begin{tabular}{ccc}
Kink/anti-kink ({\color{green} $\bullet$})	&	
Breather ({\color{red} $\bullet$} ) 
\end{tabular}
\end{center}
\end{figure}

\end{proof}

\section{Outside the light cone}
\label{sec:outside}
We now turn to the study of the asymptotic behavior when $|x/t|>c>1$. We first deal with the case $x/t>c>1$. Our starting point is RHP Problem \ref{RHP-1}. As it will become clear later, outside the light cone there are only higher order decay terms comparing to the interior of the light cone. For the purpose of brevity we are only going to display the calculations directly related to error terms. 
\subsection{$x/t >c> 1$}
\label{subsec: out}
First notice that for $v>0$, 
\begin{equation}
\label{theta-out}
4\text{Re}i\theta(z; x, t)=-\left(   1+\dfrac{x}{t} \right)v t +  \left( 1-\dfrac{x}{t} \right)\dfrac{v t}{u^2+v^2}   <0
\end{equation}
Similarly for $v<0$,
\begin{equation}
\label{theta-out'}
4\text{Re}\left[-i\theta(z; x, t)  \right] =\left(   1+\dfrac{x}{t} \right)v t +  \left( \dfrac{x}{t}-1 \right)\dfrac{v t}{u^2+v^2}  <0.
\end{equation}
Since all the pole conditions given in $\eqref{res-1}-\eqref{res-6}$ have desired decay properties, we only need the  following upper/lower factorization on $\bbR$:
\begin{equation}
\label{v-ul}
e^{-i\theta\ad\sigma_3}v(z)	=\Twomat{1}{\overline{r(z)}   e^{-2i\theta}}{0}{1} \Twomat{1}{0}{r(z)  e^{2i\theta}}{1}.
						\quad z \in\bbR 
\end{equation}
and the contour deformation:
\begin{figure}[H]
\caption{$\Sigma$-outside }
\vskip 15pt
\begin{tikzpicture}[scale=0.7]
\draw[->,thick,>=stealth] 	(3, 3) -- (2,2);						
\draw[thick]		(0,0) -- (2,2);		
\draw[->,thick,>=stealth]  		(0,0) -- (-2,2); 				
\draw[thick]	(-3,3) -- (-2,2);	
\draw[->,thick,>=stealth]		(-3,-3) -- (-2,-2);							
\draw[thick]						(-2,-2) -- (0,0);
\draw[thick,->,>=stealth]		(0,0) -- (2,-2);								
\draw[thick]						(2,-2) -- (3,-3);
\draw	[fill]							(0,0)		circle[radius=0.1];	
\node [below] at (0,0) {0};
\draw [dashed] (0,0)--(4,0);
\draw [dashed] (-4,0)--(0,0);
\node[right] at (3,3)					{$\Sigma_1$};
\node[left] at (-3,3)					{$\Sigma_2$};
\node[left] at (-3,-3)					{$\Sigma_3$};
\node[right] at (3,-3)				{$\Sigma_4$};
\node[above] at (3,0)					{$\Omega_1$};
\end{tikzpicture}
\label{fig:Painleve}
\end{figure}
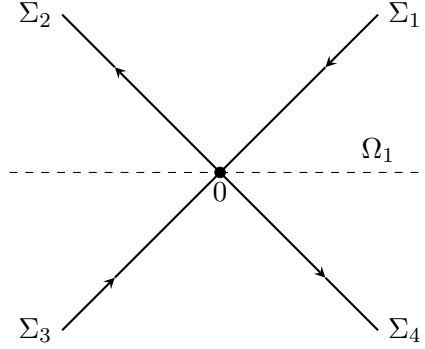
For brevity, we only discuss the situation in $\Omega_1$. In $\Omega_1$, we define
\begin{align*}
	R_1	&=	\begin{cases}
						\twomat{0}{0}{r(z)  e^{2i\theta}  }{0}		
								&	\zeta \in (0,\infty)\\[10pt]
								\\
						\twomat{0}{0}{r( 0 ) e^{2i\theta (z )  } (1-\Xi_\calZ)  }{0}	
								&	\zeta	\in \Sigma_1
					\end{cases}
	\end{align*}
and the interpolation is given by 
$$ \left( r(0)+ \left( r\left( \text{Re} z \right) -r( 0)  \right) \cos 2\phi  \right)    (1-\Xi_\calZ)=r\left( \text{Re} z \right)   \cos 2\phi     (1-\Xi_\calZ)$$
given $r(0)=0$.
So we arrive at the $\dbar$-derivative in $\Omega_1$ :
\begin{align}
\dbar R_1&= \left[ \left( \dbar( \mathbf{r}\left( z\right))  \cos 2\phi- 2\dfrac{ \mathbf{r}\left( z\right)    }{  \left\vert z \right\vert   } e^{i\phi} \sin 2\phi  \right)  (1-\Xi_\calZ) - r\left( u \right)   \cos 2\phi\, \dbar (\Xi_\calZ( z ) )\right] e^{2i\theta}
\end{align}
\begin{equation}
\label{R1.bd1}
|W|=\left| \dbar R_1  \right| 	\lesssim\left( |   \dbar( \mathbf{r}\left( z\right)) | +\dfrac{| \mathbf{r}\left( z\right)| }{  |z| } + \dbar (\Xi_\calZ( z) )  \right) e^{2\text{Re}i\theta(z; x, t)}.
\end{equation}
We proceed as in the previous section and study the integral equation related to the $\dbar$ problem. Setting $z=\alpha+i\beta$ and $\zeta=u+iv$, the region $\Omega_1$ corresponds to $u\geq v \geq 0 $. We decompose the integral operator into three parts:
$$
 \int_{\Omega_1}  \dfrac{1}{|z-\zeta|} |W(\zeta)| \, d\zeta  \lesssim  I_1 + I_2 +I_3
$$
where
\begin{align*}
I_1 	&=	\int_0^\infty \int_v^\infty \dfrac{1}{|z-\zeta|} \left\vert  \dbar[\mathbf{r}](\zeta)   \right\vert e^{\text{Re}i\theta(z; x, t)} \, du \, dv, \\[5pt]
I_2	&=	\int_0^\infty \int_v^\infty \frac{1}{|z-\zeta|}  \dfrac {1} {  \left| u  +iv \right|^{1/2} } e^{\text{Re}i\theta(z; x, t)} \, du \, dv,\\[5pt]
I_3 	&=	\int_0^\infty \int_v^\infty \dfrac{1}{|z-\zeta|} \left\vert  \dbar (\Xi_\calZ(\zeta) )   \right\vert e^{\text{Re}i\theta(z; x, t)} \, du \, dv. 
\end{align*}
Following the same proof of \eqref{omega8-}, we can conclude that
$$
 \int_{\Omega_1}  \dfrac{1}{|z-\zeta|} |W(\zeta)| \, d\zeta  \lesssim  t^{1/2-s}.
$$
We now  calculate the decay rate of the integral
$$
 \int_{\Omega_1} \dfrac{ |W(\zeta)|}{|\zeta|} \, d\zeta
.$$
Again we decompose the integral above into three parts
\begin{align*}
I_1 	&=	\int_0^{ \infty } \int_{ v}^{\infty} \dfrac{1}{|\zeta|} |\dbar[\mathbf{r}](\zeta)  | e^{\text{Re}i\theta(z; x, t)} \, du \, dv, \\[5pt]
I_2	&=	\int_0^{ \infty } \int_{ v}^{\infty} \dfrac{1}{|\zeta|^2 }|\mathbf{r}(\zeta) |  e^{\text{Re}i\theta(z; x, t) } \, du \, dv,\\
I_3	&=	\int_0^{ \infty } \int_{ v}^{\infty} \dfrac{1}{|\zeta|} \left|  \dbar (\Xi_\calZ(\zeta  ) ) \right| e^{\text{Re}i\theta(z; x, t) } \, du \, dv.
\end{align*}

For $I_1$ we apply H\"older's inequality. Taking $1<p<2$ and $q$ as its conjugate, one has
\begin{align*}
I_1 	&\lesssim \int_0^{ \infty } \int_{ 1}^{\infty} \dfrac{1}{|\zeta|} |\dbar[\mathbf{r}](\zeta) | e^{-v t} \, du \, dv +  \int_0^{ 1 } \int_{ v}^{1} \dfrac{1}{|\zeta|} |\dbar[\mathbf{r}](\zeta) | e^{-vt/u^2 } \, du \, dv\\
       &\lesssim t^{-s}.
\end{align*}
For $I_2$ we again apply H\"older's  inequality with $1<p<2$:
\begin{align}
\label{est-r-out}
I_2	 &\lesssim \int_0^{1}  \int_{v}^{1} \dfrac{ |\mathbf{r}(\zeta)|}{ |u^2+v^2|  }{e^{-v\tau/u^2}}\, du \, dv + \int_0^{\infty}   e^{-v t}   \int_{ 1}^{\infty} \dfrac{ |\mathbf{r}(\zeta)|}{ |u^2+v^2|  }\, du \, dv\\
\nonumber
   &\lesssim \norm{r}{H^s(\bbR)} \int_0^{1} \dfrac{1}{u^2} \int_{ 0}^{u} e^{-v\tau/u^2} dv   du + t^{-s}\\
   \nonumber
    & \lesssim t^{-s}.
 \end{align}
 The estimate on $I_3$ is similar to that of $I_1$.
Following the same procedure given in Section \ref{sec:large-time} we obtain the following asymptotic formulas:
\begin{align}
\label{outside-cos}
\cos f(x,t) -1 &= \mathcal{O}\left( t^{-2s}\right)\\
\label{outside-sin}
\sin f(x,t)    &=\mathcal{O}\left( t^{-s}\right).
\end{align}

\subsection{$x/t<-1$}
First notice that for $v>0$, 
\begin{equation}
\label{-theta-out}
4\text{Re}i\theta(z; x, t)=-\left(   1+\dfrac{x}{t} \right)v t +  \left( 1-\dfrac{x}{t} \right)\dfrac{v t}{u^2+v^2}  \geq -\left(   1+\dfrac{x}{t} \right)v t >0
\end{equation}
Similarly for $v<0$,
\begin{equation}
\label{-theta-out'}
4\text{Re}\left[-i\theta(z; x, t)  \right] =\left(   1+\dfrac{x}{t} \right)v t +  \left( \dfrac{x}{t}-1 \right)\dfrac{v t}{u^2+v^2}  \geq \left(   1+\dfrac{x}{t} \right)v t >0.
\end{equation}
 Define the scalar function:
\begin{equation}
\label{psi-0}
\psi(z) =\left( \prod_{k=1}^{N_1} \dfrac{z- \overline {i\zeta_k }}{ z-i\zeta_k}\right) \left(\prod_{j=1}^{N_2}\dfrac{z-\overline{z_j}}{z-z_j}\right) \left(\prod_{j=1}^{N_2} \dfrac{z+z_j}{z+\overline{z_j}}\right)\exp \left(  \dfrac{1}{2\pi i}\int_{-\infty}^{\infty}\dfrac{\log( 1+|r(s)|^2)}{s-z} {d\zeta}  \right) .
\end{equation}
It is straightforward to check that if $M(z;x,t)$ solves Problem \ref{RHP-1}, then the new matrix-valued function $m^{(1)}(z;x,t)=M(z;x,t)\psi(z)^{\sigma_3}$ has the following jump matrices:
\begin{equation}
e^{-i\theta\ad\sigma_3}v^{(1)}(z) 	=\Twomat{1}{0}{\dfrac{\psi_-^{-2}  r}{1+|r|^2}  e^{2i\theta}}{1}
					\Twomat{1}{\dfrac{\psi_+^2 \rbar}{1+|r|^2} e^{-2i\theta}}{0}{1},
						\quad z \in\bbR 
\end{equation}
\begin{align}
\label{v-soliton-p}
e^{-i\theta\ad\sigma_3}v^{(1)}(z) = 	\begin{cases}
						\twomat{1}{\dfrac{\left[ (1/\psi)'(i\zeta_k)\right]^{-2} }{\iota_k  (z-i\zeta_k)}e^{-2i\theta} }{0}{1}	&	z\in \tgamma_k, \\
						\\
						\twomat{1}{0}{\dfrac{\left[ \psi'( \overline{i\zeta_k } )\right]^{-2}}{\overline{\iota_k}  (z -\overline{i\zeta_k })} e^{2i\theta}}{1}
							&	z \in \tgamma_k^*,
					\end{cases}
\end{align}

\begin{align}
\label{v-br+}
e^{-i\theta\ad\sigma_3}v^{(1)}(z)= 	\begin{cases}
						\twomat{1}{\dfrac{\left[ (1/\psi )'(z_j)\right]^{-2} }{c_j  (z-z_j)}e^{-2i\theta} }{0}{1}	&	z\in \gamma_j, \\
						\\
						\twomat{1}{0}{\dfrac{ \left[\psi'( \overline{z_j } )\right]^{-2}}{\overline{c_j}  (z -\overline{z_j })} e^{2i\theta}}{1}
							&	z \in \gamma_j^*,\\
							\\
							\twomat{1}{0}{-\dfrac{ \left[\psi'( -{z_j } )\right]^{-2}}{ c_j  (z + z_j ) } e^{2i\theta}}{1}	&	z\in -\gamma_j, \\
						\\
						\twomat{1} {-\dfrac{ \left[ (1/\psi )'(  -\overline{z_j } )\right]^{-2} }{ \overline{c_j}  (z +\overline{z_j } ) }  e^{-2i\theta} }{0}{1}
							&	z \in -\gamma_j^*.
					\end{cases}
\end{align}

We now see that all entries in \eqref{v-soliton-p}-\eqref{v-br+} decay exponentially as $t\to\infty$, so we are allowed to reduce the RHP to a problem on $\bbR$ and perform the same contour deformation of the case when $x/t>1$ in Figure \ref{fig:Painleve}. Again for brevity, we only discuss the situation in $\Omega_1$. In $\Omega_1$, we define
\begin{align*}
	R_1	&=	\begin{cases}
						\twomat{0}{ \dfrac{\psi_+^2 \rbar}{1+|r|^2} e^{-2i\theta} }{ 0 }{0}		
								&	\zeta \in (0,\infty)\\[10pt]
								\\
						\twomat{0}{ \dfrac{\psi_+^2(0) \rbar(0)}{1+|r(0)|^2} e^{-2i\theta}\left(1-\Xi_\calZ\right) }	{0}{0}
								&	\zeta	\in \Sigma_1
					\end{cases}
	\end{align*}
and the interpolation is given by 
$$  \dfrac{ \rbar(u)}{1+|r(u)|^2}  (  \cos 2\phi)  \psi^2(z)    (1-\Xi_\calZ)  e^{-2i\theta} $$
and consequently
\begin{align}
\dbar R_1&= \left[ \left(  r'\left( u\right) \cos 2\phi- 2\dfrac{ r(u  )    }{  \left\vert z \right\vert   } e^{i\phi} \sin 2\phi  \right)  (1-\Xi_\calZ) - r\left( u \right)   \cos 2\phi \dbar (\Xi_\calZ( z ) )\right] e^{-2i\theta}
\end{align}
\begin{equation}
\label{R1.bd1}
|W|=\left| \dbar R_1  \right| 	\lesssim\left(  \left\vert  \left(  \dfrac{ \rbar(u)}{1+|r(u)|^2} \right)'\right\vert +\dfrac{|\rbar(u)| }{  |z|(   1+|r(u)|^2) } + \dbar (\Xi_\calZ( z) )  \right) e^{- 2\text{Re}i\theta(z; x, t)}.
\end{equation}
The asymptotic formulas are the same as \eqref{outside-cos}-\eqref{outside-sin}.

\section{Soliton Resolution and Full Asymptotic Stability}\label{sec:summary}
In this section, we summarize the computations in previous sections and conclude the soliton resolution for the sine-Gordon equation with  generic data. Then we use the long-time asymptotics to obtain the full asymptotic stability of reflectionless nonlinear structures.

\subsection{Soliton resolution}
\begin{theorem}
	\label{thm:maindetail}Given the generic initial data $\vec{f}\left(0\right)=\left(f_{0},f_{1}\right)\in H_{\text{sin}}^{2,s}\left(\mathbb{R}\right)\times H^{1,s}\left(\mathbb{R}\right)$ in the sense of Definition \ref{genericity} with $s>\frac{1}{2}$ and let  $f$ be the solution to the sine-Gordon equation \eqref{eq: sG}, then it can be written as the superposition of breathers, kinks,
	anti-kinks and the radiation.
	\[
	\sin\left(f(x,t)\right)=\sum_{\ell=1}^{N_{1}}\sin\left(Q_{\ell }(x,t)\right)+\sum_{\ell=1}^{N_{2}}\sin\left(B_{ \ell }\left(x,t\right)\right)+f_{r, sin}(x,t)
	\]
	and
	\[
	1-\cos\left(f(x,t)\right)=\sum_{\ell=1}^{N_{1}}\left(1-\cos\left(Q_{\ell}(x,t)\right)\right)+\sum_{\ell=1}^{N_{2}}\left(1-\cos\left(B_{\ell}(x,t\right)\right)+f_{r, cos}(x,t)
	\]
	where $1-\cos\left(B_{\ell}(x,t)\right)$, $\sin\left(B_{\ell}(x,t)\right)$
	and $1-\cos\left(Q_{\ell}(x,t)\right)$, $\sin\left(Q_{\ell}\left(x,t\right)\right)$
	are given by \eqref{br-cos}, \eqref{br-sin}, \eqref{kink-cos} and \eqref{kink-sin} respectively. As for the radiation terms, $f_{r, sin}(x,t)$ and $f_{r, cos}(x,t)$, by setting $x=\mathtt{v}t$,
	we obtain that
	\begin{itemize}
		\item[1.] For $\mathtt{v}=v_{\ell}^{K}$,  the velocity of the $\ell $th kink, $f_{r, cos}(x,t)$ and $f_{r, sin}(x,t)$ are given by \eqref{cos-asym k} and \eqref{sin-asym k}.
		
		\item[2.] For  $\mathtt{v}=v_{\ell}^{B}$, the velocity of the $\ell $th breather,  $f_{r, cos}(x,t)$ and $f_{r, sin}(x,t)$ are given by \eqref{cos-asym} and \eqref{sin-asym}.
		
		\item[3.]  For $\mathtt{v}\neq v_{\ell}^{B},v_{\ell}^{K}$ for all $\ell$, $f_{r, cos}(x,t)$ and $f_{r, sin}(x,t)$ are given by \eqref{as-solitonless} and \eqref{as-solitonless'}.
		
		\item[4.] For  $\left|\mathtt{v}\right|>1$,  $f_{r, cos}(x,t)$ and $f_{r, sin}(x,t)$ are given by 
		\eqref{outside-cos} and \eqref{outside-sin}.
		\item [5.] For $ x/t \to 1^-$ as $t \to \infty $  $f_{r, cos}(x,t)$ and $f_{r, sin}(x,t)$ are given by  \eqref{app-cos} and \eqref{app-sin}.
		\item[6.] For $ x/t \to -1^+$ as $t \to \infty $  $f_{r, cos}(x,t)$ and $f_{r, sin}(x,t)$ are given by  \eqref{app-cos-} and \eqref{app-sin-}.
	\end{itemize}
\end{theorem}

\begin{remark}
	Note that as the velocity of the reference frame $\mathtt{v}$ moving away from
	the  velocity of the breather, kink and anti-kinks, $m_{12}^{\left(br\right)}\left(z_{0}\right)$,
	$m_{12}^{(kin)}(z_{0})$ will exponentially decay in time. \footnote{
	Of
	course, these exponential decay rates depend on the gap given by the
	velocity of the frame and the velocities of breathers, kinks and anti-kinks.} Combining these exponential decay terms with 
	the remaining terms  will give us the same asymptotics 
	as the later expression in (3) above.
\end{remark}

\subsection{Full asymptotic stability}\label{subsec:fullasym}
To study the asymptotic stability, we first construct the general  reflectionless solution. Suppose we have the following
discrete scattering data
\begin{equation}
\mathcal{S}_{D}=\left\{ 0,\left\{ i\zeta_{0,k},\iota_{0,k}\right\} _{k=1}^{N_{1}},\left\{ z_{0,j},c_{0,j}\right\} _{j=1}^{N_{2}}\right\} \in H_0^{s}(\bbR)\otimes\mathbb{C}^{2N_{1}}\otimes\mathbb{C}^{2N_{2}}.\label{eq:discrete}
\end{equation}
Assuming that $z_{0,j}:=\rho_{0,j}e^{i\omega_{0,j}}$, we denote
\[
v_{0,j}^{B}=\frac{1-\rho_{0,j}^2}{1+\rho_{0,j}^2},\quad v_{0,k}^{K}=\frac{1-\zeta_{0,k}^2}{1+\zeta_{0,k}^2}.
\]
Moreover, for simplicity, we assume that all $v_{0,j}^{B}$ $v_{0,k}^{K}$
are different. If $N_{1}=N_2=0$ then the corresponding
set will be the empty set.

Then one can construct a reflectionless solution $N\left(x,t\right)$
using the discrete scattering data \eqref{eq:discrete} in Problem \ref{RHP-1}
and Proposition \ref{prop:recon}.

Moreover, we have the following asymptotics for $N\left(x,t\right)$
as $t\rightarrow\infty$: 
\begin{align}\label{eq:N1}
\sin\left(N(x,t)\right)=\sum_{j=1}^{N_{1}}\sin\left(Q_{j}(x,t)\right)+\sum_{j=1}^{N_{2}}\sin\left(B_{j}\left(x,t\right)\right)+\mathcal{O}\left(e^{-\epsilon t}\right)
\end{align}
and
\begin{align}\label{eq:N2}
1-\cos\left(N(x,t)\right)=\sum_{j=1}^{N_{1}}\left(1-\cos\left(Q_{j}(x,t)\right)\right)+\sum_{j=1}^{N_{2}}\left(1-\cos\left(B_{j}(x,t\right)\right)+\mathcal{O}\left(e^{-\epsilon t}\right)
\end{align}
where $\sin\left(B_{\ell}(x,t)\right)$, $1-\cos\left(B_{\ell}(x,t)\right)$
and $\sin\left(Q_{\ell}(x,t)\right)$, $1-\cos\left(Q_{\ell}\left(x,t\right)\right)$
are reconstructed via  Problem \ref{prob:mkdv.br} and Problem \ref{prob:mkdv.kin} using the scattering
data \eqref{eq:discrete} respectively. Therefore, by the continuity
of the scattering data, we can also consider the stability of the
sum
\[
\sum_{j=1}^{N_{1}}\sin\left(Q_{j}(x,t)\right)+\sum_{j=1}^{N_{2}}\sin\left(B_{j}\left(x,t\right)\right)
\]
and
\[
\sum_{j=1}^{N_{1}}\left(1-\cos\left(Q_{j}(x,t)\right)\right)+\sum_{j=1}^{N_{2}}\left(1-\cos\left(B_{j}(x,t\right)\right)
\]
provided the kinks, antikinks and breathers in these sums are sufficiently separated. Notice that for $N_{2}=0$ and $N_{1}=1$, $N\left(x,t\right)$ is
simply a kink/anti-kink and for $N_{2}=1$ and $N_{1}=0$, $N\left(x,t\right)$
is a breather. 

\smallskip
With preparations above and Theorem \ref{thm:maindetail}, we state a corollary regarding the asymptotic stability of $N\left(x,t\right).$
\begin{corollary}
	\label{cor:stabN}Consider the reflectionless solution $N\left(x,t\right)$
	to the sine-Gordon equation \eqref{eq: sG}. For $s>\frac{1}{2}$, suppose
	\[
	\left\Vert \vec{f}(0)-\left(N(0),\partial_{t}N(0)\right)\right\Vert _{H^{2,s}\times H^{1,s}}<\epsilon
	\]
	for $0< \epsilon\ll1$ small enough. Let $\vec{f}$ be the solution
	to the sine-Gordon equation with the initial data $\vec{f}(0)$, then
	there exist scattering data
	\begin{equation}
	\mathcal{S}=\left\{ r\left(z\right),\left\{ i\zeta_{1,k}, \iota_{1,k}\right\} _{k=1}^{N_{1}},\left\{ z_{1,j},c_{1,j}\right\} _{j=1}^{N_{2}}\right\} \in H^{s}_0(\bbR) \otimes\mathbb{C}^{2N_{1}}\otimes\mathbb{C}^{2N_{2}}\label{eq:newscatt}
	\end{equation}
	computed in terms of $\vec{f}(0)$ such that
	\[
	\left\Vert r\right\Vert _{H^{1,1}}+\sum_{\ell=1}^{N_{2}}\left(\left|i\zeta_{0,\ell}-i\zeta_{1,\ell}\right|+\left|\iota_{0,\ell}-\iota_{1,\ell}\right|\right)+\sum_{j=1}^{N_{1}}\left(\left|z_{0,j}-z_{1,j}\right|+\left|c_{0,j}-c_{1,j}\right|\right)\lesssim\epsilon.
	\]
	Moreover, with the scattering data $\mathcal{S}$ given by \eqref{eq:newscatt}, one can write the solution
	$f$ 
	\[
	\sin\left(f(x,t)\right)=\sum_{j=1}^{N_{1}}\sin\left(\tilde{Q}_{j}(x,t)\right)+\sum_{j=1}^{N_{2}}\sin\left(\tilde{B}_{j}\left(x,t\right)\right)+f_{r,sin}(x,t)
	\]
	and
	\[
	1-\cos\left(f(x,t)\right)=\sum_{j=1}^{N_{1}}\left(1-\cos\left(\tilde{Q}_{j}(x,t)\right)\right)+\sum_{j=1}^{N_{2}}\left(1-\cos\left(\tilde{B}_{j}(x,t\right)\right)+f_{r,cos}(x,t)
	\]
	where $\sin\left(\tilde{B}_{\ell}(x,t)\right)$, $1-\cos\left(\tilde{B}_{\ell}(x,t)\right)$
	and $\sin\left(\tilde{Q}_{\ell}(x,t)\right)$, $1-\cos\left(\tilde{Q}_{\ell}\left(x,t\right)\right)$
	are reconstructed via solving Problem \ref{prob:mkdv.br} and Problem \ref{prob:mkdv.kin} using scattering
	data \eqref{eq:newscatt} respectively. (If $N_{\ell}$, $\ell=1,2$
	is $0$ then the corresponding sum will be zero.) Here again the radiation
	terms $f_{r,sin}\left(x,t\right)$ and $f_{r,cos}\left(x,t\right)$ have the
	asymptotics given by Theorem \ref{thm:maindetail} with scattering data \eqref{eq:newscatt}.
\end{corollary}
\begin{remark}
	We can also consider the nonlinear structure such that the underlining  eigenvalues 
	can produce elements having the same speed. Then the perturbation should be generic in
	the sense of Definition \ref{genericity} so that we can apply our steepest descent
	computations.
\end{remark}

\section{Asymptotic Stability/Instability  in Weighted Energy Spaces}\label{sec:stability}
In this section, we study the asymptotic stability and instability of multi-soliton solutions to the sine-Gordon equation in the weighted energy space $H^{1,s}(\mathbb{R})\times L^{2,s}(\mathbb{R})$.

\subsection{Well-posedness}








Recall that the conserved energy for the sine-Gordon equation is
\begin{align*}
E_{\text{sin}}\left(t\right) & =\frac{1}{2}\int_{\mathbb{R}}\left|\partial_{t}f\right|^{2}+\left|\partial_{x}f\right|^{2}\,dx+\int\left(1-\cos f\right)\,dx\\
& =\frac{1}{2}\int_{\mathbb{R}}\left|\partial_{t}f\right|^{2}+\left|\partial_{x}f\right|^{2}\,dx+2\int\sin^{2}\left(\frac{f}{2}\right)\,dx.
\end{align*}
Therefore, the natural energy space for the sine-Gordon equation is
\[
H_{\text{sin}}^{1}\left(\mathbb{R}\right):=\left\{ f_x \in L^{2}\left(\mathbb{R}\right):\ \sin\left(f/2\right)\in L^{2}(\bbR)\right\} ,
\]
and the pseudometric distance function associated to this space is
\[
d_{\text{sin}}\left(f_{1},f_{2}\right):=\left(\left\Vert \sin\left(\frac{f_{1}-f_{2}}{2}\right)\right\Vert _{L^{2}}^{2}+\left\Vert \partial_{x}\left(f_{1}-f_{2}\right)\right\Vert _{L^{2}}^{2}\right)^{\frac{1}{2}}.
\]
First of all, when the initial data is in the weighted energy space,
one could not directly prove the solution obtained from solving the RHP is the solution to the sine-Gordon equation. We overcome
this by using the notation of the strong solution.
\begin{definition}
	\label{def:strong}We say that the function $f\left(x,t\right)$ is
	a \emph{strong solution} in $H_{\sin}^{k}\left(\mathbb{R}\right)$
	to the sine-Gordon equation
	\begin{equation}
	\partial_{tt}f-\partial_{xx}f+\sin\left(f\right)=0,\label{eq:IVP}
	\end{equation}
	with
	\[
	f\left(0\right)=f_{0}\in H_{\sin}^{k}\left(\mathbb{R}\right),\ \partial_{t}f\left(0\right)=f_{1}\in H^{k-1}\left(\mathbb{R}\right)
	\]
	if and only if $\left(f,f_{t}\right)\in C\big(\mathbb{R};H_{\text{sin}}^{k}\left(\mathbb{R}\right)\times  H^{k-1}\left(\mathbb{R}\right)\big)$
	satisfies
	\begin{align}
	f & =\cos\left(t\sqrt{-\Delta}\right)f_{0}+\frac{\sin\left(t\sqrt{-\Delta}\right)}{\sqrt{-\Delta}}f_{1}\label{eq:mild}-\int_{0}^{t}\frac{\sin\left(\left(t-s\right)\sqrt{-\Delta}\right)}{\sqrt{-\Delta}}\left(\sin\left(f\left(s\right)\right)\right)\,ds.
	\end{align}
\end{definition}
Next, we record the well-posedness of the sine-Gordon equation in the energy space. This can be established using the standard energy estimate and the Picard iteration. For the  detailed proof, see de Laire-Gravejat \cite{deLG}. 
\begin{theorem}\label{thm:GWP}
	Given $\left(f_{0},f_{1}\right)\in H_{\text{sin}}^{1}\times L^{2}$,
	there exists a unique solution $f\in C\left(\mathbb{R};H_{\text{sin}}^{1}\left(\mathbb{R}\right)\right)$
	with $\partial_{t} f \in C\left(\mathbb{R};L^{2}\left(\mathbb{R}\right)\right)$ to the sine-Gordon
	equation with initial data $\left(f_{0},f_{1}\right)$. Moreover the
	solution satisfies the following statements
	
	(1) For any positive number $T$, there exists a positive number $A$
	depending on $T$ such that the flow map $\left(f_{0},f_{1}\right)\rightarrow\left(f,  f_t \right)$
	satisfies
	\[
	d_{\text{sin}}\left( f \left(\cdot,t\right),\tilde{ f }\left(\cdot,t\right)\right)+\left\Vert f_t\left(\cdot,t\right)-\tilde{f}_t\left(\cdot,t\right)\right\Vert _{L^{2}}\leq A\left(d_{\text{sin}}\left(f_{0},\tilde{f }_{0}\right)+\left\Vert f_{1}-\tilde{f}_{1}\right\Vert _{L^{2}}\right),
	\]
	for any $t\in\left[-T,T\right]$. Here, the function $\tilde{f}$
	is the unique solution to the sine-Gordon equation with initial conditions
	$\left(\tilde{f}_{0},\tilde{f}_{1}\right)$.
	
	(2) The sine-Gordon energy $E_{\text{sin }}$ is conserved along the
	flow.
\end{theorem}

\subsection{Asymptotic stability}\label{subsec:wasymp}
Considering the sine-Gordon equation in $H^{1,s}(\mathbb{R})\times L^{2,s}(\mathbb{R})$ with $s>\frac{1}{2}$, our goal of this subsection is to establish the asymptotic stability of the nonlinear structure
given by \eqref{eq:N1} and \eqref{eq:N2} in subsection \ref{subsec:fullasym} measured in the localized energy norm.

Recall that for given $\mathfrak{v}\in\mathbb{R}$ and $\mathfrak{L}>0$ fixed,
we define the localized energy norm for a vector $\vec{f}=\left(f_{1},f_{2}\right)$
as
\[
\left\Vert \vec{f}\right\Vert _{\mathcal{E},\mathfrak{v} ,\mathfrak{L}}^{2}:=\left\Vert f_{1}\right\Vert _{L^{2}\left(\left|x-\mathfrak{v}t\right|<\mathfrak{L}\right)}^{2}+\left\Vert \partial_{x}f_{1}\right\Vert _{L^{2}\left(\left|x-\mathfrak{v}t\right|<\mathfrak{L}\right)}^{2}+\left\Vert f_{2}\right\Vert _{L^{2}\left(\left|x-\mathfrak{v}t\right|<\mathfrak{L}\right)}^{2}.
\]
Our main result in this subsection is the following:
\begin{theorem}
	\label{thm:stabNL}Consider the reflectionless solution $N\left(x,t\right)$ constructed using \eqref{eq:N1} and \eqref{eq:N2}
	to the sine-Gordon equation \eqref{eq: sG}. Suppose	\[
	\left\Vert \vec{f}(0)-\left(N(0),\partial_{t}N(0)\right)\right\Vert _{H^{1,s}\times L^{2,s}}<\epsilon
	\]
	for $0\leq\epsilon\ll1$ small enough and $s>\frac{1}{2}$. Let $\vec{f}$
	be the solution to the sine-Gordon equation with the initial data
	$\vec{f}(0)$ then there exist scattering data
	\begin{equation}
	\mathcal{S}=\left\{ r\left(z\right),\left\{ i\zeta_{1,k}, \iota_{1,k}\right\} _{k=1}^{N_{1}},\left\{ z_{1,j},c_{1,j}\right\} _{j=1}^{N_{2}}\right\} \in \left( L^2 \left(\mathbb{R}\right)\cap L^\infty(\bbR) \right)\otimes\mathbb{C}^{2N_{1}}\otimes\mathbb{C}^{2N_{2}}\label{eq:newscatt1}
	\end{equation}
	computed in terms of $\vec{f}(0)$ with $r(0)=0$ such that
	\[
	\left\Vert r\right\Vert _{ \left( L^2 \left(\mathbb{R}\right)\cap L^\infty(\mathbb{R}) \right)\  }+\sum_{k=1}^{N_{1}}\left(\left| \zeta_{0,k}-\zeta_{1, k}\right|+\left| \iota_{0, k}-\iota_{1,k}\right|\right)+\sum_{j=1}^{N_{2}}\left(\left|z_{0,j}-z_{1,j}\right|+\left|c_{0,j}-c_{1,j}\right|\right)\lesssim\epsilon.
	\]
	Moreover, with the scattering data $S$, one can write the solution
	$f$ as
	\[
	f=\sum_{\ell=1}^{N_{1}}\tilde{Q}_{\ell}\left(x,t\right)+\sum_{\ell=1}^{N_{2}}\tilde{B}_{\ell}\left(x,t\right)+f_r\left(x,t\right)
	\]
where $\tilde{B}_{\ell}\left(x,t\right)$ and $\tilde{Q}_{\ell}\left(x,t\right)$
	are reconstructed via Problem \ref{prob:mkdv.br} and Problem \ref{prob:mkdv.kin} using scattering
	data \eqref{eq:newscatt1} respectively, and the radiation term satisfies
	\[
	\lim_{t\rightarrow\infty}\left\Vert f_r(t)\right\Vert _{\mathcal{E},\mathfrak{v},\mathfrak{L}}=0
	\]
	any fixed  $\mathfrak{v}\in\mathbb{R} $ and  $\mathfrak{L}>0$.
\end{theorem}

To achieve our goal, we begin with the following lemma:

\begin{lemma}
	\label{lem:BCD}Given initial data $\vec{f}(0)\in H_{\text{sin}}^{1,s}\left(\mathbb{R}\right)\times L^{2,s}\left(\mathbb{R}\right)$
	with $s>\frac{1}{2}$ which is generic in the sense of Definition \ref{genericity},
	it produces scattering data
	\begin{equation}
	\mathcal{S}=\left\{ r\left(z\right),\left\{ i\zeta_{k},\iota_{k}\right\} _{k=1}^{N_{1}},\left\{ z_{j},c_{j}\right\} _{j=1}^{N_{2}}\right\} \in \left( L^2 \left(\mathbb{R}\right)\cap L^\infty(\bbR) \right)\otimes\mathbb{C}^{2N_{1}}\otimes\mathbb{C}^{2N_{2}}\label{eq:scattlow}
	\end{equation}
with $r(0)=0$.	(If $N_{\ell}$, $\ell=1,2$ is $0$ then the corresponding set will
	be the empty set.) Then this map is locally Liphschitz and
 the function given by the reconstruction formula \eqref{sin}-\eqref{cos}  using the
	scattering data \eqref{eq:scattlow} is a strong solution to the sine-Gordon
	equation in the sense of Definition \ref{def:strong}.
\end{lemma}

	\begin{proof}
		The Lipschitz continuity of the map $\vec{f}(0)\mapsto \mathcal{S}$ follows from the proof of Proposition \ref{prop:r} directly. Secondly, given the initial data $\vec{f}(0)\in H_{\text{sin}}^{1,s}\left(\mathbb{R}\right)\times L^{2,s}\left(\mathbb{R}\right)$
for $s>\frac{1}{2}$, the scattering data it produces can still be used to set up the Riemann-Hilbert problem and give a Beals-Coifman
solution $\mu$. With this $\mu$, one can still use the reconstruction
formula to define a space-time function. Since we do not have enough
smoothness, one could not conclude this function is a solution to
the sine-Gordon equation using the inverse scattering formalism and
compatibility conditions. But the weighted energy space is a subset
of the standard energy space, so we can apply the standard PDE Picard iterations
to construct the solution (see Theorem \ref{thm:GWP}). By the { uniqueness} of classical solutions and
the standard approximation argument, we can conclude that the space-time
function above given by the limiting reconstruction formula is indeed a strong solution
to the sine-Gordon equation in the sense of the Duhamel formalism given by Definition \ref{def:strong}.
	\end{proof}
Then we record two easy but important identities. We can write $f=N+f_r$ where $N$ corresponds to the sum of breathers,
kinks and anti-kinks and $f_r$ denotes the radiation. Elementary trigonometric identities give
\begin{align}\label{eq:rsine}
\sin(f_r) & =\sin\left(N+f_r\right)\cos\left(N\right)-\cos\left(N+f_r\right)\sin\left(N\right)\\
& =\cos\left(N\right)\left(\sin\left(N+f_r\right)-\sin\left(N\right)\right)-\sin\left(N\right)\left(\cos\left(N+f_r\right)-\cos\left(N\right)\right)\nonumber
\end{align}
and
\begin{align}\label{eq:rcosine}
1-\cos\left(f_r\right) & =1-\cos\left(f_r+N\right)\cos\left(N\right)-\sin\left(f_r+N\right)\sin\left(N\right)\\
& =-\left[\left(\cos\left(f_r+N\right)-\cos\left(N\right)\right)\cos\left(N\right)\nonumber+\left(\sin\left(N+f_r\right)-\sin\left(N\right)\right)\sin\left(N\right)\right].\nonumber
\end{align}
By Corollary \ref{cor:stabN}, we know the asymptotics of
$\sin\left(N+f_r\right)-\sin\left(N\right)$ and $\cos\left(N+f_r\right)-\cos\left(N\right)$.
Then from two identities above,  we can deduce the asymptotics of $\sin\left(f_r\right)$ and
$1-\cos\left(f_r\right)$.

To prove Theorem \ref{thm:stabNL}, in order to keep the proof short and illustrate ideas, we first prove the following two simple cases of the asymptotic stability. The first one is the stability of zero solution. Before we start the proof, we make the following comments:

\begin{itemize}
	\item[1.]  Note that from Problem \ref{RHP-1}, we have the following reconstruction formula from the solution of this RHP by taking $z\to \infty$:
	\begin{equation}
	\label{recon:fx+ft}
	f_x(x, t)+f_t(x, t)=\dfrac{1}{\pi} \int_\bbR \mu_{11} \overline{r(z)} e^{-2i\theta} dz
	\end{equation}
	and similarly 
	\begin{align}
	\label{recon:fx+ft n}
	f_{n, x}(x, t)+f_{n,t}(x, t) &=\dfrac{1}{\pi} \int_\bbR \mu_{n, 11} \overline{r_n(z)} e^{-2i\theta} dz.
	\end{align}
	There are similar reconstruction formulas for $f_{n, x}(x, t)-f_{n,t}$.  Indeed we can work with the spectral problem \eqref{Phi-x} and establish a parallel version of Proposition \ref{prop:r} and  RHP Problem \ref{RHP-1}. Then we can obtain the expressions for $f_x(x, t), f_t(x, t)$ separately. We omit this step for brevity and will only work with the difference between \eqref{recon:fx+ft}-\eqref{recon:fx+ft n}.
	\item[2.] We do not have a direct reconstruction of $f$ from the RHP \eqref{RHP-1}. To deal with the $L^2$ norm of $f$ in the definition of local energy \eqref{eq:localenergy}, we can use the fact that $1-\cos x \sim x^2/2$ for $x$ which is close to zero and the asymptotic behavior of $1-\cos f_n$ given by \eqref{as-solitonless'}. Again for brevity we omit this $L^2$ norm.
\end{itemize}

\begin{theorem}
	\label{thm:zero}
	Suppose $f\left(x, 0\right)=f_{0}\left(x\right)$ and $f_{t}\left(x, 0\right)=f_{1}\left(x\right)$
	with $\left(f_{0},f_{1}\right)\in H^{1,s}\times L^{2,s}$ with $s>\frac{1}{2}$
	such that $\left\Vert \left(f_{0},f_{1}\right)\right\Vert _{H^{1,s}\times L^{2,s}}$ is small enough.
	Then for any space-time interval $I_{t}=\{x\in\mathbb{R}|\,|x- \mathfrak{v} t|<\mathfrak{L}\}$ with any given fixed $\mathfrak{L}>0$ and any $\mathfrak{v}\in\mathbb{R}$, one has
	\[
	\lim_{t\rightarrow\infty}\int_{I_t}\left| f_x\left(\cdot, t\right)\right|^{2} + \left| f \left(\cdot, t\right)\right|^{2}+\left| f_t \left(\cdot, t\right)\right|^{2}\,dx\rightarrow0.
	\]
\end{theorem}

\begin{proof}
	To establish this, we take a sequence of
	\[
	\left(f_{0, n}\left( x\right),f_{1, n}\left(x\right)\right)\in  {H}^{2,1}\times {H}^{1,1}
	\]
	such that
	\[
	\left(f_{0, n}\left( x\right), f_{1, n} \left(x\right)\right) \rightarrow \left(f_{0}\left( x\right), f_{1}\left(x\right)\right) 
	\]
	in $H^{1,s}\times L^{2,s}$. By the smallness assumption,
	the direct scattering map from $\left(f_{0}\left( x\right), f_{1}\left(x\right)\right) $
	and $\left(f_{0, n}\left( x\right), f_{1, n} \left(x\right)\right) $ to scattering
	data will not produce discrete data. Moreover, by the mapping properties of the direct scattering from the initial data to the scattering data, the reflection coefficients
	$r_{n}$ and $r$ generated by the approximation sequence and the
	initial data satisfy $r_{n} \in H^{1,1}_0$  and $r \in L^2$ with
	$\lim_{n\rightarrow\infty}r_{n}=r$ in $L^2$. Moreover, we also
	know $r_{n}$ satisfies the regularity assumption in Proposition \ref{prop:r} from which long-time asymptotics can be computed.
	We now consider the localized energy for the original solution.
	\[
	\left\Vert \left(f \left(x, t\right), f_t \left(x, t\right)\right)\right\Vert _{H^{1}\left(I_t\right)\times L^{2}\left(I_t\right)}.
	\]
	For any $\epsilon>0$, we pick $n$ large, say, $n\geq\mathcal{N}_{0}$,
	such that
	\[
	\left\Vert ( f_0-f_{0, n}), (f_1-f_{1,n}) \right\Vert _{H^{1,s}\times L^{2,s}}\leq \epsilon.
	\]
	By the forward scattering map, see Lemma \ref{lem:BCD}, we know that for $n\geq\mathcal{N}_{0}$
	\[
	\left\Vert r_{n}-r\right\Vert _{L^{2}\cap L^\infty }\lesssim \epsilon.
	\]	
We now want to show that for any $t\in\mathbb{R}^{+}$, one has
	\begin{align*}
	\left\Vert \left( f\left(\cdot, t\right)-f_{n}\left(\cdot, t\right),f_{t}\left(t\right)-f_{n, t}\left(t\right)\right)\right\Vert _{H^{1}\times L^{2}}
	\lesssim\left\Vert r_{n}-r\right\Vert _{L^2\cap L^\infty}\lesssim \epsilon.
	\end{align*}
	Then we can compute the difference between \eqref{recon:fx+ft}-\eqref{recon:fx+ft n}:
	\begin{align}
	\label{differ}
	D_n &=\dfrac{1}{\pi}\int_\bbR \mu_{n, 11}\left(  \overline{r_n(z)}  -\overline{r(z)} \right) e^{-2i\theta} dz+ \dfrac{1}{\pi} \int_\bbR \left( \mu_{n, 11} - \mu_{11} \right) \overline{r(z)} e^{-2i\theta} dz\\
	&=: D_{n, 1} +D_{n, 2}	\nonumber
	\end{align}where
	\begin{align*}
	D_{n, 1} =\dfrac{1}{\pi}\int_\bbR( \mu_{n, 11}-1)\left(  \overline{r_n(z)}  -\overline{r(z)} \right) e^{-2i\theta} dz+\dfrac{1}{\pi}\int_\bbR \left(  \overline{r_n(z)}  -\overline{r(z)} \right) e^{-2i\theta} dz.
	\end{align*}
Consider $D_{n,1}$ above, the second term is in $L^2_x(\bbR)$ by the standard Fourier theory. For the first term,  \cite[(2.6)]{Zhou98} and the Cauchy-Schwarz inequality gives:
	\begin{align*}
	\dfrac{1}{\pi}\int_\bbR( \mu_{n, 11}-1)\left(  \overline{r_n(z)}  -\overline{r(z)} \right) e^{-2i\theta} dz &\leq \norm{\mu_{n, 11}-1}{L^2_z} \norm{ \overline{r_n(z)}  -\overline{r(z)}}{L^2_z \cap L^\infty }\\
	&\lesssim (1+x^2)^{-1}\norm{ \overline{r_n(z)}  -\overline{r(z)}}{L^2_z\cap L^\infty} .
	\end{align*}
	For $D_{n, 2}$, 
	we only have to show that $\mu_{n, 11} - \mu_{11} \in L^2_z(\bbR)$ and the rest will follow from the Fourier theory. To see this, we make the following observation:
	\begin{align*}
	\mu-I&=(1-C_w)^{-1}(C_w I)\\
	\mu_{n}-I &= (1-C_{w_n})^{-1}(C_{w_n} I)
	\end{align*}
	where operator $C_w$ is defined in \eqref{BC-int}. We first note that by \cite[Lemma 5.2]{DP},  the operator $(1-C_w)^{-1}$ indeed exists and has norm
	\begin{equation}
	\label{norm-dp}
	\norm{(1-C_w)^{-1}}{L^2}\lesssim 1+\norm{r}{L^\infty}^2.
	\end{equation}
	Then using the second resolvent identity, we can write
	\begin{equation}
	\label{resol-w}
	(1-C_{w_n} )^{-1}= (1-C_{w})^{-1}\left[ 1- (C_{w_n} - C_w ) (1-C_{w})^{-1}\right] .
	\end{equation}
	The uniform boundedness of this resolvent operator will follow from
	$$\norm{C_{w_n} - C_w }{L^2}\leq \norm{r_n-r}{L^2\cap L^\infty}.$$
	and the norm is independent of $x, t$ since $i\theta$ is purely imaginary. Now we calculate the difference:
	\begin{align}
	\label{differ-mu}
	D_{\mu, n}&=\mu-\mu_{n}\\
	\nonumber
	&=(1-C_w)^{-1}(C_w I)-(1-C_{w_n})^{-1}(C_{w_n} I)\\
	\nonumber
	&=\left[ (1-C_w)^{-1}- (1-C_{w_n})^{-1} \right](C_w I)\\
	\nonumber
	&\quad - (1-C_{w_n})^{-1} (C_{w_n} I- C_w I )\\
	\nonumber
	&=:D_{\mu, n, 1}-D_{\mu, n, 2}.
	\end{align}
Then it follows that 
$$\norm{\mu-\mu_{n}}{L^2_z}\lesssim \norm{r_n-r}{L^2 \cap L^\infty}.$$
	Here we used uniform resolvent bounds and the second resolvent identity:
	\begin{align*}
	\norm{(1-C_w)^{-1}- (1-C_{w_n})^{-1}}{L^2}\leq \norm{(1-C_w)^{-1}}{L^2} \norm{C_{w_n} - C_w }{L^2} \norm{(1-C_{w, n} )^{-1}}{L^2}.
	\end{align*}
	Thus by taking $n$ large enough, we have
	\begin{align*}
	\left\Vert \left(f \left(x, t\right), f_t \left(x, t\right)\right)\right\Vert _{H^{1}\left(I_t\right)\times L^{2}\left(I_t\right)} & \leq \left\Vert\left(f_{n}\left( x, t\right), f_{n, t} \left(x, t\right)\right)\right\Vert _{H^{1}\left(I_t\right)\times L^{2}\left(I_t\right)} \\
	&\quad +\left\Vert \left( f\left(\cdot, t\right)-f_{n}\left(\cdot, t\right),f_{t}\left(t\right)-f_{n, t}\left(t\right)\right)\right\Vert _{H^{1}\times L^{2}} \\
	& \lesssim \left\Vert\left(f_{n}\left( x, t\right), f_{n, t} \left(x, t\right)\right)\right\Vert _{H^{1}\left(I_t\right)\times L^{2}\left(I_t\right)} +\epsilon.
	\end{align*}
	By our explicit computations for the radiation term, Theorem \ref{thm:maindetail}, we also have
	\[
	\left\Vert\left(f_{n}\left( x, t\right), f_{n, t} \left(x, t\right)\right)\right\Vert _{H^{1}\left(I_t\right)\times L^{2}\left(I_t\right)}\rightarrow0.
	\]
	Taking $n=\mathcal{N}_{0}$ and letting $t_{\mathcal{N}_{0}}$ large
	enough such that
	\[
	\left\Vert \left(f_{\mathcal{N}_{0}}\left(x,  t\right), f_{\mathcal{N}_{0}, t} \left(x, t\right)\right)\right\Vert _{H^{1}\left(I_t\right)\times L^{2}\left(I_t\right)}\leq {\epsilon}
	\]
	for $t\ge t_{\mathcal{N}_{0}}$, we conclude that
	\begin{align*}
	\left\Vert \left(f \left(x, t\right), f_t \left(x, t\right)\right)\right\Vert _{H^{1}\left(I_t\right)\times L^{2}\left(I_t\right)} & \lesssim \left\Vert \left(f_{\mathcal{N}_{0}}\left(x,  t\right), f_{\mathcal{N}_{0}, t} \left(x, t\right)\right)\right\Vert _{H^{1}\left(I_t\right)\times L^{2}\left(I_t\right)} +{\epsilon} \lesssim\epsilon
	\end{align*}which implies the desired result.
\end{proof}
\begin{remark}
In \eqref{differ}, the integration is taken on $\mathbb{R}$. A byproduct of this proof will be the convergence in energy norm $f_n\to f$. By uniqueness, this guarantees that $f$ obtained from solving RHP is the solution to the sine-Gordon equation.
\end{remark}
Then we prove the stability of one-kink solution. To deal with the $L^2$ norm of $f$ in the definition of local energy \eqref{eq:localenergy}, we again use the fact that $1-\cos x \sim x^2/2$ for $x$ close to zero and the expression \eqref{eq:rcosine}. Again for brevity we omit this $L^2$ norm estimate.
\begin{theorem}For any $x_0\in\mathbb{R}$ and $v\in(-1.1)$, given the notation of kinks \eqref{eq:kink2}, suppose $f\left(x, 0 \right)=Q(x, 0; v, x_0)+f_{0}\left(x\right)$ and $f_{t}\left(x, 0\right)=\partial_{t}Q\left(x, 0; v, x_0\right)+f_{1}\left(x\right)$
	with $\left(f_{0},f_{1}\right)\in H^{1,s}\times L^{2,s}$ for $s>\frac{1}{2}$
	and $\left\Vert \left(f_{0},f_{1}\right)\right\Vert _{H^{1,s}\times L^{2,s}}\lesssim \epsilon$ small enough.
	Then there exist $\tilde{v}$ and $\tilde{x}_{0}(t)$ such that  $\left|\tilde{v}-v\right|+ \left|\tilde{x}_0(t) -x_0\right|\lesssim\epsilon$
	and the space-time interval $I_{t}=\{x\in\mathbb{R}|\,|x- \mathfrak{v} t|<\mathfrak{L}\}$ for any given fixed $\mathfrak{L}>0$ and any $\mathfrak{v}\in\mathbb{R}$, one has
	\[
	\lim_{t\rightarrow\infty}\int_{I_t}\left| f -K\left(\tilde{v},\tilde{x}_{0}\right) \right|^{2}(t)+\left|\partial_x\left( f -K\left(\tilde{v},\tilde{x}_{0}\right)\right) \right|^{2}(t)+\left|\left( f -K\left(\tilde{v},\tilde{x}_{0}\right)\right)_t\right|^{2}\,dx\rightarrow0
	\]
where 
$$K\left(\tilde{v},\tilde{x}_{0}\right)(t):=Q(x,t;\tilde{v},\tilde{x}_{0}).$$
\end{theorem}

\begin{proof}
	By the standard inverse scattering, we can construct $K\left(v,x_{0}\right)$
	from two eigenvalues $i\zeta_{0}$, $-i\zeta_{0}$ located on
	the imaginary axis with some norming constant $\iota_{0}$. By the continuity of the map from the initial condition to the scattering data,
	 we know that $\left(f(x, 0), f_t(x, 0)\right)$
	will generate eigenvalues $i\zeta_{1}$ and $-i\zeta_{1}$ such
	that $\left|\zeta_{1}-\zeta_{0}\right|\lesssim\epsilon.$ It will
	also produce the reflection coefficient $r\in L^\infty(\bbR)\cap L^{2}(\bbR)$ with $r(0)=0$. We denote this
	scattering data as $\left(r, i\zeta_{1}, \iota_1\right)$. 
	
	Next, we use a sequence of $\{ r_{n} \}_{n=1}^\infty \subset H^{1,1}_0$ to approximate $r$ in $L^{2}$.
	Then we consider the scattering data $\left(r_n, i\zeta_{1}, \iota_1\right)$.
	By the inverse scattering transform, this will also produce a sequence of solutions $\left(f_{n}\left(t\right),\partial_{t}\left(f_{n}\right)\left(t\right)\right)$.
	Note that for $t>0$
	$$|x- \mathfrak{v}t|<\mathfrak{L} \Rightarrow \ \mathfrak{v}-\dfrac{\mathfrak{L} }{t}< \dfrac{x}{t}<   \mathfrak{v}+\dfrac{\mathfrak{L}}{t}.$$
	Note that by construction, the velocity of kink $K\left(\tilde{v},\tilde{x}_{0}\right)$ is given by
	\begin{equation}
	\text{v}_K=\tilde{v}=\dfrac{1-\zeta_1^2}{1+\zeta_1^2}.
	\end{equation}
	Without losing generality, we choose the center of the kink as our frame of reference by letting 
	$$  \mathfrak{v}-\dfrac{\mathfrak{L}}{t}< \dfrac{x}{t}=\text{v}_K<   \mathfrak{v}+\dfrac{\mathfrak{L}}{t}.$$ In the spirit of nonlinear steepest descent, we conjugate the corresponding RHPs related to scattering data $\left(r, i\zeta_{1}, \iota_1\right)$ and $\left(r_n, ie, \iota_1\right)$ by
	$$ \delta^{\sigma_3}=  \left[ e^{\chi(z)} \right]^{\sigma_3}, \, \delta_n^{\sigma_3}=  \left[ e^{\chi_n(z)} \right]^{\sigma_3}.$$
	and let $t$ be large enough such that we can perform nonlinear steepest descent on the RHP related to $\left(r, i\zeta_{1}, \iota_1\right)$. By the triangle inequality,  we rewrite the limit as{\small
	\begin{align*}
	& \int_{I_t}\left|\left( f -K\left(\tilde{v},\tilde{x}_{0}(t)\right)\right)_x\right|^{2}+\left| f -K\left(\tilde{v},\tilde{x}_{0}(t)\right) \right|^{2}+\left|\left( f -K\left(\tilde{v},\tilde{x}_{0}(t)\right)\right)_t\right|^{2}\,dx\\
	& \leq  \int_{I_t}\left| \left( f_n -K_n\left(\tilde{v},\tilde{x}_{0}(t)\right)\right)_x\right|^{2}+\left|  f_n -K_n\left(\tilde{v},\tilde{x}_{0}(t)\right) \right|^{2}+\left|\left( f_n -K_n\left(\tilde{v},\tilde{x}_{0}(t)\right)\right)_t\right|^{2}\,dx\\
	&+ \int_{I_t}\left| \left(K_n\left(\tilde{v},\tilde{x}_{0}(t)\right)-K\left(\tilde{v},\tilde{x}_{0}(t)\right)\right)_x\right|^{2}+\left| K_n\left(\tilde{v},\tilde{x}_{0}(t)\right)-K\left(\tilde{v},\tilde{x}_{0}(t)\right)\right|^{2} +\left| \left(K_n\left(\tilde{v},\tilde{x}_{0}(t)\right)-K\left(\tilde{v},\tilde{x}_{0}(t)\right)\right)_t \right|^{2}\,dx\\
	&+\left\Vert \left( f\left(\cdot, t\right)-f_{n}\left(\cdot, t\right),f_{t}\left(t\right)-f_{n, t}\left(t\right)\right)\right\Vert^2 _{  (H^{1}\times L^{2} ) (I_t) }\\
	&= \mathrm{I}_{t, 1}+ \mathrm{I}_{t, 2}+\mathrm{I}_{t, 3}
	\end{align*}}
For a fixed $n$, with smooth data, our result in Theorem \ref{thm:maindetail} and identities \eqref{eq:rsine}, \eqref{eq:rcosine} imply that $\mathrm{I}_{t, 1}\to 0$ as $t\to +\infty$ since $I_t$ is a compact interval. For $\mathrm{I}_{t, 2}$, we recall from \eqref{eq:Kink-1/2u} that the difference only involves
	$$| \delta(i\zeta_1)- \delta_n(i\zeta_1) | \lesssim \norm{r-r_n}{L^2\cap L^\infty}.$$
	For $\mathrm{I}_{t, 3}$, we can apply the same proof as that of Theorem \ref{thm:zero}. So we conclude that given any $\epsilon>0$, there exists $T>0$ such that for $t>T$
	$$ \left\vert \int_{I_t}\left|\left( f -K\left(\tilde{v},\tilde{x}_{0}(t)\right)\right)_x\right|^{2}+\left| f -K\left(\tilde{v},\tilde{x}_{0}(t)\right) \right|^{2}+\left|\left( f -K\left(\tilde{v},\tilde{x}_{0}(t)\right)\right)_t\right|^{2}\,dx \right\vert \lesssim \epsilon$$
	where  $\left|\tilde{v}-v\right|, \left|\tilde{x}_0-x_0\right|\lesssim\epsilon$.
\end{proof}
Finally we turn to the proof of Theorem \ref{thm:stabNL}.
\begin{proof} [Proof of Theorem \ref{thm:stabNL}]
	By Lemma \ref{lem:BCD},
	we know the existence of scattering data
	\begin{equation}
	\mathcal{S}=\left\{ r\left(z\right),\left\{ i\zeta_{1,k},\iota_{1,k}\right\} _{k=1}^{N_{1}},\left\{ z_{1,j},c_{1,j}\right\} _{j=1}^{N_{2}}\right\} \in \left( L^2\left(\mathbb{R}\right)\cap L^\infty(\mathbb{R}) \right)\otimes\mathbb{C}^{2N_{1}}\otimes\mathbb{C}^{2N_{2}}\label{eq:newscatt1-1}
	\end{equation}
	computed in terms of $\vec{f}(0)$ with $r(0)=0$ such that
	\[
	\left\Vert r\right\Vert _{L^2\left(\mathbb{R}\right)\bigcap L^\infty (\mathbb{R})}+\sum_{\ell=1}^{N_{2}}\left(\left| i\zeta_{0,\ell}-i\zeta_{1,\ell}\right|+\left|\iota_{0,\ell}-\iota_{1,\ell}\right|\right)+\sum_{j=1}^{N_{1}}\left(\left|z_{0,j}-z_{1,j}\right|+\left|c_{0,j}-c_{1,j}\right|\right)\lesssim\epsilon.
	\]
	Then using the scattering data we can apply the Riemann-Hilbert problem
	and reconstruction formula to define the solution $f$ and find the
	reflectionless part $\tilde{N}\left(x,t\right)$ using Problem \ref{RHP-1}
	and Proposition \ref{prop:recon}.
	
	Next, we use sequence $\lbrace r_n \rbrace$ with $r_{n}\in H_0^{1,1}(\bbR)$  to approximate $r$ in $L^2\left(\mathbb{R}\right)\bigcap L^\infty(\bbR)$. Then we have a sequence
	of scattering data
	\[
	\mathcal{S}_{n}=\left\{ r_{n}\left(z\right),\left\{ i\zeta_{1,k},\iota_{1,k}\right\} _{k=1}^{N_{1}},\left\{ z_{1,j},c_{1,j}\right\} _{j=1}^{N_{2}}\right\} \in H_0^{1,1}(\bbR) \otimes\mathbb{C}^{2N_{1}}\otimes\mathbb{C}^{2N_{2}}.
	\]
	Using the scattering data $\mathcal{S}_{n}$, we again can employ the Riemann-Hilbert
	problem and reconstruction formula to define the function $f_{n}$
	and find the reflectionless part $\tilde{N}_{n}\left(x,t\right)$
	using Problem \ref{RHP-1} and Proposition \ref{prop:recon}. (Here $f_{n}$ will actually
	be a solution to the sine-Gordon equation by Zhou \cite{Zhou95}, but this fact is
	not essential here.) We are mainly interested in the case when $\text{v}$ is close to 
	a kink whose velocity is given by $(1-\zeta_\ell^2)/(1+ \zeta_\ell^2)$. Similar to the proof of the previous theorem, we conjugate the RHPs associated to both $\mathcal{S}$ and $\mathcal{S}_n$ by $\delta(z_\ell)$ and $\delta_n(z_\ell)$ as in problem \ref{prob:mkdv.kin}. Note that by construction,
	\begin{align*}
	\left\Vert \left(\tilde{N}\left(t\right)-\tilde{N}_{n}\left(t\right),\left(\tilde{N}\right)_{t}\left(t\right)-\left(\tilde{N}_{n}\right)_{t}\left(t\right)\right)\right\Vert _{\mathcal{E},\mathfrak{v},\mathfrak{L}} & \lesssim\left\Vert r_{n}-r\right\Vert _{L^2\left(\mathbb{R}\right)\bigcap L^\infty (\mathbb{R})}.
	\end{align*}
	Moreover, if we look at the corresponding integral equation representation of solutions given by \eqref{BC-int}, taking the difference,
	those parts associated with eigenvalues will be exponentially small. So by the uniform
	Lipschitz estimates for the resolvents established in Theorem \ref{thm:zero}, we know that for any $t\in\mathbb{R}^{+}$,
	one has
	\begin{align*}
	\left\Vert \left(f\left(t\right)-f_{n}\left(t\right),\left(f\right)_{t}\left(t\right)-\left(f_{n}\right)_{t}\left(t\right)\right)\right\Vert  _{\mathcal{E},\mathfrak{v},\mathfrak{L}}& \lesssim\left\Vert r_{n}-r\right\Vert _{L^2\left(\mathbb{R}\right)\bigcap L^\infty (\mathbb{R})}.
	\end{align*}
	For $f_{n}$, subtracting the reflectionless part off, we write
	\[
	f_{n}=\tilde{N}_{n}(x,t)+f_{r,n}(x,t).
	\]
	By our long-time asymptotic computations and the identities \eqref{eq:rsine}, \eqref{eq:rcosine},
	we know that
	\[
	\lim_{t\rightarrow\infty}\left\Vert f_{r,n}(t)\right\Vert _{\mathcal{E},\mathfrak{v},\mathfrak{L}}=0.
	\]
	Next, we write
	\begin{align*}
	\left\Vert f_r(t)\right\Vert _{\mathcal{E},\mathfrak{v},\mathfrak{L}}^{2}&\lesssim\left\Vert f_{r,n}(t)\right\Vert _{\mathcal{E}, \mathfrak{v},\mathfrak{L}}^{2}+\left\Vert \tilde{N}(t)-\tilde{N}_{n}(t)\right\Vert _{\mathcal{E},\mathfrak{v},\mathfrak{L}}^{2}+\left\Vert f_{n}(t)-f(t)\right\Vert _{\mathcal{E},\mathfrak{v},\mathfrak{L}}^2\\
	&\lesssim\left\Vert f_{r,n}(t)\right\Vert _{\mathcal{E},\mathfrak{v},\mathfrak{L}}^{2}+\left\Vert r_{n}-r\right\Vert _{L^2\left(\mathbb{R}\right)}^2 +\left\Vert \left(f\left(t\right)-f{}_{n}\left(t\right),\left(f\right)_{t}\left(t\right)-\left(f_{n}\right)_{t}\left(t\right)\right)\right\Vert _{H_{\text{sin}}^{1}\times L^{2}}^{2}\\
	&\lesssim\left\Vert f_{r,n}(t)\right\Vert _{\mathcal{E},\mathfrak{v},\mathfrak{L}}^{2}+\left\Vert r_{n}-r\right\Vert _{L^2\left(\mathbb{R}\right)}^2.
	\end{align*}
	This could be arbitrarily small provided $\left\Vert r_{n}-r\right\Vert _{L^2\left(\mathbb{R}\right)}$
	is small enough and $t$ large enough. So we conclude that
	\[
	\lim_{t\rightarrow\infty}\left\Vert f(t)\right\Vert _{\mathcal{E},\mathfrak{v},\mathfrak{L}}=0
	\]
	as desired.
\end{proof}

\subsection{The failure of asymptotic stability}
Finally, we consider the sine-Gordon equation in $H^{1,s}_{\text{sin}}(\mathbb{R})\times L^{2,s}(\mathbb{R})$ with $0\leq s<1/2$. By computing the norms of breathers, the difference between a wobbling kink and a kink, we conclude the failure of the asymptotic stability of solutions to the sine-Gordon equation in these weighted spaces.
\subsubsection{Breathers}\label{subsubsec:breather}

We start with the computations for breathers in weighted energy spaces. For simplicity,
we focus on breathers with zero velocity  here.

Recall that from \eqref{eq:breather}, all breathers with zero velocity of the sine-Gordon equation
are given as
\[
\ensuremath{B\left(x,t;\beta,x_{1},x_{2}\right):=B_{0}\left(x,t;\beta, x_{1},x_{2}\right)=4\arctan\left(\frac{\beta}{\alpha}\frac{\sin\left(\alpha\left(t+x_{1}\right)\right)}{\cosh\left(\beta\left(x+x_{2}\right)\right)}\right),\quad\alpha=\sqrt{1-\beta^{2}},\quad\beta\neq0.}
\]
To simplify the computation again, we take $x_{1}=x_{2}=0$.
\begin{proposition}\label{prop:breather}
	Denoting $B:=B(x,t;\beta, 0,0)$, for $\beta$ small,  at $t=\frac{\pi}{2\alpha}$, we have
	\[
	\int\left|x\right|^{2s}\left(\left|\partial_{x}B\right|^{2}+\left|\partial_{t}B\right|^{2}+\sin^{2}\left(\frac{B}{2}\right)\right)\,dx\sim\beta^{1-2s}.
	\]
\end{proposition}

\begin{proof}
	Taking the time derivative, one has
	\[
	\partial_{t}B=4\frac{1}{1+g^{2}\left(x,t\right)}\partial_{t}g
	\]
	where
	\[
	g\left(x,t\right):=\frac{\beta}{\alpha}\frac{\sin\left(\alpha t\right)}{\cosh\left(\beta x\right)}.
	\]
	Direct computations give $\partial_{t}g\left(x,t\right)=\beta\frac{\cos\left(\alpha t\right)}{\cosh\left(\beta x\right)}$
	which implies
	\begin{align*}
	\partial_{t}B &=4\frac{\alpha^{2}\cosh^{2}\left(\beta x\right)}{\alpha^{2}\cosh^{2}\left(\beta x\right)+\beta^{2}\sin^{2}\left(\alpha t\right)}\beta\frac{\cos\left(\alpha t\right)}{\cosh\left(\beta x\right)}\\
	&=4\beta\frac{\alpha^{2}\cosh\left(\beta x\right)\cos\left(\alpha t\right)}{\alpha^{2}\cosh^{2}\left(\beta x\right)+\beta^{2}\sin^{2}\left(\alpha t\right)}.
	\end{align*}
	Taking the homogeneous weighted norm, one has
	\begin{align*}
	&\int\left|x\right|^{2s}\left(4\beta\frac{\alpha^{2}\cosh\left(\beta x\right)\cos\left(\alpha t\right)}{\alpha^{2}\cosh^{2}\left(\beta x\right)+\beta^{2}\sin^{2}\left(\alpha t\right)}\right)^{2}\,dx\\
	& \quad =16\alpha^{2}\cos^{2}(\alpha t)\int\left|x\right|^{2s}\left(\beta\frac{\cosh\left(\beta x\right)}{\alpha^{2}\cosh^{2}\left(\beta x\right)+\beta^{2}\sin^{2}\left(\alpha t\right)}\right)^{2}\,dx\\
	&\quad =16\alpha^{2}\cos^{2}(\alpha t)\beta^{1-2s}\int\left|y\right|^{2s}\left(\frac{\cosh\left(y\right)}{\alpha^{2}\cosh^{2}\left(y\right)+\beta^{2}\sin^{2}\left(\alpha t\right)}\right)^{2}\,dy\\
&\quad	\sim16\alpha^{2}\cos^{2}(\alpha t)\beta^{1-2s}
	\end{align*}
	where in the third equality, we did a change of variable $y=\beta x$.
	
	
	Next, we consider $\alpha t=\frac{\pi}{2}$. In this case breather
	takes the form
	\[
	B\left(x\right)=4\arctan\left(\frac{\beta}{\alpha}\frac{1}{\cosh\left(\beta x\right)}\right)
	\]
and we compute 
	\begin{align*}
	\int\left|x\right|^{2s}B^{2}(x)\,dx & =16\int\left|x\right|^{2s}\left|\arctan\left(\frac{\beta}{\alpha}\frac{1}{\cosh\left(\beta x\right)}\right)\right|^{2}\,dx\\
	& =16\beta^{-2s}\int\left|y\right|^{2s}\left|\arctan\left(\frac{\beta}{\alpha}\frac{1}{\cosh\left(y\right)}\right)\right|^{2}\,dy\\
	& \sim\beta^{1-2s}.
	\end{align*}
	Similar computations can also be applied to $\partial_{x}B$ which will
	be smaller when $\beta$ is small since $\partial_{x}$ will introduce
	one more $\beta$ factor.
	
	To summarize the computations above, we conclude that
	\[
	\int\left|x\right|^{2s}\left(\left|\partial_{x}B\right|^{2}+\left|\partial_{t}B\right|^{2}+\left|B\right|^{2}\right)(t,x;\beta)\,dx\sim\beta^{1-2s}.
	\]
	Finally, note that when $\beta$ is small, $\left|B\right|\sim\sin(B/2)$
	which allow us to interchange the weighted $H^{1}$ norm above with
	the weighted $H_{\text{sin}}^{1}$ norm.
\end{proof}

\subsubsection{Wobbling kinks}\label{subsubsec:wobbling}
By similar ideas above, now we measure
the difference between a wobbling kink and a kink in the weighted space $H^{1,s}\times L^{2, s}$. Recall that one has the kink
\[
\ensuremath{S(x)=4\arctan\left(e^{x}\right),\quad x\in\mathbb{R}}.
\]
Then a wobbling kink has the following form:
\begin{equation}
\ensuremath{W_{\alpha}(x, t)=4\arctan\left(V_{\alpha}\left(x, t\right)/U_{\alpha}\left(x, t\right)\right)}
\label{eq:Wobbling-Kink}
\end{equation}
where
\[
\ensuremath{U_{\alpha}(x, t)=1+\frac{1+\beta}{1-\beta}e^{2\beta x}-\frac{2\beta}{1-\beta}e^{(1+\beta)x}\cos(\alpha t)}
\]
\[
V_{\alpha}(x, t)=\frac{1+\beta}{1-\beta}e^{x}+e^{(1+2\beta)x}-\frac{2\beta}{1-\beta}e^{\beta x}\cos(\alpha t).
\]
We denote $$\mathfrak{B}:=\left(S-W_{\alpha}\right).$$To simplify
the problem, we only compute these quantities at $t=\frac{\pi}{2\alpha}$
as in Kowalczyk-Martel-Muñoz \cite{KMM2}.
\begin{proposition}\label{prop:wobble}
	At $t=\frac{\pi}{2\alpha}$, we have
	\[
	\dfrac{\pi}{2\alpha}\int\left|x\right|^{2s}\left(\left|\partial_{x}  \mathfrak{B}  \right|^{2}+\left|\partial_{t}   \mathfrak{B} \right|^{2}+\sin^{2}\left(\frac{\mathfrak{B}}{2}\right)\right)\,dx\sim\beta^{1-2s}.
	\]
\end{proposition}

\begin{proof}
	By direct computations or by Kowalczyk-Martel-Muñoz \cite{KMM2},
	we know that for $t=\frac{\pi}{2\alpha}$, 
	\[
	S(x)-W_{\alpha}(x, \frac{\pi}{2\alpha})
	\]
	is small in any polynomial weighted Sobolev spaces. It remains
	to check $\partial_{t}\left(S-W_{\alpha}\right).$ Since $S\left(x\right)$
	is stationary, it suffices to check $\partial_{t}W_{\alpha}(x,\frac{\pi}{2\alpha})$
	has a uniform lower bound.
	
	At $t=\frac{\pi}{2\alpha}$, one has 
	\[
	\ensuremath{U_{\alpha}=1+\frac{1+\beta}{1-\beta}e^{2\beta x}},\ V_{\alpha}=\frac{1+\beta}{1-\beta}e^{x}+e^{(1+2\beta)x}
	\]
	and
	\[
	\partial_{t}U_{\alpha}=\frac{2\beta\alpha}{1-\beta}e^{(1+\beta)x},\ \ \partial_{t}V_{\alpha}=\frac{2\beta}{1-\beta}e^{\beta x}.
	\]
	Taking the time derivative of $W_{\alpha}$, one has
	\[
	\partial_{t}W_{\alpha}=\frac{1}{1+\left(V_{\alpha}\left(x, t\right)/U_{\alpha}\left(x, t\right)\right)^{2}}\partial_{t}\left(\frac{V_{\alpha}\left(x, t\right)}{U_{\alpha}\left(x, t\right)}\right).
	\]
	It is safe to replace $\frac{V_{\alpha}\left(x, t\right)}{U_{\alpha}\left(x, t\right)}$
	by $e^{x}$ in the denominator. 
	
	Next, we only consider the part $x\leq0$ since computations for  $x\geq 0$ are similar. Direct computations give
	us
	\[
	\partial_{t}\left(\frac{V_{\alpha}\left(x, t\right)}{U_{\alpha}\left(x, t\right)}\right)=\frac{\partial_{t}V_{\alpha}}{U_{\alpha}}-\frac{V_{\alpha}}{U_{\alpha}}\frac{\partial_{t}U_{\alpha}}{U_{\alpha}}
	\]
	where
	\[
	\frac{\partial_{t}V_{\alpha}}{U_{\alpha}}=\frac{\frac{2\beta}{1-\beta}e^{\beta x}}{1+\frac{1+\beta}{1-\beta}e^{2\beta x}}=\frac{2\beta e^{\beta x}}{1-\beta+\left(1+\beta\right)e^{2\beta x}}.
	\]
	\[
	\frac{\partial_{t}U_{\alpha}}{U_{\alpha}}=\frac{\frac{2\beta\alpha}{1-\beta}e^{(1+\beta)x}}{1+\frac{1+\beta}{1-\beta}e^{2\beta x}}\sim\beta e^{x},\ \frac{V_{\alpha}}{U_{\alpha}}=\frac{\frac{1+\beta}{1-\beta}e^{x}+e^{(1+2\beta)x}}{1+\frac{1+\beta}{1-\beta}e^{2\beta x}}\sim e^{x}.
	\]
	Therefore, using estimates above, one can conclude that for $x\lesssim-1$,
	\begin{align*}
	\partial_{t}W_{\alpha} & =\frac{1}{1+\left(V_{\alpha}\left(t,x\right)/U_{\alpha}\left(x, t\right)\right)^{2}}\partial_{t}\left(\frac{V_{\alpha}\left(x, t\right)}{U_{\alpha}\left(x, t\right)}\right)\\
	& \sim\frac{1}{2}\frac{2\beta e^{\beta x}}{1-\beta+\left(1+\beta\right)e^{2\beta x}}-\beta e^{2x}.
	\end{align*}
It follows that
	\begin{align*}
	\int_{-\infty}^{-1}\left|x\right|^{2s}\left|\partial_{t}W\right|^{2}\,dx & \sim\int_{-\infty}^{-1}\left|x\right|^{2s}\left(\frac{1}{2}\frac{2\beta e^{\beta x}}{1-\beta+\left(1+\beta\right)e^{2\beta x}}\right)^{2}\,dx\\
	&  \quad +\int_{-\infty}^{-1}\left|x\right|^{2s}\beta^{2}e^{4x}\,dx.
	\end{align*}
	The last integral above can be arbitrarily small provided that $\beta$ is small enough.
	For the first integral, note that
	\begin{align*}
	\int_{-\infty}^{-1}\left|x\right|^{2s}\left(\frac{1}{2}\frac{2\beta e^{\beta x}}{1-\beta+\left(1+\beta\right)e^{2\beta x}}\right)^{2}\,dx & \sim\beta^{1-2s}\int_{-\infty}^{-\beta}\left|y\right|^{2s}\left(\frac{ke^{y}}{1-\beta+\left(1+\beta\right)e^{2y}}\right)^2\,dy\\
	& \sim\beta^{1-2s}.
	\end{align*}
	Therefore, one can conclude that
	\[
	\int\left|x\right|^{2s}\left|\partial_{t}W_{\alpha}\left(\frac{\pi}{2\alpha}\right)\right|^2\,dx\sim\beta^{1-2s}
	\]
	and
	\[
	\frac{\pi}{2\alpha}\int\left|x\right|^{2s}\left(\left|\partial_{x}   \mathfrak{B}   \right|^{2}+\left|\partial_{t}   \mathfrak{B} \right|^{2}+\sin^{2}\left(\frac{  \mathfrak{B}  }{2}\right)\right)\,dx\sim\beta^{1-2s}
	\]
	as desired.
\end{proof}
\subsubsection{The failure of the stability}
Consider the sine-Gordon equation in $H^{1,s}(\mathbb{R})\times L^{2,s}(\mathbb{R})$.
From Proposition \ref{prop:breather}, we find that when $s<\frac{1}{2}$, as $\beta\rightarrow0$, the norm of the breather goes to $0$.  Therefore there are arbitrarily
small breathers which imply the failure of asymptotic stability
of the zero solution. Similarly, by Proposition \ref{prop:wobble}, in the weighted energy spaces with $s<\frac{1}{2}$, the difference between the wobbling kink and the kink can be arbitrarily small. Hence the kink is not asymptotically stable.

\smallskip

 On the other hand, when $s>1/2$, the norms of breathers and the difference between the wobbling kink and the kink cannot be arbitrarily small. Indeed, by our results Subsection \ref{subsec:wasymp}, the sine-Gordon equation is asymptotically stable in $H^{1,s}(\mathbb{R})\times L^{2,s}(\mathbb{R})$ measured using the localized energy norm.  Therefore the weights we require are essentially optimal if
 we consider the problem without any symmetry assumptions. Finally $s=\frac{1}{2}$ as the threshold
 remains as an interesting  problem.
\appendix

\section{Global well-posedness  in weighted spaces}
In the direct scattering process, it requires that initial data to
be in $H_{\text{sin}}^{2,s}\left(\mathbb{R}\right)\times H^{1,s}\left(\mathbb{R}\right)$ with $s>\frac{1}{2}$.
Here we provide a sketch of the proof of  the global well-posedness  in this space for
the sake of completeness.
\begin{theorem}
\label{thm:gwpweighted}Given $\left(f_{0},f_{1}\right)\in H_{\text{sin}}^{2,s}\left(\mathbb{R}\right)\times H^{1,s}\left(\mathbb{R}\right)$ for any $s>\frac{1}{2}$,
there exists a unique solution $$f\in C\left(\mathbb{R};H_{\text{sin}}^{2,s}\left(\mathbb{R}\right)\right)\,\text{
with}\,\, \partial_{t}f\in C\left(\mathbb{R};H^{1,s}\left(\mathbb{R}\right)\right)$$
to the sine-Gordon equation \eqref{eq: sG} with initial data $\left(f_{0},f_{1}\right)$.
\end{theorem}

\begin{proof}
The proof here is more or less standard. Given $f_{0}\in H_{\text{sin}}^{2}\left(\mathbb{R}\right)$, following the steps in de Laire-Gravej \cite{deLG}, we can find a smooth function $\phi\left(x\right)$ such that $f_{0}-\phi\left(x\right)\in H^{2}\left(\mathbb{R}\right)$.
 Then one can show that
for given $\left(f_{0},f_{1}\right)\in H_{\text{sin}}^{2}\left(\mathbb{R}\right)\times H^{1}\left(\mathbb{R}\right)$,
there exists  a unique solution $f-\phi\left(x\right)\in C\left(\mathbb{R};H^{2}\left(\mathbb{R}\right)\right)$
with $\partial_{t}f\in C\left(\mathbb{R};H^{1}\left(\mathbb{R}\right)\right)$
to the sine-Gordon equation \eqref{eq: sG} with initial data $\left(f_{0},f_{1}\right)$. It follows from Theorem \ref{thm:GWP} above. Basically, one can perform a
direct differentiation and use the fixed-point theorem to prove the local
well-posedness in $[-T,T]$. Then the size of $T$ can be estimated
by performing standard energy estimates. Actually, one can show the
global well-posedness in the space $H_{\text{sin}}^{k}\left(\mathbb{R}\right)\times H^{k-1}\left(\mathbb{R}\right)$
with $k>\frac{3}{2}$ by the same argument. Again we refer to de Laire-Gravej \cite{deLG}
for full details.

With unweighted results above, one can pass to the weighted space
directly. Given $\left(f_{0},f_{1}\right)\in H_{\text{sin}}^{2,s}\left(\mathbb{R}\right)\times H^{1,s}\left(\mathbb{R}\right)$, following notations above, one can find the solution $f$ such that $f\in C\left(\mathbb{R};H_{\text{sin}}^{2,s}\left(\mathbb{R}\right)\right)$
with $\partial_{t}f\in C\left(\mathbb{R};H^{1}\left(\mathbb{R}\right)\right)$. Our goal is to show that $f$ is in the weighted Sobolev spaces. We denote $\tilde{f}=f-\phi\left(x\right)$. Then $\tilde{f}$ satisfies
\begin{align}
\partial_{tt}\tilde{f}-\partial_{xx}\tilde{f} & =-\sin\left(\tilde{f}+\phi\right)+\partial_{xx}\phi\label{eq:tildefdiff}\\
 & =-\sin\left(\tilde{f}\right)-\sin\left(\tilde{f}\right)\left(\cos\left(\phi\right)-1\right)-\cos\left(\tilde{f}\right)\sin\left(\phi\right)+\partial_{xx}\phi.\nonumber 
\end{align}
We multiply the equation by $x^{2s}e^{-2\epsilon x^{2s}}\partial_{t}\tilde{f}$
and then directing differentiation results in
\begin{align}
\frac{d}{dt}E_{\sin,w}^{\epsilon}\left(\tilde{f},\tilde{f}_{t}\right) & \lesssim E_{\text{sin}}\left(\tilde{f},\tilde{f}_{t}\right)+E_{\sin,w}^{\epsilon}\left(\tilde{f},\tilde{f}_{t}\right)\label{eq:weightedenegydiff}
\end{align}
where the weighted energy is defined via
\[
E_{\sin,w}^{\epsilon}\left(g_{1},g_{2}\right)=\int\left(\frac{\left|\partial_x g_{1}\right|^{2}}{2}+\frac{\left|g_{2}\right|^{2}}{2}+\frac{1}{2}\sin^{2}\left(\frac{g_{1}}{2}\right)\right)x^{2s}e^{-2\epsilon x^{2s}}\,dx.
\]
Indeed, differentiating  the weighed energy and integrating by parts, one has
\begin{align*}
\frac{d}{dt}E_{\sin,w}^{\epsilon}\left(\tilde{f},\tilde{f}_{t}\right) & =\int\left(\partial_{tt}\tilde{f}\partial_{t}\tilde{f}+\partial_{tx}\tilde{f}\partial_{x}\tilde{f}+\sin\left(\tilde{f}\right)\partial_{t}\tilde{f}\right)x^{2s}e^{-2\epsilon x^{2s}}\,dx\\
& =\int\partial_{t}\tilde{f}\left(\partial_{tt}\tilde{f}-\partial_{xx}\tilde{f}+\sin\left(\tilde{f}\right)\right)x^{2s}e^{-2\epsilon x^{2s}}\,dx\\
& \quad-\int\partial_{t}\tilde{f}\partial_{x}\tilde{f}\partial_{x}\left(x^{2s}e^{-2\epsilon x^{2s}}\right)\,dx.
\end{align*}
Using the equation for $\tilde{ f }$, \eqref{eq:tildefdiff}, one has
\begin{align*}
\int\partial_{t}\tilde{f}\left(\partial_{tt}\tilde{f}-\partial_{xx}\tilde{f}+\sin\left(\tilde{f}\right)\right)x^{2s}e^{-2\epsilon x^{2s}}\,dx & = \int\partial_{t}\tilde{f}\left[-\sin\left(\tilde{f}\right)\left(\cos\left(\phi\right)-1\right)\right.\\
&\quad \left.  -\cos\left(\tilde{f}\right)\sin\left(\phi\right)+\partial_{xx}\phi\right]x^{2s}e^{-2\epsilon x^{2s}}\,dx.
\end{align*}
By construction, $\phi$ achieves the boundary values rapidly as $\left|x\right|\rightarrow\infty$,
whence, by the Cauchy-Schwarz inequality, it follows that

\begin{align*}
&\left|\int\partial_{t}\tilde{f}\left(-\sin\left(\tilde{f}\right)\left(\cos\left(\phi\right)-1\right)-\cos\left(\tilde{f}\right)\sin\left(\phi\right)+\partial_{xx}\phi\right)x^{2s}e^{-2\epsilon x^{2s}}\,dx\right|\\
&\quad \lesssim E_{\sin,w}^{\epsilon}\left(\tilde{f},\tilde{f}_{t}\right).
\end{align*}
Actually, using the rapid convergence of $\phi$ to boundary values,
one can use the unweighted energy on the RHS above.

Notice that explicitly, we have
\begin{align}
\left|\int\partial_{t}\tilde{f}\partial_{x}\tilde{f}\partial_{x}\left(x^{2s}e^{-2\epsilon x^{2s}}\right)\,dx\right| & \lesssim E_{\sin,w}^{\epsilon}\left(\tilde{f},\tilde{f}_{t}\right)+E_{\text{sin}}\left(\tilde{f},\tilde{f}_{t}\right).\label{eq:weightedcom}
\end{align}
where the last line is achieved by comparing the size of $|x|^s$ and $\epsilon$. 
Due to the well-posedness result in the regular Sobolev space, the
last term on the RHS of \eqref{eq:weightedcom} is bounded. 

Combining \eqref{eq:weightedenegydiff} and \eqref{eq:weightedcom} together,
Gronwall\textquoteright s inequality gives
\[
E_{\sin,w}^{\epsilon}\left(\tilde{f}\left(t\right),\tilde{f}_{t}\left(t\right)\right)\leq C_{T}E_{\sin,w}^{\epsilon}\left(\tilde{f}_{0},\tilde{f}_{1}\right)
\]
for $t\in\left[-T,T\right]$ for all $T\in\mathbb{R}$ where $\tilde{f}_0=f_0-\phi$ and $\tilde{f}_1=f_1$. Passing $\epsilon$
to $0$ by the monotone convergence theorem, we get
\[
E_{\text{sin},w}\left(\tilde{f}\left(t\right),\tilde{f}_{t}\left(t\right)\right)\leq C_{T}E_{\text{sin}}\left(\tilde{f}_{0},\tilde{f}_{1}\right)
\]
where
\[
E_{\sin,w}\left(g_{1},g_{2}\right)=\int\left(\frac{\left|\partial_x g_{1}\right|^{2}}{2}+\frac{\left|g_{2}\right|^{2}}{2}+\frac{1}{2}\sin^{2}\left(\frac{g_{1}}{2}\right)\right)x^{2s}\,dx
\]
for $t\in\left[-T,T\right]$ for all $T\in\mathbb{R}$ which implies
that $f\in C\left(\mathbb{R};H_{\text{sin}}^{1,s}\left(\mathbb{R}\right)\right)$
with $\partial_{t}f\in C\left(\mathbb{R};L^{2,s}\left(\mathbb{R}\right)\right)$.

To obtain the result at the level of $H^{2,s}_{ \text{sin}}(\mathbb{R})\times H^{1,s}(\mathbb{R})$, we can differentiate the equation \eqref{eq:tildefdiff} by $\partial_{x}$.
The same argument above results in
\begin{align*}
\int\left(\dfrac{\left|\partial_{x}^{2}\tilde{f}\left(t\right)\right|^{2}}{2}+\dfrac{\left|\partial_{x}\tilde{f_{t}}\left(t\right)\right|^{2}}{2}\right)x^{2s}e^{-2\epsilon x^{2s}}\,dx &\leq C_{T}E_{\sin,w}^{\epsilon}\left(\tilde{f}_{0},\tilde{f}_{1}\right)\\
&\quad +C_T\int\left(\dfrac{\left|\partial_{x}^{2}\tilde{f}_{0}\right|^{2}}{2}+\dfrac{\left|\partial_{x}\tilde{f_{1}}\left(t\right)\right|^{2}}{2}\right)x^{2s}e^{-2\epsilon x^{2s}}\,dx
\end{align*}
for $t\in\left[-T,T\right]$ for all $T\in\mathbb{R}$. Again passing
$\epsilon$ to $0$, the desired results follow.
\end{proof}

\section{Approaching the light cone}
In this appendix we discuss the situation when $|x/t|\to 1$ as $t\to \infty$. When $|x/t|$ is close to $1$ enough, we will not have any solitons in this space time region. So in the following section we omit all discrete scattering data for brevity.
\subsection{$0<x/t<1$} 
\label{subsec: app}
We write $z=u+iv$. Notice that for $v>0$ and $u>\sqrt{z_0}= \sqrt[4]{(t-x)/(t+x)}$
\begin{align*}
4\text{Re}i\theta(z; x, t) &=-\left(   1+\dfrac{x}{t} \right)v t +  \left( 1-\dfrac{x}{t} \right)\dfrac{v t}{u^2+v^2}\\
&  \leq -\left(   1+\dfrac{x}{t} \right)v t  +\left(   1-\dfrac{x}{t} \right) \dfrac{v}{z_0} t \\
& \leq -\left(   1+\dfrac{x}{t} \right)v t +\sqrt{1-\dfrac{x^2}{t^2}}vt\\
&\leq -vt.
\end{align*}
Given $\frac{x}{t}\rightarrow1$ which implies
$z_0 \to 0$,  \text{as} $t \to \infty$ we only need the  following upper/lower factorization on $\bbR$:
\begin{equation}
\label{v-ul}
e^{-i\theta\ad\sigma_3}v(z)	=\Twomat{1}{\overline{r(z)}   e^{-2i\theta}}{0}{1} \Twomat{1}{0}{r(z)  e^{2i\theta}}{1}
						\quad z \in\bbR .
\end{equation}
We will again perform the contour deformation and write the solution as a product of solutions to a $\dbar$-problem and a ``localized" Riemann-Hilbert problem.
\begin{figure}[H]
\caption{$\Sigma^{(1)}:\text{near the light cone}$}
\vskip 15pt
\begin{tikzpicture}[scale=0.7]
\draw[thick]		(5, 3) -- (4,2);						
\draw[->,thick,>=stealth] 		(2,0) -- (4,2);		
\draw[thick] 			(-2,0) -- (-4,2); 				
\draw[->,thick,>=stealth]  	(-5,3) -- (-4,2);	
\draw[->,thick,>=stealth]		(-5,-3) -- (-4,-2);							
\draw[thick]						(-4,-2) -- (-2,0);
\draw[thick,->,>=stealth]		(2,0) -- (4,-2);								
\draw[thick]						(4,-2) -- (5,-3);
\draw[thick]		(0,0)--(2,0);
\draw[thick,->,>=stealth] (-2,0) -- (0, 0);
\draw	[fill]							(-2,0)		circle[radius=0.1];	
\draw	[fill]							(2,0)		circle[radius=0.1];
\draw [dashed] (2,0)--(6,0);
\draw [dashed] (-6,0)--(-2,0);
\node[below] at (-2,-0.25)			{$-\sqrt{z_0}$};
\node[below] at (2,-0.25)			{$\sqrt{z_0}$};
\node[right] at (5,3)					{$\Sigma^{(1)}_1$};
\node[left] at (-5,3)					{$\Sigma^{(1)}_2$};
\node[left] at (-5,-3)					{$\Sigma^{(1)}_3$};
\node[right] at (5,-3)				{$\Sigma^{(1)}_4$};
\node[above] at (4.5,0)           {$\Omega_1$};
\node[above] at (-4.5,0)           {$\Omega_2$};
\node[below] at (-4.5,0)           {$\Omega_3$};
\node[below] at (4.5,0)           {$\Omega_4$};
\end{tikzpicture}
\label{fig:contour-scale-1}
\end{figure}
For brevity, we only discuss the $\dbar$-problem in $\Omega_1$.  In $\Omega_1$, we define
\begin{align*}
	R_1	&=	\begin{cases}
						\twomat{0}{0}{r(z)  e^{2i\theta}  }{0}		
								&	z \in (\sqrt{z_0} ,\infty)\\[10pt]
								\\
						\twomat{0}{0}{r( \sqrt{z_0} ) e^{2i\theta (z )  } }{0}	
								&	z	\in \Sigma_1^{(1)}
					\end{cases}
	\end{align*}
and the interpolation is given by
$$R_1=  r( \sqrt{z_0}  )+ \left( \mathbf{r}\left(  z \right) -r( \sqrt{z_0} )  \right) \cos 2\phi  . $$
So we arrive at the $\dbar$-derivative in $\Omega_1$ :
\begin{align}
\dbar R_1&=  \left(  \dbar[\mathbf{r}\left( z\right)] \cos 2\phi- 2\dfrac{ \mathbf{r}(z  )-   r( \sqrt{z_0}  )  }{  \left\vert z -\sqrt{z_0} \right\vert   } e^{i\phi} \sin 2\phi  \right) e^{2i\theta}.
\end{align}
\begin{equation}
\label{R1.bd app}
|W|=\left| \dbar R_1  \right| 	\lesssim\left( |   \dbar[\mathbf{r}\left( z\right)] | +\dfrac{|\mathbf{r}(z)- r( \sqrt{z_0}  ) | }{  |z-\sqrt{z_0}| }  \right) e^{- v t}.
\end{equation}
From \eqref{r=0} we deduce that
$$r(\sqrt{z_0})={o}(\sqrt{z_0})$$
so the solution to model problem on the deformed contour will take the form
\begin{equation}
\label{model-app}
I+o(\sqrt{z_0})
\end{equation}
Then we proceed as in the previous section and study the integral equation related to the $\dbar$-problem. The estimates will follow from the same arguments in section \ref{sec:outside}. All we need is a uniform estimate of the following integrals near $z_0$ as $z_0\to 0$:
{ \begin{align*}
\int_0^{ \infty } \int_{ v +\sqrt{z_0} }^{\infty} \dfrac{\left| \dbar[\mathbf{r}\left( z\right)]\right| }{ |u^2+v^2|^{1/2} }  e^{- v t } \, du \, dv & \leq \int_0^{ 1 } \int_{ v }^{1} \dfrac{|\dbar[\mathbf{r}\left( z\right)]| }{ |u^2+v^2|^{1/2} }  e^{- v t } \, du \, dv \\
&\quad+\int_0^{ \infty } \int_{ 1 }^{\infty} \dfrac{|\dbar[\mathbf{r}\left( z\right)] | }{  |u^2+v^2|^{1/2}  } e^{- v t } \, du \, dv\\
 &\lesssim t^{1/2-s}+t^{-s}.
\end{align*}}

{\small \begin{align*}
\int_0^{ \infty } \int_{ v +\sqrt{z_0} }^{\infty} \dfrac{|\mathbf{r}(z)- r( \sqrt{z_0}  ) | }{  |(u-\sqrt{z_0})^2+v^2 |^{1/2} |u^2+v^2|^{1/2} }  e^{- v t } \, du \, dv & \leq \int_0^{ \infty } \int_{ v +\sqrt{z_0} }^{\infty} \dfrac{|\mathbf{r}(z)| }{  |(u-\sqrt{z_0})^2+v^2 |^{1/2} |u^2+v^2|^{1/2} }  e^{- v t } \, du \, dv \\
&+\int_0^{ \infty } \int_{ v  +\sqrt{z_0}  }^{\infty} \dfrac{|r( \sqrt{z_0}  ) | }{  |(u-\sqrt{z_0})^2+v^2 |^{1/2} |u^2+v^2|^{1/2} }  e^{- v t } \, du \, dv\\
 &=:\tilde{I}_1+\tilde{I}_2.
\end{align*}}For $1<p<2$ using the same argument as \eqref{est-r-out}
\begin{align*}
\tilde{I}_1 \lesssim t^{-s}.
\end{align*}
For $\tilde{I}_2$, again using H\"older's inequality, for $1<p<2$ and $1/p+1/q=1$
\begin{align*}
\tilde{I}_2 \leq \int_0^{\infty}  e^{-vt}  \left( \int_{ \sqrt{z_0} }^\infty \dfrac{| r(\sqrt{z_0})  |^q}{|u|^{q}}du  \right)^{1/q} \left( \int_{ v}^{\infty} \dfrac{1}{ |v|^{p} |1+(u^2/v^2)|^{p/2}  } du  \right)^{1/p} dv.
\end{align*}
Using the fact that $\lim_{z\to 0} r(z)/z $ is bounded near the origin, a simple calculation gives 
$$ \left( \int_{ \sqrt{z_0} }^\infty \dfrac{| r(\sqrt{z_0})  |^q}{|u|^{q}}du  \right)^{1/q} <\infty$$
as $z_0 \to 0$ and choose $p=1/s$ to conclude that
\begin{align*}
\tilde{I}_2 \lesssim t^{-s}.
\end{align*}
Similarly to Section \ref{sec:outside} we obtain the following asymptotic formulas as $x/t\to 1$ for $0<x/t<1$:
\begin{align}
\label{app-cos}
\cos f(x,t) -1 &= o(z_0) +\mathcal{O}\left(  t^{1-2s}\right)\\
\label{app-sin}
\sin f(x,t)    &=o(\sqrt{z_0}) + \mathcal{O}\left( t^{1/2-s }\right).
\end{align}
\subsection{$x/t \to -1$} Finally we give a brief discussion on the asymptotic formula of \eqref{eq: sG} in the region where $x/t \to -1$ as $t\to \infty$.  We instead work with the spectral problem \eqref{Phi-x} and establish a parallel version of Proposition \ref{prop:r} and  RHP Problem \ref{RHP-1}. We then repeat the nonlinear steepest descent method on this new RHP. To further facilitate computation, we make the following change of variable:
$$z \mapsto \dfrac{1}{z}$$
then we have 
\begin{align*}
\breve{r}(z) &=r\left( \dfrac{1}{z} \right)\\
\breve{\theta}(z; x, t)&= \dfrac{1}{4}\left(  \left( \dfrac{1}{z}-z  \right)\dfrac{x}{t} +\left( z+\dfrac{1}{z}  \right) \right)t\\
\breve{z_0} &= \sqrt{ \dfrac{t+x}{t-x}}.
\end{align*}
It is easy to check that 
\begin{itemize}
\item[I.] $x/t>-1$ implies that 
for $v>0$, 
\begin{equation}
\label{theta-out-}
4\text{Re}i\theta(z; x, t)=-\left(   1 -\dfrac{x}{t} \right)v t +  \left( 1+\dfrac{x}{t} \right)\dfrac{v t}{u^2+v^2};
\end{equation}
\item[II.] $x/t <-1$ implies that 
for $v>0$, 
\begin{equation}
\label{theta-out-}
4\text{Re}i\theta(z; x, t)=-\left(   1 -\dfrac{x}{t} \right)v t +  \left( 1+\dfrac{x}{t} \right)\dfrac{v t}{u^2+v^2}<-\left(   1 -\dfrac{x}{t} \right)v t  ;
\end{equation}
\end{itemize}
For case I above we will follow the same argument in subsection \ref{subsec: app} to obtain 
\begin{align}
\label{app-cos-}
\cos f(x,t) -1 &= o( \breve{z_0}) +\mathcal{O}\left(  t^{1-2s}\right)\\
\label{app-sin-}
\sin f(x,t)    &=o(\sqrt{ \breve{ z_0} }) + \mathcal{O}\left( t^{1/2-s }\right).
\end{align}
Case II is similar to the situation in subsection \ref{subsec: out} so the formulas are 
\begin{align}
\label{app-cos-'}
\cos f(x,t) -1 &= \mathcal{O}\left(  t^{1-2s}\right)\\
\label{app-sin-'}
\sin f(x,t)    &= \mathcal{O}\left( t^{1/2-s}\right).
\end{align}
{ All the implicit constants above only depend on the Sobolev norm of the reflection coefficient  $r\in H^{s}_0(\bbR)$.} 

\begin{remark}
{In order to get a decay rate of higher order than $\frac{1}{\sqrt{t}}$, we have to make further assumptions on smoothness of initial data which result in better weighted estimates for the reflection coefficient $r$. And from there, uniform estimates will allow us to pass $z_0\rightarrow0$ as $x/t\rightarrow 1$ or $\breve{z}_0\rightarrow0$ as $x/t\rightarrow -1$ and match up the corresponding asymptotic formulas with \eqref{outside-cos}-\eqref{outside-sin} respectively. } 
\end{remark}
\begin{proposition}
    \label{prop: r-higher}
If $\vec {f}(0)\in H_{\text{sin}}^{3,s}\left(\mathbb{R}\right)\times H^{2,s}\left(\mathbb{R}\right)$
	with $s>\frac{1}{2}$ 
then
\begin{itemize}
\item[1.] $(\cdot)r(\cdot)\in H^{s}(\bbR)$;
\item[2.] $\lim_{z\to 0}{r(z)}/z=0$.
\end{itemize}
\end{proposition}
\begin{proof}
The proof of this proposition follows from the proof of Proposition \ref{prop:r} applied to the integral representation given by \eqref{uni-cont}.
\end{proof}
We then proceed to make the following estimate:
\begin{align}
    \int_0^{ 1 } \int_{ v }^{1} \dfrac{|\dbar[\mathbf{r}\left( z\right)]| }{ |u^2+v^2|^{1/2} }  e^{- v t } \, du \, dv &\leq \int_0^{ 1 } \int_{ v }^{1} \dfrac{|\dbar[\mathbf{r}\left( z\right)]| }{ u }  e^{- v t } \, du \, dv\\
    \nonumber
    &\leq  \int_0^{ 1 } \int_{ v }^{1} \dfrac{\left\vert \int_\bbR i\xi\left(  e^{i\xi u}\hat{r}(\xi)\hat{\mathcal{P}}(v\xi) + e^{i\xi u}\hat{r}(\xi)\hat{\mathcal{P}}'(v\xi) \right) d\xi\right\vert}{ u }  e^{- v t } \, du \, dv  \\
    \nonumber
    &\lesssim \norm{(\cdot)r(\cdot)}{H^s(\bbR)}t^{-s}
\end{align}
which is a consequence of integration by parts and the standard \textit{Fourier} theory. Then we can replace the decay term $\mathcal{O}(t^{1-2s})$ in \eqref{app-cos}, \eqref{app-cos-} and \eqref{app-cos-'} by $\mathcal{O}(t^{-2s})$ and $\mathcal{O}(t^{1/2-s})$ in \eqref{app-sin}, \eqref{app-sin-} and \eqref{app-sin-'} by $\mathcal{O}(t^{-s})$.

\section*{Acknowledgements}
G.C. was supported by Fields Institute for Research in Mathematical Sciences via Thematic Program on Mathematical Hydrodynamics. 
Part of this work was done when the third author was visiting the department of mathematics, University of Toronto. The third author wants to thank the department for their hospitality and Catherine Sulem for the financial support. We would also like to thank Claudio Mu\~noz, Dmitry Pelinovsky, Wilhelm Schlag for helpful suggestions. We are very grateful to the editor and the anonymous referees whose detailed comments improve the presentation of the paper significantly.

\end{document}